\documentclass[11pt]{article}

\usepackage{float} 
\usepackage{amsmath}
\usepackage{amssymb}
\usepackage{array}
\usepackage{bbm}
\usepackage{makeidx}
\usepackage{hyperref}
\usepackage{cleveref}
\usepackage{caption}
\usepackage{subcaption}
\usepackage{amsthm}
\usepackage{color}
\usepackage{mathtools}
\usepackage{graphicx} 
\usepackage[LGR,T1]{fontenc}
\usepackage{mathrsfs}
\usepackage[noend]{algorithmic}
\usepackage[ruled,vlined]{algorithm2e}
\usepackage{lipsum}
\usepackage{booktabs}
\usepackage{xcolor}
\usepackage[title]{appendix}
\hypersetup{
    colorlinks,
    linkcolor={blue!80!black},
    citecolor={green!50!black},
}
\usepackage{apptools}
\usepackage{tikz}
\usetikzlibrary{arrows, positioning, automata}
\usepackage{comment}
\usepackage[shortlabels]{enumitem}

\usepackage[mathscr]{euscript}
\DeclareSymbolFont{rsfs}{U}{rsfs}{m}{n}
\DeclareSymbolFontAlphabet{\mathscrsfs}{rsfs}

\usepackage{listings}
\usepackage{xcolor}

\usepackage{lipsum}

\newcommand\blfootnote[1]{%
  \begingroup
  \renewcommand\thefootnote{}\footnote{#1}%
  \addtocounter{footnote}{-1}%
  \endgroup
}

\usepackage{float}

\setlist[itemize,1]{label=$\blacktriangleright$}

\newfloat{sufigure}{htbp}{loc}[section] % Define a custom floating environment
\floatname{sufigure}{Supplementary Figure} % Change the caption prefix for this environment

\def\GS{E_\star}
\def\Vsol{{\rm V}}

\def\uq{{\underline{q}}}

\def\cuF{\mathscrsfs{F}}
\def\de{{\rm d}}
\def\Par{{\sf P}}

\def\sub{\mbox{\tiny \rm ub}}
\def\slb{\mbox{\tiny \rm lb}}
\def\sHD{\mbox{\tiny \rm HD}}

\def\sGF{\mbox{\tiny \rm GF}}
\def\sGD{\mbox{\tiny \rm GD}}

\def\sLip{\mbox{\tiny \rm Lip}}

\def\sRS{\mbox{\tiny \rm RS}}

\def\<{\langle}
\def\>{\rangle}

\def\sRS{\mbox{\tiny \rm RS}}
\def\s1RSB{\mbox{\tiny \rm 1RSB}}
\def\sSAT{\mbox{\tiny \rm SAT}}

\def\sens{\mbox{\tiny \rm sens}}
\def\sinit{\mbox{\tiny \rm init}}
\def\srepl{\mbox{\tiny \rm repl}}
\def\sopt{\mbox{\tiny \rm opt}}
\def\spoly{\mbox{\tiny \rm poly}}
\def\dist{{\sf dist}}
\def\sit{\mbox{\tiny \rm it}}

\def\salg{\mbox{\tiny \rm alg}}

\def\prob{{\mathbb P}}

\def\slin{\mbox{\tiny lin}}
\def\proj{{\boldsymbol P}}
\def\cL{{\mathcal L}}
\def\he{\mathfrak{h}}
\def\sep{\mbox{\rm \tiny ep}}

\usepackage{mathtools}

\newtheoremstyle{myremark} % name
    {\topsep}                    % Space above
    {\topsep}                    % Space below
    {\rm}                        % Body font
    {}                           % Indent amount
    {\bf}                        % Theorem head font
    {.}                          % Punctuation after theorem head
    {.5em}                       % Space after theorem head
    {}  % Theorem head spec (can be left empty, meaning normal)

\newtheorem{claim}{Claim}[section]
\newtheorem{lemma}[claim]{Lemma}

\newtheorem{theorem}{Theorem}
\newtheorem{proposition}[claim]{Proposition}
\newtheorem{corollary}[claim]{Corollary}
\newtheorem{definition}[claim]{Definition}

\theoremstyle{myremark}
\newtheorem{remark}{Remark}[section]

\allowdisplaybreaks
\usepackage[top=1in,bottom=1in,left=1in,right=1in]{geometry}
\usepackage[utf8]{inputenc}

\def\ed{\stackrel{{\mathrm d}}{=}}

\def\bA{{\boldsymbol A}}
\def\bB{{\boldsymbol B}}

\def\hbx{\hat{{\boldsymbol x}}}
\def\hby{\hat{{\boldsymbol y}}}
\def\hbz{\hat{{\boldsymbol z}}}

\def\hx{\hat{x}}
\def\hy{\hat{y}}

\def\hR{\hat{R}}

\def\lr{{\rm lr}}
\def\Proj{{\sf P}}

\def\cuC{\mathscrsfs{C}}

\def\cD{{\mathcal D}}
\def\cW{{\mathcal W}}

\def\cS{{\mathcal S}}

\def\hp{\hat{p}}

\def\hnu{\hat{\nu}}

\def\cV{\mathcal{V}}
\def\cF{\mathcal{F}}
\def\cE{\mathcal{E}}
\def\cG{\mathcal{G}}
\def\cS{\mathcal{S}}
\def\cO{\mathcal{O}}

\def\GOE{\mbox{GOE}}

\def\tbx{\tilde{\bx}}

\def\myangle{{\rm angle}}

\def\hbh{\hat{\boldsymbol h}}

\def\diam{{\rm diam}}

\def\Vol{{\rm Vol}}

\def\hbm{\hat{\boldsymbol m}}
\def\obm{\overline{\boldsymbol m}}
\def\obh{\overline{\boldsymbol h}}
\def\bz{{\boldsymbol z}}

\def\S{{\mathbb S}}

\def\Sol{{\sf Sol}}
\def\Ts{{\sf T}}

\def\oh{\overline{h}}
\def\og{\overline{g}}
\def\of{\overline{f}}

\def\cF{\mathcal{F}}

\def\cE{\mathcal{E}}

\def\ualpha{\underline{\alpha}}
\def\oW{\overline{W}}

\def\LL{\mathbb{L}}

\def\TT{\mathbb{T}}

\def\calE{\mathcal{E}}

\def\bA{\mathbf{A}}
\def\bB{\mathbf{B}}

\def\bD{\mathbf{D}}

\def\bF{\mathbf{F}}
\def\bG{\mathbf{G}}
\def\bH{\mathbf{H}}
\def\bI{\mathbf{I}}

\def\bK{\mathbf{K}}
\def\bL{\mathbf{L}}
\def\bM{\mathbf{M}}

\def\bO{\mathbf{O}}
\def\bP{\mathbf{P}}
\def\bQ{\mathbf{Q}}
\def\bR{\mathbf{R}}
\def\bS{\mathbf{S}}

\def\bU{\mathbf{U}}

\def\bW{\mathbf{W}}
\def\bX{\mathbf{X}}
\def\bY{\mathbf{Y}}
\def\bZ{\mathbf{Z}}

\def\ba{\boldsymbol{a}}

\def\be{\boldsymbol{e}}

\def\bg{\boldsymbol{g}}
\def\bh{\boldsymbol{h}}

\def\bm{\boldsymbol{m}}

\def\bu{\boldsymbol{u}}
\def\bv{\boldsymbol{v}}
\def\bw{\boldsymbol{w}}
\def\bx{\boldsymbol{x}}
\def\by{\boldsymbol{y}}
\def\bz{\boldsymbol{z}}

\def\bA{\boldsymbol{A}}
\def\bB{\boldsymbol{B}}

\def\bD{\boldsymbol{D}}

\def\bF{\boldsymbol{F}}
\def\bG{\boldsymbol{G}}
\def\bH{\boldsymbol{H}}
\def\bI{\boldsymbol{I}}

\def\bK{\boldsymbol{K}}
\def\bL{\boldsymbol{L}}
\def\bM{\boldsymbol{M}}

\def\bO{\boldsymbol{O}}
\def\bP{\boldsymbol{P}}
\def\bQ{\boldsymbol{Q}}
\def\bR{\boldsymbol{R}}
\def\bS{\boldsymbol{S}}

\def\bU{\boldsymbol{U}}

\def\bW{\boldsymbol{W}}
\def\bX{\boldsymbol{X}}
\def\bY{\boldsymbol{Y}}
\def\bZ{\boldsymbol{Z}}

\def\normal{{\mathsf{N}}}
\def\bcH{{\boldsymbol {\mathcal H}}}

\def\bDelta{\boldsymbol{\Delta}}

\def\blambda{\boldsymbol{\lambda}}

\def\bxi{\boldsymbol{\xi}}

\def\bzero{\boldsymbol{0}}
\def\bfone{\boldsymbol{1}}

\def\bzeta{\boldsymbol{\zeta}}

\def\bm{\boldsymbol{m}}

\def\GOE{{\sf GOE}}

\def\rr{\overline{r}}

\def\ALG{{\sf ALG}}

\def\complex{{\mathbb C}}
\def\upper{{\mathbb H}}
\def\reals{{\mathbb R}}
\def\naturals{{\mathbb N}}

\def\sT{{\sf T}}

\def\Proj{{\bf P}}

\def\Ball{{\sf B}}

\def\id{{\boldsymbol I}}

\def\ed{\stackrel{{\rm d}}{=}}
\def\op{{\mbox{\rm\tiny op}}}

\def\sB{{\sf B}}
\def\hsB{\hat{\sf B}}
\DeclareMathOperator*{\plim}{p-lim}

\renewcommand{\P}{\mathbb{P}}
\newcommand{\E}{\mathbb{E}}
\newcommand{\R}{\mathbb{R}}

\newcommand{\eps}{\varepsilon}
\newcommand{\Var}{\operatorname{Var}}

\newcommand{\argmin}{\operatorname{argmin}}

\newcommand{\sign}{\operatorname{sign}}

\newcommand{\Unif}{\operatorname{Unif}}

\newcommand{\Lip}{\operatorname{Lip}}

\newcommand{\Tr}{\operatorname{Tr}}

\newcommand{\RN}[1]{%
  \textup{\uppercase\expandafter{\romannumeral#1}}%
}

\newcommand{\RNum}[1]{\uppercase\expandafter{\romannumeral #1\relax}}



\makeatletter
\newcommand*{\rom}[1]{\expandafter\@slowromancap\romannumeral #1@}
\makeatother

\definecolor{lightergray}{rgb}{0.97, 0.97, 0.97}

\title{
 Solving systems of Random Equations\\
 via First and Second-Order Optimization Algorithms}

\author{Andrea Montanari
	\thanks{Department of Statistics and Department of Mathematics, Stanford University} 
	\and 
	Eliran Subag\thanks{Department of Mathematics, Weizmann Institute of Science}
}
\date{}
\begin{document}
\maketitle

\begin{abstract}
We revisit the classical problem of solving $n$ random equations in 
$d$ real variables, when the equations are independent realizations 
of a common rotationally-invariant  Gaussian process in $d$ dimensions.
A special case is the one of random polynomial equations, which has been studied
since the seminal contributions of Littlewood-Offord and Kac in the 1940s (who 
initiated the study of existence of solutions of random polynomials)
and of Shub and Smale in the 1990s.
The last authors first investigated the computational aspect of this problem.
In particular Smale's `17th problem' asks whether  a system of random polynomial equations
can be (approximately) solved in average case polynomial time.

We formulate this as a nonconvex optimization problem, and apply local
algorithms based on gradient or Hessian information. 
We leverage recent advances in spin glass theory to characterize the optimal algorithm
in this class, and show that the latter undergoes a phase transition at a critical
value $\alpha_{\salg}$ of the ratio $\alpha=n/d$. The present paper establishes 
that near-solutions can be found with-high probability for $\alpha<\alpha_{\salg}$,
while a companion paper proves that a broad class of efficient algorithms fail for $\alpha>\alpha_{\salg}$
(we outline the proof of this hardness result here).
We further prove that there are cases (distributions over the
random equations)
such that for $(1+\delta)\alpha_{\salg}<n/d<(1-\delta)\alpha_{\slb}$
(with $(\alpha_{\salg},\alpha_{\slb})$
a non-empty interval and $\delta>0$ arbitrarily small)
solutions exists with high probability but are not found efficiently by 
a broad class of local search algorithms.

We compare our theoretical predictions with numerical simulations
using the optimal algorithm we propose as well as 
stochastic gradient descent, and show that they are accurate for a related albeit non-Gaussian cost function.
We finally observe empirically a sensitivity phase
transition (or cross-over) in the behavior of optimization
algorithms, below $\alpha_{\salg}$. This marks a qualitative 
departure with respect to standard optimization theories.
\end{abstract}

\tableofcontents

\section{Introduction}
\label{sec:Introduction}

Given\blfootnote{The present manuscript subsumes the earlier technical report 
\cite{montanari2023solving} by the same authors.} a cost function $H:\reals^d\to\reals$ (which we will assume throughout to be smooth),
first-order optimization algorithms produce a sequence of 
candidates $\bx^1,\dots,\bx^k\in\reals^d$. 
The new candidate $\bx^{k+1}$ is a function uniquely of the 
first-order information at previous points,
namely a function of  $\{\bx^i, H(\bx^i)$, $\nabla H(\bx^i)$: $i\le k\}$.  Second-order methods 
additionally use the Hessians 
$\{\nabla^2 H(\bx^i):\; i\le k\}$. Important examples of these two classes of algorithms 
are, respectively, gradient descent and the Newton method. 
With a slight abuse of terminology, we will refer to stochastic gradient descent (SGD) also
as to a first-order method.

Because of their simplicity, these algorithms are the 
method of choice to solve 
large optimization problems. A  prominent application is training 
large machine learning (ML) models. The training problem is normally formulated
as a non-convex optimization problem whereby $\bx$ is a vector 
comprising the neural network weights (a.k.a. `model parameters').
Focusing, to be definite, on the case of square loss, the cost function
optimized in ML takes the form 
\begin{align}
H(\bx) = \frac{1}{2} \sum_{i=1}^n F_i(\bx)^2\, ,\label{eq:FirstHamiltonian}
\end{align}
where each term $F_i$ is the error incurred in
fitting the $i$-th data point (seen as a function of the 
model parameters $\bx$). 
Approximate minimizers of problems of this type are sought
 via SGD and its variants \cite{bottou2018optimization}. 

 In  
ML theory, the functions $F_i$ are assumed to be i.i.d. copies of a 
certain stochastic process $F_1:\reals^d\to\reals_{\ge 0}$ \cite{shalev2014understanding}.
For instance, in a supervised learning problem,  
we are given i.i.d. data $(y_i,\ba_i)$, with 
$y_i\in\reals$ a response and $\ba_i\in\reals^D$ a features vector.
We want to learn a ML model $f(\,\cdot\,;\bx):\reals^D\to\reals$
(parametrized by $\bx$),
that, given a new feature vector $\ba_{n+1}$, predicts the response
$y_{n+1}$ via $f(\ba_{n+1};\bx)$.
In this case, we might define  $F_i(\bx) = y_i-f(\ba_i;\bx)$,
and hence learn the parameters $\bx$ by minimizing 
$H(\bx)$ of Eq.~\eqref{eq:FirstHamiltonian}. 
While we will focus on the case 
of square loss, we believe our methods can be generalized to other losses.
We will sometimes refer to $H(\bx)$ as to the `energy.'

Despite their applicability, first and second-order 
methods are well understood only for convex cost functions $H$ \cite{nesterov2003introductory}. When
the number of parameters $d$ is large and $H$ is non-convex, we 
cannot predict whether
SGD will converge to a near global optimum in polynomial time. 

Substantial effort has been devoted to the mathematical study of
SGD dynamics in non-convex landscapes, and we will briefly review some
literature in Section \ref{sec:Related}. Broadly speaking, existing theory is limited
either to the `highly underparametrized' regime $n\gg d$, or makes crude linearity approximations
that are only accurate in the `highly overparametrized' regime $n\ll d$.

In contrast, our focus will be on the proportional asymptotics whereby 
\begin{align}
n,d\to\infty,\;\;\; \frac{n}{d}\to \alpha\in (0,\infty)\, .
\end{align}
Big-Oh notation refers to this limit, unless we write the
asymptotic variable in subscript.
 `With high probability' (whp) means `with probability converging to one
as $n,d\to\infty$.' 

We will consider settings in
which the $n$ equations  $F_i(\bx) = 0$, $i\le n$ admit a common solution,
and hence the global minimum is $\min_{\bx}H(\bx)=0$. 
When this happens, finding a near minimum of $H(\bx)$ is equivalent to approximately
solving the $n$ nonlinear equations $F_i(\bx)=0$, $i\le n$. 

For mathematical convenience, we will assume the normalization $\|\bx\|=1$
and hence (denoting by $\S^{d-1}$ the unit sphere) define the set of near-global optima as 
\begin{align}\label{eq:NearSolutions}
\Sol_{n,d}(\eps):= \Big\{\bx\in \S^{d-1}:\,\; \|\bF(\bx)\|_2^2\le n \eps\Big\}\, ,
\end{align}
where $\bF(\bx) = (F_1(\bx),\dots,F_n(\bx))^{\sT}$. We will use the term
`near solutions' to refer to elements of $\Sol_{n,d}(\eps_n)$ for some $\eps_n\to 0$ slowly.

We will consider two distributions for the processes $F_i$, defined as follows.

\begin{figure}
\begin{center}
\includegraphics[width=0.75\linewidth]{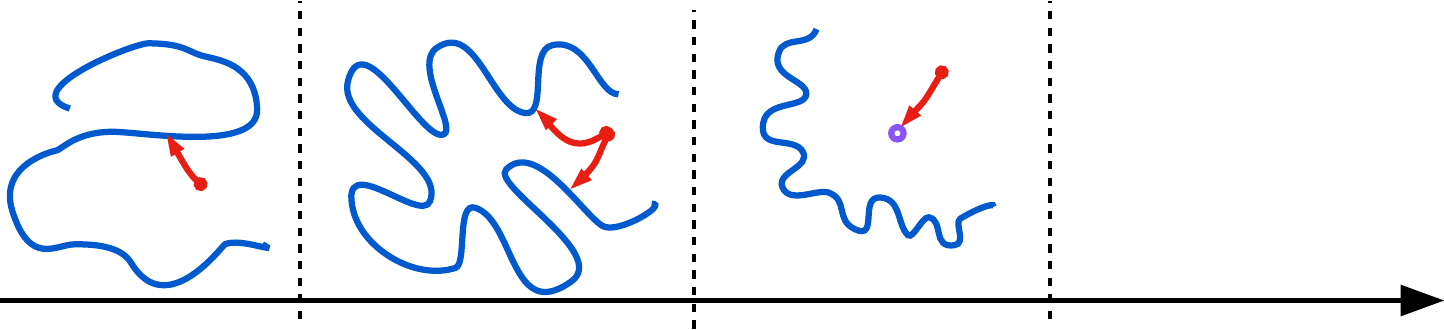}
\put(-290,-8){$\alpha_{\sens}$}
\put(-195,-8){$\alpha_{\salg}$}
\put(-100,-8){$\alpha_{\sSAT}$}
\put(-10,-8){$\alpha = n/d$}
\end{center}
\vspace{-0.2cm}
\caption{Cartoon phase diagram for random equations models ($n$ equations in $d$ unknowns).
Blue: manifold of solutions; red: local optimization trajectories. See text for description.}
		\label{fig:phase_diag}
\end{figure}

\paragraph{Average of ridges (AoR).} For $\phi:\reals\to\reals$ a fixed non-linear function,
and 
$\ba_i = (\ba_i^{(1)},\dots,\ba^{(m)}_i)\sim\normal(0,\id_D)$, $D=md$,
$\ba^{(k)}_i \in\reals^d$, we let (for $\bx\in \S^{d-1}$)
\begin{align}\label{eq:AoR_Fi}
F_i(\bx) =  \frac{1}{\sqrt{m}}\sum_{k=1}^m\phi(\<\bx,\ba^{(k)}_i\>)\, .
\end{align}
In words, each component $\phi(\<\bx,\ba^{(k)}_i\>)$ is a ridge function 
(a function of a one-dimensional projection of the variables $\bx$) along
a random direction $\ba^{(k)}_i$, and we are averaging $m$
independent copies of such a function\footnote{We emphasize that the AoR model is
not motivated here by a specific application, but rather as a natural
example of random cost function.}. 
The functions \eqref{eq:AoR_Fi} are extended to $\bx\in \reals^d$ as described 
in Appendix \ref{sec:Appendix_Numerical_Setup}. 
%\vspace{0.1cm}

\paragraph{Gaussian equations (GE) model.}  In this case the $F_i$ are i.i.d. Gaussian processes with 
zero mean and covariance:
\begin{equation}\label{eq:Fcov}
	\E\big[F_i(\bx^1)F_i(\bx^2)\big] = \xi\big(\langle \bx^1,\bx^2\rangle\big)\, .
\end{equation} 
This specific covariance structure is equivalent to the assumption the $F_i$
is invariant (in distribution) under rotations.
The fact that the covariance is positive semidefinite  implies that\footnote{We will 
assume $\xi(1+\eps)<\infty$ for some $\eps>0$
which implies that the random
 function $\bF:\bx\mapsto \bF(\bx)$ is a.s. smooth.} 
$\xi(t)=\sum_{k\ge 0}\xi_kt^k$, with $\xi_k\ge 0$ for all $k$ \cite{Schoenberg1942}.
An explicit construction of the processes $\{F_i\}$ is given by 
\begin{align}
 F_i(\bx)&:=\sum_{k\ge 0}\sqrt{\xi_k}\sum_{j_1,\dots,j_k=1}^d 
 G_{i,j_1\dots j_k}^{(k)}x_{j_1}\cdots x_{j_k}\, ,\label{eq:Representation}
 \end{align}
 where the random coefficients
 $(\bG^{(k)})_{k\ge0}:=(G_{i,j_1\dots j_k}^{(k)})_{k\ge 0, i\le n, j_1,\dots ,j_k\le d}
 \sim_{i.i.d. } \normal(0,1)$ are i.i.d. 
 Gaussian random variables. 

Our theoretical results will be stated and proved in the GE model.
We will also compare predictions from the GE model to
numerical simulations in the AoR model,
in Section \ref{sec:Numerical_First}, \ref{sec:Numerical_Second}
(while more challenging from an analytical viewpoint, the AoR model is
simpler to simulate). The two
models are related by matching the covariance of the processes $F_i$. Namely, letting 
$\phi(s) = \sum_{k\ge 0} \phi_k \he_k(s)$ be the expansion of $\phi$ into
the basis of orthonormal Hermite polynomials in $L^2(\reals,\gamma)$ (here $\gamma$
is the standard Gaussian measure), the process \eqref{eq:AoR_Fi} has covariance of the form
\eqref{eq:Fcov}, where
\begin{align}
\xi(s) = \sum_{k\ge 0} \xi_ks^k\, ,\;\;\;\;\; \xi_k = \phi_k^2\, .
\end{align}
We find good agreement (for $m$ sufficiently large), 
pointing at a remarkable universality phenomenon.

\vspace{0.2cm}

The GE model is connected to several rich lines of inquiry within
mathematics (zeros of random polynomials); statistical physics (mean field spin glasses);
machine learning (loss landscape in deep learning);
random constraint satisfaction problems (solving the system of equations 
$\bF(\bx)=\bzero$ can be viewed as a continuous constraint 
satisfaction problem).
If $\xi(t)$ is a polynomial of degree $p$, then the $F_i$'s are polynomials
of degree $p$ with Gaussian coefficients, see Eq.~\eqref{eq:Representation}. 
The zeros of such polynomials
have been studied at least since Littlewood-Offord
\cite{littlewood1939number} and Kac \cite{Kac1943}, see, e.g., \cite{Azais2009}
for a recent overview. In particular, Kac first developed the celebrated
Kac-Rice formula to compute the expected number of real solutions in one dimension.
Much less is known about the existence and computational complexity of finding real solutions 
of systems of random polynomial equations. We refer again to Section \ref{sec:Related} for 
further related work.

The analogy with discrete constraint satisfaction problems 
\cite{MezardMontanari} suggests
that the GE model should undergo
two  phase transitions in the limit $n,d\to\infty$.
(All statements below are understood to hold whp.
 We refer Fig.~\ref{fig:phase_diag} for a cartoon.)
\begin{description}
\item[$\alpha<\alpha^{\spoly}_{\salg}$.] The solution set $\Sol_{n,d}(0)$ is
 non-empty and, for any $\eps>0$,  there exists a polynomial-time algorithm
 that converges to points in $\Sol_{n,d}(\eps)$ with high probability. 
\item[$\alpha^{\spoly}_{\salg}<\alpha<\alpha_{\sSAT}$.]
 $\Sol_{n,d}(0)$ is non-empty but no polynomial time
algorithm
converges to it.
\item[$\alpha_{\sSAT}<\alpha$.] Solutions don't exist and 
$\min_{\bx\in\S^{d-1}}H(\bx)/n = u_{\sopt}(\alpha,\xi) +o_n(1)$ for a strictly positive 
$u_{\sopt}(\alpha,\xi)$.
\end{description}
The next subsection outlines our contributions towards establishing 
this conjectural picture.

\subsection{Summary of contributions}

This paper establishes rigorously several elements of the above picture 
and additionally presents numerical and heuristic evidence for certain
other phenomena that we deem worthy of future investigation. 

\paragraph{Rigorous results.}
We prove the following results about the algorithmic and satisfiability 
phase transitions in the GE model:
\begin{itemize}
\item \emph{Characterization of the algorithmic threshold.} 
We introduce a first order algorithm that finds near-solutions 
with high probability for $\alpha<\alpha_{\salg}(\xi)$, with 
$\alpha_{\salg}(\xi)$ an explicit threshold 
(see Sec.~\ref{sec:AchievableEnergy} and in particular
Eqs.~\eqref{eq:uRS_def} to \eqref{eq:AlphaAlgDef} below). 
For $\alpha>\alpha_{\salg}(\xi)$, we show that any algorithm whose output 
is Lipschitz continuous 
in the input data (more precisely, in the coefficients of the polynomials $F_i$), 
with Lipschitz constant bounded uniformly in $d$,  fails to find a
near-solution. Indeed, we will give an energy threshold $u_{\salg}(\xi,\alpha)>0$
such that any Lipschitz algorithm returns 
$\bx^{\salg}$ such that, with high-probability,
$\|\bF(\bx^{\salg})\|_2^2/2 > n (u_{\salg}(\alpha,\xi) -o_n(1))$,
for $\alpha>\alpha_{\salg}(\xi)$.

The present paper focuses on establishing
the achievability of energy $u_{\salg}(\xi,\alpha)$. We outline the proof of the matching
lower bound, deferring a full proof to a companion paper.
\item\emph{Upper and lower bounds on the  satisfiability phase transition.} 
We will give upper and lower bounds on the satisfiability threshold 
$\alpha_{\sSAT}(\xi)$,
denoted by $\alpha_{\slb}(\xi)$ and $\alpha_{\sub}(\xi)$, respectively. 
More precisely, for $\alpha<\alpha_{\slb}(\xi)$ we will show that $\Sol_{n,d}(0)$ is non-empty 
whp, and for $\alpha>\alpha_{\sub}(\xi)$ there exists $\eps>0$ such that $\Sol_{n,d}(\eps)$ is empty.
(We emphasize that we do not prove that an $d$-independent sharp threshold 
$\alpha_{\sSAT}(\xi)$ exists, but we conjecture it does and will keep using this terminology
informally.)
While these upper and lower bounds do not match in general,
 they are strong enough to imply (for certain choices of $\xi$) the existence 
of a non-empty interval 
$\alpha_{\salg}(\xi)<\alpha<\alpha_{\slb}(\xi)$ such that the set
of solutions $\Sol_{n,d}(0)$ is non-empty (and, in fact, 
is a manifold of dimension $(d-n-1)\to\infty$) but no Lipschitz algorithm finds even near-solutions.
\end{itemize}

While we expect that $\alpha_{\salg}(\xi)$ coincides with 
the threshold for general polynomial-time algorithms
 (denoted in the previous section by $\alpha^{\spoly}_{\salg}(\xi)$),
 proving it is an open problem.

 We will sometimes refer to the first order algorithm mentioned
 at the first point also as the `two-phase algorithm' or
 `Hessian descent,' because it can be alternatively implemented using
 the Hessian of the energy function.

\paragraph{Nonrigorous results.}
We compare the predictions of our theory (which is rigorously established within the GE model) 
with numerical simulations in the AoR model with matching covariance structure.
In numerical experiments we implement both SGD, 
and our optimal first order algorithm.
We make the following observations:
\begin{itemize}
\item \emph{Agreement between theoretical predictions and 
numerical simulations.} 
We find good agreement in the following aspects: 
$(i)$~The behavior of the Hessian descent algorithm agrees 
with theoretical predictions, both in terms
of the achieved energy and the spectrum of the Hessian, 
along the algorithm's trajectory;
$(ii)$~The minimum energy $H(\bx)$ achieved by SGD, 
and the spectrum of the Hessian along the 
SGD trajectory agree with theoretical predictions for the optimal algorithm.

These findings provide evidence for two phenomena.
First, a remarkable degree of \emph{universality}, since
the GE and AoR models behave very similarly and, second,
 the fact that \emph{SGD is nearly optimal for these models.}
We believe that both of these phenomena are important directions 
for future mathematical work.
\item \emph{Sensitivity phase transition.} 
We observe a qualitative change in the behavior of optimization
algorithms around a sensitivity threshold $\alpha_{\sens}(\xi)
<\alpha_{\salg}(\xi)$.
Denoting by $\bx^{\salg}_1, \bx^{\salg}_2$ the outputs of a specific algorithm for
two independent runs (on the same cost function $H$, but 
under a random perturbation of the algorithm), we observe the following behavior:
for $\alpha$ below $\alpha_{\sens}(\xi)$, 
$\|\bx^{\salg}_1 - \bx^{\salg}_2\|\ll 1$, i.e.
 two independent runs return essentially the same solution;
 for $\alpha$ above $\alpha_{\sens}(\xi)$,
$\|\bx^{\salg}_1 - \bx^{\salg}_2\|$ is of order $1$:
two runs yield substantially different solutions.

For a version 
of our optimal first order algorithm,  we can compute a theoretical prediction
for $\alpha_{\sens}(\xi)$. 
For the case of SGD, we observe an increase in  sensitivity
around the same threshold but a theoretical characterization 
of the latter is an open problem.
(Further, for SGD this appears as a
 gradual transition rather than a sharp threshold.)
 \end{itemize}

\begin{figure}
\hspace{-0.5cm}
\includegraphics[width=0.52\linewidth]{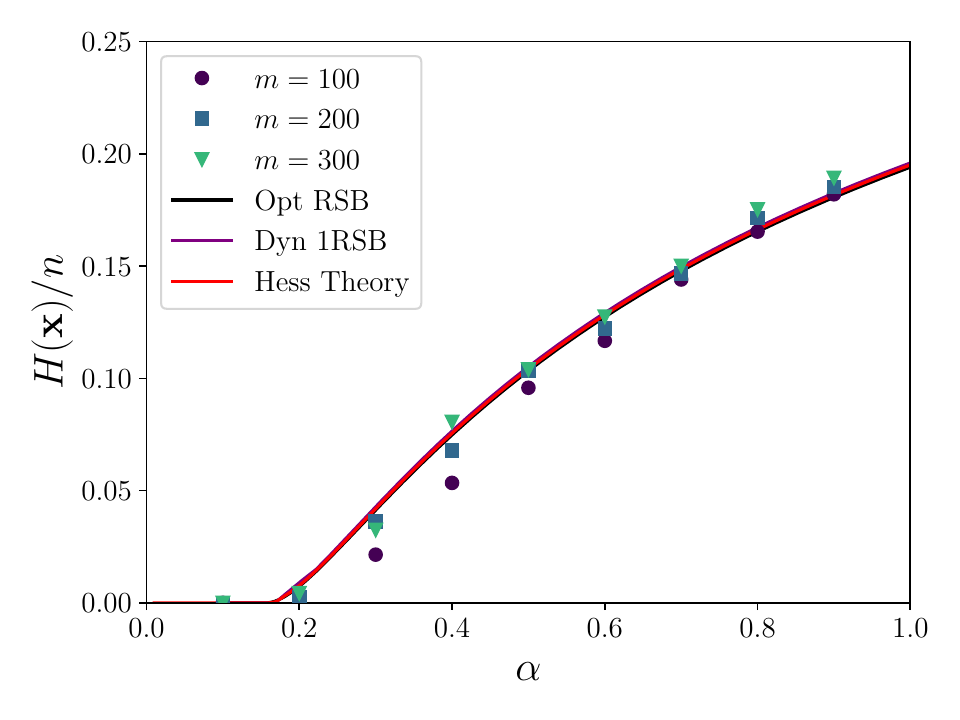}
\hspace{-0.5cm}
\includegraphics[width=0.52\linewidth]{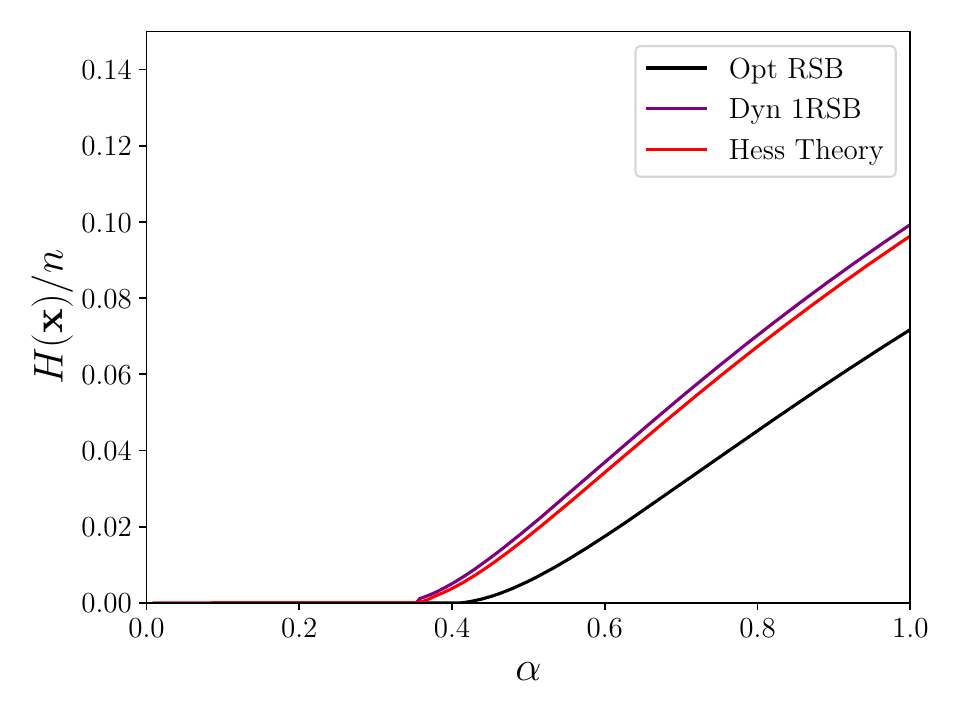}
\caption{Black lines: Replica theory predictions for the asymptotics of minimum cost  in the GE model 
\cite{urbani2023continuous}. Red lines: Algorithmic threshold energy (this paper). Purple lines:
1RSB predictions for the threshold energy. Left:
$\xi(t) =1+a^2 t^3$, $a=0.25$. Right: $\xi(t) = 1+a^2 t^{10}$,
$a=0.33$.
On the left, we also plot the energy achieved by SGD in the corresponding AoR model.
Here the learning rate is $\lr = 0.05$, number of iterations $N_{\sit}=8,000$, $d=100$ and 
vary $m$ of Eq.~\eqref{eq:AoR_Fi}.}
		\label{fig:a025_third_change_m}
\end{figure}

 Figure \ref{fig:a025_third_change_m} presents 
a preview of our results for covariance polynomials 
$\xi_1(t) = 1+a^2 t^3$, $a=0.25$
(left frame) and  $\xi_2(t) = 1+a^2 t^{10}$, $a=0.33$ (right frame).
In these figures, we plot several theoretical predictions: 
$(i)$~The (non-rigorous) prediction
for the asymptotics of the optimum 
$u_{\sopt}(\alpha,\xi) = \lim_{n,d\to\infty}\min_{\bx\in\S^{d-1}}H(\bx)/n$ 
from replica symmetry breaking (RSB) \cite{urbani2023continuous};
$(ii)$~The (non-rigorous) one-step replica symmetry breaking (1RSB) prediction for the so called
`dynamical threshold' energy 
(see, e.g., \cite{MezardMontanari} for general background); 
$(iii)$~Our (rigorous) prediction for the energy achieved by the 
optimal first-order algorithm 
$u_{\salg}(\alpha,\xi) = \lim_{n,d\to\infty} H(\bx^{\salg})/n$. 
Evaluating the explicit formulas in 
Section \ref{sec:AchievableEnergy}
(for $\alpha_{\salg}(\xi)$) and the 
formulas of \cite{urbani2023continuous}
(for $\alpha_{\sSAT}(\xi)$),
we obtain 
$\alpha_{\salg}(\xi_1)\approx 0.16$, 
$\alpha_{\sSAT}(\xi_1)\approx 0.16$
(left frame of Fig.~\ref{fig:a025_third_change_m}), 
$\alpha_{\salg}(\xi_2)\approx 0.35$.
$\alpha_{\sSAT}(\xi_2)\approx 0.41$
(right frame).
 \vspace{0.2cm}

 \subsection{Organization of the paper}

 We summarize related literature in Section 
 \ref{sec:Related}. We present our rigorous results in 
 Sections \ref{sec:Rigorous-SAT} (satisfiability phase transition) 
 and \ref{sec:Rigorous-ALG} (algorithmic phase transition).
We present the results of our numerical simulations with the
AoR model in Section \ref{sec:Numerical_First}, 
while we discuss
the sensitivity phase transition in 
Section \ref{sec:Numerical_Second}. We present the analysis of 
our algorithm in Sections \ref{sec:Algo_HD} 
(which focuses on the case $\xi'(0)=0$)
and  Section \ref{sec:TwoPhase}
(which extends the analysis to the case $\xi'(0)>0$). Finally, we draw some conclusions 
and outline open problems in Section \ref{sec:Conclusion}.

 \subsection{Definitions and notations}
 \label{app:Definitions}

 Given a symmetric matrix $\bA\in\reals^{n\times n}$, we denote by 
 $\lambda_1(\bA)\ge \lambda_2(\bA)\ge \cdots\ge
 \lambda_n(\bA)$ its eigenvalues in decreasing order. For a general matrix 
 $\bM\in\reals^{m\times n}$, $m\le n$, $\sigma_1(\bM)\ge \sigma_2(\bM)\ge \cdots\ge
 \sigma_m(\bM)\ge 0$ denote its singular values and $\|\bM\|_{\op}$ its operator norm.
 
 We let $\Ball^d(r):=\{\bx\in \R^d:\, \|\bx\|\le r\}$ denote the  ball of radius $r$ in $\R^d$,
 with $\Ball^d:=\Ball^d(1)$. 
 %We also write $\Ball^d(r_1,r_2):= \Ball^d(r_2)\setminus \Ball^d(r_1)$.
 
 Throughout, we will write $\bW\sim \GOE(N)$ if $\bW=\bW^{\sT}$ and $(W_{ij})_{i\le j\le N}$
 are independent with $W_{ii}\sim\normal(0,2)$,  $W_{ij}\sim\normal(0,1)$ for $i<j$.
 We write   $\bZ\sim \GOE(M,N)$ if $(Z_{ij})_{i\le M,j\le N}$
 are independent with $Z_{ij}\sim\normal(0,1)$.

%
%*****************************************************
%
\section{Further related work}
\label{sec:Related}
Our work is motivated and related to multiple lines of 
research investigating the behavior of optimization algorithms 
in high-dimensional random landscapes.

\paragraph{Optimization in machine learning.}
Two generic theoretical approaches have been developed to 
prove convergence of gradient based methods to fit large-neural
networks: neural tangent theory and mean field theory.

Roughly speaking, \emph{neural tangent theory} 
	\cite{jacot2018neural,du2018gradient,chizat2019lazy}
	observes that, for certain non-convex cost functions
	$H$, when $d$ is sufficiently larger than $n$, 
	the set of global optima is `near space-filling' in the sense that
	a random initialization is close to some global minimum with high probability.
	Gradient descent initialized at such a point converges in a 
	neighborhood of the unique closest global minimum.  
	We will show in Appendix \ref{sec:Algo_GD} that it 
	fails to capture all regimes in which gradient descent succeeds.
	
\emph{Mean field theory} 
\cite{mei2018mean,chizat2018global,agoritsas2018out,celentano2021high}
leverages mathematical tools developed for the study of particle systems,  
to derive an exact asymptotic characterization of the SGD
dynamics, in the limits of large  dimension and large sample size. 
However, capturing the effects of overparametrization by these 
techniques, while in principle possible
\cite{agoritsas2018out,celentano2021high},
is often too complex for analytic treatment.

\paragraph{Zeros of random polynomials.}
If $\xi(t)$ is a polynomial of degree $p$
then the $F_i$'s are polynomials
of degree $p$ with Gaussian coefficients, see Eq.~\eqref{eq:Representation}.
The zeros of such polynomials
have been studied at least since Littlewood-Offord
\cite{littlewood1939number} and Kac \cite{Kac1943}, see \cite{Azais2009}
for a recent overview.

When $n=d-1$, and the polynomials are homogeneous,
the number of \emph{complex solutions} is fixed by the classical 
Bezout theorem. Shub and Smale \cite{ShubSmaleI} explored the average-case
complexity of algorithms to find such complex solutions,
a question that was formalized as `Smale's 17th problem' \cite{smale1998mathematical,BeltranPardo1,BeltranPardo2,BurgisserCucker,Lairez1}. 
Much less is known about
the  computational complexity of finding \emph{real solutions} of systems of
random polynomial equations.
In sharp contrast with the setting of Smale's 17th problem
we will provide evidence that for certain natural 
distributions over the polynomials, solutions are abundant but
hard to find for a broad family of algorithms. 

Throughout this paper we will focus on the notion of near-solutions that 
is encapsulated by Eq.~\eqref{eq:NearSolutions} 
(with $\eps$ small or slowly vanishing with $n$).
The question of finding approximate solution in a stronger 
sense is investigated (in a more restricted setting) in 
\cite{montanari2024smale} which develops
some of the  ideas presented here.

\paragraph{Spin glass theory.}
Mean field spin glasses are random cost functions in high-dimensional space, whose covariance is often 
of the form \eqref{eq:Fcov}. Their analysis led to the development of a number of new
techniques in theoretical physics and probability theory \cite{mezard1987spin,talagrand2010mean,Pan}.
In particular, the
 GE model is closely related to the mixed $p$-spin model from spin glass theory,
and indeed, each coordinate $F_i$ is an independent copy of the so-called 
`mixed $p$-spin Hamiltonian' \cite{crisanti1992sphericalp,crisanti1995TAP,subag2017geometry}.

Note that --in the GE model-- the energy function \eqref{eq:FirstHamiltonian} is the sum of of $n$
terms, each being the square of a Gaussian process.
This Hamiltonian
was recently studied in physics by Fyodorov \cite{fyodorov2019spin,fyodorov2022optimization} and by 
Urbani \cite{urbani2023continuous,urbani2023dynamical} who proposed it as 
model for confluent tissues. In particular, 
\cite{urbani2023continuous} derives the phase diagram of this model using the 
non-rigorous replica method  and  \cite{urbani2023dynamical} derives DMFT equations
that characterize the high-dimensional asymptotics of gradient flow
in the landscape $H(\,\cdot\,)$. Our focus and approach 
will be different from the one of these works,
but we will report for comparison some of the predictions of \cite{urbani2023continuous}.

\paragraph{Random constrain satisfaction problems.}
The GE model can be seen as a random constraint satisfaction problem 
with continuous variables. Over the last two decades, spin glass
techniques were successful in characterizing random 
discrete constraint satisfaction problems such as graph-coloring
or $k$-satisfiability 
\cite{mezard2002analytic,achlioptas2008algorithmic,MezardMontanari,ding2022proof}.
The phase diagram we derive for the GE model bears strong resemblance 
with the one discovered in the discrete setting.

\paragraph{Loss landscape in deep learning.} While Gaussian models might appear 
crude analogs for the empirical risk landscape of machine learning problems,
several influential works used them to derive important insights into the latter 
\cite{dauphin2014identifying,choromanska2015loss}. In particular, 
\cite{choromanska2015loss} argued that the loss landscape of a multi-layer
perceptron is analogous to the one of a pure $p$-spin spherical spin glass.
Namely, $H$ is approximated by a Gaussian process with covariance
$\E[H(\bx_1)H(\bx_2)] = a_{n}+b_n\<\bx_1,\bx_2\>^p$. 
Motivated by this connection they demonstrated that SGD optimization typically converges to 
local minima or saddles with small index.
We notice that the Gaussian model
of \cite{choromanska2015loss} cannot capture the overparametrized regime which is our
main focus, since in that model 
$\prob(\min_{\bx\in\S^{d-1}}H(\bx) = 0) = 0$. However, for large $n/d$ (and $\xi(t)=\xi_0+\xi_p t^p$)
the model we study can be shown to be well approximated by the one of \cite{choromanska2015loss}.

Our characterization of the algorithmic phase transition
is directly related to the 
asymptotic spectral structure of the Hessian of the energy 
$\nabla ^2H(\bx)$.
A substantial literature has investigated (either empirically or heuristically)
the relation between spectrum of the Hessian and optimization dynamics in deep learning
\cite{dauphin2014identifying,choromanska2015loss,sagun2017empirical,pennington2017geometry,ghorbani2019investigation,papyan2018full,papyan2020traces}. 
The GE model allows to make this connection rigorous, and recover several facts
observed empirically in more complex models.

%
%*****************************************************
%
\section{Rigorous results (I): Satisfiability phase transition}
\label{sec:Rigorous-SAT}

In \cite{SubagPolynomialSystems}, one of the authors 
studied the GE model with independent  but not identically distributed functions
$F_i$:
 $\E\big[F_i(\bx^1)F_i(\bx^2)\big] = \xi_i\big(\langle \bx^1,\bx^2\rangle\big)$, 
such that $\xi_i(0)=\xi'_i(0)=0$ for all $i\le n$.  In that case, \cite{SubagPolynomialSystems}
proved that solutions exist with high probability for all $n\le d-1$ and
that for $n=d-1$ the number of solutions is, with high probability,
$\prod_{i\le n}(\xi_i'(1)/\xi_i(1))^{1/2}\cdot(1+o_d(1))$ (see also \cite{Wschebor} for the iid case with $\xi_i(t)=t^p$). For homogeneous polynomials 
$\xi_i(t) = t^{p_i}$, we thus get $\prod_{i\le n}p_i^{1/2}\cdot(1+o_d(1))$
real solutions, while the number of complex 
solutions is $\prod_{i\le n}p_i$ by the classical Bezout theorem \cite{bezout1779}.

Throughout this paper, we will be interested uniquely in real solutions. 
Our first observation is that the behavior of the solution set $\Sol_{n,d}(0)$
changes qualitatively when  $\xi(0)>0$ (i.e. the functions $F_i$ have a degree-zero
term). In this case,  solutions may not exist on $\S^{d-1}$ whp, even if the number of 
equations per variable is  such that $n/d\to \alpha$ for some $\alpha<1$. The next two 
propositions establish upper and lower bounds on the satisfiability threshold.
Proofs from this Section are 
deferred to Appendix \ref{sec:UpperLowerBounds}.
\begin{proposition}\label{thm:SecMoment}
	Define $\Psi(\,\cdot\, ;\alpha,\xi): [0,1]\to \reals$ via
	\begin{align}
	\Psi(r;\alpha,\xi):= \frac{1}{2}\log(1-r^2) -\frac{\alpha}{2}\log
	\left(1-\Big(\frac{\xi(r)-\xi(0)}{\xi(1)-\xi(0)}\Big)^2\right)
	-\frac{\alpha \xi(0)}{\xi(1)+\xi(r)-2\xi(0)}+\frac{\alpha \xi(0)}{\xi(1)-\xi(0)}\, .
	\label{eq:PsiDef}
	\end{align}
	Further define 
	\begin{itemize}
	\item For $\xi'(0) = 0$:
	\begin{align}
	\alpha_{\slb}(\xi)&:= \inf\Big\{\alpha\ge 0: \;\; \sup_{r\in [0,1]} \Psi(r;\alpha,\xi)>0 
	\, \Big\}\, ,
	\end{align}
	\item For $\xi'(0) > 0$:
	\begin{align}
	\alpha_{\slb}(\xi)&:=
	\frac{\alpha_{\slb}(\xi_{\#})}{1+\alpha_{\slb}(\xi_{\#})}\,, \;\;\;\;\;
	\xi_{\#}(t) := \xi(t)-\xi'(0)t\, .
	\end{align}
	\end{itemize}
	If $\alpha< \alpha_{\slb}(\xi)$ and $\xi''(0)=0$ (respectively $\xi''(0)>0$) 
	then $\Sol_{n,d}(0)\neq \emptyset$
	with probability converging to one (respectively, probablity bounded away from zero)
	as $n,d\to\infty$ with $n/d\to\alpha$.
\end{proposition}
We note that $\Psi(r;\alpha,\xi) = ((1-\alpha)/2)\log(1-r)+O(1)$ as 
$r\uparrow 1$, whence it follows that $\alpha_{\slb}(\xi)\le 1$.

In order to state the upper bound, we define:
\begin{equation}
	\GS(\xi):=  \lim_{d\to\infty}\frac{1}{\sqrt d}\E  \max_{\bx\in\S^{d-1}} F_1(\bx)\,.
\end{equation}
The limit is known to exist and is given by the Parisi formula \cite{auffinger2017parisi,JT},
which we recall in Appendix \ref{sec:Existence_UB}.
\begin{proposition}\label{propo:SimpleUB}
	Assume that $\xi(0)>0$ and define $\xi_{>0}(q) := \xi(q)-\xi(0)$ and 
	\begin{align}
		\alpha_{\sub}(\xi) := \alpha_{\sub}^{(1)}(\xi)\wedge
		\alpha_{\sub}^{(2)}(\xi)\, ,
		\;\;\;
	\alpha_{\sub}^{(1)}(\xi):=\frac{\GS(\xi_{>0})^2}{\xi(0)}\, ,\;\;
	\alpha_{\sub}^{(2)}(\xi):= \frac{\xi'(1)}{\xi(1)}\, .
	\end{align} 
	If 	$\alpha> \alpha_{\sub}$ then there exists $\eps>0$ such that,
	for $n/d\to\alpha$, $\Sol_{n,d}(\eps) = \emptyset$ with probability converging to one.

	In fact, $\Sol_{n,d}(\eps) = \emptyset$  
	with probability converging to one
	for any $\eps< 2\, u_{0}(\xi,\alpha)$, where
	\begin{align}
	u_0(\alpha,\xi)&:= u_0^{(1)}(\alpha,\xi) \vee u_0^{(2)}(\alpha,\xi)\, ,\\
	u_0^{(1)}(\alpha,\xi)&:=\frac{1}{2}
	\xi(0)
	\left(1-\sqrt{\frac{\alpha_{\sub}^{(1)}}{\alpha}}\right)_+^2\, , \;\;\;\;
	u_0^{(2)}(\alpha,\xi):=\frac{1}{2}\left(\sqrt{\xi(1)}-
	\sqrt{\frac{\xi'(1)}{\alpha}}\right)_+^2\, .\label{eq:U0_LB_def}
	\end{align}
\end{proposition}
We note that an equivalent formulation of the last proposition 
is 
\begin{align}
\lim_{n,d\to\infty}\P\left(\frac1n
\min_{\bx \in\S^{d-1}}H(\bx) \ge u_0(\xi,\alpha)-\delta\right) = 1\;\;\;
\forall\delta>0\, .
\end{align}

While the above upper and lower bounds do not match in general, 
they get close to each other for certain sequences of 
high-degree polynomials.
As an example, consider the case $\xi_{p,\gamma_0}(t) = \xi_{0,p}+t^p$, $\xi_{0,p}=\gamma_0\log p$, 
for some constant $\gamma_0>1$. This corresponds to $F_i(\bx) = \sqrt{\xi_{0,p}} g_i+\tilde{F}_i(\bx)$
where the $\tilde{F}_i$ are homogeneous degree-$p$ polynomials with Gaussian coefficients.
Then, we have the large-$p$ asymptotics
(cf. Appendix \ref{app:LB-asymp})
\begin{align}
\alpha_{\slb}(\xi_{p,\gamma_0})= \frac{1}{\gamma_0}\cdot \big(1+o_p(1)\big)\,,\;\;\;\alpha_{\sub}(\xi_{p,\gamma_0})  = \frac{1}{\gamma_0}\cdot \big(1+o_p(1)\big)\, .\label{eq:LB-asymp}
\end{align}

We further note that ---in the language of spin glass theory---
$u_0^{(1)}(\xi,\alpha)$ is the `replica symmetric' lower bound.
We will see that in some cases this will be achieved by the algorithm
in the next section and hence be tight.

%
%*********************************************************
%
\section{Rigorous results (II): Algorithmic phase transition}
\label{sec:Rigorous-ALG}

\subsection{Achievable energy}
\label{sec:AchievableEnergy}

In order to state our main results, 
we introduce the notion of generalized first order method (GFOMs),
a terminology that we adapt from \cite{celentano2020estimation}.
 
A GFOM is an iterative algorithm. At iteration $\ell$, 
the algorithm keeps track of $2\ell$ vectors
$\bm^1$, \dots, $\bm^\ell\in \reals^d$, 
$\bh^1$, \dots, $\bh^\ell\in \reals^n$. The new vectors
$\bm^{\ell+1}$ and $\bh^{\ell+1}$ are functions 
of the previous ones $(\bm^i,\bh^i:i\le \ell)$, 
and of the evaluation 
$\bF(\bm^{\ell})$, $\bD\bF(\bm^{\ell})$,
and are further linear in $\bF(\bm^{\ell})$, $\bD\bF(\bm^{\ell})$.

We next present our achievability result for
this class of algorithms. Throughout, we will assume\footnote{There is no real loss of generality in this assumption when discussing near-solutions,
because we an always replace $\bF(\bx)$ by $\bF_{\eps_0}(\bx) :=\eps_0\bg+\bF(\bx)$ 
with $\eps_0>0$ and $\bg\sim \normal(\bzero,\bI_n)$. This only results in an $O(n\eps_0)$ change in
the energy, which is irrelevalt since we are only interested in near-solutions.} $\xi(0)>0$.
Define the `replica symmetry breaking radius' 
(uniqueness of the minimizer is shown in Section \ref{sec:TwoPhase})
\begin{align}\label{eq:qrsDef}
q_{\sRS}(\xi):= \argmin_{q\ge 0} \frac{\xi(q)\xi'(q)}{q} \, .
\end{align}
Here it is understood that the right-hand side is extended by 
continuity to $q=0$ when $\xi'(0)=0$ and therefore $q_{\sRS}(\xi)=0$
in that case (we set by convention $q_{\sRS}(\xi)=\infty$ if 
$\xi(t) =\xi_0+\xi_1 t$). 
We further dedine $q_0=q_0(\alpha,\xi)$  as the solution 
of 
\begin{align}\label{eq:q0Def}
\frac{q_0\xi'(q_0)}{\xi(q_0)} = \alpha\, .
\end{align}
(The solution exists uniquely, unless 
$\xi(t) =\xi_0+\xi_1 t$ and $\alpha\ge 1$, in which case we set
$q_0(\alpha,\xi)=\infty$.)

Next we define 
\begin{align}\label{eq:uRS_def}
u_{\sRS}(q,\alpha,\xi) := 
\frac{1}{2}\Big(\sqrt{\xi(q)}-\sqrt{\frac{1}{\alpha}q\xi'(q)}\Big)_+^2\, ,
\end{align}
and note that $q_0$ is the first point at which 
$q\mapsto u_{\sRS}(q,\alpha,\xi)$ vanishes.

For $a,b\in\reals_{\ge 0}$, define
$z_*(\alpha,a,b)$  via
\begin{align}
 Q(m;\alpha,a,b) &:= -\frac{1}{m}+\frac{\alpha b}{1+bm}-a^2m\, ,\\
 z_*(\alpha,a,b) & := -\sup_{m>0}  Q(m;\alpha,a,b)\, .
 \end{align}
 %
%Further, given $q_{\sRS}(\xi)$, $q_0(\alpha,\xi)$,
%as in Eqs. \eqref{eq:qrsDef} and \eqref{eq:q0Def}, and 
%set $q_*:=  q_{\sRS}(\xi)\wedge q_0(\alpha,\xi)$.
We distinguish two cases.

\paragraph{Case 1: $q_{\sRS}(\xi)\ge 1$.} We set 
 $u_{\sRS}(\alpha,\xi):=u_{\sRS}(1,\alpha,\xi)$ and
\begin{align}
	u_{\salg}(\alpha,\xi) := u_{\sRS}(\alpha,\xi)=
	\frac{1}{2}\Big(\sqrt{\xi(1)}-\sqrt{\frac{\xi'(1)}{\alpha}}
	\Big)_+^2\, .\label{eq:REenergy}
\end{align}

\paragraph{Case 2: $q_{\sRS}(\xi)< 1$.} 
 Let $u(t) = u(t ;\alpha,\xi,q)$,  be the unique solution of 
 the ordinary differential equation 
 \begin{align}
 \frac{\de u}{\de t}(t) = -\frac{1}{2\alpha}z_*\big(\alpha; \sqrt{2\alpha u(t) \xi''(q+t)}, \xi'(q+t)\big)\,,
 \;\;\;\;\; u(0) = u_{\sRS}(q,\alpha,\xi)\, .\label{eq:U_ODE}
 \end{align}
Finally we define the algorithmic threshold energy as
\begin{align}\label{eq:AchievableEnergyDef}
 u_{\salg}(\alpha,\xi)= \begin{cases}
	u(1-q_{\sRS};\alpha,\xi,q_{\sRS}) & \mbox{if $q_{\sRS}(\xi)
	<q_0(\alpha,\xi)$,}\\
	0 & \mbox{otherwise.}
 \end{cases}
\end{align}
In particular, $q_0(\alpha,\xi)\le q_{\sRS}(\xi)$ then the above 
yields $u_{\salg}(\alpha,\xi)= 0$.

We are now in position to state our achievability result,
establishing that we can achieve an 	energy arbitrarily close
to $u_{\salg}(\alpha,\xi)$ in polynomial time.
Below, we denote by $\chi_{\bF}$, $\chi_{\bD\bF}$, respectively,
the complexity of evaluating
$\bF$ and of matrix vector multiplication by the Jacobian $\bD\bF$.
\begin{theorem}\label{thm:MainAchievability}
For any $\eps>0$, there exists a constant $C(\eps)$ and
a GFOM algorithm with complexity $C(\eps)(\chi_{\bF} +\chi_{\bD\bF}+nd)$
such that, denoting by $\bx^{\salg}\in \S^{d-1}$ the algorithm output, 
we have, for any $\eps>0$,
 \begin{align}
 \lim_{d,n\to\infty}\prob\Big(\frac{1}{n}H(\bx^{\salg})\le 
 u_{\salg}(\alpha,\xi)+\eps
 \Big)=1\, .
 \end{align} 
 \end{theorem}
 The proof of this theorem is presented in Sections
 \ref{sec:Algo_HD} and \ref{sec:TwoPhase}.

In view of Theorem \ref{thm:MainAchievability}, we define 
the algorithmic threshold as
\begin{align}\label{eq:AlphaAlgDef}
\alpha_{\salg}(\xi):= \sup\big\{\alpha\ge 0 \;\;
\mbox{such that}\;\; u_{\salg}(\alpha,\xi) = 0\big\}\, .
\end{align}
 It is possible to show (see Appendix \ref{sec:Monotonicity}) that 
 $\alpha\mapsto
 u_{\salg}(\alpha,\xi)$ is continuous non-decreasing,
 and hence $u_{\salg}(\alpha,\xi)=0$ for all 
 $\alpha\in [0,\alpha_{\salg}]$. We thus get the following
 immediate consequence.
 \begin{corollary}\label{coro:RS_Opt}
For $\alpha\in [0,\alpha_{\salg}(\xi)]$ and  any $\eps>0$ 
there exists an algorithm, with the same complexity as in
Theorem \ref{thm:MainAchievability}, which, with high probability,
outputs $\bx^{\salg}\in \Sol_{n,d}(\eps)$. 
 \end{corollary}

A particularly simple case is the one in which $q_{\sRS}(\xi)>1$,
which by the definition \eqref{eq:qrsDef} happens if and only if
$\xi(1)\xi'(1)\ge \xi'(1)^2+\xi(1)\xi''(1)$.  In this case, we never 
enter the second phase of the algorithm, and the algorithmic threshold
$u_{\salg}(\alpha,\xi)$ coincides with the lower bound
$u^{(1)}_0(\alpha,\xi)$ of Proposition \ref{propo:SimpleUB}. 
\begin{corollary}
Define $u_{\sRS}(\alpha,\xi)$
as per Eq.~\eqref{eq:REenergy}.
If $\xi(1)\xi'(1)\ge \xi'(1)^2+\xi(1)\xi''(1)$, then,
denoting by $\bx^{\salg}$ the output of the algorithm of
Theorem \ref{thm:MainAchievability}, we have
\begin{align}
\plim_{n,d\to\infty}\frac{1}{n}\min_{\bx\in\S^{d-1}}
H(\bx) = \plim_{n,d\to\infty}\frac{1}{n}H(\bx^{\salg}) = 
u_{\sRS}(\alpha,\xi)\, .
\end{align}
In particular, under this condition, 
$\alpha_{\salg}(\xi) =\xi'(1)/\xi(1)\in (0,1]$ 
and this coincides with the threshold for the existence of near solutions.
Namely, for $\alpha<\alpha_{\salg}(\xi)$ $\Sol_{n,d}(\eps)\neq\emptyset$
whp for any $\eps>0$, while, 	for $\alpha>\alpha_{\salg}(\xi)$ there exists $\eps_0>0$
such that $\Sol_{n,d}(\eps_0)=\emptyset$.
\end{corollary}

\begin{remark}[Numerical evaluation of $u_{\salg}(\alpha,\xi)$]
	Evaluating $u_{\salg}(\alpha,\xi)$ numerically 
	by solving the ODE \eqref{eq:U_ODE} is 
	straightforward, and we provide examples in 
	Fig.~\ref{fig:a025_third_change_m}, 
	as well as in Section \ref{sec:Numerical_First}
	and Appendix \ref{sec:Appendix_Numerical} below.
\end{remark}

\begin{remark}
	The complexities $\chi_{\bF}$, $\chi_{\bD\bF}$ of evaluating
	$\bF$ and of matrix vector multiplication by the Jacobian $\bD\bF$ at a query point 
	$\bx$ depends on the details of the computation model in use. In a model in which sums 
	and products in $\reals$ can be carried out in $O(1)$ time, if $\xi_{k}=0$ for all $k>k_{\max}$
	(i.e. $\bF$ is a polynomial), then $\chi_{\bF}$, $\chi_{\bD\bF}=O(d^{k_{\max}+1})$.
	
	If $\xi_k\neq 0$ for infinitely many $k$ (i.e. $\bF$ is not a polynomial),
	we can truncate it at a large level $k_{\max}$, and  hence approximate matrix-vector multiplication
	by using the truncated Jacobian. Similarly, if only integer operations are allowed, we 
	can use finite-precision approximations of the entries of 
	$\bF(\bx)$, $\bD\bF(\bx)$. It is easy to 
	show that these modifications do not change the claim of Theorem \ref{thm:Hessian}.
\end{remark}

\begin{figure}
	\begin{center}
		\includegraphics[width=0.45\linewidth]{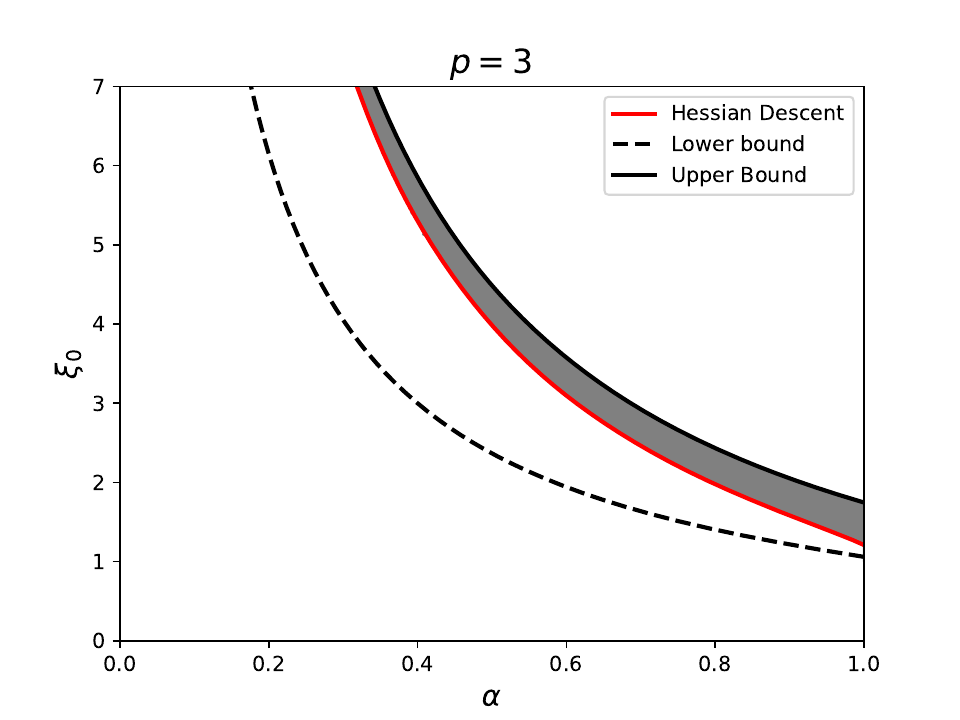}
			\includegraphics[width=0.45\linewidth]{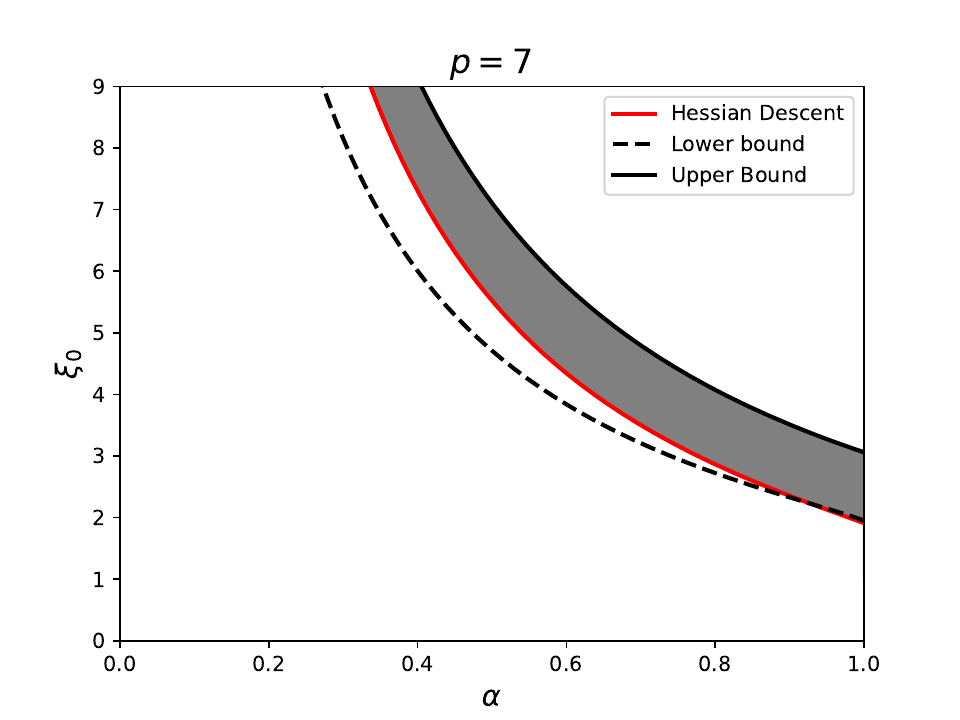}\\
				\includegraphics[width=0.45\linewidth]{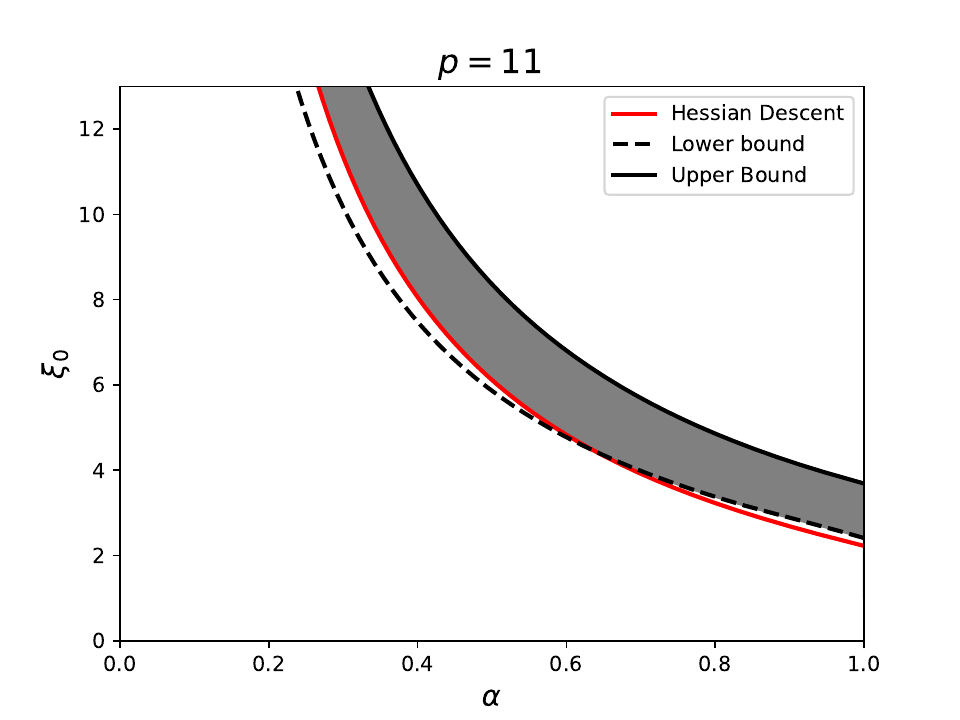}
			\includegraphics[width=0.45\linewidth]{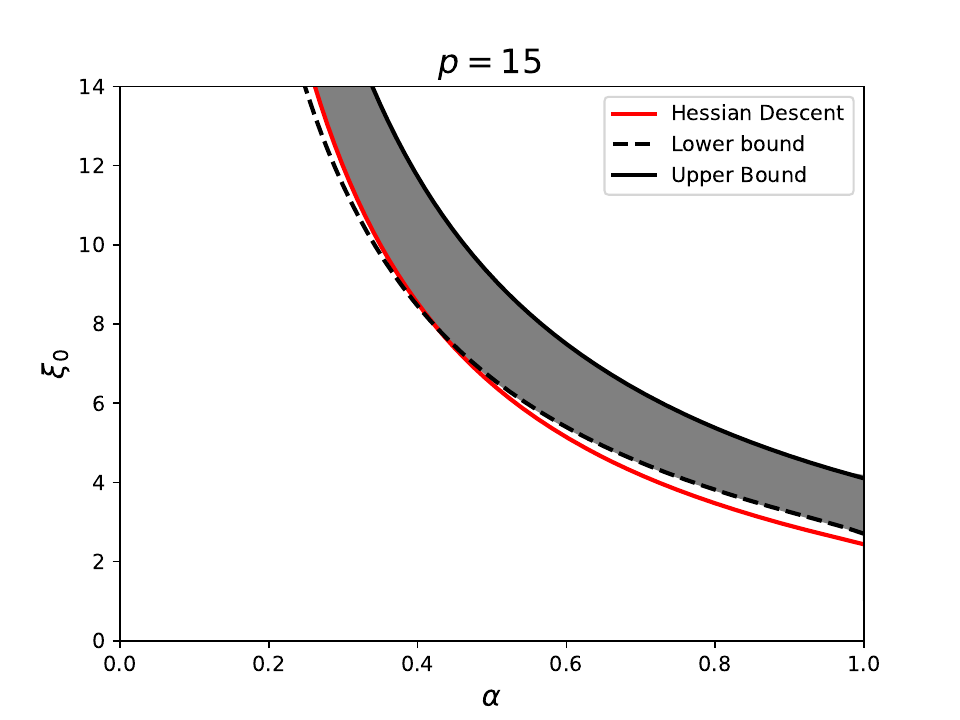}
		\caption{Phase diagram for $\xi(q)= \xi_0+q^p$, $p\in\{3,7,11,15\}$. 
		Upper (solid) and lower (dashed) black lines are the upper and lower bounds on the 
		threshold for existence of solutions $\alpha_{\sub}$ and $\alpha_{\slb}$
		obtained respectively by Gaussian comparison and second moment method. 
		Red line: The threshold $\alpha_{\salg}$ below which the Hessian descent algorithm finds solutions 
		with high probability.}
		\label{fig:PhaseDiagP11}
	\end{center}
\end{figure} 
 \subsection{Lipschitz hardness}

 Lipschitz algorithms are algorithms  whose output $\bx^{\salg}$ is a 
  Lipschitz function of the input.
 Denote by $\cuF(n,d)$ the space of polynomials
 $\bF:\reals^d\to \reals^n$, $F_i(\bx) = \sum_{\ell\ge 0}^{\ell_{\max}}\<\bG_{i}^{(\ell)},\bx^{\otimes \ell}\>$
 which we identify with the set of coefficients $\bG_{i}^{(\ell)} \in(\reals^{d})^{\otimes \ell}$
 and endow with the distance 
 \begin{align}
 \dist(\bF,\tilde{\bF})^2 : = \frac{1}{d} \sum_{\ell=0}^{\ell_{\max}}\sum_{i=1}^n 
 \big\|\bG_{i}^{(\ell)} -\tilde\bG_{i}^{(\ell)}\big\|^2_F \, .
 \end{align}
  An algorithm $\bx^{\salg}: \cuF(n,d) \to \S^{d-1}$, $\bF\mapsto \bx^{\salg}(\bF)$
 is Lipschitz if there exists a constant $M$, independent of $n,d$ such that
 for any $\bF,\tilde\bF\in  \cuF(n,d)$, $\|\bx^{\salg}(\bF)-\bx^{\salg}(\tilde\bF)\|\le
 M\dist(\bF,\tilde{\bF})$.

 We prove the next theorem
in a companion paper  \cite{hardness_in_preparation}, 
by using recently developed techniques in average-case complexity
and spin glass theory \cite{gamarnik2021overlap,huang2022tight}.
An outline of the proof is presented in Appendix \ref{sec:optimalitypfsketch}.
\begin{theorem}\label{thm:Optimality}
	For any Lipschitz algorithm $\bx^{\salg}$ and any $\eps>0$,
	\begin{align}\label{eq:optimality}
		\lim_{d,n\to\infty}\P\left(\frac1n H(\bx^{\salg}) \ge 
		u_{\salg}(\alpha,\xi)-\eps\right) = 1\, .
	\end{align}
	\end{theorem}
We note that the algorithm of Theorem \ref{thm:MainAchievability}
is Lipschitz and hence together this two theorems characterize 
the optimal energy achieved by this class of algorithms.

\begin{remark}\label{rmk:LBeqUB}
	The proof of Theorem  \ref{thm:Optimality}
	yields an alternative characterization of 
	$u_{\salg}(\alpha,\xi)$, which is of independent interest.
	Let $\cuC^1:=\{p:[0,1]\to [0,1]:\; p(0)=0, p(1)=1, p\in C^1((0,1))\cap C^0([0,1]), \mbox{non-decreasing}\}$, and for any
	$p\in\cuC^1$ define $u_{\slb}(t) = u_{\slb}(t;p)$ as the solution of the ordinary differential 
	equation\footnote{Since $p\in\cuC^1$, we have 
	$u_{\slb}(0) =   0$ by definition. We prefer to write the initialization
	in the form \eqref{eq:ulb}, since it generalizes to 
	other choices of $p$ \cite{hardness_in_preparation}.}
	\begin{align}\label{eq:ulb}
		\frac{\de u_{\slb}}{\de t}(t) &= \max_{s>0} \cL(s;p,\dot{p},u_{\slb},t)\, , \;\;\;\;
		u_{\slb}(0) =   \frac{1}{2}p(0)\xi(0)\, ,
	\end{align}
	where we write $\dot{p}$ (and below $\dot{\xi}$) for the derivative of $p$ (similarly for $\xi$)
	and 
	\begin{align}
		\cL(s;p,\dot{p},u,\cdot) = -(\dot{p}\dot{\xi}+p\ddot{\xi})  u s+\frac{1}{2}\dot{p}\xi-\frac{1}{2\alpha s}
		+ \frac{p\dot{\xi}}{2(1+p\dot{\xi} s)}(1-\dot{p}\xi s)\, .\label{eq:DriftLowerBound}
	\end{align}
	Finally, we define
	\begin{align}\label{eq:ALG-LB}
		u_{\salg,\slb}(\alpha,\xi):=\inf_{p\in\cuC^1}  u_{\slb}(1;p)\, .
	\end{align}
	In \cite{hardness_in_preparation}, we prove that 
	$u_{\salg,\slb}(\alpha,\xi)=u_{\salg}(\alpha,\xi)$.
\end{remark}

 \begin{remark}
Figure \ref{fig:PhaseDiagP11} compares numerical evaluations of  
the threshold $\alpha_{\salg}(\xi)$
with the upper and lower bounds
$\alpha_{\slb}(\xi)$ and $\alpha_{\sub}(\xi)$ 
on the satisfiability threshold.
We observe that there are  choices of $\xi$, 
such that there  exists a non-empty interval 
$\alpha\in  (\alpha_{\salg},\alpha_{\sSAT})$
for which  solutions exist whp, but Lipschitz algorithms 
do not find them.

As a specific example, recall the model
$\xi_{p,\gamma_0}(t)= \xi_{0,p}+t^p$, $\xi_{0,p} =
 \gamma_0\log p$. For large $p$
 (see Corollary \ref{coro:Hessian} and Eq.~\eqref{eq:AlphaHessianBounds})
\begin{align}
	\alpha_{\salg}(\xi_{\gamma_0,p}) = \frac{4}{\gamma_0\log p}(1+o_p(1))\, .
	\end{align}
For all $p$ large enough, this is strictly smaller than 
$\alpha_{\slb}(\xi_{p,\gamma_0})$, see Eq.~\eqref{eq:LB-asymp}.
	\end{remark}

\subsection{Description of the algorithm}
\label{sec:AlgoDescr}

We will describe a second-order algorithm that achieves 
the threshold energy of Theorem \ref{thm:MainAchievability}.
Then, we will show that the Hessian evaluation can in fact
be approximated by a sequence of gradient evaluations,
hence reducing it to a first-order algorithm.

The algorithm proceeds in two phases:
$(i)$~We first compute $\bm^L$ which is an approximate stationary point  of $H(\bm)$
subject to $\|\bm\|^2=q_*$. We use gradient 
information and constant stepsize $\gamma_*$, with $q_*$ and $\gamma_*$ 
computed from our theory. $(ii)$~We use the vector $\bm^L$ as initialization for the next phase.
We restrict to the hyperplane orthogonal to this vector and run a Hessian descent 
algorithm to optimize the energy $H(\bx)$ in the intersection of this hyperplane 
with the sphere $\S^{d-1}$.

 The first phase is defined by letting  $\bm^0 = \bh^{0} = \bzero$, and,
 for $\ell\ge 0$,
 \begin{align}
\bh^{\ell+1}= \bh^{\ell}+\frac{1}{\sqrt{n}}\proj^{\perp}_{\bh,\ell}\bF(\bm^{\ell}) ,\;\;\;\;
\bm^{\ell+1}= \bm^{\ell}-\frac{\gamma}{\sqrt{d}}
\proj^{\perp}_{\bm,\ell}\bD\bF(\bm^{\ell})^{\sT}\bh^{\ell}\, .\label{eq:Ortho_2_maintext}
 \end{align}
Here $\proj^{\perp}_{\bh,\ell}\in\reals^{n\times n}$ is the projector orthogonal to
 $(\bh^j:\, j\le \ell)$, and
$\proj^{\perp}_{\bm,\ell}\in\reals^{d\times d}$ is the projector orthogonal to
$(\bm^j:\, j\le \ell)$. The update \eqref{eq:Ortho_2_maintext}  is repeated for $L$ iterations, with fixed
stepsize $\gamma$. 

It turns out that the above iteration is asymptotically equivalent
to an approximate message passing (AMP) algorithm \cite{DMM09,bayati2011dynamics}, and this allows us to leverage the theory
of those algorithms for its analysis. 

In the second phase we take $K$ iterations with 
stepsize parameter $\delta$, related via $K\delta = 1-\|\bm^L\|^2$.
At each step $k\in\{0,1,\dots , K-1\}$, given current state $\bx^k$, 
we  compute the Hessian $\nabla^2H(\bx^k)$ along the hyperplane orthogonal to
$\bx^k$ and take a step of length $\sqrt{\delta}$ along the optimal Hessian direction.
 The pseudocode the algorithm is given in 
 Algorithm \ref{algo:HD+AMP}.
 
\vspace{0.3cm}

\begin{algorithm}[H]
\SetAlgoLined
\KwData{Couplings $\{\bG^{(k)}\}_{0\le k\le k_{\max}}$, iteration number $L$, stepsize $\delta$,
AMP parameter $\gamma$}
\KwResult{Approximate optimizer $\bx^{\salg}\in \S^{d-1}$}
Initialize $\bm^0 = \bh^{0} = \bzero$\;
\For{$\ell\in\{0,\dots,L-1\}$}{
Compute $\bh^{\ell+1}$, $\bm^{\ell+1}$ via Eq.~\eqref{eq:Ortho_2_maintext}\;
}
Set $\bx^0= \bm^{L}$, $V_{L}:= \{\bx\in\reals^d:\, \<\bx,\bm^L\>=0\}$\;
Set $K=(1-\|\bm^L\|^2)/\delta$\;
\For{$k\in\{0,\dots,K-1\}$}{
Compute $\bv = \bv(\bx^k)\in \Ts_{\bx^k}\cap V_L$ such that $\|\bv\|_2=1$ and
\begin{align*}
\<\bv, \nabla^2H(\bx^k)\bv\> \le \lambda_{\min} (\nabla^2H(\bx^k)|_{\Ts_{\bx^k}\cap V_L})+d\delta\, ;
\end{align*}
Set $s_k:= \sign(\<\bv(\bx^k),\nabla H(\bx^k)\>)$\;
$\bx^{k+1} = \bx^k-s_k\sqrt{\delta} \, \bv(\bx^k)$\;
}
\KwRet $\bx^{\salg} = \bx^K/\|\bx^K\|_2$\;
\caption{Two-Phase Algorithm}\label{algo:HD+AMP}
\end{algorithm}

 \begin{remark}
The calculation of the approximate eigenvector $\bv(\bx^k)$
can be performed by  power iteration, by performing $O(1/\delta)$
matrix-vector multiplications.
Hence, for arbitrary constants $L$, $\delta$,
 the complexity of this algorithm is of the same order 
 as matrix vector multiplication by 
 $\nabla^2H(\bx)$ at a query point $\bx$.

 In fact evaluating the Hessian is not necessary, and this calculation
 can be approximated by evaluating the gradient $\nabla H$,
 resulting in a pure first-order method.
 We refer to Section \ref{sec:TwoPhase} for the details.

An alternative approach to obtain a first order method 
is provided by the IAMP technique of 
\cite{montanari2019optimization,el2021optimization,montanari2024exceptional}.
 \end{remark}

\section{Numerical results (I):\\ Comparing theory and simulations for 
Hessian descent and SGD}
\label{sec:Numerical_First}

As mentioned above, the
theoretical prediction of Theorem \ref{thm:MainAchievability} is 
easily evaluated numerically.
We carried out extensive numerical simulation 
within the AoR model with matched covariance using:
$(i)$~The 
algorithm of Theorem \ref{thm:MainAchievability}, described in 
Section \ref{sec:AlgoDescr}; $(ii)$~Projected SGD.
We present some of our results in this section and further 
ones in Appendix \ref{sec:Appendix_Numerical}.

We find good agreement between analytical predictions derived for
 the optimal algorithm within the GE model, and 
 numerical results within the AoR model obtained either with the same algorithm or
 with SGD. 
Providing a rigorous explanation of the near optimality of SGD,
and of universality beween AoR and GE models is an outstanding
mathematical challenge. 

We note that $\xi'(0)=0$ for the examples considered here, and therefore
the optimal algorithm described in Section \ref{sec:AlgoDescr}
is purely Hessian descent in these cases.

\subsection{Hessian descent}

\begin{figure}
	\includegraphics[width=0.5\linewidth]{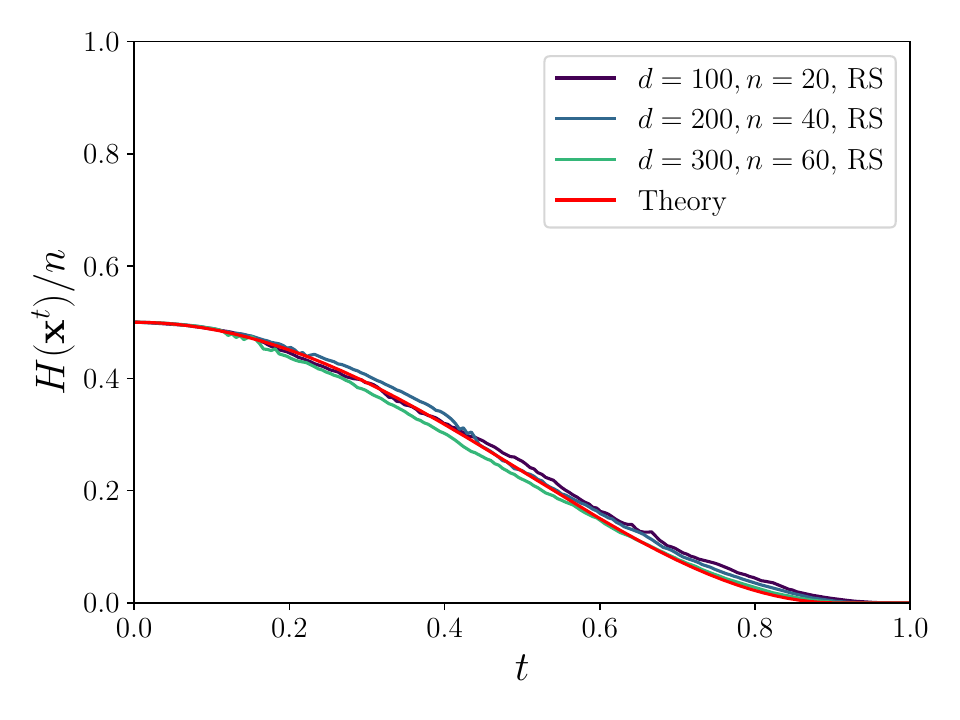}
	\includegraphics[width=0.5\linewidth]{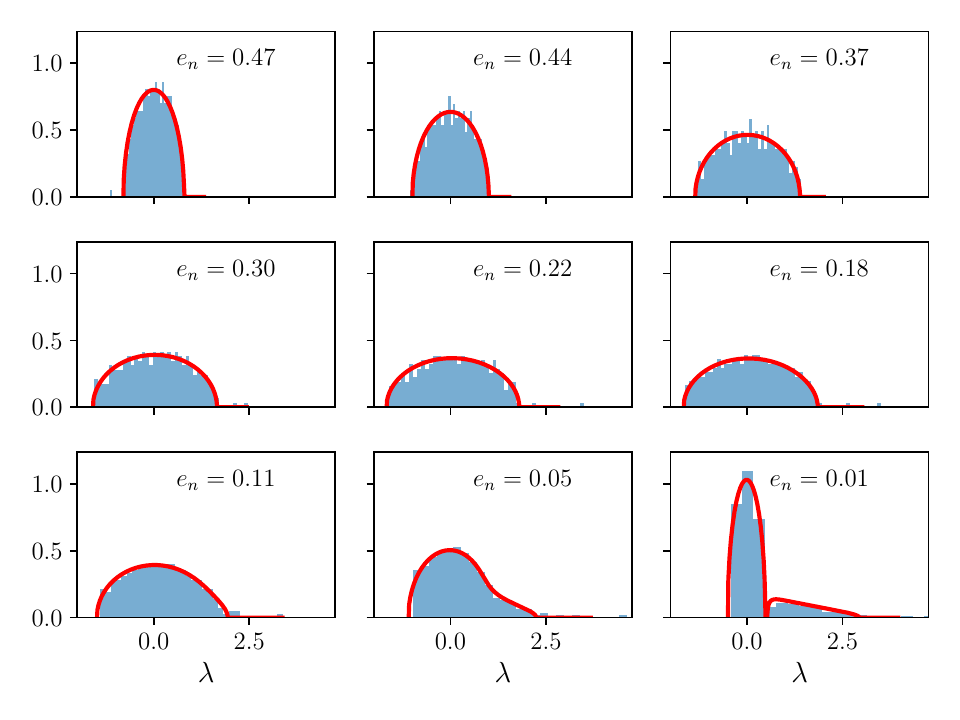 }
	\caption{Trajectory of Hessian descent, \emph{with random signed steps},
	for the AoR model of Eq.~\eqref{eq:AoR_Fi} with 	
	$\phi(t) = 1+a \he_4(t)$, where $a=0.33$ ams $\he_4$ 
	is the fourth Hermite polynomial. 
	Left, energy as a function of time (radius): $\alpha=0.2$, $m=16d$, $K=200$. 
	Red lines are theoretical predictions for the
	algorithm evolution.
	Right: Empirical spectral distribution of the Hessian along Hessian descent 
	trajectory, with $e_n$ the energy achieved. 
	(Here  $d=300$, $n=20$, $m=4800$, $K=200$.)}
			\label{fig:a033_fourth_alpha02_RS}
	\end{figure}
	 Figure \ref{fig:a033_fourth_alpha02_RS} reports the 
	 evolution of the cost $H(\bx)/n$
	 along the Hessian descent evolution for the AoR model \eqref{eq:AoR_Fi}
	 (for a specific choice of $\phi$, $\alpha$), 
	 and compares empirical results with asymptotic
	 predictions for the GE model from Theorem \ref{thm:MainAchievability}.
	  We see good agreement despite the fact that in the AoR model
	   the processes
	  $F_i(\,\cdot\,)$ are non-Gaussian. Here we use a slight variant of 
	  Algorithm \ref{algo:HD+AMP}, whereby each Hessian descent step is multiplied by a 
	  uniformy random sign. We refer to Appendix \ref{sec:Appendix_Numerical} 
	  for similar examples in which the sign is chosen to align with the $\nabla H$.

In the right frame of Fig.~\ref{fig:a033_fourth_alpha02_RS}, we plot the evolution of the 
spectrum of the Hessian along the Hessian descent trajectory, as the energy decreases.
 This is compared with 
the asymptotic prediction for the GE model (red lines). 
Again the agreement is excellent.
	
Establishing Theorem \ref{thm:Hessian} requires to analyze
	the eigen-structure of the Hessian $\nabla^2 H(\bx)$ at a point 
	$\bx\in\reals^d$:
\begin{align}
	\nabla^2
	 H(\bx) = \bD\bF(\bx)^\sT\bD\bF(\bx) + \sum_{i=1}^nF_i(\bx) 
	 \nabla^2F_i(\bx)\, ,
	\end{align}
	A key remark is that the restriction of the Hessian to
	$\Ts_{\bx}$ (the orthogonal complement of $\bx$) has a simple 
	distribution. 
	Namely, for $\bx\in\reals^d$ with $\|\bx\|^2_2=q>0$, let $\bU_{\bx}\in\reals^{d\times(d-1)}$
	be an orthonormal basis for the orthogonal complement of $\bx$ and
	  $\bcH(\bx)  := \bU_{\bx}^{\sT}\nabla^2 H(\bx)\bU_{\bx}$
	be the restriction of the Hessian  to the tangent space. Then, the joint
	distribution of the Hessian and cost at $\bx$ is given by
	\begin{align}
	\bcH(\bx)  = \sqrt{\xi(q)\xi''(q)}\, \|\bg\|_2\bW+\xi'(q)\,\bZ^{\sT}\bZ\, ,\;\;\;\;
	H(\bx)    = \frac{1}{2}\xi(q)\|\bg\|_2^2\, ,\label{eq:DistributionHessian}
	\end{align}
	where
	$(\bg,\bW,\bZ)\sim \normal(0,\id_n)\otimes \GOE(d-1)\otimes \GOE(n,d-1)$
	(cf. Section \ref{app:Definitions} for definitions). 

	The distribution of the Hessian determined by Eq.~\eqref{eq:DistributionHessian}
	 holds at a deterministic point $\bx$ (or at a random point $\bx$ independent of $H(\,\cdot\,)$).
	In particular, it implies that the empirical spectral distribution converges (for large
	$n,d$) to the free convolution of suitably scaled Wigner (for the term $\bW$) 
	and Marchenko-Pastur (for the term $\bZ^{\sT}\bZ$) laws. 
	
	The distribution of the Hessian implied by Eq.~\eqref{eq:DistributionHessian}
	does not hold for a point $\bx$ dependent on the cost function $H(\,\cdot\, )$,
	such as those produced by the algorithm. 
	Nevertheless, in Section \ref{sec:Algo_HD} we prove that 
	--in the GE model-- the empirical spectral distribution 
	is close to the one
	predicted by the model  \eqref{eq:DistributionHessian}
	\emph{uniformly over the unit ball $\bx \in \Ball^{d-1}(1)$}.

	The theoretical predictions in the right frame of 
	Fig.~\ref{fig:a033_fourth_alpha02_RS} are the free convolution of Wigner and Marchenko-Pastur 
	spectra 
	from
	Eq.~\eqref{eq:DistributionHessian}.
	
	 For $\alpha<1$, the Wishart component has $(1-\alpha) d$ flat directions, and hence
$\nabla^2 H(\bx)$ has of the order of $d$ negative eigenvalues\footnote{This however does 
not imply that zero energy is achieved for any $\alpha<1$ because of the spherical constraint
$\|\bx\|=1$.}. 
As the energy decreases, the weight of the Wigner component decreases, and this leads to 
the decomposition of the support of the spectrum in two intervals.

\subsection{Stochastic gradient descent}

\begin{figure}
\includegraphics[width=0.5\linewidth]{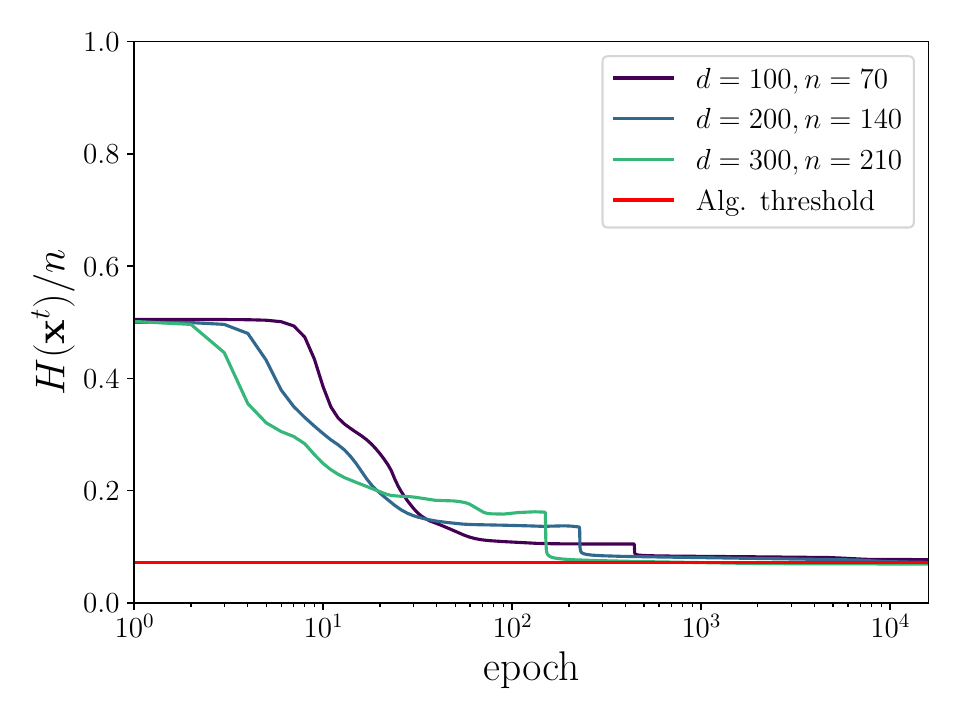}
\includegraphics[width=0.5\linewidth]{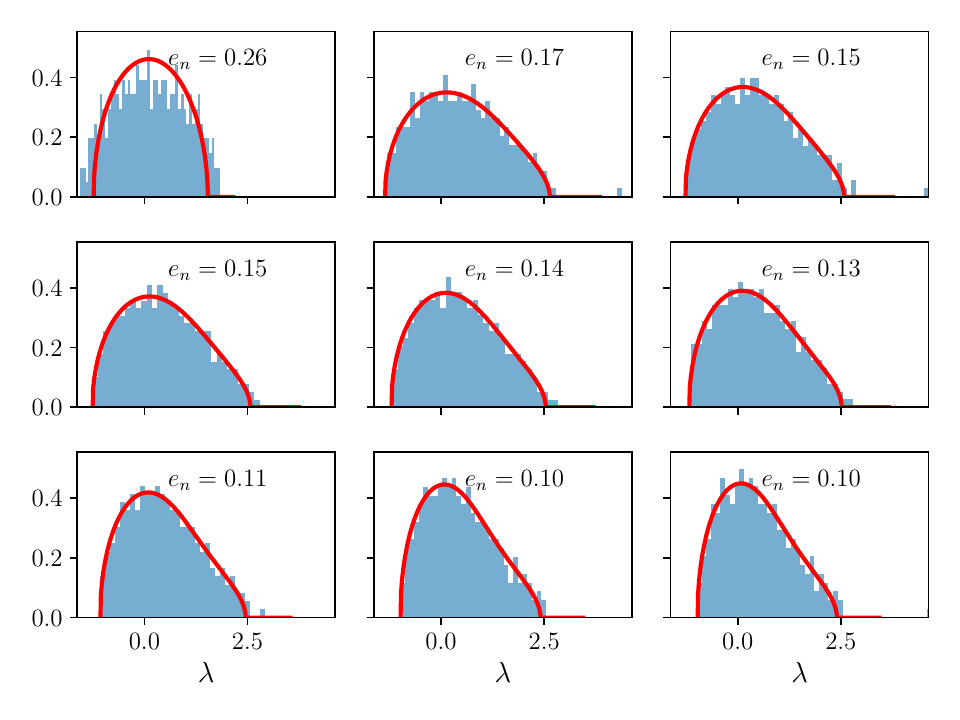}
\caption{Trajectory of SGD ,
for $\phi(t) = 1+ a \he_4(t)$, with $a=0.33$.
Left, energy as a function of time (epochs):
 $d=300$, $n=210$, $m=4800$, $N_{\sit}=16,000$,
${\rm lr}=0.01$, whence $\alpha=0.7$ (see caption of Fig.~\ref{fig:a025_third_change_m} for notations).
Right: Empirical spectral distribution of the Hessian along SGD trajectory,
with theoretical predictions as red lines.}
		\label{fig:a033_fourth_alpha07_SGD3}
\end{figure}

Figure \ref{fig:a025_third_change_m}, left frame, 
 reports the results of 
numerical simulations using SGD within the AoR  model 
(where we use the matching activation $\phi(x) = 1+a \he_3(x)$).
We run SGD in the unit ball $\bx\in\Ball^d(1)$, initializing with $\bx_0=\bzero$.
We refer to Appendix \ref{sec:Appendix_Numerical} for details of the SGD implementation as well as results
with other choices of the SGD hyperparameters (number of epochs, step size).

In Figure \ref{fig:a033_fourth_alpha07_SGD3}, we plot the evolution of the cost
$H(\bx^t)/n$ along three SGD trajectories for increasing values of $n$, $d$, at fixed 
$n/d=\alpha$. We use the AoR model with $\phi(t) = 1+a \he_3(t)$, $a=0.25$ and $\alpha=0.7$.
Once again, we compare these simulation results with 
analytical predictions for the threshold energy $u_{\salg}(\alpha,\xi)$
obtained within the GE model for the Hessian descent algorithm. 

As for Hessian descent, in the right frame of Fig.~\ref{fig:a033_fourth_alpha07_SGD3}, we plot the evolution of the 
spectrum of the Hessian along the SGD trajectory, as the energy decreases. 
This is compared with 
the asymptotic prediction for the GE model (red lines). 
Again the agreement is excellent.

	\section{Numerical results (II): Sensitivity phase transition}
	\label{sec:Numerical_Second}
	
\subsection{Experiments with stochastic gradient descent}
\label{sec:SensitivityMain}

\begin{figure}
	\hspace{-1cm}\includegraphics[width=0.5\linewidth]{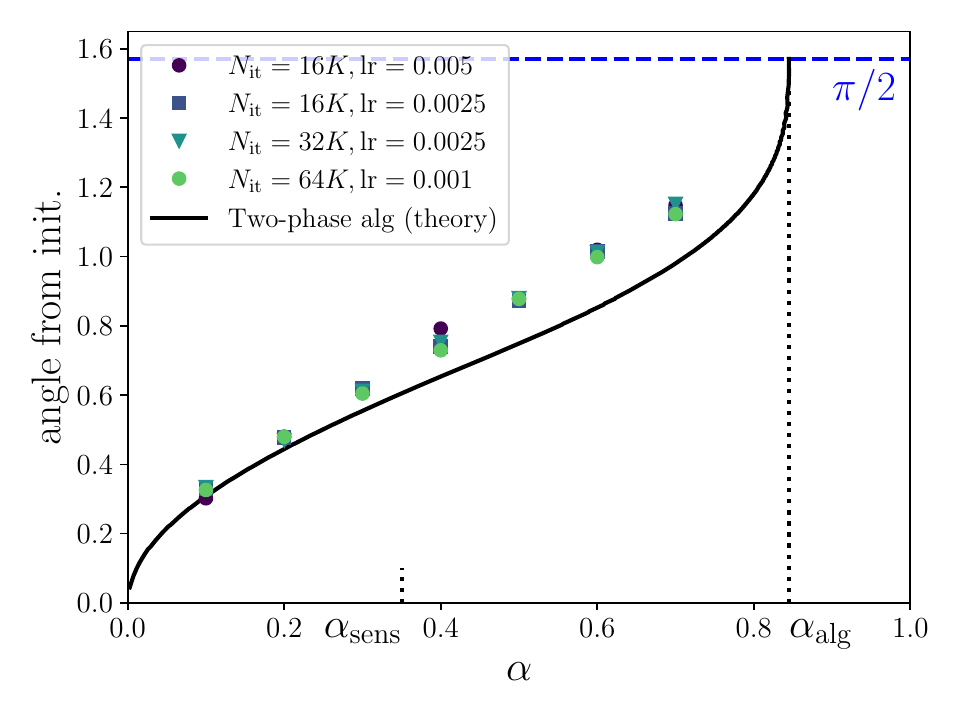}
				\includegraphics[width=0.5\linewidth]{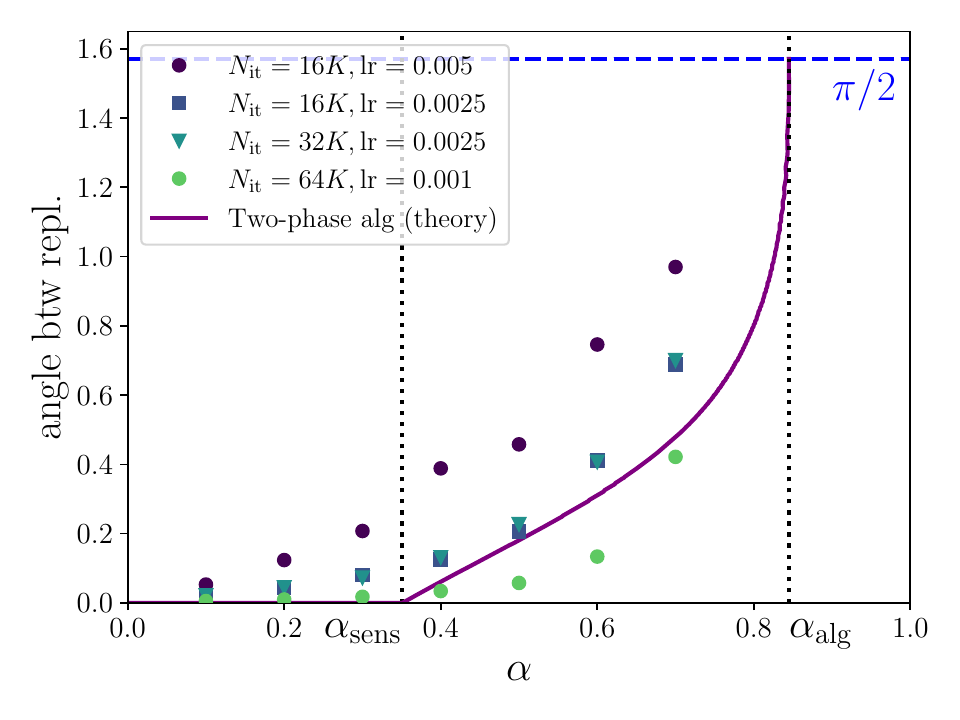}
				
			\caption{Geometry of SGD (symbols) and two-phase algorithm (continuous
			lines, asymptotic theory) in the AoR model $\phi(t) = 1+ a \he_3(t)$, with $a=0.33$. Here $d=200$, $m=1600$
			and we vary $N_{\sit}$, ${\rm lr}$, $\alpha$
			(see caption of Fig.~\ref{fig:a025_third_change_m} for notations). Left: angular distance of the algorithm output
			from the random initialization $\bx_0\in \S^{d-1}$. Right: angular distance between 
			outputs between two runs of the optimization algorithm that differ by a perturbation
			(reshuffling for SGD).}
			\label{fig:Geometry}
	\end{figure}

	It is interesting to revisit the neural tangent (NT) theory 
	\cite{jacot2018neural}, since it 
	predicts convergence to a global minimum,  provided the problem is sufficiently overparametrized, i.e. 
	$n/d$ is small enough. At first sight, this might appear to explain the transition at 
	$\alpha_{\salg}(\xi)$. 
	
	A closer look reveals that NT theory is both quantitatively and qualitatively inaccurate. 
	From a quantitative viewpoint, in Appendix \ref{sec:Algo_GD} we carry out an
	NT analysis of the GE model, and show that this applies ony to 
	$\alpha< \ualpha_{\sGD}(\xi)$, and  $\ualpha_{\sGD}(\xi)$ can be significantly smaller than
	$\alpha_{\salg}(\xi)$.
	
	From a qualitative point of view,
	 NT theory approximates the cost function $H(\bx) = \|\bF(\bx)\|^2/2$ by
	linearizing $\bF(\bx)$ around the initialization $\bx_0$:
	\begin{align}
	H_{\slin}(\bx) = \frac{1}{2}\big\|\bF_0+\bL(\bx-\bx_0)\big\|^2\, ,\;\;\;
	\bF_0 = \bF(\bx_0) \, ,\;\; \bL = \bD\bF(\bx_0)\, .\label{eq:Linear}
	\end{align}
	Whenever this approximation is accurate it implies convergence of SGD
	to a neighborhood of the solution of $\bF_0+\bL(\bx-\bx_0)=\bzero$  
	which minimizes $\|\bx-\bx_0\|$ \cite{bartlett_montanari_rakhlin_2021}. 
	In particular the global minimizer 
	$\bx^*\in \Sol_{n,d}(0)$ selected by SGD should be 
	insensitive to small perturbation of the initialization. 
	We refer to Appendix \ref{sec:Insensitivity} 
	for formal statements.

	Does this qualitative behavior hold in practice?
	We carry out the following numerical experiment in the AoR model. 
	We initialize SGD at a random point $\bx^0\sim\Unif(\S^{d-1})$ and
	run it twice after reshuffling the data. We thus generate two trajectories
	$\bx^t_1$, $\bx^t_2$, $t\ge 0$, which 
	we refer to as the two `replicas' (borrowing from the statistical physics jargon).
	 We stop SGD whenever $H(\bx_{a}^t)\le n\eps$ (in simulations
	$\eps=10^{-4}$) or we reach a pre-specified maximum number of iterations. 
	We thus obtain two configurations $\bx^*_{1}=\bx^*_{1}(\eps)$, 
	$\bx^*_2 = \bx^*_2(\eps)$ which are near solutions
	when the early stopping criterion is achieved (which,
	empirically, happens with high probability for 
	$\alpha$ sufficiently smaller than $\alpha_{\salg}$).
	
	Figure \ref{fig:Geometry} reports the results of these experiments,
	for $\xi(t) =1+a^2t^3$, $a=0.75$.
	In this case, Theorems \ref{thm:MainAchievability} 
	and \ref{thm:Optimality} predict the algorithmic phase transition
	to occur at
	 $\alpha_{\salg}(\xi)\approx 0.8450$.
	We plot the average angular distance 
	between any of the two near solutions
	$\bx^*_1$, $\bx^*_2$, and the initialization $\bx^0$,
	and the average angular distance between  
	$\bx^*_1$ and $\bx^*_2$:
	\begin{align}
	\myangle_{\sinit}^2(\alpha;n) := \E
	\arccos\<\bx^*_a,\bx^0\>\, ,\;\;\;\;\;\;
	\myangle_{\srepl}^2(\alpha;n) :=\E
	\arccos\<\bx^*_1,\bx^*_2\>\, .
	\end{align}
	We expect  $\myangle_{\sinit}, \myangle_{\srepl}^2 \in[0,\pi/2]$, with
	the  value $\pi/2$ corresponding to  orthogonal vectors.
	The following pattern emerges for large $n,d$ and small stepsize 
	(see also Appendix \ref{app:SensitivityNumer}):
	\begin{itemize}
	\item The distance from initialization $\myangle_{\sinit}^2(\alpha;n)$ grows  with $\alpha$.
	It is close to $0$ for $\alpha$ small, and close to $\pi/2$ when 
	$\alpha$ approaches $\alpha_{\salg}(\xi)$. For $\alpha\approx \alpha_{\salg}(\xi)$, the algorithm stops 
	converging to a global minimum. Finally, $\myangle_{\sinit}^2(\alpha;n)$ is roughly independent 
	of the stepsize (parametrized by the learning rate ${\rm lr}$).
	\item The replicas distance $\myangle_{\srepl}^2(\alpha;n)$ is very small for $\alpha\lesssim 0.3$,
	then grows rapidly and becomes comparable with $\myangle_{\sinit}^2(\alpha;n)$ when  
	$\alpha\approx\alpha_{\salg}(\xi)$. Further, the resulting distance depends strongly in the stepsize.
	The rapid growth in  $\myangle_{\srepl}^2(\alpha;n)$ does not appear to
	be a phase transition in this case. 
	\end{itemize}

\subsection{Theoretical predictions for the two-phase algorithm}

We can show that sensitivity undergoes a phase transition for a randomized
variant of the two-phase algorithm described in Section \ref{sec:AlgoDescr}.
Namely, denoting the output of two independent executions of the algorithm
by $\bx^{\salg}_1$ and $\bx^{\salg}_2$, respectively, and letting
$\dist_{\srepl,n}^2(\alpha) = \|\bx^{\salg}_1-\bx^{\salg}_2\|^2$, for 
$n=\lfloor d\alpha \rfloor$:
	\begin{equation}\label{eq:SensitivityDef}
	\begin{aligned}
	\alpha<\alpha_{\sens}(\xi)  \;\Rightarrow\; 
	\lim_{n\to\infty}\E\big\{\dist_{\srepl,n}^2(\alpha)\big\}
	 = 0\, ,\\
	\alpha> \alpha_{\sens}(\xi)  \;\Rightarrow\; \lim_{n\to\infty}
	\E\big\{\dist_{\srepl,n}^2(\alpha)\big\}>0 \, .
	\end{aligned}
\end{equation}
We will also use the notation $\dist_{\sinit,n}^2(\alpha):=\|\bx^0-
\bx^{\salg}_1\|^2$.

The modified two-phase algorithm proceeds as follows.
We draw an `initialization' $\bx^0\sim\Unif(\S^{d-1})$
and for $q\in Q_{\Delta}:=\{1-\Delta, 1-2\Delta, \dots, \Delta,0\}$, 
we attempt to minimize $H(\bx)$ over
$\bx\in S_{\bx^0,q}:=\{\bx \in \S^{d-1}: \<\bx^0,\bx\>=q\}$. Note that,
by a change of coordinates, this is equivalent to minimizing the Hamiltonian
\begin{align}
H_{\bx^0,q}(\bz) = \frac{1}{2}\|\bF_{\bx^0,q}(\bz)\|^2\, ,
\;\;\; \bF_{\bx^0,q}(\bz) := \bF\big(q\bx^0+\sqrt{1-q^2} \bU_{\bx^0} \bz\big)\, ,
\end{align}
over $\bz\in \S^{d-2}$, where $\bU_{\bx^0}\in\reals^{d\times (d-1)}$ 
denotes the orthogonal complement of $\bx^0$.
We run the two-phase algorithm to minimize $H_{\bx^0,q}(\bz)$ over 
$\bz\in \S^{d-2}$.  Denoting by  $\bx^{\salg}(q)$
the output of the two-phase algorithm at $q\in Q_{\Delta}$, 
we return $\bx^{\salg}=\bx^{\salg}(\hat{q})$, where 
$\hat{q}$ is the largest value in $Q_{\Delta}$ 
for which $\bx^{\salg}(\hat{q})\in\Sol_{n,d}(\eps)$
for a pre-assigned tolerance $\eps$. 

The randomization consists in choosing radomly the Hessian descent step 
$\bv(\bx^k)$. Namely, we let $\bv(\bx^k)$ be uniformly random in the 
eigenspace associated to eigenvalues of the Hessian 
smaller\footnote{For obtaining a polynomial time algorithm, this
step will need to approximated by an efficiently computable one,
but we omit such considerations in this section for simplicity.}
 than $\lambda_{\min}(\nabla^2H(\bx^k)|_{\Ts_{\bx^k}\cup V_L})+d\delta$.
(Of course the whole algorithm is applied to $H$ replaced by $H_{\bx^0,q}$.)

The process $\bF_{\bx^0,q}(\bz)$ is again centered Gaussian with covariance
function $\xi_q(t) = \xi(q^2+(1-q^2) t)$. 
We can therefore apply the theory developed in this section to derive
predictions for the quantities $\dist_{\sinit,n}^2(\alpha)$
and $\dist_{\srepl,n}^2(\alpha)$, which are plotted in 
Fig.~\ref{fig:Geometry}.

Let $u_{\salg}(\alpha,\xi)$
be the algorithmic energy threshold defined in Section \ref{sec:AchievableEnergy},
see Eq.~\eqref{eq:AchievableEnergyDef}, and 
$q_{\sRS}(\alpha;\xi)$ be defined as per Eq.~\eqref{eq:qrsDef}. 
By construction, the trajectories
coincide for the first phase
of the algorithm, i.e. up to the computation of $\bm^L$ in Algorithm
\ref{algo:HD+AMP}. On the other hand,  increments in the second phase
to be roughly orthogonal $\<\bx_1^{k}-\bm^L,\bx_2^{k}-\bm^L\>=o_n(1)$.
We thus obtain the following prediction
\begin{proposition}
Define $q_{\min}=q_{\min}(\alpha,\xi):= 
\min\big\{q: \; u_{\salg}(\alpha;\xi_q)=0\big\}$, and 
\begin{align}
\dist_{\sinit,\infty}^2(\alpha) &= 2-2q_{\min}\, ,\label{eq:DistInitPred}\\
\dist_{\srepl,\infty}^2(\alpha) &= (1-q_{\min}^2)\cdot(2-2q_{\sRS}(\alpha,\xi_{q_{\min}}))_+\, .\label{eq:DistReplPred}
\end{align}
Then, for the modified two-phase algorithm described above, for any
 $\eta_0>0$, there exists $\delta_0$, $L_0>0$, $\eps_0>0$, $\Delta_0>0$
such that for $\delta\in (0,\delta_0)$, $L\ge L_0$, 
$\eps\in (0,\eps_0)$, $\Delta\in (0,\Delta_0)$, we have
\begin{align}
&\lim_{n\to\infty}\P\Big(\big|\dist_{\sinit,n}(\alpha)- \dist_{\sinit,\infty}(\alpha)\big|\ge \eta_0\Big)=0\, ,\\
&\lim_{n\to\infty}\P\Big(\big|\dist_{\srepl,n}(\alpha)- \dist_{\srepl,\infty}(\alpha)\big|\ge \eta_0\Big)=0\, .
\end{align}
\end{proposition}
The asymptotic predictions on the right-hand side of 
Eqs.~\eqref{eq:DistInitPred} and \eqref{eq:DistReplPred} 
 are reported as the curves in Figure \ref{fig:Geometry}.
We note that $\dist_{\srepl,\infty}^2(\alpha)=0$
in an interval $[0,\alpha_{\sens}(\xi)]$, thus confirming
the behavior of Eq.~\eqref{eq:SensitivityDef}. 

While the above result is obtained for a version
of the algorithm in which the Hessian steps are completely randomized,
we expect to remain valid for weaker forms of randomization.

Evaluating Eq.~\eqref{eq:DistReplPred}  for the example 
of Figure \ref{fig:Geometry}, we get
$\alpha_{\sens}(\xi) \approx 0.3506$. We also observe
that the curves $\dist_{\sinit,\infty}^2(\alpha)$,
$\dist_{\sinit,\infty}^2(\alpha)$ are not far from the empirical 
results for SGD. Explaining this agreement is an open problem.

We conclude by observing that the occurrence of a sensitivity phase as
	described in Eq.~\eqref{eq:SensitivityDef} has important practical implications. 
	As an example, perturbative measures (a.k.a. `influence functions') are often used to estimate 
	the effect of a single datapoint in machine learning \cite{park2023trak,kolossovtowards,xialess}. 
	When the algorithm behavior is highly sensitive to small changes, 
	such perturbative arguments have obvious limitations.

%
%************************************************
%
\section{Analysis of Hessian descent}
\label{sec:Algo_HD}

For the sake of clarity, we present the analysis of the 
two phase algorithm, Algorithm \ref{algo:HD+AMP}, in two separate
sections, corresponding to the two phases. In this section, we focus on the 
second phase, namely Hessian descent. 
It will follow from the analysis in the next section that the first 
phase can be skipped  if $\xi'(0)=0$ (i.e. $\bF(\bx)$ does not contain 
terms linear in $\bx$). 

We complete the analysis of the first phase in Section \ref{sec:TwoPhase}. 

The pseudocode for Hessian descent is specified in 
Algorithm \ref{algo:HD} below. This is is simply a special case
of Algorithm \ref{algo:HD+AMP} with $L=0$, and we reproduce it here for the reader's convenience.
In Section \ref{sec:Hessian-FirstOrder}, we show how to implement Hessian descent 
via a first order oracle.

\vspace{0.3cm}

\begin{algorithm}[H]
\SetAlgoLined
\KwData{Couplings $\{\bG^{(k)}\}_{k\ge 0}$, stepsize $\delta$, with $1/\delta\in\naturals$}
\KwResult{Approximate optimizer $\bx^{\sHD}\in \S^{d-1}$}
Initialize $\bx^0= \bzero$, $\bx^1\sim \sqrt{\delta}\cdot \Unif(\S^{d-1})$\;
\For{$k\in\{1,\dots,K:= 1/\delta-1\}$}{
Compute $\bv = \bv(\bx^k)\in \Ts_{\bx^k}$ such that $\|\bv\|_2=1$ and
\begin{align}
\<\bv, \nabla^2H(\bx^k)\bv\> \le \lambda_{\min} (\nabla^2H(\bx^k)|_{\Ts,\bx^k})+d\delta\, ;
\label{eq:ConditionEvector}
\end{align}
$s_k:= \sign(\<\bv(\bx^k),\nabla H(\bx^k)\>)$\;
$\bx^{k+1} = \bx^k-s_k\sqrt{\delta} \, \bv(\bx^k)$\;
}
\KwRet $\bx^{\sHD} = \bx^{K}$\;
\caption{Hessian Descent}\label{algo:HD}
\end{algorithm}

\vspace{0.3cm}

The next theorem bounds the value achieved by the Hessian descent 
algorithm, and is a special case of Theorem \ref{thm:MainAchievability} 
for the case $\xi'(0)=0$. 
\begin{theorem}\label{thm:Hessian}
For $\alpha\in (0,\infty)$, $a,b\in\reals_{\ge 0}$, define
\begin{align}
 Q(m;\alpha,a,b) &:= -\frac{1}{m}+\frac{\alpha b}{1+bm}-a^2m\, ,\\
 z_*(\alpha,a,b) & := -\sup_{m>0}  Q(m;\alpha,a,b)\, .
 \end{align}
 Let $u(\,\cdot\, ;\alpha,\xi):[0,1]\to \reals$ be the unique solution of 
 the ordinary differential equation
 \begin{align}
 \frac{\de u}{\de t}(t) = -\frac{1}{2\alpha}z_*\big(\alpha; \sqrt{2\alpha u(t) \xi''(t)}, \xi'(t)\big)\,,
 \;\;\;\;\; u(0) = \frac{1}{2}\xi(0) \, .\label{eq:EnergyEvolutionHD}
 \end{align}
Then there exists constants $C_0 =C_{0}(\alpha,\xi)$, $\delta_0=\delta_0(\alpha,\xi)>0$ 
depending uniquely on $\alpha,\xi$ 
such that the Hessian descent algorithm, with  stepsize parameter $\delta\le \delta_0$,
 outputs $\bx^{\sHD}\in \S^{d-1}$ such that, with probability converging to one
 as $n,d\to\infty$ ($n/d\to\alpha$)
 \begin{align}
\frac{1}{2n}\big\|\bF(\bx^{\sHD})\big\|^2_2\le u(1;\alpha,\xi)+C_0\delta\, .
 \end{align}
 Further, the algorithm has complexity at most $(C_0\chi_{n,d}/\delta)\log(1/\delta)$, 
 where $\chi_{n,d}$ is the complexity of a single matrix vector multiplication by 
 $\nabla^2H(\bx)$ at a query point $\bx\in\Ball^d(1)$.
 \end{theorem}

We next provide a sketch of the proof  
of Theorem \ref{thm:Hessian} with a detailed presentation deferred to 
next subsections.
The following simple lemma (whose conclusions were anticipated in Eq.~\eqref{eq:DistributionHessian})
 characterizes the joint distribution of $H(\bx)$, $\nabla H(\bx)$, $\nabla^2 H(\bx)$.
\begin{lemma}\label{lemma:HessianDistr}
For a fixed $\bx\in\reals^d$ with $\|\bx\|^2_2=q$,  we have $\bF(\bx) = \sqrt{\xi(q)}\, \bg$,
$\bD\bF(\bx)\bU_{\bx} = \sqrt{\xi'(q)}\, \bZ$, $\bU_{\bx}^{\sT}\nabla^2 F_{\ell}(\bx)
\bU_{\bx} = \sqrt{\xi''(q)}\,\bW_{\ell}$,
where $\bg, (\bW_{\ell})_{\ell\le n}, \bZ$ are mutually independent with
\begin{align}
\bg\sim\normal(0,\id_n)\, ,\;\;\; \bW_{\ell}\sim\GOE(d-1)\, ,\;\;\; \bZ\sim\GOE(n,d-1)\, .
\end{align}
As a consequence, letting $\bcH(\bx)  := \bU_{\bx}^{\sT}\nabla^2 H(\bx)\bU_{\bx}$
be the restriction of the Hessian  to the tangent space, we have
\begin{align}
\bcH(\bx) & = \sqrt{\xi(q)\xi''(q)}\, \|\bg\|_2\bW+\xi'(q)\,\bZ^{\sT}\bZ\, ,\label{eq:HessianDistr}\\
H(\bx)    &= \frac{1}{2}\xi(q)\|\bg\|_2^2\, ,
\end{align}
where
$(\bg,\bW,\bZ)\sim \normal(0,\id_n)\otimes \GOE(d-1)\otimes \GOE(n,d-1)$.
\end{lemma}
This lemma suggests to estimate the energy decrease at step $k$
of Hessian descent, by computing the minimum eigenvalue of 
$\lambda_{\min}(\bcH(\bx))$ at a point $\bx$  with $\|\bx\|_2^2=k\delta$ and  $H(\bx)/n=u$.
If $\bx$ is a point independent of the Gaussian process  $\bF(\,\cdot\,)$,
$\lambda_{\min}(\bcH(\bx))$ turns out to concentrate around $-z_{\#}(t=k\delta)d$ where 
$z_{\#}(t):=z_*(\alpha; \sqrt{2\alpha u \xi''(t)}, \xi'(t)) $.
By summing this energy decrement over $k\in\{,\dots,\lfloor 1/\delta\rfloor-1\}$
and letting $\delta$ be small, this calculation yields the value $u(1;\alpha,\xi)$ of 
Theorem \ref{thm:Hessian}.

At first sight, such a derivation might seem incorrect because $\bx^k$ 
is not independent of $\bcH(\bx^k)$. However, the fast decay of 
the probability of upper deviations
of the minimum eigenvalue allows to establish the claim nevertheless.

 The formula $u(1;\alpha,\xi)$ for the energy achieved by Hessian descent,
 cf. Theorem \ref{thm:Hessian}, is somewhat implicit. The next corollary provides
 simple upper and lower bounds. Its proof follows immediately from Theorem 
 \ref{thm:Hessian} using the bounds on $z_*$ given in Lemma \ref{lemma:RMT-Sharp}, part 3.
 \begin{corollary}\label{coro:Hessian}
 Define 
 \begin{align}
 u_{\sHD,\slb}(\alpha,\xi)& := \frac{1}{2}
 \left(\sqrt{\xi(0)}-\sqrt{\frac{1}{\alpha}}\int_0^1\sqrt{\xi''(s)} \right)_+^2\, ,\\
 u_{\sHD,\sub}(\alpha,\xi)& := \frac{1}{2}
 \left(\sqrt{\xi(0)}-\sqrt{\frac{1-\alpha}{\alpha}}\int_0^1\sqrt{\xi''(s)} \right)_+^2\, .
 \end{align}
 Then the energy achieved by Hessian descent satisfies
 \begin{align}
u_{\sHD,\slb}(\alpha,\xi) \le \plim_{d,n\to\infty}\frac{1}{2n}\big\|\bF(\bx^{\sHD})\big\|^2_2\le u_{\sHD,\sub}(\alpha,\xi)\, .
 \end{align}
 In particular, 
 letting $\alpha_{\sHD}(\xi):= \inf\{\alpha>0: u_{\sHD,\slb}(\alpha,\xi)>0\}$ denote
 the critical point of the algorithm, we have 
 \begin{align}
&\frac{A(\xi)}{1+A(\xi)} \le \alpha_{\sHD}(\xi) \le A(\xi)\, ,\;\;\;\;\
A(\xi) := \left(\int_{0}^1\sqrt{\frac{\xi''(t)}{\xi(0)}}\, \de t\right)^2\, .
\end{align}
 \end{corollary}

For our running example of a pure model $\xi(t) = \xi_0+t^p$, $\xi_0=\gamma_0\log p$
(denoting the corresponding threshold by $\alpha_{\sHD}(\gamma_0,p)$), the last bounds yields
 \begin{align}
&\frac{4(p-1)}{p\xi_0+4(p-1)} \le \alpha_{\sHD}(\gamma_0,p) 
\le \frac{4(p-1)}{p\xi_0}\,  .\label{eq:AlphaHessianBounds}
\end{align}
In particular, for large $p$, we obtain 
\begin{align}
\alpha_{\sHD}(\gamma_0,p) = \frac{4}{\gamma_0\log p}(1+o_p(1))\, .
\end{align}
As shown in Appendix \ref{sec:Algo_GD},
this is substantially better than the guarantee obtained there for gradient descent,
but still far from the 
maximum value of $\alpha$ in which we know that solutions exist, $\alpha <
\alpha_{\slb}(\gamma_0,p) \asymp 1/\gamma_0$, cf. Eq.~\eqref{eq:LB-asymp}.

\subsection{Implementing Hessian descent via a first order oracle}
\label{sec:Hessian-FirstOrder}

By Lemma \ref{lemma:Derivatives}, we can assume that  
$\sup_{\|\bx\|\le 1}\|\nabla^2H(\bx)\|_{\op}\le C_*d$ for some constant $C_*$.
Define $\bA_k:= \bP_{\bx^k}^{\perp}(\bI-\nabla^2H(\bx^k)/(2C_*d))\bP_{\bx^k}^{\perp}$
where $\bP_{\bx^k}^{\perp}$ is the projection orthogonal to $\bx^k$.
Recall that $\lambda_1(\bA_k)\ge \lambda_2(\bA_k)\ge \cdots$ denote 
the eigenvalues of $\bA_k$ in decreasing order. 
We then have $1\ge \lambda_1(\bA_k)\ge \dots \ge \lambda_d(\bA_k)
\ge 0$. Achieving condition \eqref{eq:ConditionEvector} is equivalent 
to 
\begin{align}\label{eq:ConditionApproxPCA}
	\<\bv, \bA_k\bv\> \ge \lambda_{1}(\bA_k)-\delta'\, ,
\end{align}
for $\delta'=\delta/2C_*$. 
Then the proof of 
Theorem \ref{thm:Hessian} outlined above  (in particular, the 
fact that, for each $\eta>0$, the matrix $\bcH(\bx)$ of Eq.~\eqref{eq:HessianDistr}
has with high probability $c(\eta)d$ eigenvalues 
smaller than $-z_*(\alpha,a,b)-\eta$, where $a,b$ are the argument of $z_*$ in
Eq.~\eqref{eq:EnergyEvolutionHD}), 
implies that there exists a constant $c>0$ such that
letting $m=cd$, with high probability
$\lambda_{m}(\bA_k)\ge \lambda_{1}(\bA_k)-\delta'$. 

We can then apply the power method with a random initialization 
$\bu_0\sim\Unif(\S^{d-1})$: 
\begin{align}
\bu^{\ell+1} = \frac{\bA_k\bu^{\ell}}{\|\bA_k\bu^{\ell}\|_2}\, .
\end{align}
A standard argument implies that there exists $\ell_*$ such that, 
with probability $1-2\exp(-c'd)$ over the
random initialization $\bu_0$, setting $\bv=\bu^{\ell_*}$ satisfies
Eq.~\eqref{eq:ConditionApproxPCA}.

The above power iteration procedure provides an implementation 
of Eq.~\eqref{eq:ConditionEvector} using a second order oracle.
We obtain a first order implementation by initializing
$\bv_0\sim\Unif(\S^{d-1})$, and iterating: 
\begin{align}
\bv^{\ell+1} &= \frac{\tilde\bv^{\ell+1}}{\|\tilde\bv^{\ell+1}\|_2}\, ,\\
	\tilde\bv^{\ell+1}&=\bv^{\ell}-
	\frac{1}{2C_*d\eps}
\big[\nabla H(\bx^k+\eps\bv^{\ell})-\nabla H(\bx^k)\big]\, .
\end{align}
Using the Lipschitz continuity of the Hessian 
(see Lemma \ref{lemma:BoundLip}), it is easy to see that,
taking $\eps= n^{-3}$, yields $\|\bv^{\ell}-\bu^{\ell}\|\le n^{-1}$
for all $\ell\le \ell_*$ and all $n,d$ large enough.

\subsection{Random matrix theory estimates}
\label{app:RMT}

Throughout, for $M\le N$,  we let $\bW\sim \GOE(N)$ independent of $\bZ\sim \GOE(M,N)$ and
define 
\begin{align}
\bA =\bA_{M,N} := a\sqrt{N}\, \bW +b \bZ^{\sT}\bZ\, .
\end{align}
\begin{lemma}\label{lemma:RMT-Sharp}
Assume $a,b, \alpha\in\reals_{\ge 0}$, and recall the definition of 
$Q, z_*$ from Theorem \ref{thm:Hessian}, namely
\begin{align}
 Q(m;\alpha,a,b) &:= -\frac{1}{m}+\frac{\alpha b}{1+bm}-a^2m\, ,\\
 z_*(\alpha,a,b) & := -\sup_{m>0}  Q(m;\alpha,a,b)\, .
 \end{align}
 Further, let $S(\,\cdot\, ;\alpha,a,b):\upper\to \complex$ be the only analytic function 
 on the upper half
 plane, such that: $(i)$ $S(\,\cdot\, ;\alpha,a,b) = -1/z+o(1/z)$ as $z\to i\infty$;
 $(ii)$~ $S(z ;\alpha,a,b)$ solves $Q(S;\alpha,a,b) =z$.
 
 Then the following hold almost surely in the limit $M,N\to\infty$ with $M/N\to\alpha\in (0,1]$.
 \begin{enumerate}
\item  
 \begin{align}
 \lim_{M,N\to\infty} \frac{1}{N}\lambda_{\min}(\bA_{M,N}) = -z_*(\alpha,a,b)\, .
\end{align}
\item For $z\in\upper$ (the upper half complex plane)
 \begin{align}
 \lim_{M,N\to\infty} \frac{1}{N}\Tr\big((\bA_{M,N}/N-z\id )^{-1}\big) = S(z ;\alpha,a,b)\, .
\end{align}
\item If $a>0$, $\alpha<1$ then,
\begin{align}
2a\sqrt{1-\alpha}<z_*(\alpha,a,b)< 2a\, .
\label{eq:SimpleZstarBounds}
\end{align}
\item Letting $\hnu_{M,N}(\,\cdot\, ;a,b)$ denote the empirical spectral distribution of
$\bA_{M,N}/N$, there exists a non-decreasing, deterministic function $\nu_0(\, \cdot\, ;a,b):\reals\to\reals_{\ge 0}$ such that
$\nu_0(t;a,b)>0$ for all $a,b\ge 0$ and $z_* = z_*(\alpha,a,b)$, almost surely
\begin{align}
\lim_{M,N\to\infty}\hnu_{M,N}(-z_*+t;a,b)\ge \nu_0(t;a,b)\, .
\end{align}
\item For any $a_0,b_0>0$, $t>t'>0$, we have
\begin{align}
&\lim_{M,N\to\infty}
\inf_{a\in [0,a_0], b\in[0,b_0]} \hnu_{M,N}(-z_*+t;a,b)\ge \nu_{\min}(t';a_0,b_0)\, ,
\label{eq:UniformLBEmpirical}\\
&\nu_{\min}(t';a_0,b_0) := \inf_{a\in [0,a_0], b\in[0,b_0]}\nu_0(t';a,b)>0\, .
\label{eq:LB_numin}
\end{align}
\end{enumerate}
\end{lemma}
\begin{proof}
The asymptotic R-transforms of the random matrices $\bX:=\bW/\sqrt{N}$, $\bY:= \bZ^{\sT}\bZ/N$
are \cite{mingo2017free}
\begin{align}
R_{\bX}(z) = z\, ,\;\;\;\;\; R_{\bY}(z) = \frac{\alpha}{1- z}\, .
\end{align}
Since $\bX$, $\bY$ are asymptotically free, 
\begin{align}
R_{\bA/N}(z)& = a R_{\bX}(a z)+b R_{\bY}(b z)\\
& = a^2 z + \frac{\alpha b}{1- b z}\, .
\end{align}
Therefore, we have $ Q(m;\alpha,a,b) = R_{\bA/N}(-m)-m^{-1}$. Point 2 follows then from the
standard connection between Stieltjes transform and S-transform. 

Point 1 follows from \cite{capitaine2011free}.

For part 3, Eq. \eqref{eq:SimpleZstarBounds} follows from
the  inequalities 
\begin{align}
 -\frac{1}{m}-a^2m < Q(m;\alpha,a,b) <  -\frac{1}{m}+\frac{\alpha }{m}-a^2m\, .
 \end{align}

For part 4, note that $Q(S;\alpha,a,b) = z$ is a third order algebraic equation for the Stieltjes transform $S$, 
with coefficients that depend continuously on $a,b,z$. The equation reduces
to a second order one if $a=0$ or $b=0$.
By the definition of $z_*$,  there exists $r>0$ such that, for $z\in (-z_*,-z_*+r)$ this 
equation has three solutions, of which two are complex conjugates. The imaginary
part of these solutions gives (up to a constant factor) the asymptotic density of 
empirical spectral distribution, which is strictly positive, hence implying the claim (see e.g. \cite[Theorem 2.4.3]{Guionnet}).

Finally, we consider part 5. On the favorable event 
$\cG := \{\|\bW\|^2_{F}\le 2N^2, \|\bZ^{\sT}\bZ\|_F^2\le 2N^3\}$ (which holds with probability
at least $1-\exp(-cN)$), we have
\begin{align}
\Big\|\frac{1}{N}\bA_{M,N}(a_1,b_1)-\frac{1}{N}\bA_{M,N}(a_2,b_2)\Big\|_F\le
4 \sqrt{N} \|(a_1,b_1)-(a_2,b_2)\|_2\, ,
\end{align}
where we noted explicitly the dependence of $\bA_{M,N}$ on parameters $a,b$.
Hence, by the Wielandt-Hoffman inequality, for any Lipschitz function $f:\reals\to\reals$,
the following holds on $\cF$ for  all $a_1,a_2,b_1,b_2\ge 0$:
\begin{align}
\left|\int f(x)\, \hnu_{M,N}(\de x;a_1,b_1)-\int f(x)\, \hnu_{M,N}(\de x;a_2,b_2)\right|
\le 4 \|f\|_{\sLip}\cdot \|(a_1,b_1)-(a_2,b_2)\|_2\, .
\end{align}
Using the function 
\begin{align}
f_{t,\eps}(x) = \begin{cases}
1 & \mbox{ if $x\le -z_*+t-\eps$,}\\
(-z_*+t-x)/\eps& \mbox{ if $-z_*+t-\eps<x<-z_*+t$,}\\
0 & \mbox{ if $-z_*+t\le x$,}
\end{cases}
\end{align}
we get,  for  all $a_1,a_2,b_1,b_2\ge 0$:
\begin{align*}
\hnu_{M,N}(-z_*+t;a_1,b_1)&\ge \int f_{t,\eps}(x)\, \hnu_{M,N}(\de x;a_1,b_1)\\
& \ge  \int f_{t,\eps}(x)\, \hnu_{M,N}(\de x;a_2,b_2) -\frac{4}{\eps} \|(a_1,b_1)-(a_2,b_2)\|_2\\
&\ge \hnu_{M,N}(-z_*+t-\eps;a_2,b_2) -\frac{4}{\eps} \|(a_1,b_1)-(a_2,b_2)\|_2\, .
\end{align*}
Let $S_{\delta}$ be a finite $\delta$-net in $[0,a_0]\times [0,b_0]$.
Using the last inequality and the result at point 4 on $S_{\delta}$,
alongside Borel-Cantelli (which implies that $\cG$ holds eventually almost surely),
we get
\begin{align*}
\lim_{M,N\to\infty}
\inf_{a\in [0,a_0], b\in[0,b_0]} \hnu_{M,N}(-z_*+t;a,b)&\ge 
\min_{(a,b)\in S_{\delta}} \nu_0(t-\eps;a,b)-\frac{4\delta}{\eps}\, .
\end{align*}
The lower bound \eqref{eq:UniformLBEmpirical} follows by taking $\delta\to 0$.

Finally to prove Eq.~\eqref{eq:LB_numin}, i.e. 
$\nu_{\min}(t';a,b) >0$ strictly, we note that
\begin{align}
 \nu_0(t';a,b)\ge \nu_0(f_{t',\eps};a,b):= \int f_{t',\eps}(x)\, \nu_0(\de x;a,b)\, .
 \end{align}
 Further, $(a,b)\mapsto \nu_0(f_{t',\eps};a,b)$ is Lipschitz continuous
 by point 4 and the argument given above (with Lipschitz modulus $4/\eps$), and 
 $\nu_0(f_{t',\eps};a,b)\ge \nu_0(t'-\eps;a,b)>0$ for any $a,b$.
Therefore, by taking $\eps\in (0,t')$, we get $\nu_{\min}(t';a,b) >\inf_{a,b\in [0,a_0]\times[0,b_0]}
\nu_0(f_{t',\eps};a,b)>0$.
\end{proof}

\begin{lemma}\label{lemma:RMT-LD}
Let $a_0,b_0\in\reals_{\ge 0}$, $\alpha\in(0,1]$ and, for $(a,b)\in [0,a_0]\times[0,b_0]$
let  $z_*=z_*(\alpha;a,b)$ be defined as
in Lemma \ref{lemma:RMT-Sharp}. Let $M=M(N)$ be a sequence such that $M/N\to\alpha$ as $N\to\infty$. Then for any fixed $a_0,b_0,t,c>0$, there exists 
$C(t) = C(t,c,a_0,b_0,\alpha)$ such that, if $(a,b)\in [0,a_0]\times[0,b_0]$ then for large enough $N$,
\begin{align}
\prob\big(\lambda_{\min}(\bA_{M,N})\ge N(-z_*+t)\big)\le C(t)\, e^{-N^2/C(t)}\, .
\end{align}
\end{lemma}
\begin{proof}
%The case $b=0$ is standard \cite{Guionnet}, so we can and will assume, without loss of generality, $b=1$
Throughout the proof, we denote by $C$ constants that depend on
$a_0,b_0,\alpha,t,c$, and possibly other quantities as indicated, but not on $a,b$.

Define $\bW/\sqrt{N}:=\frac{1}{\sqrt2}(\bG+\bG^{\sT})$ and $\bS :=\bZ/\sqrt{N}$ so that $N^{-1}\bA_{M,N} = 
\frac{a}{\sqrt2}(\bG+\bG^{\sT})+b\bS^{\sT}\bS$,
and $\bG,\bS$ are matrices containing i.i.d. random variables $\normal(0,1/N)$.
Further, for $\eps>0$, define
\begin{align}
\bB^{\eps} & := \frac{a}{\sqrt{2}}(\bG+\bG^{\sT})+ b \bM^{\eps}(\bS)\, ,\\
\bM^{\eps}(\bS)&:=  \bS^{\sT}(\id+\eps\bS\bS^{\sT})^{-1}\bS\, .
\end{align}
Since $N^{-1}\bA_{M,N} \succeq \bB^\eps$, it is sufficient to prove the following claim:
\begin{quote}
{\bf Claim.} For any $a_0,b_0,\alpha,t,c >0$, there exists $\eps,C\ge 0$ such that 
$\prob\big(\lambda_{\min}(\bB^{\eps})\ge -z_*+t\big)\le C\, \exp(-N^2/C)$
for all  $(a,b)\in [0,a_0]\times[0,b_0]$ and large $N$.
\end{quote}

In order to prove this claim, we view $\bB^{\eps}$ as a function of the $N^2+NM$ random variables $(G_{ij},S_{ij})$.
Note that 
\begin{align}
\|\bB^\eps(\bG_1,\bS_1)- \bB^\eps(\bG_2,\bS_2)\|_F\le&2a \|\bG_1-\bG_2\|_{F}  + b\|\bM^{\eps}(\bS_1)-\bM^{\eps}(\bS_2)\|_{F}\, . 
\end{align}
Further 
\begin{align*}
 \|\bM^{\eps}(\bS_1)-\bM^{\eps}(\bS_2)\|_{F} 
 &\overset{(a)}{=}\frac{1}{\eps}\big\|(\id+\eps\bS_1^{\sT}\bS_1)^{-1}-(\id+\eps\bS_2^{\sT}\bS_2)^{-1}\big\|_F\\
 & \overset{(b)}{=}\big\|(\id+\eps\bS_1^{\sT}\bS_1)^{-1}(\bS_1^{\sT}\bS_1-\bS_2^{\sT}\bS_2)(\id+\eps\bS_2^{\sT}\bS_2)^{-1}\big\|_F\\
 & \overset{(c)}{\le} \big\|(\id+\eps\bS_1^{\sT}\bS_1)^{-1}\bS_1^{\sT}(\bS_1-\bS_2)\big\|_F+
  \big\|(\id+\eps\bS_2^{\sT}\bS_2)^{-1}\bS_2^{\sT}(\bS_1-\bS_2)\big\|_F\\
  &\le \left(\big\|(\id+\eps\bS_1^{\sT}\bS_1)^{-1}\bS_1^{\sT}\big\|_{\op}+
  \big\|(\id+\eps\bS_2^{\sT}\bS_2)^{-1}\bS_2^{\sT}\big\|_{\op}\right)\big\|\bS_1-\bS_2\big\|_F\\
  &\overset{(d)}{\le} \frac{1}{\sqrt{\eps}}\big\|\bS_1-\bS_2\big\|_F\, .
 \end{align*}
 The equality (a) can be proved by using the singular value decomposition of $\bS_i$ to show that the two matrices whose Frobenius norm is computed define the same bilinear form. To verify (b) add and subtract $\id$ from the middle term. For (c) add and subtract $\bS_1^{\sT}\bS_2$ from the middle term in the left-hand side of the inequality and use the fact that $\id+\eps\bS_i^{\sT}\bS_i\succeq \id$.
 Finally, for (d) we use the elementary inequality $x/(1+\eps x^2)\le \eps^{-1/2}/2$.
 
 Recall that, by \cite[Lemma 2.3.1]{Guionnet}, for any $L$-Lipschitz function
 $f:\reals\to\reals$, $N^{-1}\sum_{i=1}^Nf(\lambda_i(\bB^{\eps}))$ is $\sqrt{2/N}\cdot L$
 Lipschitz function of $\bB^{\eps}$. Therefore, by the above and Gaussian concentration,
 for any such function, any $u\ge 0$, and any $(a,b)\in [0,a_0]\times[0,b_0]$
 \begin{align}
&\prob\left(
\left|\cF(\bB^{\eps})-\E\Big\{\cF(\bB^{\eps})\Big\}
\right|\ge u\right) \le 2\, \exp\Big(-\frac{\eps N^2 u^2}{C(a_0,b_0)L^2}\Big)\, ,\label{eq:ConcBeps}\\
&\cF(\bB^{\eps}):=\frac{1}{N}\sum_{i=1}^Nf(\lambda_i(\bB^{\eps})) \, ,
\end{align}
where we can take $C(a_0,b_0)= C_0(a_0^2+b_0^2)$ for a suitable numerical constant $C_0>0$.

We next take 
\begin{align}
f(x) = \begin{cases}
0 & \mbox{if $-z_*+t\le x$,}\\
(2/t)(-x-z_*+t)  & \mbox{if $-z_*+t/2<x<-z_*+t$,}\\
1&  \mbox{if $x\le -z_*+t/2$.}
\end{cases}
\end{align}
We have
\begin{align}
\frac{1}{N}\sum_{i=1}^Nf(\lambda_i(\bB^{\eps})) & \ge 
\frac{1}{N}\sum_{i=1}^{N}f(\lambda_i(\bB^{0}))-
\frac{2}{t}\cdot\frac{1}{N}\sum_{i=1}^{N}|\lambda_i(\bB^{0})-\lambda_i(\bB^{\eps})|\\
& \stackrel{(a)}{\ge} 
\frac{1}{N}\sum_{i=1}^{N}\bfone\big\{\lambda_i(\bB^{0})\le -z_*+t/2\big\}\;-
\frac{2b_0}{t\sqrt{N}}\|\bB^{0}-\bB^{\eps}\|_F\\
& \stackrel{(b)}{\ge}  \nu_{\min}(t/3;a_0,b_0) - \frac{2b_0}{t\sqrt{N}} \big\|\eps(\bS^{\sT}\bS)^2(\id+\eps\bS^{\sT}\bS)^{-1}\big\|_F\\
& \ge  \nu_{\min}(t/3;a_0,b_0)  - \frac{2b_0}{t} \eps\|\bS\|_{\op}^4\,.
\end{align}
Here in $(a)$ we used Wielandt-Hoffman and $(b)$ holds eventually almost surely, by part 5 of Lemma
\ref{lemma:RMT-Sharp}. Recall that $\|\bS\|_{\op}\le 2+\sqrt{\alpha}\le 3$ with probability $1-\exp(-cN)$.
Hence, taking $\eps =  t\nu_{\min}(t/3;a_0,b_0)/( 4\cdot 3^4 b_0)$,
 we get that,  with high probability
\begin{align}
\frac{1}{N}\sum_{i=1}^Nf(\lambda_i(\bB^{\eps})) & \ge  \frac{1}{2}\nu_{\min}(t/3;a_0,b_0) \, .
\end{align} 
 Since $f$ is $2/t$ Lipschitz, Eq.~\eqref{eq:ConcBeps} implies
 \begin{align}
\prob\left(
\frac{1}{N}\sum_{i=1}^Nf(\lambda_i(\bB^{\eps}))\le  \frac{1}{2}\nu_{\min}(t/3;a_0,b_0)-u\right) &
\le 2\, 
\exp\Big(-\frac{\eps N^2 t^2u^2 }{C(a_0,b_0)}\Big)\\
& \le 2\, 
\exp\Big(-\frac{N^2 t^4\nu_{\min}(t/3;a_0,b_0) u^2 }{C'(a_0,b_0)}\Big)\, ,
\end{align}
and the proof of the claim follows by noting that 
 \begin{align}
\prob\left(\lambda_{\min}(\bB^{\eps})\ge -z_*+t\right)\le 
\prob\left(
\frac{1}{N}\sum_{i=1}^Nf(\lambda_i(\bB^{\eps}))\le 0\right)\, .
\end{align}
\end{proof}

\begin{lemma}\label{lemma:RMT-Hessian}
Assume we let $n,d\to\infty$ with $n/d\to\alpha\in (0,1]$. %Assume $c_0d \le n\le C_1d$ for some constants $c_0,C_1>0$, and 
Recall
the definition of Hamiltonian $H(\bx):=\|\bF(\bx)\|_2^2/2$, and that $\Ts_{\bx}$
denotes the tangent space to the sphere of radius $\|\bx\|_2$ at $\bx$. 

Let $\lambda(\bx):= \lambda_{\min}(\nabla^2H(\bx)|_{\Ts,\bx})$
and define
\begin{align}
z_0(\bx): = -z_*(\alpha;\sqrt{2\alpha \xi''(\|\bx\|^2)H(\bx)/n},\xi'(\|\bx\|^2))\, .
\end{align}
Then, for any $\eps>0$ there exists a constant $C(\eps)>0$ such that for large enough $n$ and $d$,
\begin{align}
\prob\Big(\forall \bx\in\Ball^d(1): \lambda(\bx)\le  d(z_0(\bx)+\eps)
 \Big)\ge 1-C(\eps)\, e^{-d/C(\eps)}\, .
\end{align}
\end{lemma}
\begin{proof}
 Lemmas \ref{lemma:HessianDistr} and \ref{lemma:RMT-LD} imply that 
for any $\eps, C_0>0$, there exists a constant $C(C_0,\eps)>0$ (depending on $\xi$
as well)  such that, for large $n,d$
for any $\bx\in\Ball^d(1)$
\begin{align}
\prob\big(\lambda(\bx)\ge d(z_0(\bx)+\eps);H(\bx)\le C_0 d\big)\le C(C_0,\eps)\, e^{-2d^2/C(C_0,\eps)}\, .
\end{align}

Applying Proposition \ref{propo:MaxValue}, Lemma \ref{lemma:BoundLip}, and Weyl's inequality,
we can work on the event 
\begin{align*}
\cG:=\Big\{ &\max_{\bx\in\Ball^d(1)}H(\bx)\le C_0d,\;
\max_{\bx\in\Ball^d(1)}\|\bD\bF|_{\Ts,\bx}\|_{\op}\le C_0\sqrt{d},\;  \Lip(H;\Ball^d(\bzero;1))\le C_0\, d,\\
&\Lip(\bD\bF;\Ball^d(1))\le C_0\, d,\; 
 \Lip(\lambda;\Ball^d(1))\le C_0\, d
\Big\}\, ,
\end{align*}
for $C_0$ a sufficiently large constant (dependent on $\xi$).
Indeed $\prob(\cG)\ge 1-C_1e^{-d/C_1}$ for a certain constant $C_1$ and large $n,d$.
Further, on the same event $\Lip(z_0;\Ball^d(1))\le C_0$ (eventually
enlarging $C_0$).

Let $N^d(\delta_d)$ be a $\delta_d$-net in $\Ball^d(1)$, with $\delta_d=1/d$.
Then for large $n,d$,
\begin{align*}
\prob\big(\exists \bx \in\Ball^d(1):
\; \lambda(\bx)\ge & d(z_0(\bx)+\eps);\cG\big)\le \prob\big(\exists \bx \in N^d(\delta_d):
\; \lambda(\bx)\ge d(z_0(\bx)+\eps) -C_0'd\delta_d ;\cG\big)\\
& \le |N^d(\delta_d)|\, \max_{\bx \in N^d(\delta_d)}\prob\big(
\; \lambda(\bx)\ge d(z_0(\bx)+\eps/2) ; \, H(\bx)\le C_0 d\big)\\
& \le \exp\Big(C_{1}d\log d-\frac{d^2}{C(C_0,\eps/2)}\Big)\, ,
\end{align*}
which proves the desired claim.
\end{proof}

\subsection{Proof of Theorem \ref{thm:Hessian}}
\label{app:HessianProof}

Recall the definition of $z_0(\bx)$ in Lemma \ref{lemma:RMT-Hessian}
and define the event $\cE(\eps)$ by
\begin{align}
\cE(\eps) := \Big\{\forall \bx\in\Ball^d(1):& \; \lambda_{\min}(\nabla^2H(\bx)|_{\Ts,\bx})\le  d(z_0(\bx)+\eps);\\
&\;\;\Lip(H;\Ball^d(1))\le C_0d;\,  \Lip(\nabla^2H;\Ball^d(1))\le C_0d\, \Big\}\, .\nonumber
\end{align}
By Lemma \ref{lemma:BoundLip} and Lemma \ref{lemma:RMT-Hessian} we can choose $C$
sufficiently large, so that, for any $\eps>0$, $\prob(\cE(\eps)) \ge 1-C(\eps)\exp(-d/C(\eps))$ for large $n,d$.

Consider the sequence $\bx^k$ produced by the Hessian descent algorithm.
 By Taylor's theorem, on the event $\Lip(\nabla^2H;\Ball^d(1))\le C_0d$, we have
 \begin{equation}\label{eq:HdecreamentHesian}
 \begin{aligned}
H(\bx^{k+1})&\le H(\bx^k)-s_k\sqrt{\delta} \, \<\bv(\bx^k),\nabla H(\bx^k)\>
+\frac{1}{2}\delta \, \<\bv(\bx^k),\nabla^2 H(\bx^k)\bv(\bx^k)\>+ C_0d\,\delta^{3/2}\\
& \le H(\bx^k) + \frac{1}{2}\delta\lambda_{\min}(\nabla^2 H(\bx^k)|_{\Ts,\bx^k})+ 2C_0d\delta^{3/2}\,,
 \end{aligned}
\end{equation}
where in the second line we used the definition of $s_k$ and $\bv(\bx^k)$. 
We will hereafter work on the event $\cE(\eps)$, with $\eps=\delta^{3/2}$. %,
%whence, defining $u_k:=H(\bx^k)/n$, $t_k=k\delta$, 
We thus have that 
\begin{align}
H(\bx^{k+1})& \le H(\bx^k) -\frac{1}{2}\delta d z_0(\bx_k) + 3C_0d\delta^{3/2}\,.\label{eq:BasicHessianRecursion}
\end{align}
For $t=k\delta\in[0,1]$ with integer $k$ define $U(t):=H(\bx^{t/\delta})/n$, and on $(k\delta,(k+1)\delta)$ define $U(t)$ by linearly interpolating the two end values.
Defining $[t]_{\delta}:=\delta \lfloor /\delta\rfloor$
and 
\begin{align}
\Phi(u,t):=-\frac{1}{2\alpha}z_*\big(\alpha; \sqrt{2\alpha \xi''(t)u}, \xi'(t)\big) \, ,
\end{align}
we can rewrite Eq.~\eqref{eq:HdecreamentHesian} as (excluding the points  $t=k\delta$)
\begin{align}
\frac{\de U}{\de t}(t)& \le \Phi\big(U([t]_{\delta}),[t]_{\delta}\big) + C'\delta^{3/2}\,,\\
U(0) & \le \frac{1}{2}\xi(0)+\delta\, .
\end{align}
where the bound in the last line holds wth high probability since $H(\bzero)/n$ 
concentrates around its expectation. 
Since $u\mapsto \Phi(u,t)$ is a Lipschitz function of $u$ (a fact that follows from 
Weyl's inequality), for all $\delta$ small enough, $U((k+1)\delta)$ is a monotone increasing function
of $U(k\delta)$, and hence it follows that $U(t)\le U_*(t)$ where
\begin{align}
\frac{\de U_*}{\de t}(t)& = \Phi\big(U_*([t]_{\delta}),[t]_{\delta}\big) + C'\delta^{3/2}\, ,\\
U_*(0) & =\frac{1}{2}\xi(0)+\delta\, .
\end{align}
Finally, letting $u(t)= u(t ;\alpha,\xi)$ denote the solution of the ODE
\eqref{eq:EnergyEvolutionHD}, and using the fact that $t\mapsto \Phi(u,t)$ is a Lipschitz
 function of $t$, we obtain $|U_*(t)-u(t)|\le C''\sqrt{\delta}$ by standard 
 discretization arguments for ODEs.

%
%***********************************************
%
\section{Analysis of the two-phase algorithm:\\
 Completing the proof of Theorem \ref{thm:MainAchievability}}
\label{sec:TwoPhase}

Our analysis of the first phase of the algorithm, will leverage the connection
with approximate message passing (AMP) algorithms and known results about the latter 
\cite{bayati2011dynamics,el2021optimization}.
Namely, we will compare the iterated $(\bm^{\ell})_{\ell\ge 0}$, $(\bh^{\ell})_{\ell\ge 0}$ in the 
first phase with those produced by the AMP iteration which we next define.
We let, for $\ell\ge 0$,
 \begin{align}
\hbh^{\ell+1}&= \frac{1}{\sqrt{n}}\bF(\hbm^{\ell})- \gamma B_\ell\hbh^{\ell-1}\, ,\label{eq:AMP1}\\
\hbm^{\ell+1}& = \frac{\gamma}{\sqrt{d}}\bD\bF(\hbm^{\ell})^{\sT}\hbh^{\ell}-
\gamma C_{\ell} \hbm^{\ell-1}-\gamma^2D_{\ell}\hbm^{\ell-1}\, .\label{eq:AMP2}
 \end{align}
 with initialization 
 \begin{align}
 \hbm^0 = \hbm^{-1}= \hbh^{0} = \bzero\, .\label{eq:AMP_init}
 \end{align}
 Further, $\gamma$ is a (non-random) constant to be fixed in the course of the proof and
 $B_{\ell}, C_{\ell}, D_{\ell}$ are given by
 \begin{align}
B_{\ell} & := \frac{1}{\sqrt{\alpha}}\xi'\big(\<\hbm^{\ell},\hbm^{\ell-1}\>\big)\, ,\label{eq:Onsager_1}\\
C_{\ell} & := \sqrt{\alpha}\, \xi'\big(\<\hbm^{\ell},\hbm^{\ell-1}\>\big)\, ,\\
D_{\ell} & :=  \xi''\Big(\<\hbm^{\ell},\hbm^{\ell-1}\>\Big) \, 
\<\hbh^{\ell},\hbh^{\ell-1}\>\, .\label{eq:Onsager_3}
 \end{align}

We next state a theorem that characterizes this AMP algorithm.
\begin{theorem}\label{thm:AMP}
Assume $\xi(0),\xi'(0)>0$, $\alpha\in (0,1)$,
and let $q_{\sRS}(\xi)$, $q_0(\alpha,\xi)$, and $u_{\sRS}(q,\alpha,\xi)$
be defined as in Section \ref{sec:AchievableEnergy}.
For $q\in (0,q_{\sRS}(\xi))$, define $\gamma_*(q,\alpha,\xi)$ via
\begin{align}
\gamma_*(q,\alpha,\xi)  = -\sqrt{\frac{\uq}{\xi(\uq)\xi'(\uq)}}\, ,\;\;\;\;
\uq:= q\wedge q_0(\alpha,\xi)\, .
\end{align}

Then the first phase of Algorithm \ref{algo:HD+AMP},
with input parameter $\gamma=\gamma_*(q,\alpha,\xi)$, outputs $\bm^L$ such that
\begin{align}
&\Big|\plim_{n,d\to\infty} \frac{1}{d}\big\|\hbm^{L}\big\|^2-q\wedge q_0(\alpha,\xi)\Big|\le C e^{-L/C}\, ,
\label{eq:AMP-Result-1}\\
&\Big|\plim_{n,d\to\infty} \frac{1}{2n}\big\|\bF(\hbm^{L})\big\|^2-u_{\sRS}(q,\alpha,\xi)
\Big|\le C e^{-L/C}\, .\label{eq:AMP-Result-2}
\end{align}
\end{theorem}
For the proof of Theorem \ref{thm:MainAchievability}, we will
choose $q\in (0,q_{\sRS}(\xi)\wedge 1)$.
However, the statement above makes sense and holds also for $q\ge 1$.

Notice that $q\mapsto V(q):=\xi(q)\xi'(q)/q$ is a strictly convex function on $(0,q_{\max})$, where $q_{\max}\in [1,\infty]$
is the maximum radius of convergence of $\xi$ (because $\xi(q)$
is real analytic with non-negative coefficients, and hence so is $(\xi(q)\xi'(q)-\xi(0)\xi'(0))/q$). 
Further, since $\xi(0),\xi'(0)>0$, we have $V(q)\uparrow\infty$ as $q\downarrow 0$. 
Therefore $V(q)$ has a unique minimum at $q_{\sRS}$.
Further $q\mapsto g(q):= q\xi'(q)/\xi(q)$ is strictly increasing with $g(0)=0$, 
$\lim_{q\to\infty}g(q)=\infty$. Hence $q_0(\alpha,\xi)$ is well defined as well.

\begin{proposition}\label{propo:EquivalenceAMP}
Consider the AMP iteration of Eqs.~\eqref{eq:AMP1}, \eqref{eq:AMP2} with initialization
given by Eq.~\eqref{eq:AMP_init}, and the gradient-based algorithm of 
Eq.~\eqref{eq:Ortho_2_maintext}.
Then, for any $\ell\ge 0$, we have 
\begin{align}
\plim_{n,d\to\infty}\|\bm^{\ell}-\hbm^{\ell}\|=0\, ,\;\;\;\;
\plim_{n,d\to\infty}\|\bh^{\ell}-\hbh^{\ell}\|=0\, .\label{eq:EquivalenceAMP}
\end{align}
As a consequence, Eqs.~\eqref{eq:AMP-Result-1} and \eqref{eq:AMP-Result-2} 
also hold for with $\hbm^L$ replaced by $\bm^L$.
\end{proposition}
The proofs of Theorem \ref{thm:AMP} and Proposition \ref{propo:EquivalenceAMP}
are given in Section \ref{sec:ProofThmAMP} and 
\ref{sec:ProofPropEquivalenceAMP}, respectively. They are based on the following state evolution characterization
of the AMP iteration of Eqs.~\eqref{eq:AMP1}, \eqref{eq:AMP2}.
We use these results 
to prove Theorem \ref{thm:MainAchievability} in Section \ref{app:TP2}.
\begin{proposition}\label{propo:SE}
Consider the AMP iteration of Eqs.~\eqref{eq:AMP1}, \eqref{eq:AMP2} with initialization
given by Eq.~\eqref{eq:AMP_init}, and Onsager coefficients in Eqs.~\eqref{eq:Onsager_1} to
\eqref{eq:Onsager_3}. Let $(M_{\ell}: \ell\ge 0)$, $(H_{\ell}: \ell\ge 0)$ 
be two independent centered Gaussian processes with covariances 
$\bQ^M = (Q^M_{k,l})_{k,l\ge 0}$, $\bQ^H = (Q^H_{k,l})_{k,l\ge 0}$  defined recursively via
\begin{align}
Q^H_{k+1,l+1}= \xi(Q^M_{k,l})\, ,\;\;\;\;\;  Q^M_{k+1,l+1} = \gamma^2 
\xi'(Q^M_{k,l})Q^H_{k,l}\, .\label{eq:StateEvolution}
\end{align}
with initialization $Q^{M}_{0,\ell}= Q^{H}_{0,\ell}= 0$ for all $\ell\ge 0$.

Then, for any $L$, and any locally Lipschitz function $\psi:\reals^{L+1}\to\reals$,
with $|\psi(\bx)|\le C(1+\|\bx\|_2^2)$, we have
\begin{align}
&\plim_{n,d\to\infty} \frac{1}{n}\sum_{i=1}^n \psi(\sqrt{n}h^0_i,\dots, \sqrt{n}h_i^L) =
\E\big\{\psi(H_0,\dots, H_L)\big\}\, ,\\
&\plim_{n,d\to\infty} \frac{1}{d}\sum_{i=1}^d \psi(\sqrt{d}m^0_i,\dots, \sqrt{d}m_i^L) =
\E\big\{\psi(M_0,\dots, M_L)\big\}\, .
\end{align}
\end{proposition}
We prove Proposition \ref{propo:SE} in Section \ref{sec:ProofSE}.
%
%***************
%
\subsection{Proof of Theorem \ref{thm:AMP}}
\label{sec:ProofThmAMP}
By construction, for any $k\ge 0$, we have that $Q^{M}_{\ell,\ell+k} =Q^{H}_{\ell,\ell+k} =q_{\ell}$,
where the sequence $(q_{\ell})_{\ell\ge 0}$ satisfies
\begin{align}
q_{\ell+1} = \gamma^2\, \xi(q_{\ell})\xi'(q_{\ell})\, ,\;\;\;\;\; q_{0} = 0\, .\label{eq:Q_Rec}
\end{align}
Fixed points of this map are solutions of the equation 
\begin{align}
\frac{1}{\gamma^2} = V(q)\, ,\;\;\;\; V(q):=\frac{\xi(q)\xi'(q)}{q}\, .\label{eq:FixedPointQ}
\end{align}
We have $\xi(q)\uparrow\infty$ as $q\uparrow\infty$ unless $\xi$ is linear, in which
case $V(q)$ is monotone decreasing, whence
 $q_{\sRS}<\infty$ if and only if $\xi$ is nonlinear.
In the nonlinear case, this implies that the fixed point equation \eqref{eq:FixedPointQ} has
 two solutions
 $q_1(\gamma)<q_{\sRS}<q_{2}(\gamma)$ if $\gamma^2<1/V_*:=1/\min_{q\in(0,q_{\max})} V(q)$
 and no solution for $\gamma^2>1/V_*$.
 In the linear case, we have only one solution $q_1(\gamma)<q_{\sRS}=\infty$.

 Since $q_1(\gamma)$ decreases continuously from $q_{\sRS}$ to $0$
 as $\gamma$ decreases from $1/\sqrt{V_*}$ to $0$, for any $q\in (0,q_{\sRS})$
 we can select $\gamma = \gamma_0(q,\xi)$, where
 \begin{align}
 \gamma_0(q,\xi)  := -\sqrt{\frac{q}{\xi(q)\xi'(q)}}\, 
 \end{align}
 and this yields that the smallest fixed point of \eqref{eq:Q_Rec} coincides with $q$.
 Further, $V'(q)<0$ (again by strict convexity). Letting 
 $f(r):=\gamma^2\, \xi(r)\xi'(r)$, this implies  $f'(q)<1$ strictly. Since $f$ is 
 convex non-decreasing, we conclude that there exists a constant $C>0$ 
 (depending on $\xi,q$) such that, if $\gamma = \gamma_0(q,\xi)$, then
 \begin{align}
 q-C\, e^{-\ell/C}\le q_{\ell}\le q\;\;\;\; \forall \ell\, .
 \end{align}
 Notice that this implies that, for all $\ell,k\ge 0$
 (the constant $C$ will change from line to line):
 \begin{align}
 \big|Q^M_{\ell,\ell+k}-q\big|\le Ce^{-\ell/C}\, , \;\;\;\;\; 
 \big|Q^H_{\ell,\ell+k}-\xi(q)\big|\le Ce^{-\ell/C}\, .
 \end{align}

 By Proposition \ref{propo:SE}, this implies that,  for all $\ell,k\ge 0$
 \begin{align}
\plim_{n,d\to\infty} \big|\|\hbm^{\ell}\|^2_2-q\big|\le Ce^{-\ell/C} 
 &&\plim_{n,d\to\infty} \big\|\hbm^{\ell}-\hbm^{\ell+k}\big\|_2^2\le Ce^{-\ell/C}\, ,\label{eq:ConvSE1}\\
\plim_{n,d\to\infty} \big|\|\hbh^{\ell}\|^2_2-\xi(q)\big|\le Ce^{-\ell/C}  
&& \plim_{n,d\to\infty} \big\|\hbh^{\ell}-\hbh^{\ell+k}\big|_2^2\le Ce^{-\ell/C}\, ,\label{eq:ConvSE2}
 \end{align}
 and therefore,
 \begin{align}
 \plim_{n,d\to\infty} \Big|B_{\ell} -\frac{1}{\sqrt{\alpha}}\xi'(q)\Big| &\le Ce^{-\ell/C}\, ,\\
 \plim_{n,d\to\infty} \big|C_{\ell} -\sqrt{\alpha}\xi'(q)\big| &\le Ce^{-\ell/C}\, ,\\
 \plim_{n,d\to\infty} \big|D_{\ell} -\xi(q)\xi'(q)\big| &\le Ce^{-\ell/C}\, .
 \end{align}
 Therefore, using Eq.~\eqref{eq:AMP1}, we get
 \begin{align*}
 \plim_{n,d\to\infty}
 \frac{1}{n}\|\bF(\hbm^{\ell})\|^2
 & =\plim_{n,d\to\infty}
\|\hbh^{\ell+1} + \gamma B_\ell\hbh^{\ell-1}\|^2\\
 & = \plim_{n,d\to\infty}\big(1+\gamma_0(q,\xi)B_{\ell}\big)^2\cdot \|\hbh^{\ell}\|^2_2 +O(e^{-\ell/C})\\
 & = \Big(1+\gamma_0(q,\xi)\frac{1}{\sqrt{\alpha}}\xi'(q)\Big)^2\xi(q)+ O(e^{-\ell/C})\\
 & = \Big(\sqrt{\xi(q)}-\sqrt{\frac{1}{\alpha}q\xi'(q)}\Big)^2+ O(e^{-\ell/C})\, .
 \end{align*}
The claim of the theorem now follows by choosing
 \begin{align}
 \gamma = \gamma_*(q,\alpha,\xi)  = \gamma_0(q\wedge q_0(\alpha),\xi)\, .
 \end{align}

While not needed for the proof of the theorem we also note that our analysis implies that 
\begin{align}
\lim_{n,d\to\infty}\frac{1}{n}\big\|\Proj_{\Ts,\hbm^L}\nabla H(\hbm^L)\big\|_2 \le 
C\, e^{-L/C}\, .\label{eq:HClaimFP}
\end{align}
(Recall that the typical scale of $\|\nabla H(\hbm)\|_2$ at a non-random point
$\hbm$ is $\Theta(n)$.)
To prove this claim, recall that $\nabla H(\hbm^{\ell}) = \bD\bF(\hbm^{\ell})^{\sT}\bF(\hbm^{\ell})$
and therefore, using Eqs.~\eqref{eq:AMP1} and \eqref{eq:AMP2}
\begin{align*}
\frac{1}{n}\nabla H(\hbm^\ell) &=   
\frac{1}{\sqrt{n}}\bD\bF(\hbm^{\ell})^{\sT}\big(\hbh^{\ell+1} + \gamma B_\ell\hbh^{\ell-1}\big)\\
& = \frac{1}{\sqrt{n}}(1+\gamma B_\ell)\bD\bF(\hbm^{\ell})^{\sT}\hbh^{\ell}+ \be^{\ell}_1\\
& = \frac{1}{\gamma\sqrt{\alpha}}(1+\gamma B_\ell)
\big(\hbm^{\ell+1}+\gamma C_{\ell} \hbm^{\ell-1}+\gamma^2D_{\ell}\hbm^{\ell-1}\big)+ \be^{\ell}_1\\
 &= \frac{1}{\gamma\sqrt{\alpha}}(1+\gamma B_\ell)
\big(1+\gamma C_{\ell}+\gamma^2D_{\ell}\big)\hbm^{\ell}+ \be^{\ell}_1+\be^{\ell}_2\, ,
\end{align*}
where, by Eqs.~\eqref{eq:ConvSE1}, \eqref{eq:ConvSE2}, we have 
\begin{align*}
\plim_{n,d\to\infty}\Big\{\|\be^{\ell}_1\|_2+\|\be^{\ell}_2\|_2 \Big\}\le C\, \exp(-\ell/C)\, .
\end{align*}
This immediately implies Eq.~\eqref{eq:HClaimFP}.

\subsection{Proof of Proposition \ref{propo:EquivalenceAMP}}
\label{sec:ProofPropEquivalenceAMP}

As already noted in the proof of Theorem \ref{thm:AMP}, we have,
for all $k$, $\ell\ge 0$, 
$Q^{M}_{\ell,\ell+k} =Q^{H}_{\ell,\ell+k} =q_{\ell}$,
where the sequence $(q_{\ell})_{\ell\ge 0}$ satisfies Eq.~\eqref{eq:Q_Rec}.
By Proposition \ref{propo:SE}, this implies that for $\ell\ge 0$
\begin{align}
\plim_{n,d\to\infty}\<\hbm^{\ell+1}-\hbm^{\ell}, \hbm^{j}\> =0\, ,\;\;\;\;\;
\plim_{n,d\to\infty}\<\hbh^{\ell+1}-\hbh^{\ell}, \hbh^{j}\> =0\, \;\;\forall \; j\le \ell.
\label{eq:OrthogonalityAMP}
\end{align}

We proceed by induction over $\ell$, the induction hypothesis being 
given by Eq.~\eqref{eq:EquivalenceAMP}. This holds by construction for the base case $\ell=0$. 
We will therefore assume that 
\begin{align}
\plim_{n,d\to\infty}\|\bm^{j}-\hbm^{j}\|=0\, ,\;\;\;\;
\plim_{n,d\to\infty}\|\bh^{j}-\hbh^{j}\|=0\, \;\;\; \forall j\le \ell,
\end{align}
and prove this holds for $j=\ell+1$. To this end we define $\obm^{\ell+1}$,
$\obh^{\ell+1}$, as follows
 \begin{align}
\obh^{\ell+1}= \hbh^{\ell}+\frac{1}{\sqrt{n}}\proj^{\perp}_{\hbh,\ell}\bF(\hbm^{\ell}) ,\;\;\;\;
\obm^{\ell+1}= \hbm^{\ell}-\frac{\gamma}{\sqrt{d}}
\proj^{\perp}_{\hbm,\ell}\bD\bF(\hbm^{\ell})^{\sT}\hbh^{\ell}\, .\label{eq:Ortho-Int}
 \end{align}
 namely, we apply the update \eqref{eq:Ortho_2_maintext} to the AMP iterates.
 We note that 
 \begin{align}
&\plim_{n,d\to\infty}\frac{1}{\sqrt{n}}\|\bF(\hbm^{\ell})-\bF(\bm^{\ell})\| = 0\, ,
\label{eq:Fcont}\\
&\plim_{n,d\to\infty}\frac{1}{\sqrt{d}} \|\bD\bF(\hbm^{\ell})-\bD\bF(\hbm^{\ell})\|_{\op} =0\, ,
\label{eq:DFcont}
\end{align}
by the Lipschitz property of $\bF$ and $\bD\bF$ (cf. Lemma \ref{lemma:BoundLip}) 
and the induction hypothesis. Further, denoting by $\bH_{\ell}\in\reals^{n\times \ell}$
the matrix with columns $\hbh^1,\dots, \hbh^{\ell}$ and by 
$\bM_{\ell}\in\reals^{n\times \ell}$
the matrix with columns $\hbm^1,\dots, \hbm^{\ell}$. By Proposition \ref{propo:SE},
we have
\begin{align}
\plim_{n,d\to\infty}\lambda_{\min}\Big(\bH_{\ell}^{\sT}\bH_{\ell}\Big)&=
\lambda_{\min}\big(\bQ^H_{\le \ell}\big)>0\, ,\\
\plim_{n,d\to\infty}\lambda_{\min}\Big(\bM_{\ell}^{\sT}\bM_{\ell}\Big)&=
\lambda_{\min}\big(\bQ^M_{\le \ell}\big)>0\, ,
\end{align}
where the last inequality holds because $Q^{M}_{j,k} =Q^{H}_{j,k} =q_{j\wedge k}$,
and Eq.~\eqref{eq:Q_Rec} implies that $q_{j+1}-q_j>0$ strictly for all $k$. 
Similarly $\lambda_{\max}(\bH_{\ell}^{\sT}\bH_{\ell})$,
$\lambda_{\max}(\bM_{\ell}^{\sT}\bM_{\ell})$ are bounded uniformly as $n,d\to\infty$. 
These facts, together with the induction hypothesis, imply 
\begin{align}
&\plim_{n,d\to\infty}\big\|\proj^{\perp}_{\hbh,\ell}-\proj^{\perp}_{\bh,\ell}\big\|_{\op} =0\, ,\\
&\plim_{n,d\to\infty}\big\|\proj^{\perp}_{\hbm,\ell}-\proj^{\perp}_{\bm,\ell}\big\|_{\op} =0\, .
\end{align}
Together with Eqs.~\eqref{eq:Fcont}, \eqref{eq:DFcont}, these imply 
\begin{align}
\plim_{n,d\to\infty}\|\bm^{\ell+1}-\obm^{\ell+1}\|=0\, ,\;\;\;\;
\plim_{n,d\to\infty}\|\bh^{\ell+1}-\obh^{\ell+1}\|=0\, .
\end{align}

We are left with the task of bounding $\|\hbm^{\ell+1}-\obm^{\ell+1}\|$,
$\|\hbh^{\ell+1}-\obh^{\ell+1}\|$. Consider, to be definite the latter,
Comparing Eq.~\eqref{eq:AMP1} and Eq.~\eqref{eq:Ortho-Int}, we obtain
\begin{align*}
\big\|\hbh^{\ell+1}-\obh^{\ell+1}\big\|^2= \big\|\proj_{\hbh,\ell}\hbh^{\ell+1}-
\proj_{\hbh,\ell}\obh^{\ell+1}\big\|^2\\
& = \big\|\proj_{\hbh,\ell}(\hbh^{\ell+1}-\hbh^{\ell})\big\|^2\,.
\end{align*}
Hence using \eqref{eq:OrthogonalityAMP} and, once more, the fact that
$\lambda_{\min}(\bH_{\ell}^{\sT}\bH_{\ell})$, $\lambda_{\max}(\bH_{\ell}^{\sT}\bH_{\ell})$ 
are bounded and bounded away from zero, we get
\begin{align*}
\plim_{n\to\infty}\big\|\hbh^{\ell+1}-\obh^{\ell+1}\big\|=0\, .
\end{align*}
The same argument applies to $\|\hbm^{\ell+1}-\obm^{\ell+1}\|$. 
%
%*******************************************************************
%
\subsection{Proof of Proposition \ref{propo:SE}}
\label{sec:ProofSE}
This statement is a direct consequence of \cite[Theorem 6]{el2021optimization}.
We begin by restating the latter in a form that is adapted to our use here. 
For each $k\ge 1$, let $\bW^{(k)}\in (\reals^N)^{\otimes k}$ be an independent, Gaussian 
symmetric tensor. More precisely, if $(\bG^{(k)})_{k\ge 0}$ is a collection of tensors with i.i.d. 
$\normal(0,1)$ entries, then 
\begin{align}
\bW^{(k)} := \frac{1}{k!}\sqrt{\frac{k}{N^{k-1}}}\sum_{\pi \in {\mathfrak S}_k}\bG^{(k)}_{\pi}\, ,
\end{align}
where the sum is over permutations of $k$ objects, and $\bG^{(k)}_{\pi}$ is obtained 
by permuting the indices of  $\bG^{(k)}$. 
Let $\nu(t)= \sum_{k\ge 1} c_k^2 t^k$ for some coefficients $c_k$, 
$D\ge 1$ be an integer, $\Phi_{\ell}:\reals^{D}\times \reals\to \reals^D$
be Lipschitz functions, and denote by
$\bD \Phi_{\ell}(z;v)\in\reals^{D\times D}$ the Jacobian of $\Phi_{\ell}$ with respect
to its first argument (which exists in weak sense).

Define
the sequence of random variables $V$, $(Z_{\ell})_{\ell\ge 0}$ by letting $(Z_0,V)\sim p_{0}$
be an arbitrary random vector taking values in $\reals^D\times\reals$, and $(Z_{\ell})_{\ell\ge 1}$ 
be a centered Gaussian process, again with $Z_{\ell}$ again taking values in $\reals^D$
and covariance $\bQ_{jk}:=\E[Z_jZ_k^{\sT}]$ determined recursively via
\begin{align}
\bQ_{j+1,k+1} = \nu\Big(\E\big\{\Phi_{j}(Z_j;V)\Phi_{k}(Z_k;V)^{\sT}\big\}\Big)\, .
\end{align}
(Throughout this section we will not use boldface for $D$-dimensional vectors.)

For any dimension $N\in\naturals$,
consider iterates $\bz^\ell\in\reals^{N\times D}$ given by
\begin{align}
\bz^{\ell+1} &= \sum_{k=1}^{\infty}c_k\bW^{(k)}\{\Phi_{\ell}(\bz^\ell;\bv)\}
- \Phi_{\ell-1}(\bz^{\ell-1};\bv)\sB_{\ell}\, ,\\
\sB_{\ell} &:= \nu'\Big(\E[\Phi_{\ell}(Z_{\ell};V)\Phi_{\ell-1}(Z_{\ell-1};V)^{\sT}]\Big)
\odot \E[\bD \Phi_{\ell}(Z_{\ell};V)]\, ,
\end{align}
where $\bv\in\reals^N$ and $\bz^0\in\reals^{N\times D}$ are non-random,
and $\nu'$ is applied entrywise, and $\odot$ denotes entrywise multiplication.
Further $\Phi_{\ell}(\bz^\ell;\bv)\in\reals^{N\times D}$ denotes the matrix whose 
$i$-th row is $\Phi_{\ell}(\bz^\ell;\bv)_i = \Phi_{\ell}(\bz_i^\ell;v_i)$.
Further, for $\bx\in\reals^{N\times D}$, $\bx= (x_1,\dots,x_N)^{\sT}$,
$x_i\in\reals^D$, we denote by $\bW^{(k)}\{\bx\}\in\reals^{N\times D}$ the matrix  whose 
$i$-th row is 
\begin{align}
\bW^{(k)}\{\bx\}_i = \sum_{j_1,\dots,j_{k-1}=1}^N
W^{(k)}_{ij_1\dots j_{k-1}} \bx_{j_1}\odot \cdots \odot \bx_{j_{k-1}}\, .
\end{align}
With these definitions, we have the following.
\begin{proposition}[Theorem 6 in \cite{el2021optimization}]\label{propo:SE_prev}
Assume that the empirical distribution of $\bv,\bz^0$, namely $\hp_{\bv,\bz^0}:=
N^{-1}\sum_{i\le N}\delta_{v_i,z^0_i}$
converges in $W_2$ distance to $p_{0}$. Then, for any $L\ge 0$
and any function $\psi:\reals^{D(L+1)+1}\to\reals$ that is locally Lipschitz and at most quadratic growth
(i.e. $|\psi(\bx)|\le C(1+\|\bx\|_2^2)$), we have
\begin{align}
\plim_{N\to\infty}\frac{1}{N}\sum_{i=1}^N\psi(\bz^0_i,\dots,\bz^L_i;v_i) &=
\E\Big\{\psi(Z_0,\dots,Z_{L};V)\Big\}
\end{align}
The same holds if $\sB_{\ell}$ is replaced by its empirical version
\begin{align}
\hat{\sB}_{\ell} &:= \nu'\Big(\frac{1}{N}\sum_{i=1}^N\Phi_{\ell}(z_i^{\ell};v_i)\Phi_{\ell-1}(z_i^{\ell-1};v_i)^{\sT}]\Big)
\odot \frac{1}{N}\sum_{i=1}^N\bD \Phi_{\ell}(z_i^{\ell};v_i)\, .
\end{align}
\end{proposition}
\begin{proof}
As mentioned, this is an adaptation  from \cite[Theorem 6]{el2021optimization},
with three points of difference. However, each of these points can be reduced
to the version stated in \cite{el2021optimization}.
\begin{enumerate}
\item In \cite{el2021optimization} it is assumed that $D=1$,
but memory across iterations is allowed. Namely, iterates $\bx^\ell\in\reals^N$
are updated via
\begin{align}
\bx^{\ell+1} &= \sum_{k=1}^{\infty}c_k\bW^{(k)}\{f_{\ell}(\bx^0,\dots,\bx^{\ell})\}
- f_{\ell-1}(\bx^0,\dots,\bx^{\ell})\tilde\sB_{\ell}\, ,\, .
\end{align}
It is clear that this case also covers the seemingly more general one
with iterates $\bz^{\ell}\in\reals^{N\times D}$ by grouping the 
$\bz^{\ell} = (\bx^{\ell D},\bx^{\ell D+1},\cdots, \bx^{\ell D+D-1})$.
Indeed, it covers a more general form of the present statement in
which $\Phi_{\ell}$ depends on $\bz^0,\dots,\bz^{\ell}$ rather than just on $\bz^{\ell}$.
\item The statement in \cite{el2021optimization} does not allow
for dependency on the vector $\bv$. However, the case with $\bv$
can be reduced to the one without $\bv$ simply by encoding $\bv$ as 
an additional column of $\bz^0$, and increasing $D$ to $D+1$.
Since, as mentioned in the previous point, the result of \cite{el2021optimization} 
implies the case in which both $D$ is arbitrary, and $\Phi_{\ell}$
depends on $\bz^0$, \dots, $\bz^{\ell}$, dependence on $\bv$ can be captured.
\item Finally, \cite{el2021optimization} uses the deterministic version of coefficient $\sB_{\ell}$.
One can show by induction over $\ell$ that: $(i)$ $\hsB_{\ell} = \sB_{\ell}+o_P(1)$; 
$(ii)$~Denoting by $\hbz^{\ell}$ the iterates that result from using the empirical version, we have
$\|\bz^{\ell}-\hbz^{\ell}\|_2=o_P(1)$. 
\end{enumerate}
Normalizations are different but equivalent to the ones of  \cite{el2021optimization}.
\end{proof}

We next set $N=n+d$ and
\begin{align}
\bz^{\ell}= \left(
\begin{matrix}
\bx^{\ell}\\
 \by^{\ell}
 \end{matrix}\right)\, ,
 \;\;\; \bx^{\ell}\in\reals^{n\times D},\;\;\;  \by^{\ell}\in\reals^{d\times D}\, .
 \end{align}
Further, we set $v_i = 0$ for $i\le n$, $v_i=1$ for $i>n$,  and
write $J_{\ell}(x)=\Phi_{\ell}(x;0)$, $H_{\ell}(x)=\Phi_{\ell}(x;1)$,
and write $\oW^{(k,q)}_{i;i_1,\dots,i_q,j_1,\dots j_{k-q}} = 
W^{(k)}_{i,i_1,\dots,i_q,n-1+j_1,\dots n-1+j_{k-q}}$, for $q\le k-1$. 
In what follows indices denoted by $i,i_1,i_2,\dots$ run over $[n]$,
and indices denoted by $j,i_1,i_2,\dots$ run over $[d]$. 
We can then rewrite the iteration as
\begin{align}
x_i^{\ell+1} = &\sum_{k=1}^{\infty}c_k\sum_{q=0}^{k-1}
\binom{k-1}{q}\sum_{i_1\dots i_{q}\le n}  \sum_{j_{q+1}\dots j_{k-1}\le n}
\oW^{(k,q)}_{i;i_1,\dots,i_q,j_1,\dots j_{k-1-q}}\cdot\\
&\phantom{AAAA}\cdot J_{\ell}(x^{\ell}_{i_1})\odot \cdots \odot
J_{\ell}(x^{\ell}_{i_q})\odot H_{\ell}(y^{\ell}_{j_1})\odot \cdots \odot
H_{\ell}(y^{\ell}_{j_{k-q}}) - \sB_{\ell}^{\sT}F_{\ell-1}(x^{\ell-1}_{i})\, ,\nonumber\\
y_i^{\ell+1} = & \sum_{k=1}^{\infty}c_k\sum_{q=0}^{k-1}
\binom{k-1}{q}\sum_{i_1\dots i_{q}\le n}  \sum_{j_{q+1}\dots j_{k-1}\le n}
\oW^{(k,q)}_{i_1;i_2,\dots,i_q,j,j_1\dots j_{k-1-q}}\cdot\\
&\phantom{AAAA}\cdot J_{\ell}(x^{\ell}_{i_1})\odot \cdots \odot
J_{\ell}(x^{\ell}_{i_q})\odot H_{\ell}(y_{j_1})\odot \cdots \odot
H_{\ell}(y^{\ell}_{j_{k-q}}) - \sB_{\ell}^{\sT}H_{\ell-1}(y^{\ell-1}_{j})\,  .\nonumber
\end{align}
We next set $D=3$ and
\begin{align}
J_{\ell}(x=(x_1,x_2,x_3)) &= \Big(
0,\;\;
\of_{\ell}(x_1),\;\; 0\Big)\, ,\\
H_{\ell}(y=(y_1,y_2,y_3)) &= 
\Big(\oh_{\ell}(y_2-y_3),\;\;
\og_{\ell}(y_2-y_3),\;\;\og_{\ell}(y_2-y_3)\Big)\, ,
\end{align}
Writing $\hbx^{\ell}$ for the first column of $\bx^{\ell}$ and $\hby^{\ell}$, $\hby^{\ell}_0$
for the second and third columns of $\by^{\ell}$.
In terms of these variables, the AMP iteration reads
\begin{align}
\hx_i^{\ell+1} & = 
\sum_{k=1}^{\infty}c_k \sum_{j_1,\dots,j_{k-1}>n}
\oW^{(k,0)}_{i;i_1\dots j_{k-1}} \oh_{\ell}(\hy_{j_1}^{\ell}-\hy_{0,j_1}^{\ell})\cdots  
\oh_{\ell}(\hy_{j_{k-1}}^{\ell}-\hy_{0,j_{k-1}}^{\ell}) -b_{1,\ell} \of_{\ell-1}(\hx_i^{\ell-1})\, ,\\
\hy_j^{\ell+1} &= \sum_{k=1}^{\infty}c_k\sum_{q=0}^{k-1}
\binom{k-1}{q}\sum_{i_1\dots i_{q}\le n}  \sum_{j_{q+1}\dots j_{k-1}\le n}
\oW^{(k,q)}_{i_1;i_2,\dots,i_q,j,j_1\dots j_{k-1-q}}\cdot\label{eq:FirstY}\\
&\phantom{AAAAAAAA}\cdot 
\of_{\ell}(\hx_{i_1}^{\ell})\cdots   \of_{\ell}(\hx_{i_q}^{\ell})
\og_{\ell}(\hy_{j_{1}}^{\ell}-\hy_{0,j_{1}}^{\ell}) \cdots 
\og_{\ell}(\hy_{j_{k-q}}^{\ell}-\hy_{0,j_{k-q}}^{\ell})\nonumber\\
&\phantom{AAAAAAAA} -
b_{2,\ell} \og_{\ell-1}(\hy_{j}^{\ell-1}-\hy_{0,j}^{\ell-1})
-b_{3,\ell} \oh_{\ell-1}(\hy_{j}^{\ell-1}-\hy_{0,j}^{\ell-1})
 \, ,\nonumber\\
\hy_{0,j}^{\ell+1} &= \sum_{k=1}^{\infty}c_k\sum_{j_1,\dots,j_{k-1}\le d}  
\oW^{(k,0)}_{j_1;j,j_2,\dots,j_{k-1}} 
\og_{\ell}(\hy_{j_{1}}^{\ell}-\hy_{0,j_{1}}^{\ell}) \cdots 
\og_{\ell}(\hy_{j_{k-1}}^{\ell}-\hy_{0,j_{k-1}}^{\ell})
 \label{eq:SecondY}\\
&\phantom{AAAAAAAA} -
b_{4,\ell} \og_{\ell-1}(\hy_{j}^{\ell-1}-\hy_{0,j}^{\ell-1})
-b_{5,\ell} \oh_{\ell-1}(\hy_{j}^{\ell-1}-\hy_{0,j}^{\ell-1})\, ,\nonumber
\end{align}
for suitable sequence of constants $b_{1,\ell},\dots,b_{5,\ell}$.
It is straightforward (albeit tedious) to write expressions for these constants,
as well as the state evolution characterization that follows from Proposition
\ref{propo:SE_prev}.

We next introduce a small parameter $\eps$, and set 
 $\of_{\ell}(x) = \eps f_{\ell}(x/\eps)$, $\og_{\ell}(x) = g_{\ell}(x/\eps)$
$\oh_{\ell}(x) = h_{\ell}(x/\eps)$.
By defining $x^\ell_{i}:= (\hy_{j}^{\ell}-\hy_{0,j}^{\ell})/\eps$, 
and taking the difference of Eqs.~\eqref{eq:FirstY} and \eqref{eq:SecondY},
we get
\begin{align}
\hx_i^{\ell+1} & = 
\sum_{k=1}^{\infty}c_k \sum_{j_1,\dots,j_{k-1}\le d}
\oW^{(k,0)}_{i;j_1\dots j_{k-1}} h_{\ell}(x_{j_1}^{\ell})\cdots  
h_{\ell}(x_{j_{k-1}}^{\ell}) -b_{*,\ell} f_{\ell-1}(\hx_i^{\ell-1})+O(\eps)\, ,\\
x_j^{\ell+1} &= \sum_{k=1}^{\infty}c_k(k-1)\sum_{i_1\le n}  
\sum_{j_{1}\dots j_{k-2}\le d}
\oW^{(k,0)}_{i_1;j,j_1\dots, j_{k-2}}
f_{\ell}(\hx_{i_1}^{\ell})
g_{\ell}(x_{j_{1}}^{\ell}) \cdots g_{\ell}(x_{j_{k-2}}^{\ell})  \nonumber\\
&\phantom{AAAAAAAA} -
a_{1,\ell} g_{\ell-1}(x_{j}^{\ell-1})
-a_{2,\ell} h_{\ell-1}(x_j^{\ell-1}) +O(\eps)\, .
\end{align}
The order $\eps$ term can be controlled uniformly over $n,d$ on a high probability 
event (controlling the operator norms of tensors $\bW^{(k)}$),
hence allowing to derive a state evolution characterization of the recursion in which 
$O(\eps)$ terms are dropped.

We finally define 
\begin{align}
F_i(\bx) : = \sum_{k=1}^{\infty}c_k \sum_{j_1,\dots,j_{k-1}\le d}
\oW^{(k,0)}_{i;j_1\dots j_{k-1}} x_{j_1}^{\ell}\cdots  
x_{j_{k-1}}^{\ell}\, ,
\end{align}
and note that the above recursion (dropping $O(\eps)$ terms) can be rewritten
as
\begin{align}
\hbx^{\ell+1} & = \bF(h_{\ell}(\bx^{\ell})) -b_{*,\ell} f_{\ell-1}(\hbx^{\ell-1})\, ,\\
\bx^{\ell+1} &= \bD\bF(\bx^{\ell})^{\sT}
f_{\ell}(\hbx^{\ell}) -
a_{1,\ell} g_{\ell-1}(\bx^{\ell-1})
-a_{2,\ell} h_{\ell-1}(\bx^{\ell-1})\, .
\end{align}
This is exactly the form of the algorithm introduced in the main text,
\eqref{eq:AMP1}, \eqref{eq:AMP2}, where functions $f_{\ell}, h_{\ell}, g_{\ell}$
are all linear. 

The proof of Proposition  \ref{propo:SE} follows from keeping track of the 
coefficients in the Onsager terms
(the memory terms in the AMP recursions), as well as of the state evolution recursion, along the 
chain of reductions just defined.

%
%*********************
%
\subsection{Proof of Theorem \ref{thm:MainAchievability}}
\label{app:TP2}

Theorem \ref{thm:MainAchievability} follows from Theorem \ref{thm:AMP} 
using the same exactly the same argument as for Theorem \ref{thm:Hessian}, with 
the following adaptations:
\begin{enumerate}
\item The initial value of the energy for the second phase is given by $H(\bm^L)/n = u_{\sRS}(q,\alpha,\xi)+o_P(1)$.
\item The Hessian descent algorithm takes palce in the $(d-1)$-dimensional space
$V_{L}:= \{\bx\in\reals^d:\, \<\bx,\bm^L\>=0\}$. 
\item Within this subspace, we attempt to solve an optimization problem that has the same 
form as the original one, namely
\begin{align}
\mbox{minimize}&\;\;\;\;\; \|\tilde{\bF}(\bx;\bm^L)\|_2^2\, ,\;\;\;\;\;\;\;
\tilde{\bF}(\bx;\bm^L) := \bF(\bm^L+\bx)\, ,\\
\mbox{subj. to}&\;\;\;\;\; \bx\in V_L, \;\; \|\bx\|_2^2 = 1-\|\bm^L\|_2^2
\end{align}
Of course, the reduction from 
dimension $d$ to $d-1$ is immaterial, and the change in the radius constraint does not affect 
the argument given for Theorem \ref{thm:AMP}.
\item For a non-random $\bv\in\Ball^d(1)$, the random function 
$\tilde{\bF}(\,\cdot \,;\bv)$ restricted to $\bv^{\perp}$ is again a centered
Gaussian process with covariance
\begin{align}
\E\big\{\tilde{F}_i(\bx_1)\tilde{F}_j(\bx_1;\bv) \big\} = \delta_{ij}\, 
\xi\big(\|\bv\|^2_2+\<\bx_1,\bx_2\>\big)\, .
\end{align}
Hence, by Theorem \ref{thm:Hessian}, Hessian descent on the subspace $\bv^{\perp}$
with respect to $\tilde{\bF}$ outputs $\bx^*\in\S^{d-1}$ with energy 
$\|\bF(\bx^{*})\|^2/(2n)\le u(t_*;\bv)+C\delta+o_P(1)$, where $t_*:=1-\|\bv\|^2_2$, where
$u(\,\cdot\,;\bv)$ solves
 \begin{align}
 \frac{\de u}{\de t}(t;\bv) = -\frac{1}{2\alpha}z_*\big(\alpha; \sqrt{2\alpha u(t) \xi''(\|\bv\|^2+t)}, \xi'(\|\bv\|^2+t)\big)\,,
 \;\;\;\;\; u(0) =  \frac{1}{2n}\|F(\bv)\|^2\, .
 \end{align}
\item The proof of Theorem \ref{thm:Hessian} only uses a uniform upper bound over the 
minimum eigenvalue of Hessian  of the energy function, and therefore the last claim actually
 applies uniformly over $\bv\in\Ball^d(1)$. As a consequence it also applies to the random
 point $\bm^L$.
\end{enumerate}
%

%***********************************************
%

\section{Conclusion and future directions}
\label{sec:Conclusion}

The GE model allows to study in detail the algorithmic and 
satisfiability phase transitions that are summarized in 
Figure \ref{fig:phase_diag}.
In particular, we showed that global minima can be found efficiently
for $n/d$ below a constant $\alpha_{\salg}(\xi)$ while
Lipshitz algorithms fail below  $\alpha_{\salg}(\xi)$.
This algorithmic threshold is explicitly given in Eq.~\eqref{eq:AlphaAlgDef}.

This threshold in general, is
strictly smaller than the satisfiability phase transition 
$\alpha_{\sSAT}(\xi)$, suggesting the existence of a regime in
which solutions exist but are not found by polynomial time algorithms.

We also found out that the tractable regime 
$\alpha\in (0,\alpha_{\salg})$  can only
partially be understood using theories that are popular in non-convex
optimization and learning.
Indeed, for $\alpha\in (0,\alpha_{\salg})$ but close to
$\alpha_{\salg}$, optimization algorithms converge to points that 
are highly sensitive to small perturbations.

While the GE model is obviously far from optimization
problems arising in applications (e.g., in machine learning),
we expect these phase transitions to take place for other 
overparametrized optimization problems
of the form $H(\bx) = \|\bF(\bx)\|^2/2$. 

Our work suggests a number of
concrete mathematical problems whose solution would contribute
bridging the gap between the GE model and more complex random optimization
problems:
\begin{enumerate}
\item\emph{Universality.} We observed that predictions from the GE model are accurate
for the AoR model as well. This agreement is not explained by
current mathematical theory\footnote{A direct application of
 high-dimensional central limit theorems
\cite{eldan2021non} yields that the the two model match for $m\gg d^{5k-1/2}$ where 
$k$ is the degree of $\xi$, but we observe good agreement for much smaller values of $m$.}
and it would be important to obtain a rigorous explanation.
It also suggests that Gaussian models might be more broadly predictive than expected.
\item\emph{Dependence on the covariance.} If some form of universality holds,
the behavior of the cost function $H(\,\cdot\,)$ is characterized by the covariance of the process 
$F_i(\,\cdot\,)$. We studied a particularly simple covariance that is rotationally
invariant. It would be interesting to understand whether the qualitative behavior 
outlined in Fig~\ref{fig:phase_diag} depends on the covariance structure.
\end{enumerate}

\section*{Acknowledgements}

AM was supported by the NSF through award DMS-2031883, the Simons Foundation
through Award 814639 for the Collaboration on the Theoretical Foundations of Deep Learning,
the NSF grant CCF-2006489 and the ONR grant N00014-18-1-2729. 
%Part of this work was carried out while Andrea Montanari was on partial leave from Stanford and a Chief Scientist at Ndata Inc dba
%Project N. The present research is unrelated to AM’s activity while on leave. 
ES was supported by the Israel Science Foundation (Grant
Agreement No. 2055/21), European Union (ERC, PolySpin, 101165541) and a research grant from the Center for Scientific Excellence at the Weizmann Institute of Science; and is the incumbent of the Skirball Chair in New Scientists.

\newpage

\newpage

\newpage

\bibliographystyle{amsalpha}

\newcommand{\etalchar}[1]{$^{#1}$}
\providecommand{\bysame}{\leavevmode\hbox to3em{\hrulefill}\thinspace}
\providecommand{\MR}{\relax\ifhmode\unskip\space\fi MR }
\providecommand{\MRhref}[2]{%
  \href{http://www.ams.org/mathscinet-getitem?mr=#1}{#2}
}
\providecommand{\href}[2]{#2}

\newpage

\appendix

\section{Numerical simulations in the AoR model}
\label{sec:Appendix_Numerical}

\subsection{Setup for the simulations}
\label{sec:Appendix_Numerical_Setup}

All of our simulations are carried out by generating random instances of
the AoR model. Namely, we generate weight vectors 
$\ba_i = (\ba_i^{(1)},\dots,\ba^{(m)}_i)$ with $(\ba_i^{(k)})_{i\le N, k\le m}\sim
\normal(0,\id_d)$ and attempt to minimize the energy function 
\begin{align}
H(\bx) := \frac{1}{2}\sum_{i=1}^nF_i(\bx)^2\, ,
\end{align}
with functions $F_i(\bx)$ defined as in Eq.~\eqref{eq:AoR_Fi} for $\|x\|=1$.
We extend the definition of $F_i(\bx)$ to $\bx\in \reals^d$ as follows.
Letting $\phi(t) = \sum_{k=0}^{\infty}\phi_k \he_k(t)$ be the Hermite decomposition 
of $\phi$, we define
\begin{align}
\phi_e(t;r) = \sum_{k=0}^{\infty}\phi_k r^{k}\he_k(t/r)\, ,
\end{align}
and generalize Eq.~\eqref{eq:AoR_Fi}  by
\begin{align}\
	F_i(\bx) =  \frac{1}{\sqrt{m}}\sum_{k=1}^m\phi_e\big(\<\bx,\ba^{(k)}_i\>;\|\bx\|\big)\, .
\end{align}
The rationale for this extension is that the resulting 
covariance matches the one of the GE model \eqref{eq:Fcov} for any $\bx_1, \bx_2\in \reals^d$.

We carry out experiments using two algorithms, as described below.
 
 \subsubsection{Stochastic gradient descent} 
 
 We view the function $H(\bx)$ as an average over terms $F_i(\bx)^2$
 (in a machine learning interpretation, each term $F_i$ corresponds to one datapoint)
 At each step $k$, we move in the direction of a gradient with respect to a subset of these:
 \begin{align}
 \tbx^{k+1} &= \bx^k - \frac{s_{\ell(k)}} {|B_k|}\sum_{i\in B_k} F_i(\bx^k)\nabla_{\bx}
 F_i(\bx^k)\, ,\\
 \bx^{k+1} & = \frac{\tbx^{k+1}}{1\wedge \|\tbx^k\|}\, .
 \end{align}
 Here $s_{\ell(k)}$ is the stepsize and $B_k$ is a random subset of $[n]$
 of size $|B_k|=b$. Throughout our simulations we take $b=4$,
 and shuffle the data between epoch (via the {\sf pytorch} option {\sf shuffle =True}). 
In other words, at the beginning of each epoch, the $n$ terms in the energy functions 
are randomly permuted,
 grouped in contiguous batches of size $b$, and one pass through the data is taken 
 (one epoch). In the equation above, we denote by $\ell(k)$ the epoch of step k,
 and note that the stepsize depends on the epoch.

 We initialize the SGD iteration with $\bx^0=\normal(0,\eps^2\id_d/d)$.
 We use $\eps=1/100$ in all simulations except for the ones to investigate
 the sensitivity phase transition. in that case, we let $\eps=1$ (i.e. we initialize close
 the the unit sphere) because we are interested in estimating the distance from initialization. 
 
 Note that in principle the SGD
 iteration could stop at at a point in the interior of the unit ball.
 In other world, we are attempting to solve the relaxed problem 
 whereby the constraint $\bx\in\S^{d-1}$ is replaced by $\bx\in\Ball^d(1)$. 
 However, we observe empirically that (for the settings we explore),
 the final point reached has $\|\bx^k\|=1$, unless $\alpha$ is small, in which
 case a solution is found quickly. We conclude that the energy achieved by SGD within 
 the relaxed model is achievable also
 in the unrelaxed one.
 
 Given a total number of epoch $N_{\sep}$,  we use the following stepsize schedule
 (for the learning rate $\lr>0$ a scale parameter)
 \begin{align}
s_{\ell}= \lr\cdot\big[(2+\ell-N_{\sep}/2)\wedge 1)\big]^{-1/3}\, .
\end{align}
We carried out experiments with other schedules, but did not observe major differences.

We used the {\sf pytorch} implementation of SGD. We note that the model 
\eqref{eq:AoR_Fi} can be instantiated as a convolutional neural network with 
$d$ input channels, $1$ output channel, and kernel-size $1$.
Input sample $i$ is formed by vectors $\ba^{(1)}_i,\dots,\ba^{(m)}_i$
of \eqref{eq:AoR_Fi}
whereby each $\ba^{(k)}_i$ comprises the $m$ input channels.
We reproduce simplified code for our implementation below.

\begin{lstlisting}
def __init__(self, a, m, d, eps):
    super().__init__()
    self.m = m
    self.conv1 = nn.Conv1d(d, 1, 1, padding=0, dtype=torch.float, bias=False)
    self.conv1.weight.data = (eps/np.sqrt(d))*torch.randn((1,d,1))
    self.lin2 = nn.Linear(m,1,bias=False)
    self.lin2.weight.data[0,:m//2] = a/np.sqrt(m)
    self.lin2.weight.data[0,m//2:] = -a/np.sqrt(m)
    self.lin2.weight.requires_grad = False
    self.act = Myact(type=act_type)

def forward(self, x):
    x0 = self.conv1(x)
    wnorm = torch.linalg.norm(self.conv1.weight)
    x1 = self.act(x0)  ### Normalization included
    x2 = self.lin2(einops.rearrange(x1, 'b c l -> b (c l)'))
    return x2
\end{lstlisting}

Summarizing, for each run of SGD, we specify the following parameters
(apart from the function $\phi$ that enters the definition \eqref{eq:AoR_Fi}):
\begin{itemize}
\item Dimension $d$.
\item Number of equations $n$.
\item Number of `units' $m$ in Eq.~\eqref{eq:AoR_Fi}.
\item Stepsize scale $\eta$.
\item Number of epochs $N_{\sep}$.
\end{itemize}

\subsubsection{Hessian descent}

The Hessian descent algorithm is described in the second part
of Algorithm \ref{algo:HD+AMP}. It remains to specify the calculation of 
the approximate eigenvector $\bv(\bx^k)\in \Ts_{\bx^k}$, such that 
\begin{align*}
\<\bv, \nabla^2H(\bx^k)\bv\> \le \lambda_{\min} (\nabla^2H(\bx^k)|_{\Ts_{\bx^k}})+d\delta\, .
\end{align*}
We form the Hessian using automatic differentiation and implemented two procedures for approximating the 
eigenvector with eigenvalue $\lambda_{\min}$: $(i)$~Power iteration;
$(ii)$~Eigenvector computation using {\sf torch.linalg.eigh}. 
The two procedures yield comparable results. 

As mentioned in the main text, we carry out simulations with two different procedures
to select the signs $s_k$ in Algorithm \ref{algo:HD+AMP}:
$(i)$~The optimal procedure outlined in  Algorithm \ref{algo:HD+AMP}; $(ii)$~Uniformly
random signs. We observe that the first procedure yields smaller values of
the energy, but the discrepancy is compatible with a finite $d$, $n$ effect.
\begin{itemize}
\item Dimension $d$.
\item Number of equations $n$.
\item Number of `units' $m$ in Eq.~\eqref{eq:AoR_Fi}.
\item Number of steps $K=1/\delta$.
\end{itemize}

\subsection{Numerical results omitted from the main text}

In this appendix we report further simulation results omitted from the main text, mainly to 
investigate the dependency on various parameter choices.

\subsubsection{Stochastic gradient descent}

In this section we report the following results of simulations with 
SGD:
\begin{itemize}
\item Figure \ref{fig:a025_third_change_param} reports the energy achieved by
 SGD in the AoR model, with
$\phi(t) = 1+a \he_3(t)$, with $a=1/4$. Results are averaged over $10$ realizations. 
This is the same setting as in Figure \ref{fig:a025_third_change_m}
(left frame), but we test the sensitivity of our results to variouschoices of the parameters.
\item Figures \ref{fig:a025_third_alpha04_SGD1},
 \ref{fig:a025_third_alpha04_SGD2} plots single trajectories of SGD for 
 $\phi(t) =1+ a \he_3(t)$, with $a=0.25$, $\alpha=0.4$.
(This setting is similar to the one of Figure \ref{fig:a033_fourth_alpha07_SGD3}.)
\item Figures \ref{fig:a033_fourth_alpha07_SGD1} and \ref{fig:a033_fourth_alpha07_SGD3}
plots single trajectories of SGD for 
 $\phi(t) =1+ a \he_4(t)$, with $a=0.33$.
\end{itemize}

\begin{figure}
\phantom{A}\hspace{-1cm}
\includegraphics[width=0.34\linewidth]{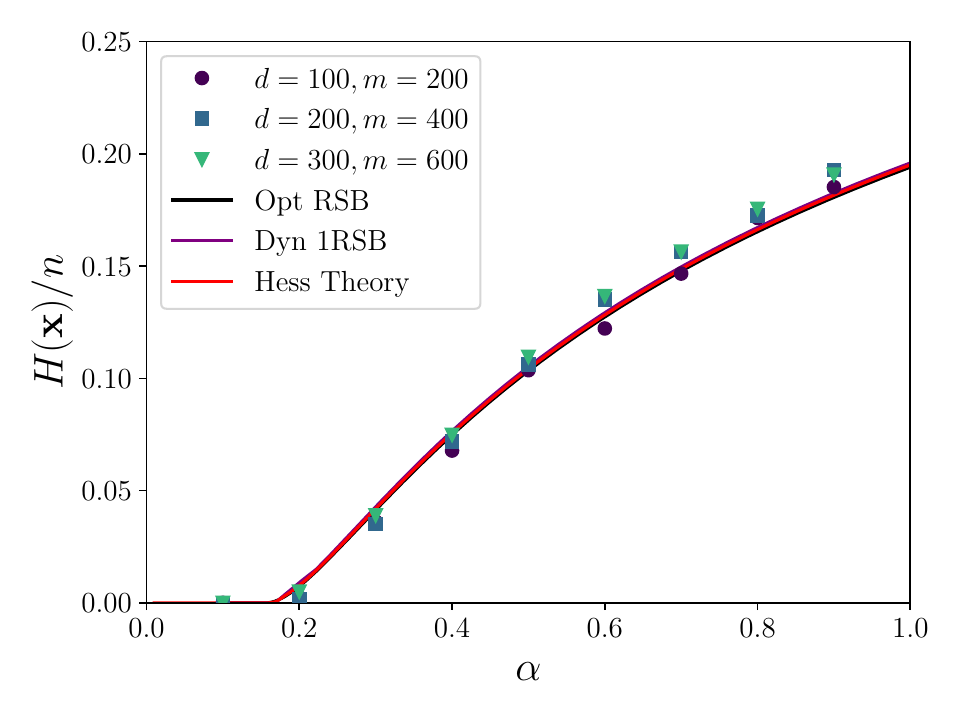}
\includegraphics[width=0.34\linewidth]{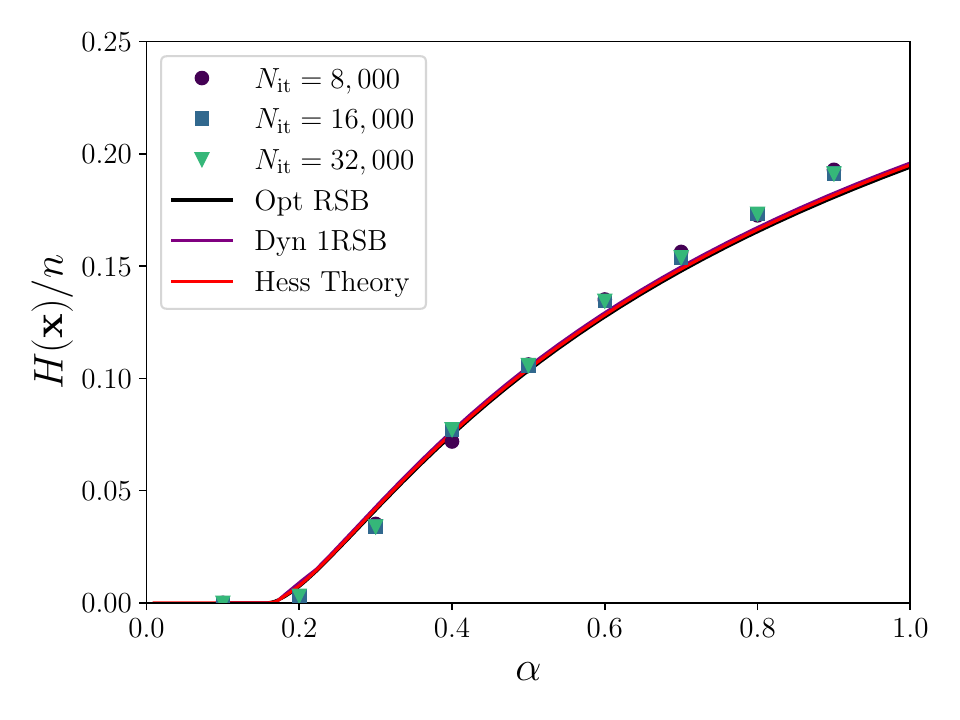}
\includegraphics[width=0.34\linewidth]{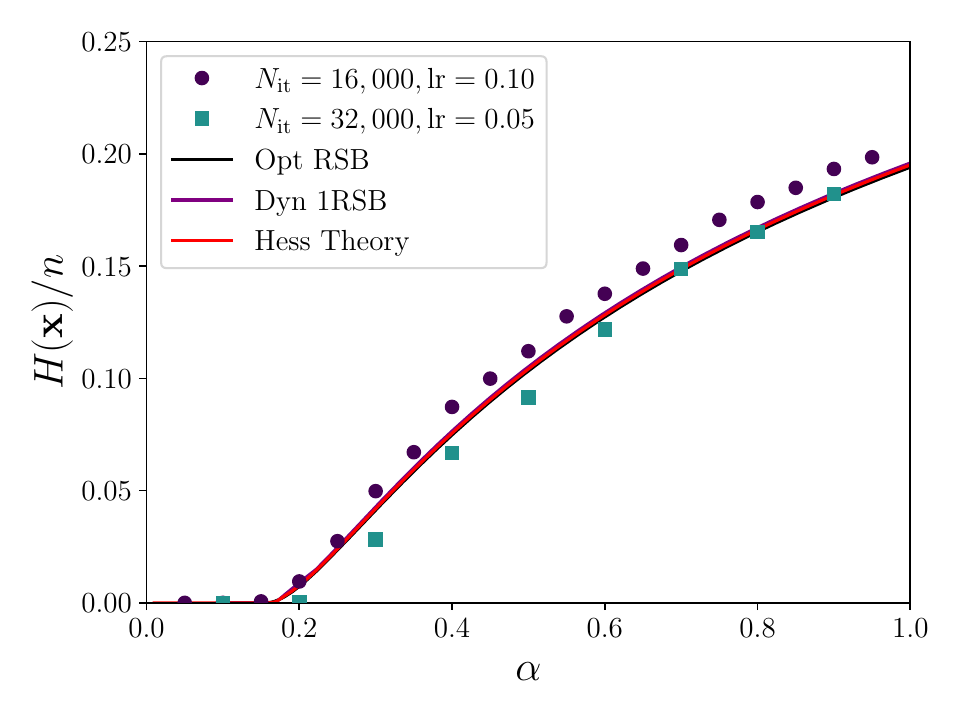}
\caption{Energy achieved with SGD in the AoR model, with
$\phi(t) = 1+a \_3(t)$, with $a=1/4$.
Left: We fix $\lr = 0.05$, $N_{\sit}=8,000$, $m=2d$ and vaary $d$.
Center: We fix $\lr = 0.1$, $m=200$, $d=200$ and change $N_{\sit}$.
Right:  We fix, $m=200$, $d=200$ and change $N_{\sit}$ and $\lr$.}
		\label{fig:a025_third_change_param}
\end{figure}

\begin{figure}
\includegraphics[width=0.5\linewidth]{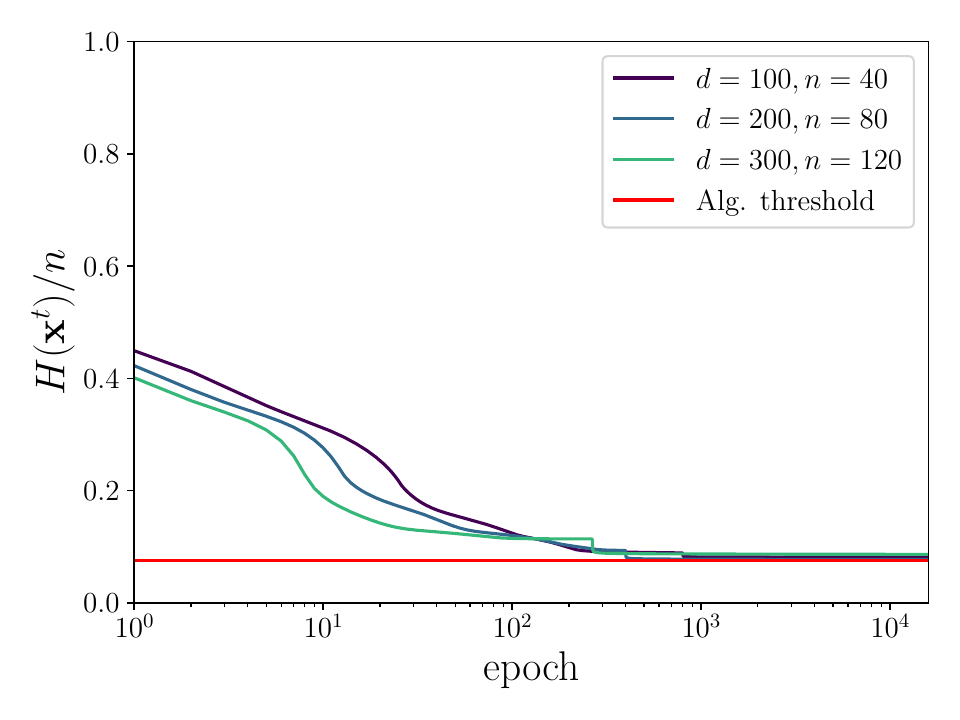}
\includegraphics[width=0.5\linewidth]{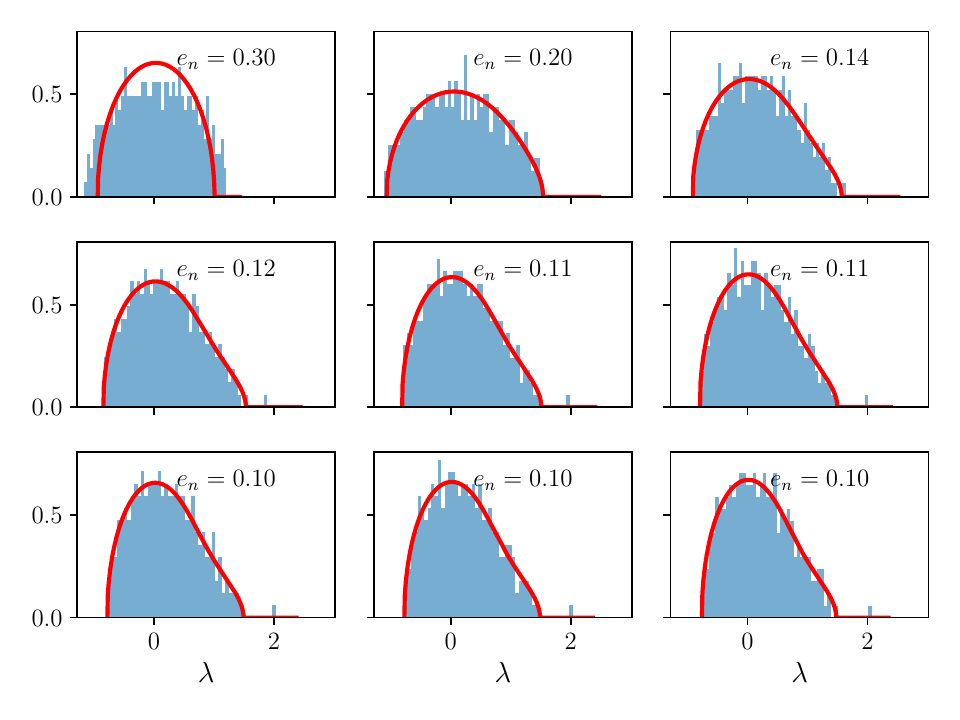}
\caption{Trajectory of SGD ,
for $\phi(t) = 1+a \he_3(t)$, with $a=0.25$.
Left, energy as a function of time (epochs):
 $\alpha=0.4$, $m=8d$,  $N_{\sit}=16,000$
${\rm lr}=0.01$. Red lines are theoretical predictions for the
algorithmic threshold.
Right: Empirical spectral distribution of the Hessian along SGD trajectory,
with theoretical predictions as red lines.
(Here $d=300$, $n=120$, $m=4800$.)}
		\label{fig:a025_third_alpha04_SGD1}
\end{figure}

\begin{figure}
\includegraphics[width=0.5\linewidth]{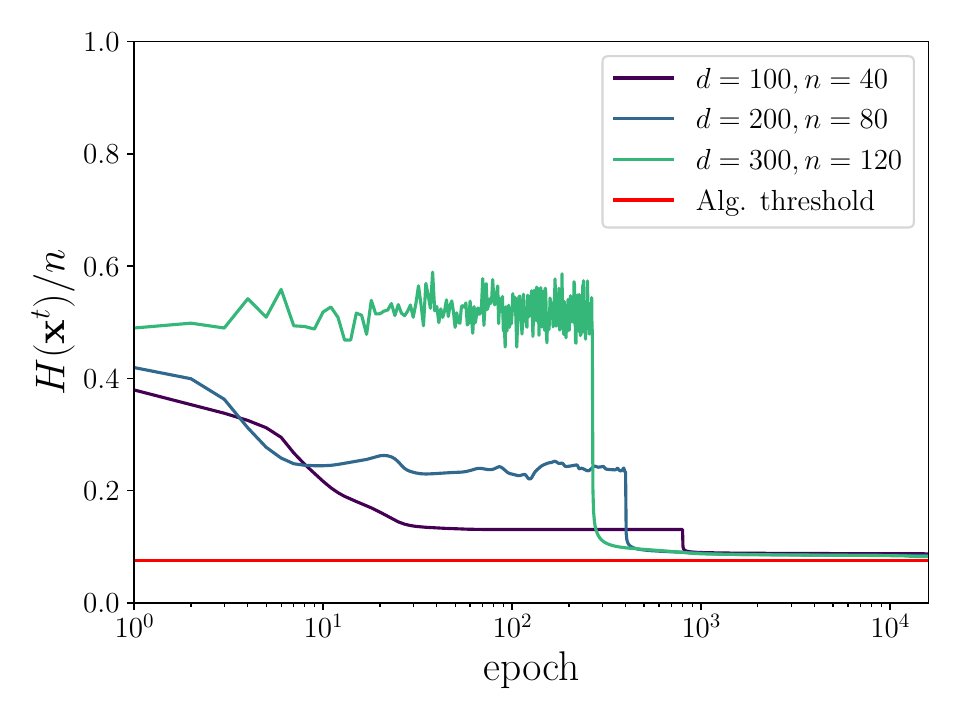}
\includegraphics[width=0.5\linewidth]{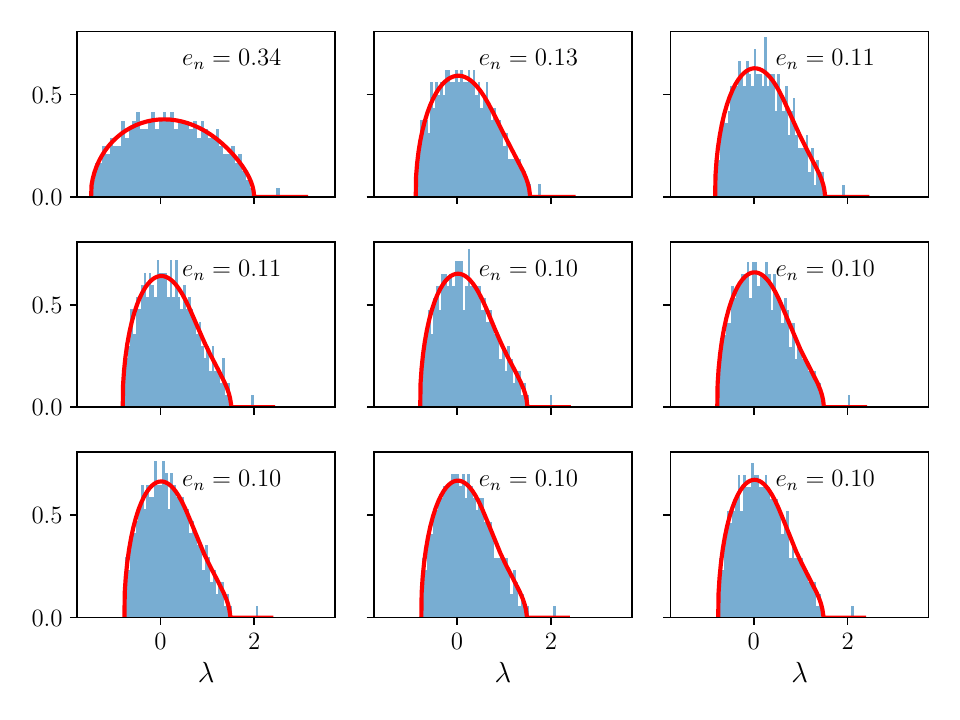 }
\caption{Trajectory of SGD ,
for $\phi(t) =1+ a \he_3(t)$, with $a=0.25$.
Left, energy as a function of time (epochs): $\alpha=0.4$, $m=8d$, $N_{\sit}=16,000$,
${\rm lr}=0.05$.  Red lines are theoretical predictions for the
algorithmic threshold.
Right: Empirical spectral distribution of the Hessian along SGD trajectory,
with theoretical predictions as red lines.
(Here $d=300$, $n=120$, $m=4800$.)}
		\label{fig:a025_third_alpha04_SGD2}
\end{figure}

\begin{figure}
\includegraphics[width=0.5\linewidth]{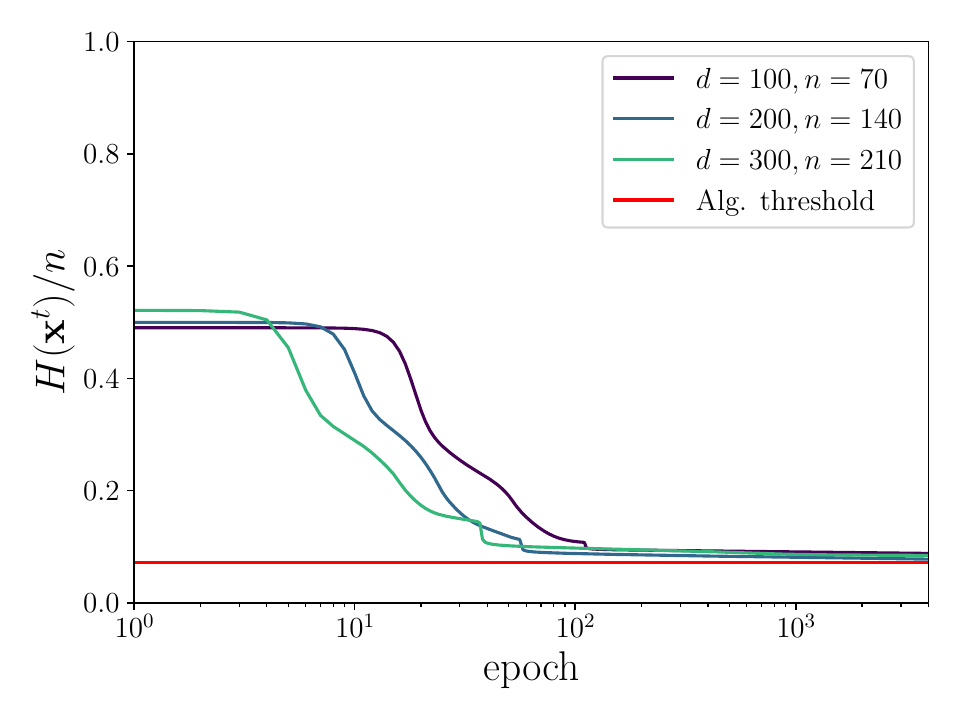}
\includegraphics[width=0.5\linewidth]{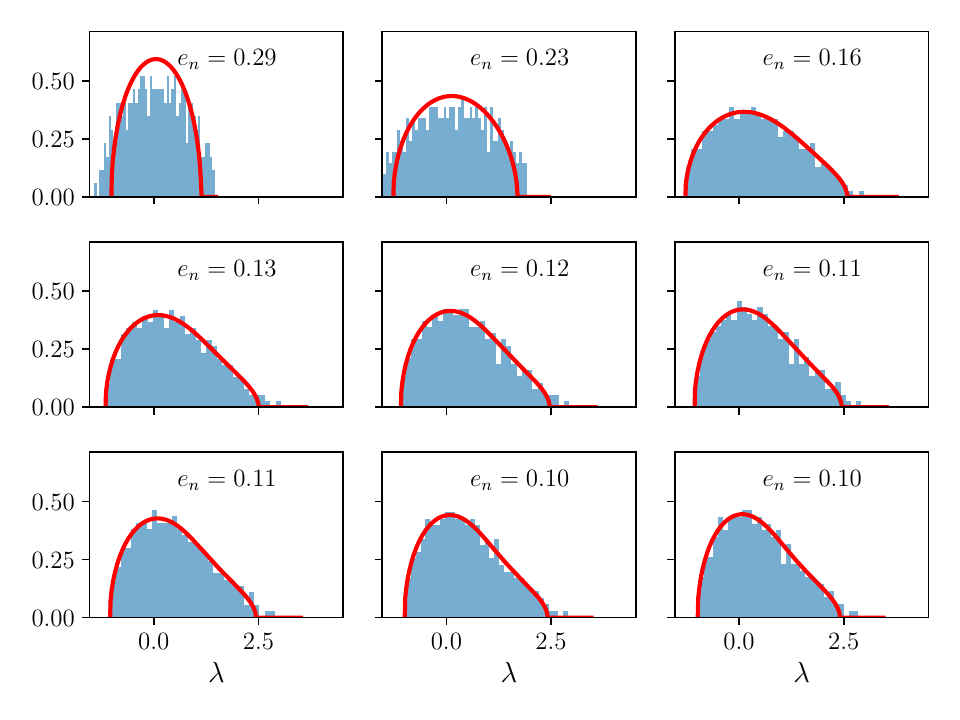}
\caption{Trajectory of SGD ,
for $\phi(t) = 1+a \he_4(t)$, with $a=0.33$.
Left, energy as a function of time (epochs): $\alpha=0.7$, $m=8d$, $N_{\sit}=4000$,
${\rm lr}=0.005$.
 Red lines are theoretical predictions for the
algorithmic threshold.
Right: Empirical spectral distribution of the Hessian along SGD trajectory,
with theoretical predictions as red lines. (Here $d=300$, $n=210$, $m=4800$.)}
		\label{fig:a033_fourth_alpha07_SGD1}
\end{figure}

\subsubsection{Hessian descent}
 
 In this section we report the following results of simulations with 
Hessian descent:
 \begin{itemize}
 \item Figures \ref{fig:a025_third_alpha01}, \ref{fig:a025_third_alpha02}, \ref{fig:a025_third_alpha04}
 report trajectories of Hessian descent for $\phi(t) =1+ a \he_3(t)$, with $a=0.25$ and 
 (respectively) $\alpha\in\{0.1,0.2,0.4\}$. In these cases we use the optimal choice of
 signs $s_k$, as in Algorithm \ref{algo:HD+AMP}.
 \item 
 Figures \ref{fig:a033_fourth_alpha02_OS},
\ref{fig:a033_fourth_alpha02_RS},
\ref{fig:a033_fourth_alpha07_OS},
\ref{fig:a033_fourth_alpha07_RS} 
 report trajectories of Hessian descent for $\phi(t) =1+ a \he_4(t)$, with $a=0.33$ and 
 $\alpha = 0.2$ (Figures \ref{fig:a033_fourth_alpha02_OS},
\ref{fig:a033_fourth_alpha02_RS}), $\alpha = 0.7$ (Figures \ref{fig:a033_fourth_alpha07_OS},
\ref{fig:a033_fourth_alpha07_RS}).  

Figures \ref{fig:a033_fourth_alpha02_OS} and
\ref{fig:a033_fourth_alpha07_OS}  use the optimal choice of
 signs $s_k$, as in Algorithm \ref{algo:HD+AMP}.
Figures 
\ref{fig:a033_fourth_alpha02_RS},
\ref{fig:a033_fourth_alpha07_RS} use random signs $s_k$.
\end{itemize}

 \begin{figure}
\includegraphics[width=0.5\linewidth]{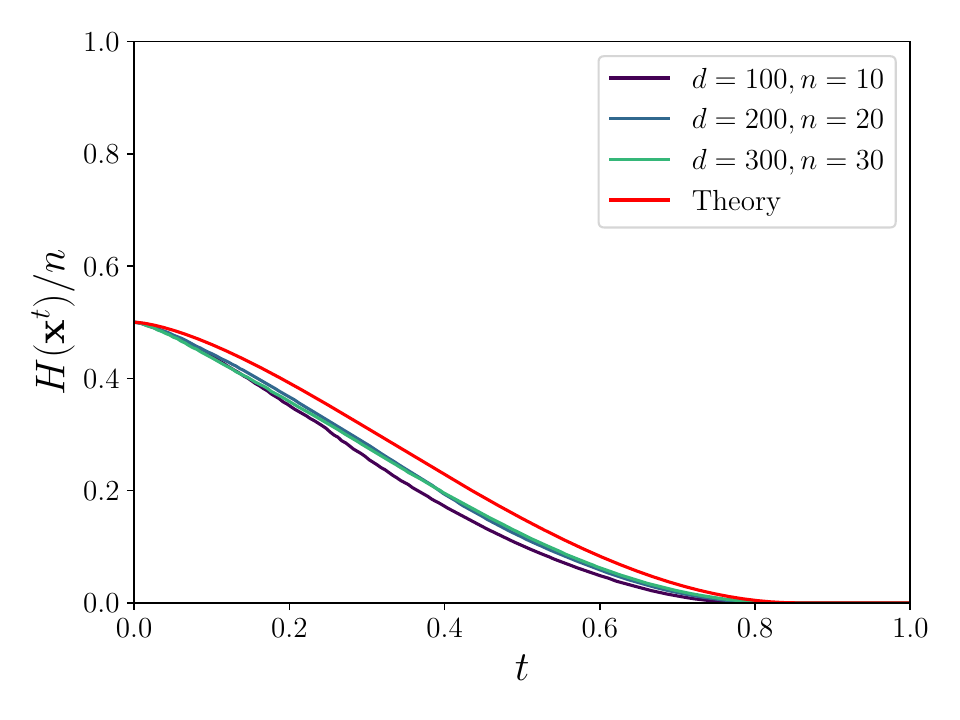}
\includegraphics[width=0.5\linewidth]{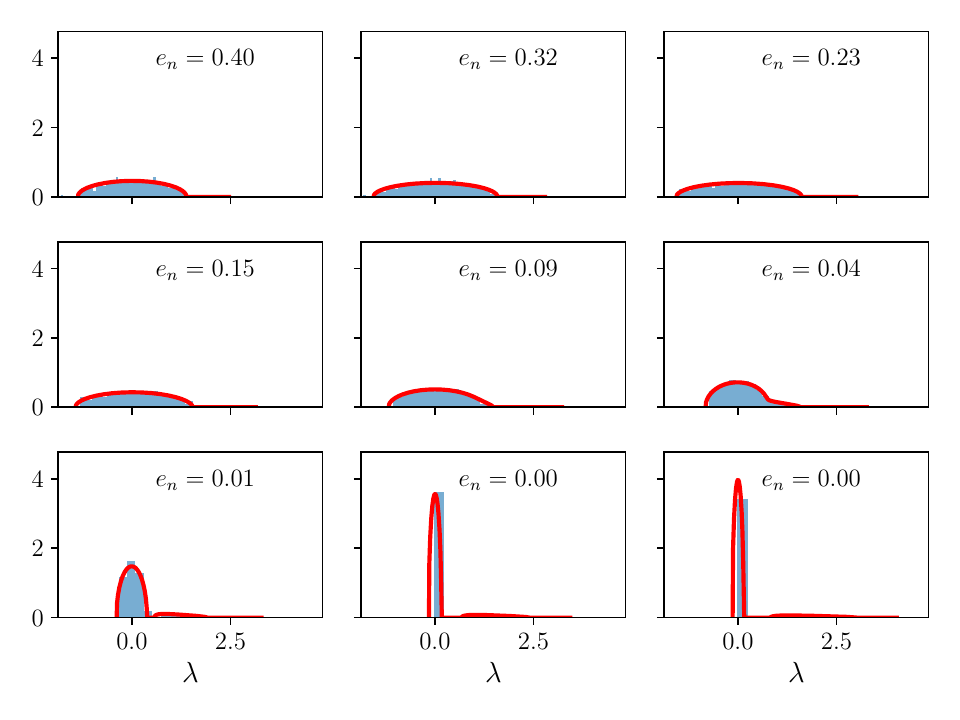}
\caption{Trajectory of Hessian descent, for $\phi(t) = 1+ a \he_3(t)$, with $a=1/4$. 
Left, energy as a function of time (radius): $\alpha=0.1$
 $m=2d$, $K=200$.
Red lines are theoretical predictions for the
algorithm evolution.
Right: Empirical spectral distribution of the Hessian along Hessian descent trajectory
with theoretical predictions as red lines. (Here $d=300$, $n=30$, $m=600$.)}
		\label{fig:a025_third_alpha01}
\end{figure}

\begin{figure}
\includegraphics[width=0.5\linewidth]{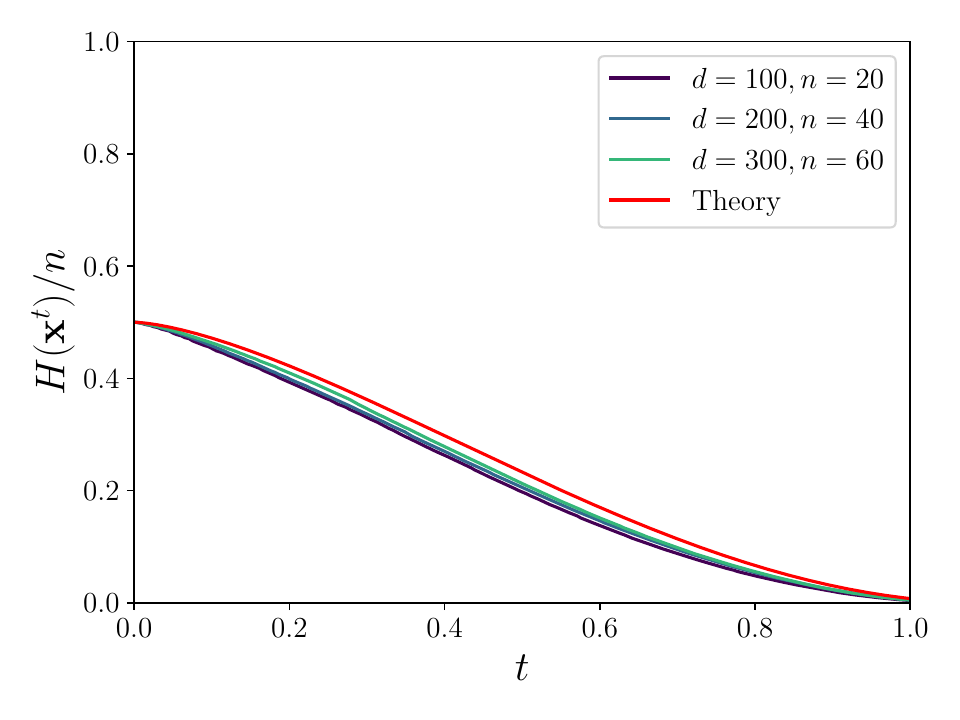}
\includegraphics[width=0.5\linewidth]{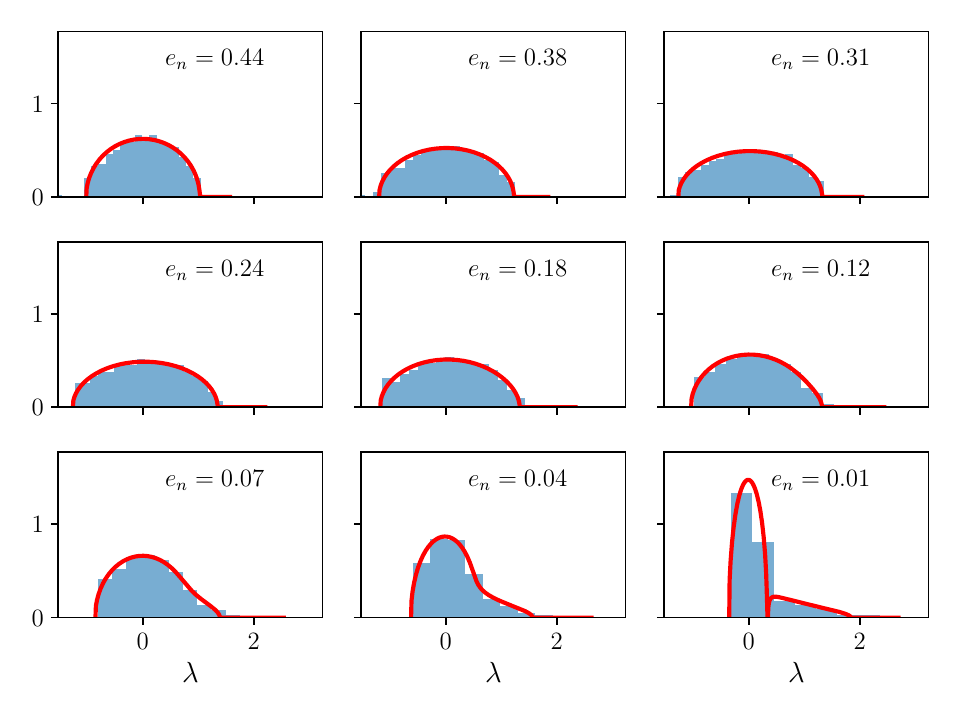}
\caption{Trajectory of Hessian descent, for $\phi(t) =1+ a \he_3(t)$, with $a=1/4$. 
Left, energy as a function of time (radius): $\alpha=0.2$, $m=2d$, $K=200$. 
Red lines are theoretical predictions for the
algorithm evolution.
Right: Empirical spectral distribution of the Hessian along Hessian descent trajectory
with theoretical predictions as red lines. (Here $d=300$, $n=60$, $m=600$.)}
		\label{fig:a025_third_alpha02}
\end{figure}

 \begin{figure}
\includegraphics[width=0.5\linewidth]{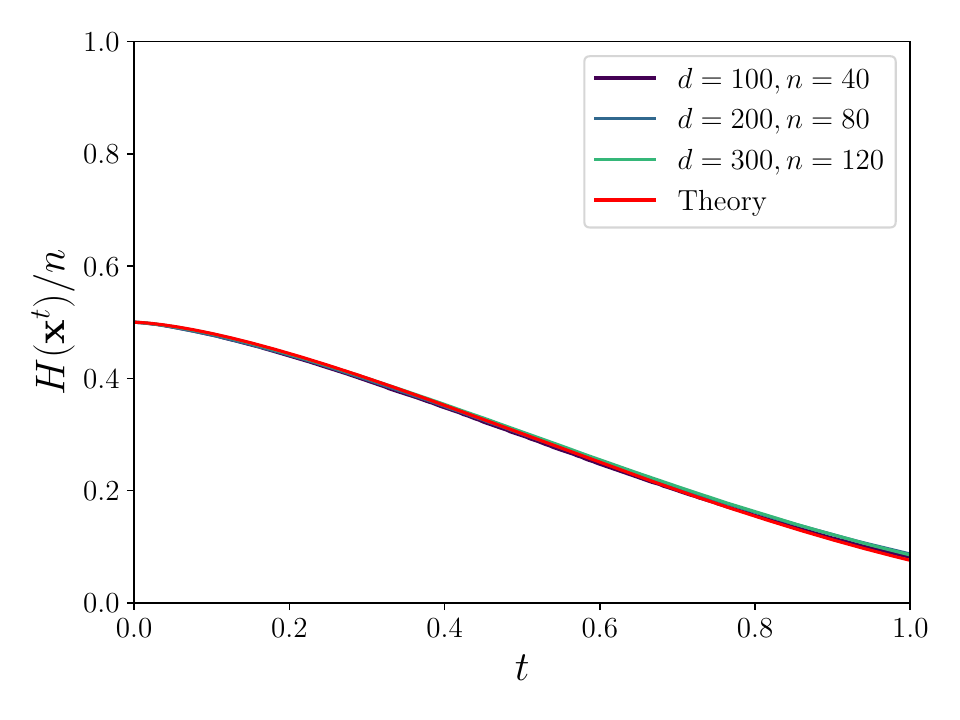}
\includegraphics[width=0.5\linewidth]{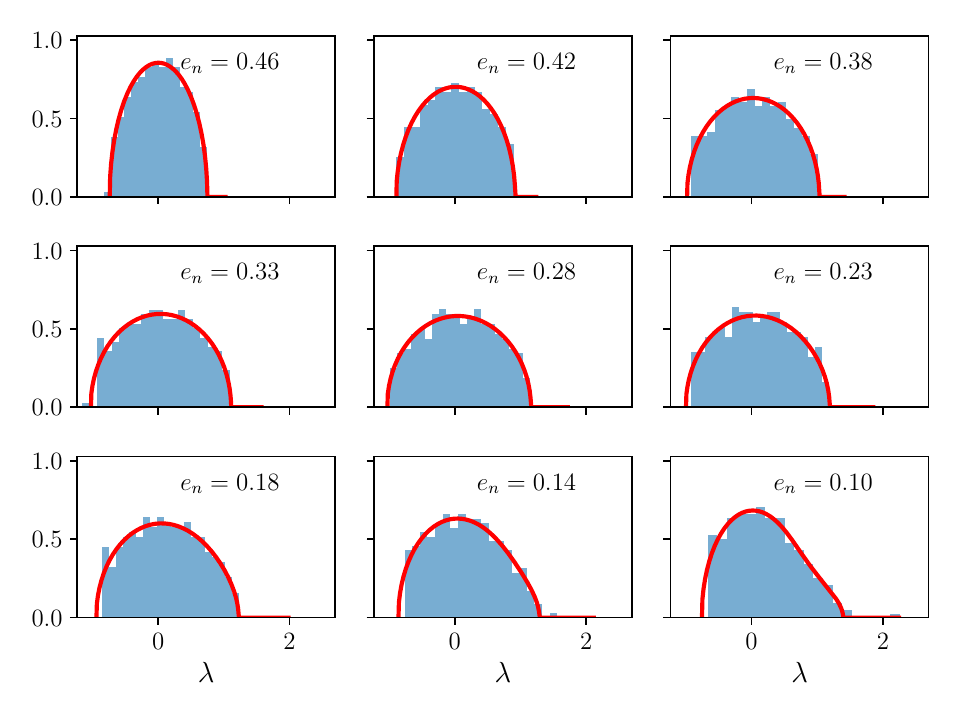}
\caption{Trajectory of Hessian descent, for $\phi(t) =1+ a \he_3(t)$, with $a=1/4$. 
Left, energy as a function of time (radius):  $\alpha =0.4$, $m=2d$, $K=200$. 
Red lines are theoretical predictions for the
algorithm evolution.
Right: Empirical spectral distribution of the Hessian along Hessian descent trajectory
with theoretical predictions as red lines. (Here $d=300$, $n=120$, $m=600$.)}
		\label{fig:a025_third_alpha04}
\end{figure}

\begin{figure}
\includegraphics[width=0.5\linewidth]{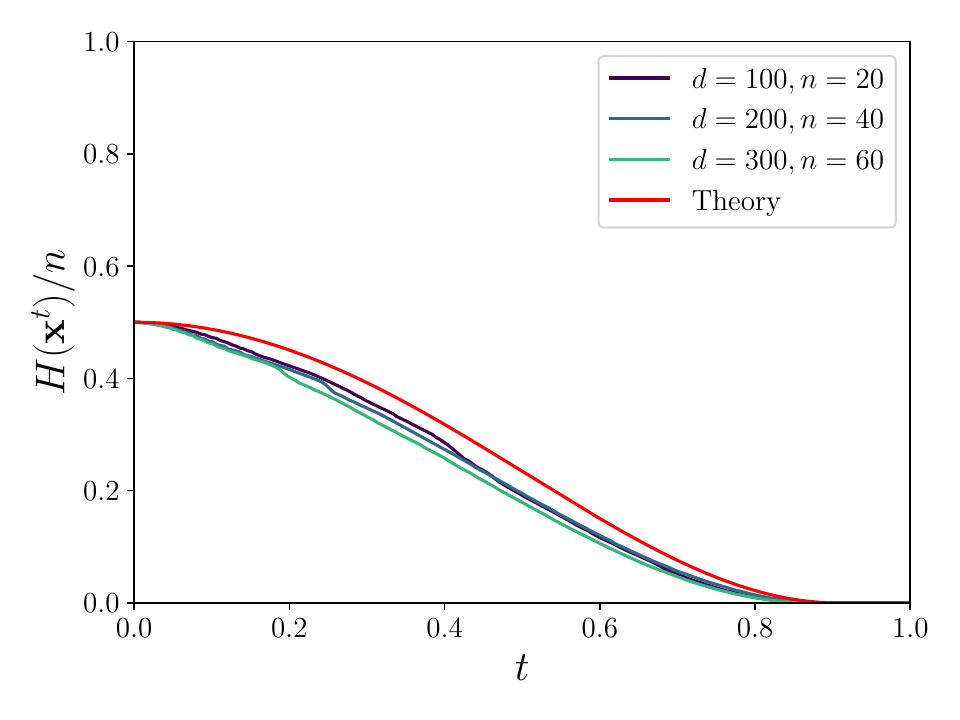}
\includegraphics[width=0.5\linewidth]{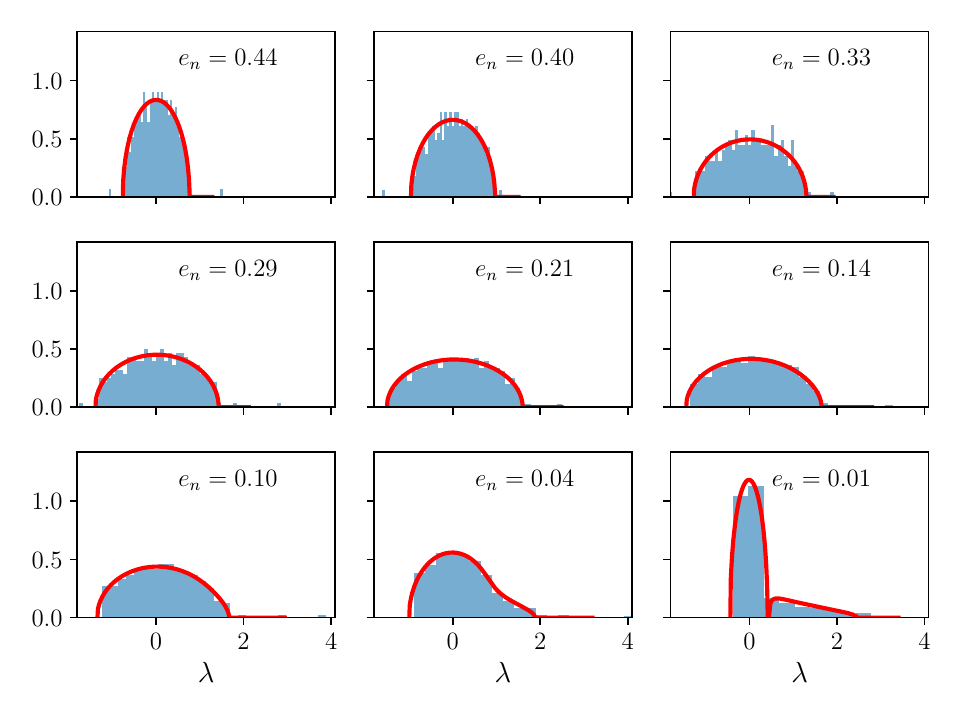 }
\caption{Trajectory of Hessian descent, \emph{with optimally signed steps},
for $\varphi(t) = 1+a \he_4(t)$, with $a=0.33$.
Left, energy as a function of time (radius): $\alpha=0.2$ $m=16d$, $K=200$. 
Red lines are theoretical predictions for the
algorithm evolution.
Right: Empirical spectral distribution of the Hessian along Hessian descent trajectory.
(Here  $d=300$, $n=60$, $m=4800$).}
		\label{fig:a033_fourth_alpha02_OS}
\end{figure}

\begin{figure}
\includegraphics[width=0.5\linewidth]{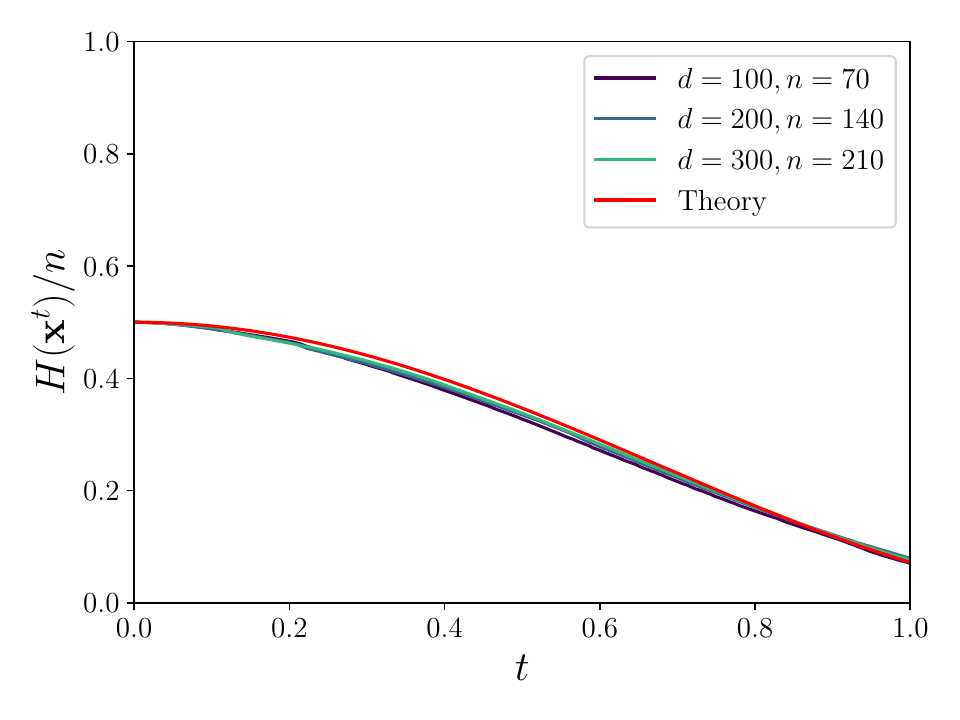}
\includegraphics[width=0.5\linewidth]{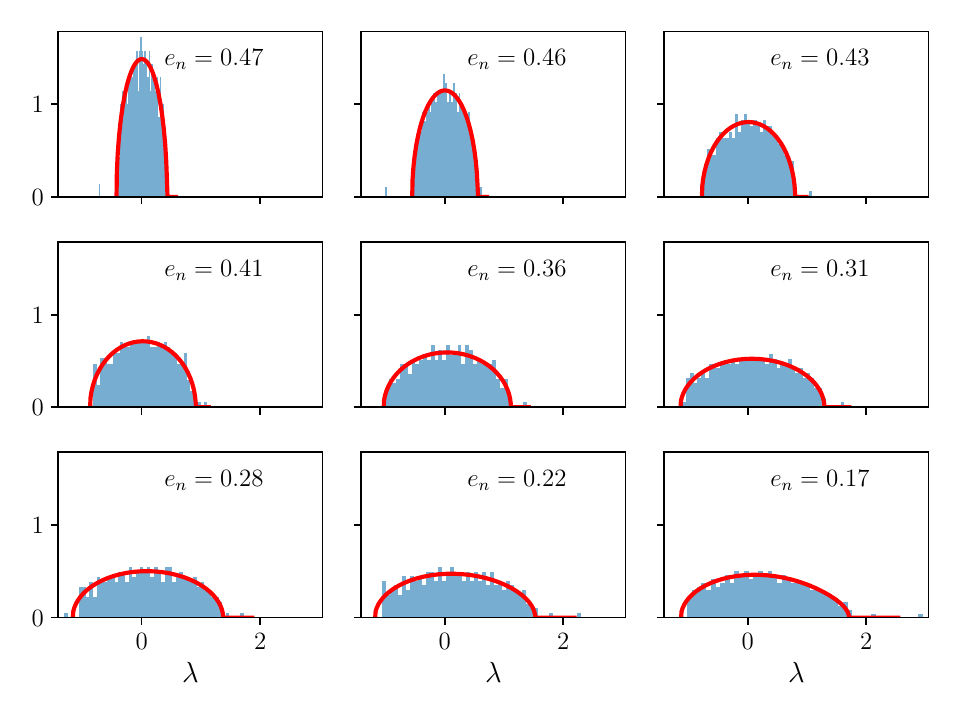}
\caption{Trajectory of Hessian descent, \emph{with optimally signed steps},
for $\varphi(t) = 1+a \he_4(t)$, with $a=0.33$.
Red lines are theoretical predictions for the
algorithm evolution.
Left, energy as a function of time (radius): $\alpha=0.7$, $m=8d$, $K=200$.
Right: Empirical spectral distribution of the Hessian along Hessian descent trajectory.
(Here  $d=300$, $n=210$, $m=4800$.)}
		\label{fig:a033_fourth_alpha07_OS}
\end{figure}

\begin{figure}
\includegraphics[width=0.5\linewidth]{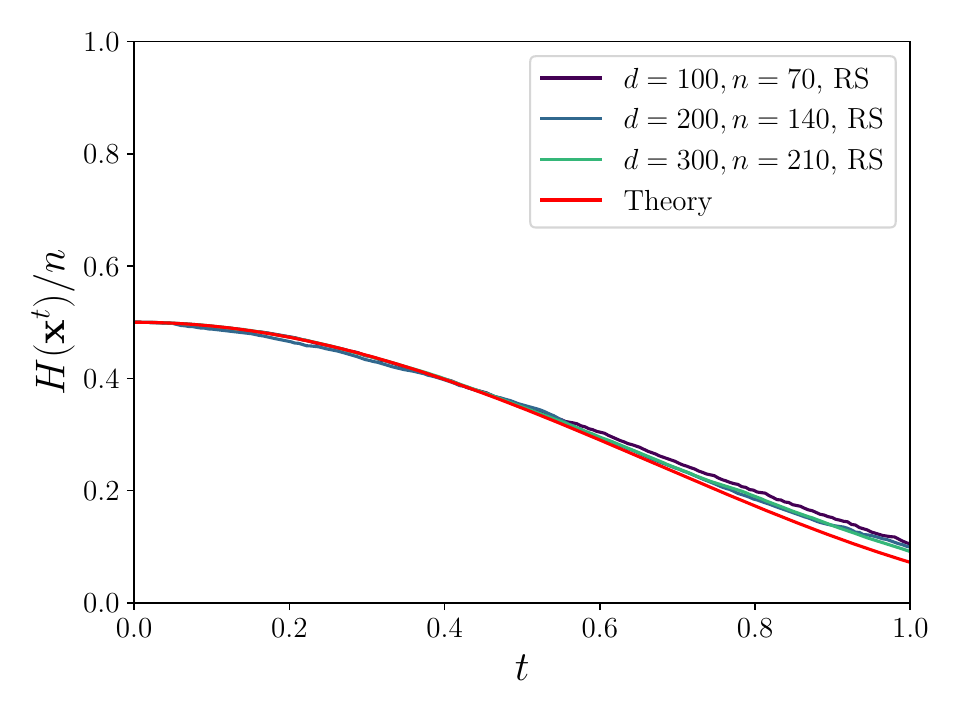}
\includegraphics[width=0.5\linewidth]{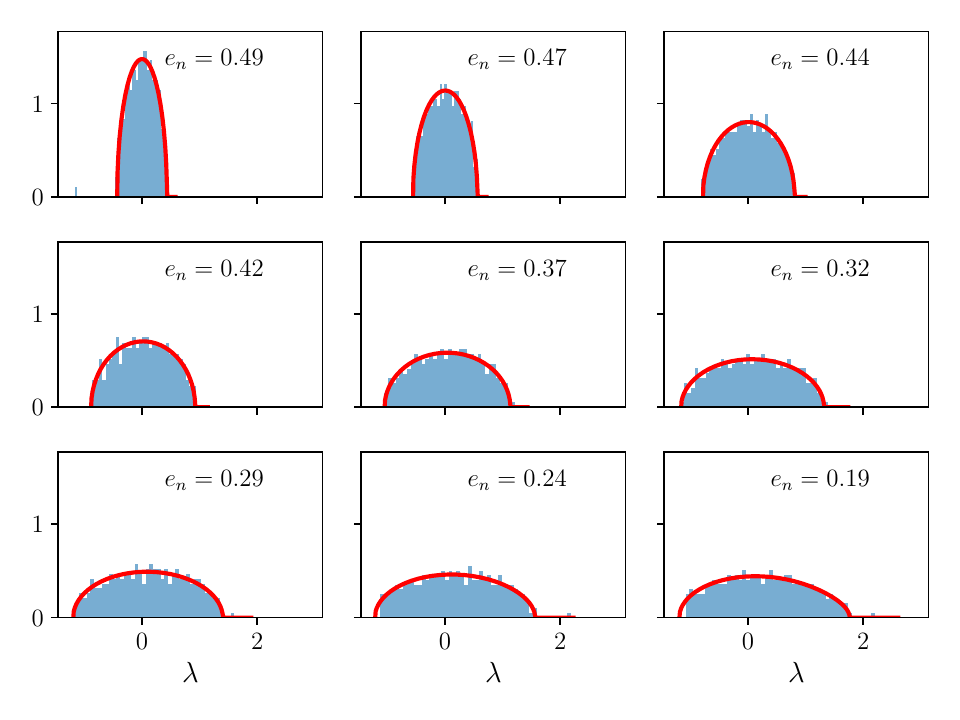 }
\caption{Trajectory of Hessian descent, \emph{with random signed steps},
for $\phi(t) = 1+a \he_4(t)$, with $a=0.33$.
Red lines are theoretical predictions for the
algorithm evolution.
Left, energy as a function of time (radius): $\alpha=0.7$, $m=8d$, $K=200$.
Right: Empirical spectral distribution of the Hessian along Hessian descent trajectory.
(Here  $d=300$, $n=210$, $m=4800$.)}
		\label{fig:a033_fourth_alpha07_RS}
\end{figure}

\subsection{Sensitivity phase transition}
\label{app:SensitivityNumer}

\begin{figure}
	\hspace{-1cm}
		\includegraphics[width=0.5\linewidth]{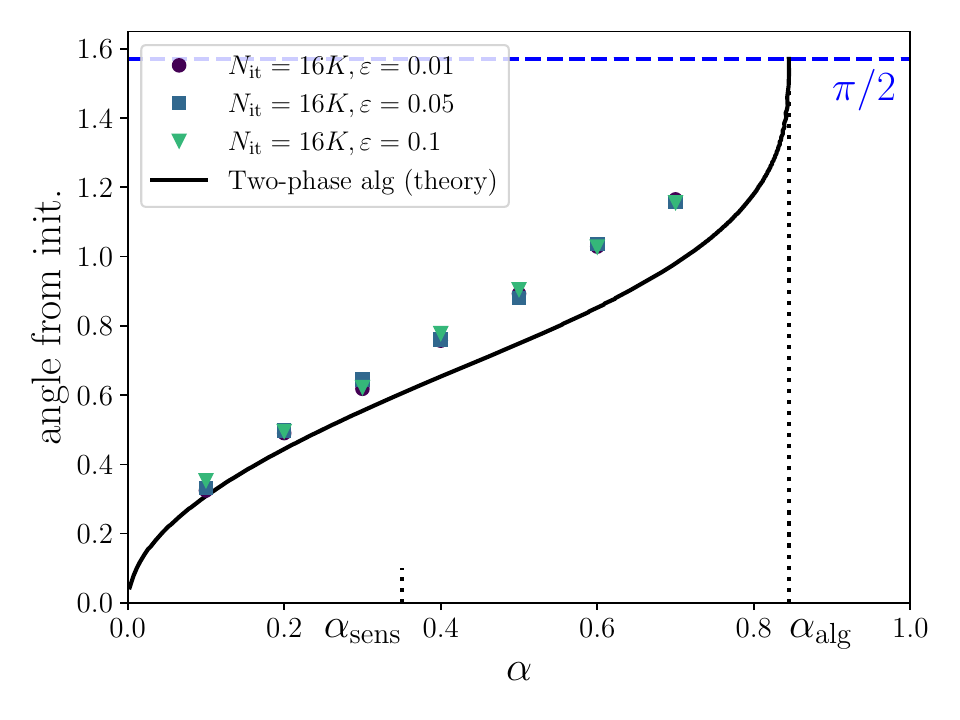}
			\includegraphics[width=0.5\linewidth]{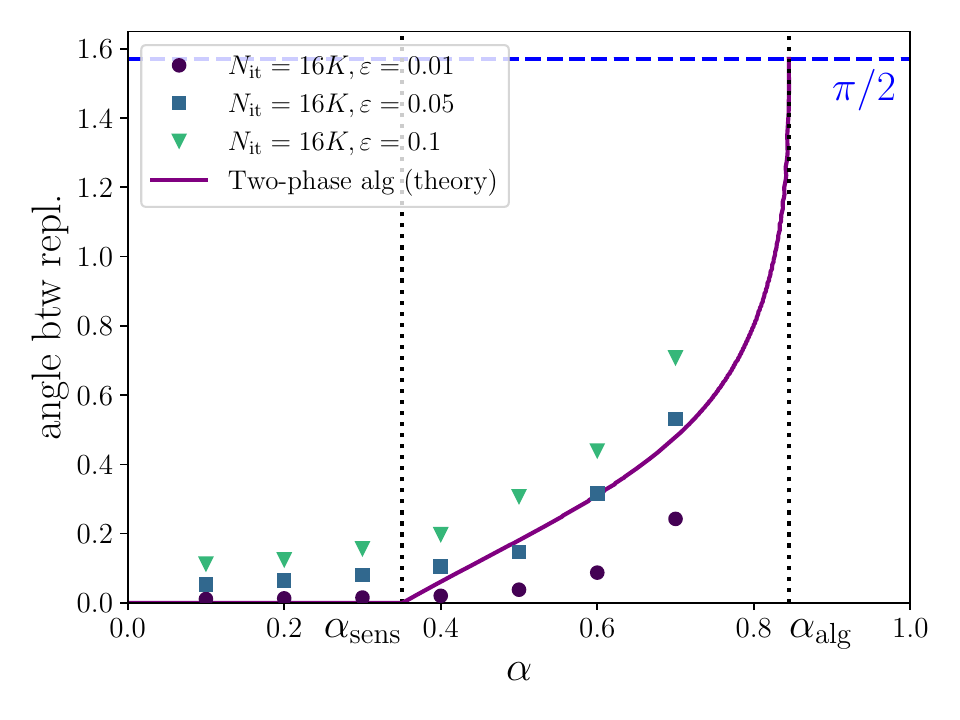}
		\caption{Geometry of SGD (symbols) and two-phase algorithm (continuous
		lines, asymptotic theory) in the AoR model $\phi(t) = 1+ a \he_3(t)$, with $a=0.33$. Here 
		$d=200$, $m=1600$,  $N_{\sit}=16000$, ${\rm lr}=$. Left: angular distance of the algorithm output
		from the random initialization $\bx_0\in \S^{d-1}$. Right: angular distance between 
		outputs between two runs of the optimization algorithm that differ by a perturbation
		of the initialization $\bx_0$.}
		\label{fig:SensitivityInit}
\end{figure}

\begin{figure}
	\hspace{-1cm}
		\includegraphics[width=0.5\linewidth]{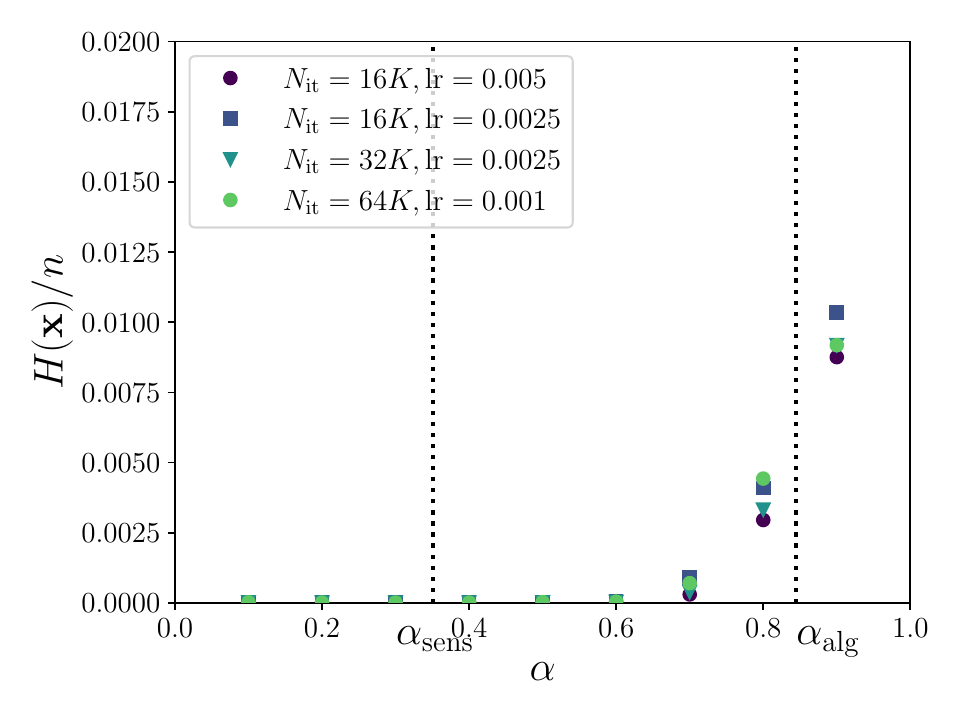}
			\includegraphics[width=0.5\linewidth]{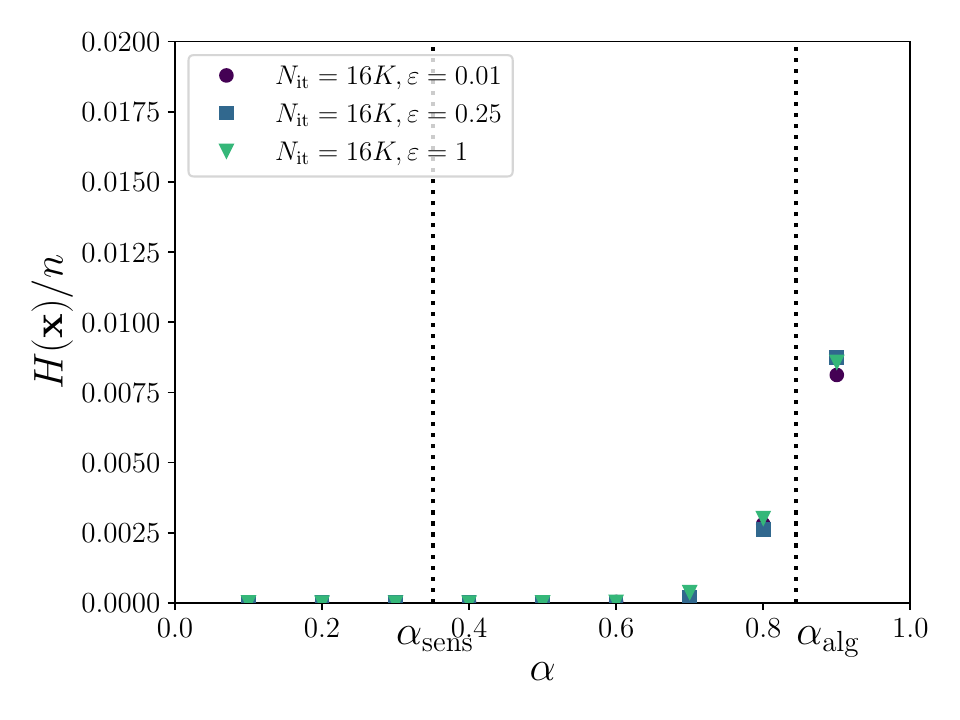}
		\caption{Energy achieved during the sensitivity experiments reported in Section 
		\ref{sec:SensitivityMain} and the present section, see text.}
		\label{fig:SensitivityEnergy}
\end{figure}

In Figure \ref{fig:SensitivityInit} we repeat the sensitivity experiment of Figure 
\ref{fig:Geometry} but change the nature of the perturbation. Namely, we 
generate approximate solutions $\bx_1^*$, $\bx_2^*$, by running SGD with the same
seed for the random number generator (hence, same ordering of the batches),
but different initializations $\bx_{0,1}$, $\bx_{0,2}$ so that 
$\<\bx_{0,1},\bx_{0,2}\>\approx \eps$.
More precisely, we let 
\begin{align}
(\bx_{1,0},\bx_{2,0})\sim \normal(0,\bQ_{\eps}\otimes \id/d)\, ,
\;\;\;\;\; \bQ_{\eps} := \left(\begin{matrix}1 & \eps\\
\eps & 1\end{matrix}\right)\, .
\end{align}
(In other words $\{(x_{1,0,i},x_{2,0,i})\}_{i\le d}\sim_{iid}\normal(0,\bQ_{\eps})$.)
We note in passing that $\bx_{1,0},\bx_{2,0}$ have not unit norm, but rather
$\|\bx_{1,0}\|,\|\bx_{2,0}\| = 1+O(d^{-1/2})$. However, since we run projected SGD, 
the vectors are rapidly projected to the unit sphere.

We observe that the results of this experiment are vary similar to the ones presented in
the main text, cf. Fig. \ref{fig:Geometry}. This suggests that the conclusions 
in the main text are robust to the choice of the perturbation.

In Figure \ref{fig:SensitivityEnergy} we report the average energy $H(\bx_*)/n$ achieved
in the two sets of experiments. We observe that the average of $H(\bx^*)/n)$ is below $0.001$
for $\alpha\in \{0.1,\dots,0.7\}$ awhile it becomes larger for $\alpha\in \{0.8,0.9\}$.
We thus exclude data for the last two values of $\alpha$ from 
Figs.~\ref{fig:Geometry} and \ref{fig:SensitivityInit}, since a sufficiently 
good near-optimum is not achieved in these cases. (This failure is not unexpected because 
of the proximity of $\alpha_{\salg}$.)

%*****************************************************
%
\section{Basic concentration estimates}
\label{sec:Basic}

This appendix contains some basic estimates for the process $(\bF(\bx))_{\bx\in\reals^d}$,
which will be useful in the proofs. Throughout $(F_{\ell})_{\ell\ge 0}$ are i.i.d.
Gaussian processes with $\E[F_{\ell}(\bx) F_{\ell}(\by)] = \xi(\<\bx,\by\>)$.

We begin by bounding the expected maximum. Notice that the limiting
value of this quantity is exactly given by Parisi's formula, 
cf.~Eq.~\eqref{eq:SphericalVariational}. However, we find it convenient to derive some explicit estimates.
\begin{proposition}\label{propo:MaxValue}
There exist absolute constants $C_1, C_2$ such that (writing $\log_+(t):=\max(1,\log t)$)
\begin{align}
C_1\sqrt{\frac{\xi'(1)^2}{\xi(1)\xi''(1)}
\log_+ \frac{\xi''(1)}{\xi'(1)}} \le \frac{1}{\sqrt{d\xi(1)}}\E\max_{\bx\in\S^{d-1}} F_1(\bx)\le  
C_2\sqrt{ \log_+ \frac{\xi'(1)}{\xi(1)}}\, .
\end{align}
The same bounds hold if the maximum is taken over $\bx\in\Ball^d(1)$.
\end{proposition}
\begin{proof}
Given $\bx,\by\in\S^{d-1}$, we denote their Euclidean and rescaled canonical distances
by $r$ and $u$  respectively
\begin{align}
r & := \|\bx-\by\|_2 = \sqrt{2(1-\<\bx,\by\>)}\, ,\\
u & := \sqrt{\frac{1}{\xi(1)}\E\big\{\big[F_1(\bx)-F_1(\by)\big]^2\big\}}
= \sqrt{2\Big(1-\frac{\xi(\<\bx,\by\>)}{\xi(1)}\Big)}\, .\label{eq:u}
\end{align}
We denote by $r(u)$ the (strictly increasing) function that maps $u$ to $r$, namely
the unique solution of the equation
\begin{align}
1-\frac{u^2}{2}=\frac{\xi(1-r^2/2)}{\xi(1)} \, .
\end{align}
Since $\xi'(t)\leq \xi'(1)$ on $[0,1]$, %Considering the even and odd
%parts of $\xi$ separately and using convexity of $\xi$ on $[0,1]$), we obtain
%
\begin{align}
\frac{\xi(1-r^2/2)}{\xi(1)} \ge 1-\frac{\xi'(1)}{\xi(1)}\cdot \frac{r^2}{2}\, ,
\end{align}
whence 
\begin{align}
r(u)\ge \sqrt{\frac{\xi(1)}{\xi'(1)}}\, u\, .
\end{align}
Let $N_d(r)\le [(10/r)\vee 1]^d$ (respectively, $N_d^{F_1}(r)$) be the covering number of $\S^{d-1}$ w.r.t. the Euclidean distance (respectively, the canonical distance of $F_1$). Noting the $1/\xi(1)$ factor in \eqref{eq:u}, we have that
$$N_d^{F_1}(\sqrt{\xi(1)}u)=  N_d(r(u))\le N_d\bigg(\sqrt{\frac{\xi(1)}{\xi'(1)}}u\bigg)\le \bigg(\sqrt{\frac{\xi'(1)}{\xi(1)}}\frac{10}{u}\vee 1\bigg)^d.$$
Finally notice that, under the canonical distance, ${\rm diam}_{F_1}(\S^{d-1})\le 2\sqrt{\xi(1)}$.
By Dudley's inequality, %letting $N_d(r)\le [(10/r)\vee 1]^d$ be the covering number of $\S^{d-1}$,
we have
\begin{align*}
\E\max_{\bx\in\S^{d-1}} F_1(\bx)&\le 24 \int_0^{2\sqrt{\xi(1)}} \sqrt{\log N_d^{F_1}(u)}\, \de u\\
&\le 24\sqrt{\xi(1)} 
\int_0^2 \sqrt{\log N_d^{F_1}(\sqrt{\xi(1)}u)}\, \de u\\
& \le 24\sqrt{d\xi(1)} \int_0^2 \sqrt{\log_+ \Big(\sqrt{\frac{\xi'(1)}{\xi(1)}}\cdot \frac{10}{u}
\Big)}\, \de u\\
& \le C \sqrt{d\xi(1)\log_+\frac{\xi'(1)}{\xi(1)}}\, .
\end{align*}
This proves the desired upper bound.

In order to prove the lower bound, we use Sudakov's inequality.  Recall that the $r$-packing number is lower bounded by the $2r$-covering number. 
For any  $0<r<1$, letting $t=r^2/2$, we have
\begin{align*}
\E\max_{\bx\in\S^{d-1}} F_1(\bx) & \ge C'
\sqrt{\big(\xi(1)-\xi(1-(r^2/2)\big) \log N_d(2r) }\\
& \ge C' \sqrt{d\big(\xi(1) -\xi(1-t)\big) \log_+ \frac{1}{t}} \\
& \ge C' \sqrt{d\big(\xi'(1)t-\xi''(1)t^2/2\big)   \log_+ \frac{1}{t}  }\, ,
\end{align*}
where $C'>0$ is a constant that may change from line to line and where in the last step we used the intermediate value theorem, and the fact that 
$\xi''(s)\le \xi''(1)$ for $s\le 1$.
The desired inequality follows by choosing $t = \xi'(1)/(2\xi''(1))$.
\end{proof}
\begin{remark}
Let $L$ be an integer-valued random variable with probability distribution $\prob(L=\ell)=\xi_{\ell}/\xi(1)$.
Then the upper and lower bounds of the last proposition are within a constant provided 
$\E[L^2]\le C\E[L]^2$.
\end{remark}

We will next establish upper bounds on the derivatives of the process $F_\ell(\,\cdot\,)$.
Before doing it, we state a useful lemma.
\begin{lemma}\label{lemma:SmallDiam}
Let $(Z(\bx,\bv): \; \bx \in\Ball^d(1),\bv\in\S^{d-1})$ be a continuous centered Gaussian process, 
and assume that, for some $r_0>0$, the following holds for all
sets $A\subseteq \Ball^d(1)$:
\begin{align}
{\rm diam}(A)\le r_0 \;\;\Rightarrow\;\; \E[\sup_{\bx\in A, \bv\in\S^{d-1}}Z(\bx,\bv)]\le 
M(r_0)\, .
\end{align}
Then there exists a universal constant $C$ such that 
\begin{align}
&\E[\sup_{\bx\in \Ball^d(1), \bv\in\S^{d-1}}Z(\bx,\bv)]\le 
M(r_0)+C \sqrt{ d V_*\log_+(1/r_0)}\, ,\\
&V_*:=\sup_{\bx\in\Ball^d(1), \bv\in\S^{d-1}} \Var(Z(\bx,\bv))\, .
\end{align}
\end{lemma}
\begin{proof}
Let $(A_{\ell}:\; \ell\le N)$ be a covering of $\Ball^{d}(1)$ with sets of diameter 
$\diam(A_{\ell})\le r_0$, such that $\log N\le C d\log_+(1/r_0)$. Define
\begin{align}
Z_{\ell}:=\sup_{\bx\in A_{\ell}, \bv\in\S^{d-1}}Z(\bx,\bv)\, .
\end{align}
We then have
\begin{align*}
\prob\Big(\sup_{\bx\in \Ball^d(1), \bv\in\S^{d-1}}Z(\bx,\bv)\ge M(r_0)+t\Big)
&= \prob\Big(\max_{\ell\le N}Z_{\ell}\ge M(r_0)+t\Big)\\
&\le N\cdot \max_{\ell\le N}\prob\Big(Z_{\ell}\ge \E Z_{\ell}+t\Big)\\
& \le \exp\Big\{ C d\log_+(1/r_0) - \frac{t^2}{2V_*}\Big\}\, ,
\end{align*}
where the last step follows from the Borell-TIS inequality. This proves the claim.
\end{proof}

\begin{lemma}\label{lemma:Derivatives}
Recall that $\Proj_{\Ts,\bx}$ denotes the projector onto the tangent space to the sphere 
of radius $\|\bx\|_2$ at $\bx$. Write, as above, $\log_+(t)=\max(\log t;1)$
Then there exist an absolute constant $C$ such that, letting $\hbx:=\bx/\|\bx\|_2$, we have 
\begin{align}
\E\max_{\bx\in\Ball^{d}(1)} \|\Proj_{\Ts,\bx}\nabla F_1(\bx)\|_2& \le  
C\sqrt{d\xi'(1) \log_+ \frac{\xi''(1)}{\xi'(1)}}\, ,\label{eq:PerpGradBasic}\\
\E\max_{\bx\in\Ball^{d}(1)} \big|\<\hbx,\nabla F_1(\bx)\>\big|& \le  
C\sqrt{d\xi''(1) \log_+ \frac{\xi'''(1)}{\xi''(1)}}\, ,\label{eq:ParGradBasic}\\
\E\max_{\bx\in\Ball^{d}(1)} \|\Proj_{\Ts,\bx}\nabla^2 F_1(\bx)\Proj_{\Ts,\bx}\|_{\op}& \le  
C\sqrt{d\xi''(1) \log_+ \frac{\xi^{(3)}(1)}{\xi''(1)}}\, ,\label{eq:HessGradBasic1}\\
\E\max_{\bx\in\Ball^{d}(1)} \|\Proj_{\Ts,\bx}\nabla^2 F_1(\bx)\hbx\|_{\op}& \le  
C\sqrt{d\xi^{(3)}(1) \log_+ \frac{\xi^{(4)}(1)}{\xi^{(3)}(1)}}\, ,\label{eq:HessGradBasic2}\\
\E\max_{\bx\in\Ball^{d}(1)} \<\hbx,\nabla^2 F_1(\bx)\hbx\>& \le  
C\sqrt{d\xi^{(4)}(1) \log_+ \frac{\xi^{(5)}(1)}{\xi^{(4)}(1)}}\,. \label{eq:HessGradBasic3}
\end{align}
(Here $\xi^{(\ell)}$ denotes the $\ell$-th derivative of $\xi$.)
\end{lemma}
\begin{proof}
All these inequalities are analogous to the upper bound in 
Proposition \ref{propo:MaxValue}. We will repeatedly use Lemma
\ref{lemma:SmallDiam} to restrict ourselves to sets of some small diameter $r_0$.

Consider Eq.~\eqref{eq:PerpGradBasic}. Define the following Gaussian process
indexed by $\bx\in \Ball^{d}(1)$, $\bv\in \Ts_{\bx}$, $\|\bv\|_2=1$
\begin{align}
G_{\perp}(\bv,\bx) := \<\bv ,\nabla F_1(\bx)\>\, .
\end{align}
We have 
\begin{align}
\E\big\{G_{\perp}(\bv_1,\bx_1)G_{\perp}(\bv_2,\bx_2)\big\} = 
\<\bv_1,\bx_2\>
\<\bx_1,\bv_2\>\,\xi''(\<\bx_1,\bx_2\>)+\<\bv_1,\bv_2\>\,\xi'(\<\bx_1,\bx_2\>)\, .
 \end{align}
 This, and other covariance calculations below, can be checked either by writing the derivatives of $F_1$ explicitly or using that the order of taking derivatives and expectations can be interchanged. 
 Letting $d_{G,\perp}(\bv_1,\bx_1;\bv_2,\bx_2)$ denote the canonical distance associated to
 this process, we have
 \begin{align*}
 d_{G,\perp}(\bv_1,\bx_1;\bv_2,\bx_2)^2 &=
\xi'(\< \bx_1,\bx_2\>)\|\bv_1-\bv_2\|_2^2+\xi'(\|\bx_1\|^2_2)-2\xi'(\<\bx_1,\bx_2\>)+
 \xi'(\|\bx_2\|^2_2)\\
 & \phantom{AA}+(\<\bv_1,\bx_2\>-\<\bv_2,\bx_1\>)^2\xi''(\<\bx_1,\bx_2\>) \\
 &\phantom{AA}- 
 \<\bv_1,\bx_2\>^2\xi''(\<\bx_1,\bx_2\>)- \<\bv_1,\bx_2\>^2\xi''(\<\bx_1,\bx_2\>)\\
 &\le \xi'(1)\|\bv_1-\bv_2\|_2^2+\xi'(\|\bx_1\|^2_2)-2\xi'(\<\bx_1,\bx_2\>)+
 \xi'(\|\bx_2\|^2_2)\\
 &\phantom{AA} +\xi''(1) \|\bx_1-\bx_2\|^2\|\bv_1-\bv_2\|^2 \\
 &\le 2 \xi'(1)\|\bv_1-\bv_2\|_2^2 +\xi'(\|\bx_1\|^2_2)-2\xi'(\<\bx_1,\bx_2\>)+
 \xi'(\|\bx_2\|^2_2)\, ,
\end{align*}
where the last inequality follows for $\|\bx_1-\bx_2\|_2\le r_0:=\xi'(1)/\xi''(1)$.
Therefore, for any set $A\subseteq \Ball^d(1)$, $\diam(A)\le r_0$, using the Sudakov-Fernique inequality we
have
\begin{align}
\E\max_{\bx\in A}\max_{\bv\in \Ts_{\bx}, \|\bv\|_2=1}
G_{\perp}(\bv,\bx)\le C\E\max_{\|\bv\|_2=1} G_1(\bv) + \E\max_{\|\bx\|_2\le 1} G_2(\bx)\, ,
\end{align}
where $G_1$ is a process with canonical distance  $\xi'(1)\|\bv_1-\bv_2\|_2^2$
and $G_2$  is a process with canonical distance  $\xi'(\|\bx_1\|^2_2)-2\xi'(\<\bx_1,\bx_2\>)+
 \xi'(\|\bx_2\|^2_2)$. 
 By applying the previous proposition to $G_2$ (with $\xi$ replaced by $\xi'$), we obtain
\begin{align}
\E\max_{\bx\in A}\max_{\bv\in \Ts_{\bx}, \|\bv\|_2=1}
G_{\perp}(\bv,\bx)\le C\sqrt{s\xi'(1)} +C\sqrt{d\xi'(1)\log_+ \frac{\xi''(1)}{\xi'(1)}}\, .
\end{align}
Finally, using Lemma \ref{lemma:SmallDiam}, we see that
\begin{align}
\E\max_{\bx\in \Ball^{d}(1)}\max_{\bv\in \Ts_{\bx}, \|\bv\|_2=1}
G_{\perp}(\bv,\bx)\le C\sqrt{d\xi'(1)\log_+ \frac{\xi''(1)}{\xi'(1)}}+ C\sqrt{d\xi'(1)\log_+(1/r_0)}
\, ,
\end{align}
which yields the bound of Eq.~\eqref{eq:PerpGradBasic}. 
 
 Next consider Eq.~\eqref{eq:ParGradBasic}. Recalling the representation
 \eqref{eq:Representation}, we have  
 \begin{align}
\<\hbx,\nabla F_1(\bx)\> = \frac{1}{\|\bx\|_2}\sum_{k\ge 0}k\sqrt{\xi_k}\sum_{j_1,\dots,j_k=1}^d 
G_{1,j_1\dots j_k}^{(k)}x_{j_1}\cdots x_{j_k}\, ,
\end{align}
and therefore
 \begin{align}
\E\{\<\hbx_1,\nabla F_1(\bx_1)\>\<\hbx_2,\nabla F_1(\bx_2)\>\big\} = 
\frac{1}{\|\bx_1\|_2\|\bx_2\|_2}[\<\bx_1,\bx_2\>^2\xi''(\<\bx_1,\bx_2\>)+\<\bx_1,\bx_2\>\xi'(\<\bx_1,\bx_2\>)]\, .
\end{align}
In other words, $\<\hbx,\nabla F_1(\bx)\>$ has the same structure as $F_1(\bx)$
provided we replace $\xi$ by $\xi''+\xi'$. Therefore the claim follows 
by argument in Proposition \ref{propo:MaxValue}.

The bounds \eqref{eq:HessGradBasic1}, \eqref{eq:HessGradBasic2}, \eqref{eq:HessGradBasic3}
 follow from similar arguments, and we limit ourselves to defining the relevant Gaussian process 
 and computing its covariance:
 \begin{itemize}
 \item For Eq.~\eqref{eq:HessGradBasic1}, we define $H_{\perp}(\bv,\bx)$ indexed
 by $\bx\in \Ball^{d}(1)$ and $\bv\in \Ts_{\bx}$ via
 \begin{align}
 H_{\perp}(\bv,\bx) := \<\bv,\nabla^2 F_1(\bx)\bv\>\, .
 \end{align}
 Its covariance is (denoting by $\xi^{(\ell)}$ the $\ell$-th derivative of $\xi$
 and letting $q:=\<\bx_1,\bx_2\>$)
 \begin{align*}
 \E\big\{ H_{\perp}(\bv_1,\bx_1) H_{\perp}(\bv_2,\bx_2)\big\}
 =&\, \xi^{(4)}(q) \, \<\bx_1,\bv_2\>^2 \<\bv_1,\bx_2\>^2+
 4 \, \xi^{(3)}(q) \<\bx_1,\bv_2\> \<\bv_1,\bx_2\>\<\bv_1,\bv_2\>\\
 &+
 2 \, \xi^{(2)}(q) q^2\<\bv_1,\bv_2\>^2\, .
 \end{align*}
 The associated canonical distance is given by
 \begin{align*}
 d_{H,\perp}(\bv_1,\bx_1;\bv_2,\bx_2)^2 &= \Delta_{1,2} +
 2\xi^{(2)}(\|\bx_1\|^2_2)-4\xi^{(2)}(\<\bx_1,\bx_2\>)+
 2\xi^{(2)}(\|\bx_2\|^2_2)\, ,
 \end{align*}
 where $\Delta_{1,2}$ can be bounded as follows for $\|\bx_1-\bx_2\|_2\le c_0$
 with $c_0$ a small absolute constant (recall that $\Proj_{\Ts,\bx_i}$ is the projector orthogonal to $\bx_i$)
 \begin{align*}
 \Delta_{1,2} & = -2\xi^{(4)}(q) \, \<\bx_1,\bv_2\>^2 \<\bv_1,\bx_2\>^2-8
  \, \xi^{(3)}(q) \<\bx_1,\bv_2\> \<\bv_1,\bx_2\>\<\bv_1,\bv_2\>\\
 &\phantom{=}+
 4 \, \xi^{(2)}(q) q^2(1-\<\bv_1,\bv_2\>^2)\\
 &\le 8\xi^{(3)}(q)
  \big|\<\Proj_{\Ts,\bx_2}\bx_1,\bv_1-\bv_2\>\<\Proj_{\Ts,\bx_1}\bx_2,\bv_1-\bv_2\>\big|\\
  &\phantom{=}+8\xi^{(2)}(q) \|\bv_1-\bv_2\|_2^2\\
  &\le 8\big(\xi^{(3)}(q)\|\bx_1-\bx_2\|_2^2+\xi^{(2)}(q)\big)\|\bv_1-\bv_2\|^2_2\\
  &\le 8\big(\xi^{(3)}(q)(1-q)+ \xi^{(2)}(q)\big)\|\bv_1-\bv_2\|^2_2\\
  &\le 8 \xi^{(2)}(1)\|\bv_1-\bv_2\|^2_2\, .
 \end{align*}
 Hence the proof of Eq.~\eqref{eq:HessGradBasic1} follows by applying Propopsition
 \ref{propo:MaxValue} and Lemma \ref{lemma:SmallDiam}.
 \item For Eq.~\eqref{eq:HessGradBasic2}, we define $H_{2}(\bv,\bx)$ indexed
 by $\bx\in \Ball^{d}(1)$ and $\bv\in \Ts_{\bx}$ via
 \begin{align}
 H_{2}(\bv,\bx) := \<\hbx,\nabla^2 F_1(\bx)\bv\>\, .
 \end{align}
 Its covariance is (here $q:=\<\bx_1,\bx_2\>$)
 \begin{align*}
 \E\big\{ H_{2}(\bv_1,\bx_1) H_{2}(\bv_2,\bx_2)\big\}=&
 \frac{\<\bv_1,\bv_2\>}{\|\bx_1\|_2\|\bx_2\|_2}\big[\xi^{(3)}(q)q^2+\xi^{(2)}(q)q\big]\\
 &+ \frac{\<\bv_1,\bx_2\>\<\bx_1,\bv_2\>}{\|\bx_1\|_2\|\bx_2\|_2}
 \big[\xi^{(4)}(q)q^2+3\xi^{(3)}(q)q
 +\xi^{(2)}(q)\big]\, .
 \end{align*}
 \item Finally, for \eqref{eq:HessGradBasic3} we define
 \begin{align*}
 H_3(\bx):= \<\hbx,\nabla^2 F_1(\bx)\hbx\> \, ,
 \end{align*}
and note that  (again $q:=\<\bx_1,\bx_2\>$)
 \begin{align*}
 \E\big\{ H_{3}(\bx_1) H_{3}(\bx_2)\big\}=& \frac{1}{\|\bx_1\|_2^2\|\bx_2\|_2^2}
 \big[\xi^{(4)}(q)q^4+4\xi^{(3)}(q)q^3+2\xi^{(2)}(q)q^2\big]\, .
 \end{align*}
 \end{itemize}
\end{proof}

\begin{proposition}\label{propo:OPNorm}
Let $\bD\bF(\bx)\in \reals^{n\times d}$ be the Jacobian of $\bF$ at $\bx$ and
$\bD^2\bF(\bx)\in \reals^{n\times d\times d}$ the tensor of second derivatives.
We equivalently view the latter as a linear operator 
$\bD^2\bF(\bx):\reals^{d\times d}\to \reals^n$, and 
write $\bD^2\bF(\bx)|_{\Ts\otimes \Ts}$ for its restriction to $\Ts_{\bx}\otimes \Ts_{\bx}$
and similarly for $\bD^2\bF(\bx)|_{\Ts\otimes \hbx}$ and $\bD^2\bF(\bx)|_{\hbx\otimes \hbx}$.

Then there exists an absolute constant $C$ such that
(using the notation $\log_+(t):=\max(\log(t);1)$ introduced above)
\begin{align}
\E\big[\max_{\bx\in\Ball^d(1)} \|\bD \bF(\bx)|_{\Ts}\|_{\op}\big]\le 
C\sqrt{(n\vee d) \, \xi'(1)\log_+\frac{\xi''(1)}{\xi'(1)} }\, ,\label{eq:JacobianBasicBound1}\\
\E\big[\max_{\bx\in\Ball^d(1)} \|\bD \bF(\bx)|_{\hbx}\|_{\op}\big]\le 
C\sqrt{(n\vee d) \, \xi^{''}(1)\log_+\frac{\xi^{(3)}(1)}{\xi''(1)} }\, ,\label{eq:JacobianBasicBound2}\\
\E\big[\max_{\bx\in\Ball^d(1)} \|\bD^2 \bF(\bx)|_{\Ts\otimes \Ts}\|_{\op} \big]\le 
C\sqrt{(n\vee d) \, \xi''(1)\log_+\frac{\xi^{(3)}(1)}{\xi''(1)} }\, ,\label{eq:DDFnorm1}\\
\E\big[\max_{\bx\in\Ball^d(1)} \|\bD^2 \bF(\bx)|_{\Ts\otimes \hbx}\|_{\op} \big]\le 
C\sqrt{(n\vee d)\, \xi^{(3)}(1)\log_+\frac{\xi^{(4)}(1)}{\xi^{(3)}(1)} }\, ,\label{eq:DDFnorm2}\\
\E\big[\max_{\bx\in\Ball^d(1)} \|\bD^2 \bF(\bx)|_{\hbx\otimes \hbx}\|_{\op} \big]\le 
C\sqrt{(n\vee d)\, \xi^{(4)}(1)\log_+\frac{\xi^{(5)}(1)}{\xi^{(4)}(1)} }\, .\label{eq:DDFnorm3}
\end{align}
\end{proposition}
\begin{proof}
All of these bounds follow from Lemma \ref{lemma:Derivatives} via the same 
argument. We will focus to be definite on the bound \eqref{eq:JacobianBasicBound1}.
We define the following Gaussian 
process indexed by $\bu\in \reals^n$, $\|\bu\|_2=1$,
 $\bx\in\Ball(1)$, $\bv\in\Ts_{\bx}$, $\|\bv\|_2=1$: 
\begin{align}
Z(\bu,\bv,\bx):=\<\bu,\bD \bF(\bx)\bv\>\, ,
\end{align}
and notice that of course 
\begin{align}
\max_{\bx\in\Ball^d(1)} \|\bD \bF(\bx)|_{\Ts}\|_{\op}=\max_{\bu\in\reals^n,\|\bu\|_2=1} 
\max_{\bv\in\Ts_{\bx},\|\bv\|_2=1}\max_{\bx\in\Ball^d(1)} Z(\bu,\bv,\bx)\, .
\end{align}
We next compute the canonical distance of  $Z$, to get 
\begin{align}
d_Z(\bu_1,\bv_1,\bx_1;\bu_2,\bv_2,\bx_2)^2 = 
\frac{1}{2}
\|\bu_1-\bu_2\|^2(\xi'(\|\bx_1\|^2)+\xi'(\|\bx_2\|^2))
+\<\bu_1,\bu_2\> d_{G,\perp}(\bv_1,\bx_1;\bv_2,\bx_2)^2\,  ,
 \end{align}
 where $d_{G,\perp}(\bv_1,\bx_1;\bv_2,\bx_2)^2$ is the canonical distance of
 the process $G_{\perp}(\bv,\bx) := \<\bv ,\nabla F_1(\bx)\>$
 (this was derived in the proof of Lemma~\ref{lemma:Derivatives} but will not be needed 
 here).  We bound this distance as

\begin{align}
d_Z(\bu_1,\bv_1,\bx_1;\bu_2,\bv_2,\bx_2)^2 \le
\xi'(1)
\|\bu_1-\bu_2\|^2
+d_{G,\perp}(\bv_1,\bx_1;\bv_2,\bx_2)^2\,  ,
 \end{align}
 Therefore by Sudakov-Fernique, letting $\bg\sim\normal(0,\id_{n})$,
 \begin{align}
 \E\big[\max_{\bx\in\Ball^d(1)} \|\bD \bF(\bx)|_{\Ts}\|_{\op}\big]\le 
 \E\big[\max_{\bu\in\S^{n-1}} \sqrt{\xi'(1)} \<\bg,\bu\>\big]+
 \E\max_{\bx\in\Ball^{d}(1)} \|\Proj_{\Ts,\bx}\nabla F_1(\bx)\|_2\, .
 \end{align}
 The bound \eqref{eq:JacobianBasicBound1} then follows from Lemma \ref{lemma:Derivatives}.
 The other bounds are proved analogously.
\end{proof}

We next use the above estimates to control the Lipschitz constant of $F$ and its derivatives.
Recall that, for $\bx_1,\bx_2\in\reals^d$, $\bx_i\neq \bzero$, we let $\bR_{\bx_1,\bx_2}\in\reals^{d\times d}$ 
be the rotation in the plane $\bx_1,\bx_2$ such that 
$\bR_{\bx_1,\bx_2}\bx_1/\|\bx_1\| = \bx_2/\|\bx_2\|$. Among the two rotations that leave the orthogonal
complement of $\bx_1,\bx_2$ unchanged, choose the one with smallest angle, and break ties arbitrarily.
(In other words, $\bR_{\bx_1,\bx_2}$ is the parallel transport on $\S^{d-1}$.)

\begin{remark}\label{rmk:MaxValBorell}
By Borell-TIS inequality, the bounds in Proposition \ref{propo:MaxValue}, Lemma \ref{lemma:Derivatives} and
Proposition \ref{propo:OPNorm} with modified constants also hold with probability at least $1-\exp(-C(\xi)d)$
for some $C(\xi)>0$. 
\end{remark}

\begin{definition}\label{def:Lip}
For $\bx\in \reals^d$, we choose $\bU_{\bx}\in\reals^{d\times(d-1)}$ such that
$\bU_{\bx}^{\sT}\bU_{\bx}=\id_{d-1}$, $\bU_{\bx}^{\sT}\bx = \bzero$
(i.e. $\bU_{\bx}$ is an orthonormal basis for the tangent space $\Ts_{\bx}$.)
Further, for $\bx_1,\bx_2\in \reals^d$, define 
$\bU_{\bx_1,\bx_2}:= \bR_{\bx_1,\bx_2}\bU_{\bx_1}$.

For $\Omega\subseteq \reals^d$, we
define the following Lipschitz constants
\begin{align}
\Lip(\bF;S) &:= \sup_{\bx_1\neq\bx_2\in\Omega}\frac{\|\bF(\bx_1)-\bF(\bx_2)\|}{\|\bx_1-\bx_2\|_2}\, ,\\
\Lip_{\perp}(\bD\bF;\Omega)& :=\sup_{\bx_1\neq\bx_2\in\Omega}
\frac{\|\bD\bF(\bx_1)\bU_{x_1}-\bD\bF(\bx_2)\bU_{\bx_1,\bx_2}\|_{\op}}{\|\bx_1-\bx_2\|_2}\, ,\\
\Lip_{\perp}(\nabla^2\bF;\Omega)& :=\sup_{\bx_1\neq\bx_2\in\Omega}
\max_{\ell\le n}
\frac{\|\bU_{\bx_1}^{\sT}\nabla^2 F_{\ell}(\bx_1)\bU_{\bx_1}-
\bU_{\bx_1,\bx_2}^{\sT}\nabla^2 F_{\ell}(\bx_2)\bU_{\bx_1,\bx_2}\|_{\op}}{\|\bx_1-\bx_2\|_2}\, .
\end{align}
\end{definition}
\begin{lemma}\label{lemma:BoundLip}
Recall the notation $\log_+(t):=\max(\log(t);1)$.
Then there exists a universal constant $C_1$, and a $\xi$-dependent constant $C_*(\xi)$
such that the following hold with probability at least $1-\exp(-(n\vee d)/C_*(\xi))$:
\begin{align}
\Lip(\bF;\Ball^d(1)) & \le C_1 \sqrt{(n\vee d) \, \xi''(1)\log_+\frac{\xi^{(3)}(1)}{\xi''(1)} } \, ,
\label{eq:LipF}\\ 
\Lip_{\perp}(\bD\bF;\Ball^d(1))&  \le C_1\sqrt{(n\vee d)\, \xi^{(4)}(1)\log_+\frac{\xi^{(5)}(1)}{\xi^{(4)}(1)} }
\, ,\label{eq:LipDf}\\ 
\Lip_{\perp}(\nabla^2\bF;\Ball^d(1))& \le C_*(\xi)\sqrt{n\vee d}\, .\label{eq:LipDDf}
\end{align}
Further, the following tighter Lipshitz constant holds on $\S^{d-1}$:
\begin{align}
 \Lip(\bF;\S^{d-1}) & \le C_1 \sqrt{(n\vee d) \, \xi'(1)\log_+\frac{\xi''(1)}{\xi'(1)} } \, ,
 \label{eq:LipF_Sphere}\\ 
\Lip_{\perp}(\bD\bF;\S^{d-1})&  \le C_1\sqrt{(n\vee d)\, \xi''(1)\log_+\frac{\xi^{(3)}(1)}{\xi''(1)} }\, .
\label{eq:LipDf_Sphere}
\end{align}
\end{lemma}
\begin{proof}
Throughout, we make use of Remark \ref{rmk:MaxValBorell}.
\begin{itemize}
\item The bound \eqref{eq:LipF} follows from Eqs.~\eqref{eq:JacobianBasicBound1},
\eqref{eq:JacobianBasicBound2}. 
\item The bound \eqref{eq:LipDf} follows from Eqs.~\eqref{eq:DDFnorm1},
\eqref{eq:DDFnorm2}, \eqref{eq:DDFnorm3}.
\item The proof of Eq.~\eqref{eq:LipDDf} follows by bounding $\max_{\bx}\|\nabla^3F_{\ell}(\bx)\|$,
which in turn can be done by the same technique as in the proof of Lemma \ref{lemma:Derivatives}.
Since we are not seeking a precise characterization of the constant  $C_*(\xi)$
it suffices to notice that the canoninical distance of the process 
$\<\nabla^3F_{\ell}(\bx),\bv^{\otimes 3}\>$ is bounded by a smooth function of the distances
$\|\bv_1-\bv_2\|$, $\|\bx_1-\bx_2\|$.
\item  The bound \eqref{eq:LipF_Sphere} follows from Eqs.~\eqref{eq:JacobianBasicBound1} 
by noting that, for $\bx_1,\bx_2\in \S^{d-1}$, letting $\bx(t)$ denote a geodesic on the sphere
parametrized by arclength, we have
\begin{align}
\bF(\bx_2)-\bF(\bx_1) = \int_{0}^{t_*}\bD \bF(\bx(t))\dot\bx(t) \, \de t\, .
\end{align}
Here $t_* :=\arccos(\<\bx_1,\bx_2\>)$ and we notice that $\dot\bx(t)\in \Ts_{\bx(t)}$. 
\item Finally, for the bound \eqref{eq:LipDf_Sphere} we consider
 a geodesic path between $\bx_1$ and $\bx_2$ in the sphere. 
For any two vectors $\bv_1\in \Ts_{\bx_1}$, $\bv_2 =\bR_{\bx_1,\bx_2}\bv_1\in\Ts_{\bx_2}$,
let $\bv(t)$ be the parallel transport of $\bv_1$ along the geodesic connecting $\bx_1$
to $\bx_2$. We then have
\begin{align}
\bD\bF(\bx_2)\bv_2-\bD\bF(\bx_1)\bv_1 = 
\int_{0}^{t_*}\big\{\bD^2 \bF(\bx(t))\{\dot\bx(t),\bv(t)\} + 
\bD \bF(\bx(t))\dot\bv(t)\big\} \, \de t\, .
\end{align}
Hence the claim follows from Eqs.~\eqref{eq:JacobianBasicBound2} and
\eqref{eq:DDFnorm1}
\end{itemize}
\end{proof}
%
%****************************************************************** 
%
\section{Upper and lower bounds on the satisfiability threshold}
\label{sec:UpperLowerBounds}
\subsection{Lower bound:
Proof of Proposition \ref{thm:SecMoment}}
\label{sec:Existence_LB}

In this section and next ones, we prove our bounds on the satisfiability
threshold.
An illustration of these bounds is given in Figure \ref{fig:PhaseDiagP11}.

Before proceeding, we recall two elementary facts, in the next two remarks.
\begin{remark}
The solutions set is almost surely  empty for $n>d-1$.
To see this, write $F(\bx) = \bg+\bF_{>0}(\bx)$ where $\bg\sim\normal(\bzero,\xi_0\id_n)$
is independent of $\bF_{>0}$ (which has covariance determined by $\xi_{>0}(t) = \xi(t)-\xi(0)$).
Then solutions exists only if $\bg\in \cS_{d-1}:=\bF_{>0}(\S^{d-1})$ (the image of $\S^{d-1}$
under $\bF_{>0}$). However $\Vol_{n}(\cS_{d-1})=0$ on the event 
$\sup_{\bx\in\S^{d-1}}\|\bD\bF_{>0}(\bx)\|_{\op}<\infty$ (by a change of variables argument),
and the latter happens with probability one (by Proposition \ref{propo:OPNorm}).
\end{remark}

\begin{remark}\label{rmk:Concentration}
By Gaussian concentration we can
often focus on $\E \min_{\bx\in\S^{d-1}}\|\bF(\bx)\|_2$.

Indeed, for each $i\le n$, and each $\bx\in\S^{d-1}$
the function $(G^{(k)}_{i,j_1,\dots,j_k})\mapsto F_i(\bx)$ is Lipschitz continuous
(in $\ell_2$), with Lipschitz norm bounded by $\sqrt{\xi(1)}$ (recall that $\xi(1)=\sum_k \xi_k$). As a consequence
$(G^{(k)}_{i,j_1,\dots,j_k})\mapsto \min_{\bx\in\S^{d-1}}\|\bF(\bx)\|_2/\sqrt{n}$
is also Lipschitz hence it concentrates:
\begin{align}
\prob\Big(\big|\min_{\bx\in\S^{d-1}}\|\bF(\bx)\|_2 - \E \min_{\bx\in\S^{d-1}}\|\bF(\bx)\|_2
\big|\ge t\Big)\le 2\, e^{-t^2/(cn)}\, .
\end{align}
\end{remark}

Finally, for the proof of Proposition \ref{thm:SecMoment},
we will assume without loss of generality $n\le d-1$.

\subsubsection{Proof of Proposition \ref{thm:SecMoment}: Reduction to $\xi'(0)=0$}
\label{sec:SecMomentReduction}

In the case $\xi'(0)=0$,
$\bF(\bx)$ does not contain a linear component.
Let $\bF_{> 0}(\bx):= \bF(\bx)-\bF(\bzero)$, and note that 
$\bF_{> 0}(\, \cdot\, )$ is independent of $\bF(\bzero)\sim\normal(\bzero,\xi(0)\id_{n})$ and distributed as the original 
process, with $\xi(q)$ replaced by  $\xi_{> 0}(q) =\xi(q)-\xi(0)$.

Consider the modified set of solutions
\begin{align}
\Sol_{n,d}(\bu;\eps):= \Big\{\bx\in \S^{d-1}:\,\; \|\bF_{> 0}(\bx)-\bu\|_2^2\le n\xi(1)\cdot \eps\Big\}\, .
\label{eq:SolDef}
\end{align}
In the case $\xi''(0)=0$, we will prove 
in the next subsection  that, for $\alpha<\alpha_{\slb}$ and $\delta_d\ge 0$
any deterministic sequence such that $\delta_d\to 0$, we have
\begin{align}
    \lim_{n,d\to\infty}\sup_{\|\bu\|_2/\sqrt{n}\in[\xi_0^{1/2}-\delta_d,\xi_0^{1/2}+\delta_d]}
    \prob(\Sol_{n,d}(\bu;0)= \emptyset) =0\, .\label{eq:ProbSolEmpty}
\end{align}
In the case $\xi''(0)>0$, we will prove that $\sup_{\|\bu\|_2/\sqrt{n}\in[\xi_0^{1/2}-\delta_d,\xi_0^{1/2}+\delta_d]}
    \prob(\Sol_{n,d}(\bu;0)= \emptyset)$ is uniformly bounded away from $1$.

This of course implies the claim of the theorem
for $\xi'(0)=0$ since $\|\bF(\bzero)\|_2/\sqrt{n}$
concentrates around $\xi^{1/2}_0$.

 In the case $\xi'(0)>0$, we write
$\bF(\bx) = \bF_{\#}(\bx)+c_1\bG\bx$, $c_1^2=\xi'(0)$, where 
$\bF_{\#}$ is a Gaussian map
with covariance given by $\xi_{\#}(t)=\xi(t)-\xi'(0)t$, and $\bG\sim\GOE(n,d)$. 
We let $\bU\in \reals^{d\times(d-n)}$ by an orthogonal matrix whose columns span
the null space of $\bG$ (which has dimension $d-n$ almost surely). We then
look for a solution of the form $\bx = \bU\bz$, and is thus sufficient to solve
$\bF_{\#}(\bU\bz)=\bzero$. The claim then follows by applying the previous result 
to $\tilde\bF(\bz):= \bF_{\#}(\bU\bz)$.

\subsubsection{Proof of Proposition \ref{thm:SecMoment}: Case $\xi'(0)=0$}
\label{app:SecondMoment}

In this section, we  prove Eq.~\eqref{eq:ProbSolEmpty},
which, as explained above, directly  implies
Proposition \ref{thm:SecMoment} for $\xi'(0)=0$.

In order to upper bound the probability
of $\Sol_{n,d}(\bu;0)= \emptyset$, 
 we introduce the random variable
\begin{align}
\Vsol(\bu)& := \Vol_{d-n-1} \big(\Sol_{n,d}(\bu;0))\, ,\label{eq:VolDef}
\end{align}
where $\Vol_i$ is the Hausdorff measure of dimension $i$,
or the counting measure when $i=0$.

We will use Kac-Rice formula \cite{Kac1943,Rice1945} to compute the
first two moments of $\Vsol(\bu)$ and obtain, by a second moment argument,
 a sufficient condition for it to concentrate around the mean.
For $\bu=\bzero$ this calculation was carried out in \cite{SubagPolynomialSystems}.

For $\xi'(0)=\xi''(0)=0$, we will 
 prove that, for any deterministic sequence $\delta_d>0$ with 
  $\lim_{d\to\infty}\delta_d= 0$,
\begin{align}
 \lim_{n,d\to\infty}\sup_{\|\bu\|_2/\sqrt{n}\in[\xi_0^{1/2}-\delta_d,\xi_0^{1/2}+\delta_d]}
 \frac{\E\{\Vsol(\bu)^2\}}{\E\{\Vsol(\bu)\}^2} =1\, .
\end{align}
For $\xi'(0)=0$, $\xi''(0)>0$, the same argument implies that the ratio 
is bounded by a constant.
This will complete the proof sketch given in Section 
\ref{sec:SecMomentReduction}.

We begin by computing the first moment. Recall that $\Vol_i$ is the Hausdorff measure of dimension $i$,
or the counting measure when $i=0$.
\begin{proposition}\label{prop:1stmoment}
	For any $\bu\in \R^n$ and $n\le d-1$, 
	\begin{equation}\label{eq:meansol}
		\E\,\Vsol(\bu) = \cV_{d-n-1} \Big(\frac{\xi'(1)}{\xi(1)}\Big)^{\frac{n}{2}}e^{-\frac{\|\bu\|^2}{2\xi(1)}},
	\end{equation}
	where $\cV_{i}:=\Vol_{i}(\{\bx\in\R^{i+1}:\,\|\bx\|=1\})=2\pi^{\frac{i+1}{2}}/\Gamma(\frac{i+1}{2})$ is the volume of the sphere.
\end{proposition}
\begin{proof}
	The proof here is identical to the case $\bu=0$ treated in \cite{SubagPolynomialSystems}. We repeat it for completeness, using the notation of the current paper.
	
	Recall that, for $\bx\in \reals^d$ on the sphere $\bU_{\bx}\in\reals^{d\times(d-1)}$ 
	denotes a matrix whose columns form a basis of the tangent space $\Ts_{\bx}$ 
	to the sphere of radius $\|\bx\|_2$ at $\bx$ (equivalently, of the orthogonal complement
	of $\bx$ in $\reals^d$).
	We write $\bD_{\perp}\bF(\bx):=\bD\bF(\bx)\bU_{\bx}\in \R^{n\times d-1}$. 
	
	By a variant of the Kac-Rice formula in \cite[Theorem 6.8]{Azais2009},
	\begin{align}
		\E\Vsol(u) & =\int_{S^{d-1}}\E\left[J(\bD_{\perp}\bF(\bx))\,\Big|\,\bF(\bx)=\bu\right]\,
		p_{\bF(\bx)}(\bu)\, \de\Vol_{d-1}(\bx),%\label{eq:KR1stmoment1}
	\end{align}
	where $J(\bA)=\sqrt{\det \bA\bA^{\sT}}$, $\de\Vol_{d-1}(\bx)$ is the $(d-1)$-dimensional volume
	 element on the sphere  and 
	\[
	p_{\bF(\bx)}(\bu)=\big(2\pi\xi(1)\big)^{-\frac{n}{2}}e^{-\frac{\|\bu\|^2}{2\xi(1)}}
	\]
	is the 
	density of $\bF(\bx)$ at $\bu$. Using the Kac-Rice formula requires checking
 that certain conditions are satisfied, this was done in Remark 3 of  \cite{SubagPolynomialSystems}.
	
	Recall that by Lemma \ref{lemma:HessianDistr},	
	$\bD_{\perp}\bF(\bx) \ed \sqrt{\xi'(q)}\, \bZ$ where $q=\|\bx\|^2$, $\bZ\sim\GOE(n,d-1)$ and 
	$\bD_{\perp}\bF(\bx)$ is independent of $\bF(\bx)$. Thus,
	\begin{equation}\label{eq:1stmoment}
		\E\Vsol(\bu) = \cV_{d-1} \Big(\frac{\xi'(1)}{2\pi\xi(1)}\Big)^{n/2}e^{-\frac{\|\bu\|^2}{2\xi(1)}} \E J(\bZ).
	\end{equation}
	
	In the case $\xi(t)=t$ and $\bu=0$, the above is the volume of the set of points on the sphere orthogonal to $n$ independent Gaussian vectors, hence equal to $\cV_{d-n-1}$. Hence, 
	\begin{equation}\label{eq:EJ}
		\cV_{d-n-1}=\cV_{d-1} \Big(\frac{1}{2\pi}\Big)^{n/2} \E\, J(\bZ)
	\end{equation} 
and \eqref{eq:meansol} follows.
\end{proof}

The next theorem establishes the desired second moment bound.
To verify this note that the function $\Phi(r)$ defined here matches $\Psi(r;\alpha,\xi)$ of Eq.~\eqref{eq:PsiDef},
if we replace $\gamma^2$ by $\alpha \xi(0)$ and $\xi(t)$ by $\xi_{>0}(t) =\xi(t)-\xi(0)$, 
and that condition \eqref{eq:WeakConditionMoment}
is implied by $\alpha<\alpha_{\slb}(\xi)$.
\begin{theorem}
	Assume that $\xi(0)=\xi'(0)=0$. Let $\gamma\ge 0$ and $\alpha\in(0,1]$ and define
	\begin{align*}
		\Phi(r):= \frac{1}{2}\log(1-r^{2})- \frac{\alpha}{2}\log\Big(1-\frac{\xi(r)^{2}}{\xi(1)^{2}}  \Big)+\frac{\gamma^2}{\xi(1)}-\frac{\gamma^2}{\xi(1)+\xi(r)}\,.
	\end{align*}
Further assume that for any $\eps>0$,
	\begin{align}
		\Phi(0)>\sup_{r\in(\eps,1)}\Phi(r)\,.\label{eq:WeakConditionMoment}
	\end{align}
	
If $\xi''(0)=0$, then for any deterministic sequence $\delta_d>0$ such that $\lim_{d\to \infty}\delta_d= 0$,
\begin{equation}\label{eq:momentsmatching_a}
	\lim_{d,n\to\infty}\frac{\E\big(\Vsol(\bu)^2\big)}{\big(\E\Vsol(\bu)\big)^2}=1\, ,
\end{equation}
uniformly in $n=n(d)< d-1$ and $\bu=\bu^{(n,d)}\in\R^n$ such that $\lim_{d\to\infty}n(d)\to\alpha$ 
and $\sqrt d (\gamma-\delta_d)\leq \|\bu\|\leq \sqrt d(\gamma+\delta_d)$.
% and therefore for any $\delta>0$,
%\begin{equation}\label{eq:concentration}
%\lim_{d\to\infty}\prob\bigg(
%\Big| \frac{\Vsol(\bu)}{\E\Vsol(\bu)}-1 \Big| \leq\delta
%\bigg) = 1.
%\end{equation}

If $\xi''(0)>0$, then there exists $C(\gamma)$ bounded for $\gamma$ bounded, such that
 \begin{equation}\label{eq:momentsmatching}
	\lim_{d,n\to\infty}\frac{\E\big(\Vsol(\bu)^2\big)}{\big(\E\Vsol(\bu)\big)^2}\le C(\gamma)\, ,
\end{equation}
uniformly in $n=n(d)< d-1$ and $\bu=\bu^{(n,d)}\in\R^n$ such that $\lim_{d\to\infty}n(d)\to\alpha$ 
and $\sqrt d (\gamma-\delta_d)\leq \|\bu\|\leq \sqrt d(\gamma+\delta_d)$.
\end{theorem}
\begin{proof}
	To simplify the proof, throughout we will assume that $n\leq d-2$. The case 
	in which $n=d-1$ infinitely often (in particular $\alpha=1$) was treated for 
	$\bu=\bzero$ in \cite{SubagPolynomialSystems} and the argument can be adapted to 
	our situation with $\bu\neq 0$, but to save space we will omit the proof. We refer 
	the interested reader to Lemmas 8 and 9 in \cite{SubagPolynomialSystems}. 
	
	For any union of intervals $I\subset[-1,1]$, define 
	\[
	\Vsol^{(2)}(\bu,I):=\Vol_{2(d-n-1)}\Big(\left\{ (\bx,\by):\,\<\bx,\by\>\in I,\,\bF(\bx)=
	\bF(\by)=\bu,\bx,\by\in\S^{d-1}\right\}\Big) .
	\]

As in the proof of Proposition \ref{prop:1stmoment}, we may apply the Kac-Rice formula 
to the random field  $(\bx,\by)\to(\bF(\bx),\bF(\by))$ defined
on the domain
\begin{equation}
	\cD(I):=\left\{ (\bx,\by):\,\<\bx,\by\>\in I, \|\bx\|=\|\by\|=1 \right\}
\end{equation}
to compute the expectation of $\Vsol^{(2)}(\bu,I)$, see Section 3 of \cite{SubagPolynomialSystems} 
for the details. This yields that
\begin{align}
	\E\,\Vsol^{(2)}(\bu,I) &=\big(\xi'(1)\big)^n\int_{(\bx,\by)\in\cD(I)}p_{\bF(\bx),\bF(\by)}(\bu,\bu)
	\,T(\bu,\<\bx,\by\>)\,\de\Vol_{d-1}(\bx)\de\Vol_{d-1}(\by)\nonumber\\
	&=\cV_{d-1}\cdot \big(\xi'(1)\big)^n\int_{\{\by\in \S^{d-1}:\,\<\bx,\by\>\in I\}}
	p_{\bF(\bx),\bF(\by)}(\bu,\bu)T(\bu,\<\bx,\by\>)d\Vol_{d-1}(\by),\label{eq:SecondMomIntegral}
\end{align}
where in the second line $\bx$ is an arbitrary point on the sphere,
and writing $\<\bx,\by\>=r$, we have 
\[
p_{\bF(\bx),\bF(\by)}(\bu,\bu)
=(2\pi)^{-n}(\xi(1)^{2}-\xi(r)^{2})^{-\frac{n}{2}}e^{-\frac{\|\bu\|^2}{\xi(1)+\xi(r)}}
\]
is the density of $(\bF(\bx),\bF(\by))$ at $(\bu,\bu)$
since the covariance matrix of $(F_i(\bx),F_i(\by))$ is 
\[
\left(\begin{array}{cc}
	\xi(1) & \xi(r)\\
	\xi(r) & \xi(1)
\end{array}\right),
\]
and we define (always for $\<\bx,\by\>=r$)
\begin{equation}
	T(\bu,r):=\E\left[J\Big(\frac{\bD_{\perp}\bF(\bx)}{\sqrt{\xi'(1)}}\Big)J\Big(\frac{\bD_{\perp}\bF(\by)}{\sqrt{\xi'(1)}}\Big)\,\bigg|\,\bF(\bx)=\bF(\by)=\bu\right]\,.\label{eq:D(r)}
\end{equation}
By rotational invariance $T(\bu,r)$ does not depend on the choice of $\bU_{\bx}$ and $\bU_{\by}$ and of 
$\bx$ and $\by$ as long as $\<\bx,\by\>=r$.

Using the co-area formula with the function $\rho(\by)=\<\bx,\by\>$,
we may express the integral w.r.t. $\by$ in Eq.~\eqref{eq:SecondMomIntegral}
as a one-dimensional integral
over a parameter $r$ (the volume of the inverse-image $\rho^{-1}(r)$
and the inverse of the Jacobian are given by $\cV_{d-2}(1-r^{2})^{(d-2)/2}$
and $(1-r^{2})^{-1/2}$, respectively). This yields 
\begin{equation}\label{eq:VI}
\E\, \Vsol^{(2)}(\bu,I) =\cV_{d-1}\cV_{d-2}(2\pi)^{-n}\int_{I}\left(\frac{\xi'(1)^{2}}{\xi(1)^{2}-\xi(r)^{2}}\right)^{n/2}
(1-r^{2})^{\frac{d-3}{2}}T(\bu,r)e^{-\frac{\|\bu\|^2}{\xi(1)+\xi(r)}}\de r\, .
\end{equation}

From  Eqs.~\eqref{eq:1stmoment} and \eqref{eq:VI},we get
\begin{equation}\label{eq:momentsratio}
	\frac{\E\,\Vsol^{(2)}(\bu,I)}{\big(\E\, \Vsol(\bu)\big)^2}=
	\frac{\cV_{d-2}}{\cV_{d-1}} \int_{I} \frac{T(\bu,r)}{\big(\E J(\bZ)\big)^2} 
	\left(1-\frac{\xi(r)^{2}}{\xi(1)^{2}}\right)^{-n/2}(1-r^{2})^{(d-3)/2}
	e^{\frac{\|\bu\|^2}{\xi(1)}-\frac{\|\bu\|^2}{\xi(1)+\xi(r)}}\de r.
\end{equation}
The next lemma provide a useful bound on the integrand.
\begin{lemma}\label{lem:T_upperbd}
	Assume that  $\xi(0)=\xi'(0)=0$. Define 
	\begin{equation}
	\label{eq:kappa}
	\kappa=\kappa(r)= \frac{\xi'(r)}{\xi(1)+\xi(r)}\sqrt{\frac{1-r^2}{\xi'(1)}}\,.
	\end{equation}
	Then for some universal constant $c$ and sequence $\tau_d\to0$,  for any $r\in(-1,1)$, 
	\[
	\frac{T(\bu,r)}{\big(\E J(\bZ)\big)^2}\leq \big(1+\tau_d\big) \big(1+2r^2\sqrt d\big)\big(
	1+c\kappa(r)\|\bu\|+c\kappa(r)^2\|\bu\|^2
	\big)\,.
	\]
\end{lemma}

Before proving this lemma, let us complete the proof of the theorem.
Note that $\frac{\cV_{d-2}}{\cV_{d-1}}=\sqrt{\frac{d}{2\pi}}(1+o_d(1))$. Hence, we have that
\begin{equation*}%\label{eq:2ndmombd1}
	\begin{aligned}
		\frac{1}{d}\log\frac{\E\, \Vsol^{(2)}(\bu,I)}{\big(\E\, \Vsol(\bu)\big)^2}   \leq  \frac1d\log\int_{I}\nu(\bu,r)\, \de r+o_d(1)\,,
	\end{aligned}
\end{equation*}
where the $o_d(1)$ term is independent of $\bu$ and $n$, and
\[
\nu(\bu,r):=\left(1-\frac{\xi(r)^{2}}{\xi(1)^{2}}\right)^{-n/2}(1-r^{2})^{(d-3)/2}\big(
1+c\kappa(r)\|\bu\|+c\kappa(r)^2\|\bu\|^2
\big)e^{\frac{\|\bu\|^2}{\xi(1)}-\frac{\|\bu\|^2}{\xi(1)+\xi(r)}}\,.
\]

Note that $\nu(\bu,r)\leq (1-r^2)^{-1/2}e^{d\Phi(r)+o(d)}$, uniformly in $n\leq d-2$, 
$\sqrt d(\gamma-\delta_d)\leq \|\bu\|\leq \sqrt d(\gamma+\delta_d)$ and $r\in(-1,1)$.  Assuming that $n=n(d)\to\alpha$, 
since $(1-r^2)^{-1/2}$ is integrable on $(-1,1)$, we have that for any $\eps>0$, uniformly in 
$\sqrt d(\gamma-\delta_d)\leq \|\bu\|\leq \sqrt d(\gamma+\delta_d)$,
\[
\limsup_{d\to\infty}\frac{1}{d}\log\frac{\mathbb{E}\Vsol^{(2)}(\bu,[-1,1]\setminus(-\eps,\eps))}{\big(\E\Vsol(\bu)\big)^2}\leq \sup_{r\in[\eps,1]} \Phi(r)<\Phi(0)=0\,,
\]
where the second inequality follows by assumption. 

Of course, the same bound holds with some sequence $\eps_d\to0$ instead of  $\eps$. 
Since the ratio of moments in  \eqref{eq:momentsmatching_a} is lower bounded by $1$, to prove the same equation it remains to show that, uniformly,
\begin{equation}\label{eq:momentsmatchingepsilon}
	\limsup_{d\to\infty} \frac{\mathbb{E}\Vsol^{(2)}(\bu,(-\eps_d,\eps_d))}{\big(\E\Vsol(\bu)\big)^2}\leq 1\,.
\end{equation}

First consider the case $\xi(0)=\xi'(0)=\xi''(0)=0$.
Note that $\frac{\xi(r)}{\xi(1)}=O(r^3)$ and $\kappa(r)=O(r^2)$ for small $r$,  and recall that $\frac{\cV_{d-2}}{\cV_{d-1}}=\sqrt{\frac{d}{2\pi}}(1+o_d(1))$. 
From \eqref{eq:momentsratio} we obtain that the ratio in \eqref{eq:momentsmatchingepsilon} is equal to 
\begin{align*}
	&(1+o_d(1))\sqrt{\frac{d}{2\pi}}
	  \int_{-\eps_d}^{\eps_d} \frac{T(\bu,r)}{\big(\E J(\bZ)\big)^2} \left(1-O(r^6)\right)^{-n/2}(1-r^{2})^{d/2}e^{\frac{\|\bu\|^2}{\xi(1)}O(r^3)}\de r\\
	  &=(1+o_d(1))\sqrt{\frac{d}{2\pi}}
	  \int_{-\eps_d}^{\eps_d}\exp\Big(
	  -\frac{d}{2}r^2+2\sqrt d r^2
	  +\frac{\gamma^2 d}{\xi(1)}O(r^3)+dO(r^6)
	  + c\gamma\sqrt d O(r^2)+c\gamma^2 dO(r^4)
	  \Big) \de r
\end{align*}
where we used Lemma \ref{lem:T_upperbd} and that $\log(1+t)\leq t$. This proves \eqref{eq:momentsmatchingepsilon} and completes the proof.
\end{proof}

\begin{proof}[Proof of Lemma \ref{lem:T_upperbd}]
Fix some $r\in(-1,1)$. Let $\bM^1$ and $\bM^2$ be jointly Gaussian, centered random matrices such that 
\begin{equation}\label{eq:Mcov}
	\mbox{cov}(\bM^t_{ij}, \bM^s_{kl}) = \delta_{ik}\delta_{jl}\cdot \begin{cases}
		1&t=s,\,j<d-1\\
		1-\frac{\xi(1)}{\xi'(1)}\frac{\xi'(r)^2(1-r^2)}{\xi(1)^2-\xi(r)^2}&t=s,\,j=d-1\\
		\frac{\xi'(r)}{\xi'(1)}&t\neq s,\,j<d-1\\
		r\frac{\xi'(r)}{\xi'(1)}-\frac{\xi''(r)}{\xi'(1)}(1-r^2)-\frac{\xi(r)}{\xi'(1)}\frac{\xi'(r)^2(1-r^2)}{\xi(1)^2-\xi(r)^2}&t\neq s,\,j=d-1\\
		0 &\mbox{otherwise.}
	\end{cases}
\end{equation}
Suppose $\bx$ and $\by$ are such that $\<\bx,\by\>=r\in(-1,1)$. In Lemma 6 of \cite{SubagPolynomialSystems} it was shown that 
for an appropriate choice of  $\bU_{\bx}$ and $\bU_{\by}$ (bases of the tangent spaces $\Ts_{\bx}$,
$\Ts_{\by}$), 
\[
\bigg(
\frac{\bD_{\perp}\bF(\bx)}{\sqrt{\xi'(1)}},\,\frac{\bD_{\perp}\bF(\by)}{\sqrt{\xi'(1)}}\bigg)\ed\Big(\bM^1-\kappa \bu\cdot \be_{d-1}^{\sT},\,\bM^2+\kappa \bu\cdot \be_{d-1}^{\sT}\Big)
\,,
\]
where $\be_{d-1}\in \R^{d-1}$ is the last (column) vector in the standard basis and $\kappa$ is defined in \eqref{eq:kappa}. (To be precise, only the case $\bu=0$ was treated there, but only the conditional expectation depends on $\bu$ and is easily computed in the same way.) 

Let $\bO$ be some orthogonal matrix such that $\bO\bu=(0,\ldots,0,\|\bu\|)^{\sT}$. Since 
$J(\bA)=J(\bO \bA)$ and $(\bO\bM^1,\bO\bM^2)\ed (\bM^1,\bM^2)$,
\[
\bigg(
J\bigg(\frac{\bD_{\perp}\bF(\bx)}{\sqrt{\xi'(1)}}\bigg),\,J\bigg(\frac{\bD_{\perp}\bF(\by)}{\sqrt{\xi'(1)}}\bigg)\bigg)
\ed \Big(J\big(\hat\bM^1\big),\,J\big(\hat\bM^2\big)\Big)
\,,
\]
where we denote $\hat\bM^1:=\bM^1-\kappa \|\bu\| \be_n\be_{d-1}^{\sT}$ and $\hat\bM^2:=\bM^2+\kappa \|\bu\|\be_n \be_{d-1}^{\sT}$.   Recall that for any matrix
 $\bA$, $J(\bA)=\prod_{i=1}^{n}\Theta_i(\bA)$, where we define $\Theta_i(\bA)$ as the norm of the projection of the $i$-th 
 row of $\bA$ to the orthogonal space to its first $i-1$ rows. Note that  $\Theta_i(\hat\bM^j)= \Theta_i(\bM^j)$ for $i<n$ and $\Theta_n(\hat\bM^j)\leq \Theta_n(\bM^j)+\kappa \|\bu\|$. 
 Therefore, using independence of the rows anf the orthogonal invariance of Gaussians, 
\begin{align*}
	T(\bu,r)&\leq\E\Big[\prod_{i=1}^{n-1}\Theta_i(\bM^1)\big(
	\Theta_n(\bM^1)+\kappa \|\bu\|
	\big)
	\prod_{i=1}^{n-1}\Theta_i(\bM^2)\big(
	\Theta_n(\bM^2)+\kappa \|\bu\|
	\big)
	\Big]\\
	&= \E \big[J(\bM^1)J(\bM^2)\big]\Big(
	1+\frac{2\kappa\|\bu\|\E \Theta_n(\bM^1)}{\E\big[ \Theta_n(\bM^1)\Theta_n(\bM^2)\big]}+\frac{\kappa^2\|\bu\|^2}{\E\big[ \Theta_n(\bM^1)\Theta_n(\bM^2)\big]}
	\Big) \,.
\end{align*}

Using that $\Theta_n(\bM^i)\sim\chi_{n-d}$, we have that, for some universal constant  $c$,
\[
T(\bu,r)\leq \E \big[J(\bM^1)J(\bM^2)\big]
\big(
1+c\kappa(r)\|\bu\|+c\kappa(r)^2\|\bu\|^2
\big)\,.
\]

Lemma 10 of \cite{SubagPolynomialSystems} bounds $T(\bu,r)$ for $\bu=0$, which is exactly equal to $\E \big[J(\bM^1)J(\bM^2)\big]$. Precisely, it states that
\[
T(\bzero,r)= \E \big[J(\bM^1)J(\bM^2)\big]
\leq
\big(1+\tau_d\big) \big(1+2r^2\sqrt d\big)\,,
\]
for some universal $\tau_d\to0$. This completes the proof. 
\end{proof}

%
%**************************************************************************
%
\subsection{Upper bound:
Proof of Proposition \ref{propo:SimpleUB}}
\label{sec:Existence_UB}

In this section prove the two upper bounds $\alpha_{\sub}^{(1)}$ and $\alpha_{\sub}^{(2)}$
on the satisfiability threshold, plus an additional one which
 we refer to as $\alpha^{(3)}(\xi_0,p)$, 
 which only holds for `pure' model $\xi(q) =\xi_0+q^p$.

\subsubsection{Upper bound $\alpha_{\sub}^{(1)}$}

Define
\begin{equation}\label{eq:GS}
	\GS(\xi):=  \lim_{d\to\infty}\frac{1}{\sqrt d}\E  \max_{\bx\in\S^{d-1}} F_1(\bx)\,.
\end{equation}
The limit is known to exist and is given by the Parisi formula \cite{auffinger2017parisi,JT},
which we next recall.

Given $\gamma:[0,1)\to\reals_{\ge 0}$ and $L\ge \int_0^1 \gamma(s)\de s$, we define
\begin{align}
\Par(\gamma,L) & = \frac{1}{2}\int_0^1\left(\xi''(t)\Gamma(t) +\frac{1}{\Gamma(t)}\right)\de t\,,\label{eq:ParisiSpherical}\\
\Gamma(t) & := L-\int_0^t \gamma(s)\de s\,  .
\end{align}
Equivalently, we can view $\Par$ as a function of $\Gamma:[0,1]\to\reals_{\ge 0}$ which is 
continuous and non-increasing.
We then have 
\begin{align}
  \GS(\xi) &:= \inf \Big\{\Par(\gamma,L) :\;  \gamma \;\mbox{non-decreasing}\, \;\; L\ge \int_0^1\gamma(t)\, \de t\;\Big\}\, .
  \label{eq:SphericalVariational}
\end{align}
We next restate the claimed bound, in equvalent form, for the reader's 
convenience.
\begin{proposition}\label{propo:SimpleUB_app}
	Assume that $\xi(0)>0$ and define $\xi_{>0}(t) := \xi(t)-\xi(0)$.
	Then, for  $u_0^{(1)}(\alpha,\xi)$ defined as in 
	Eq.~\eqref{eq:U0_LB_def}, we have, provided $n/d\to\alpha$,
\begin{align}
	\lim_{n,d\to\infty}\P\left(\frac1n
	\min_{\bx \in\S^{d-1}}H(\bx) \ge u^{(1)}_0(\xi,\alpha)-\delta\right) = 1\;\;\;
	\forall\delta>0\, .
	\end{align}
	Finally, there exists an absolute constant $C$ such that 
	\begin{align}\label{eq:SimpleUB-Explicit}
	\alpha_{\sub}^{(1)}(\xi)\le C\frac{\xi(1)}{\xi(0)}\max \Big(\log\frac{\xi'(1)}{\xi(1)},1\Big)\, .
	\end{align} 
\end{proposition}
\begin{proof}
As in the previous section, we write $\bF(\bx)=\bF(\bzero)+\bF_{>0}(\bx)$
and note that the two summands are independent.
By Remark \ref{rmk:Concentration} (applied to $\bF_{>0}$)  and  Eq.~\eqref{eq:GS}
\begin{align}
\plim_{n,d\to\infty} \frac{1}{\sqrt{d}}\max_{\bx\in\S^{d-1}}
\left|\frac{\<\bF(\bzero),\bF_{>0}(\bx)\>}{\|\bF(\bzero)\|_2}\right|
= \GS(\xi)\, .
\end{align}
Therefore, 
\begin{align}
\frac{1}{\sqrt{n}}\min_{\bx\in\S^{d-1}}\| \bF(\bx)\|_2&\ge \frac{1}{\sqrt{n}} \|\bF(\bzero)\|_2
-\frac{1}{\sqrt{n}} \max_{\bx\in\S^{d-1}}\left|\frac{\<\bF(\bzero),\bF_{>0}(\bx)\>}{\|\bF(\bzero)\|_2}\right|\\
& \ge \sqrt{\xi(0)} - \frac{1}{\sqrt{\alpha}} \GS(\xi_{>0})-o_P(1)\, ,
\end{align}
where $o_P(1)$ denotes a term which converges in probability to $0$ as $n,d\to\infty$, 
which implies the claim. 

Finally, the explicit upper bound \eqref{eq:SimpleUB-Explicit} follows by using Proposition
\ref{propo:MaxValue}.
\end{proof}

\subsubsection{Upper bound $\alpha_{\sub}^{(2)}$}

Define the Gaussian process $(\bx,\blambda)$ indexed by 
$(\bx,\blambda)\in\S^{d-1}\times\S^{n-1}$ via
\begin{align}
G(\bx,\blambda) = \<\blambda,\bF(\bx)\>+g_0\sqrt{\xi'(1)}\, ,
\end{align}
where $g_0$ is standard normal and independent of $\bF(\, \cdot\,)$.
Clearly $G$ is centered and has covariance 
\begin{align}
	\E\big\{G(\bx_1,\blambda_1)G(\bx_2,\blambda_2)\big\}
	 = \<\blambda_1,\blambda_2\>\, \xi(\<\bx_1,\bx_2\>)+\xi'(1)\, .
\end{align}
We compare this to the Guassian process
\begin{align}
J(\bx,\blambda) :=\sqrt{\xi(1)}\<\blambda,\bg\>+
\sqrt{\xi'(1)}\<\bx,\bh\>\, ,
\end{align}
where $\bg\sim\normal(\bzero,\bI_n)$ and $\bh\sim\normal(\bzero,\bI_d)$
are independent Gaussian vectors.
This has covariance
\begin{align}
	\E\big\{J(\bx_1,\blambda_1)J(\bx_2,\blambda_2)\big\}
	 = \<\blambda_1,\blambda_2\>\, \xi(1)+\<\bx_1,\bx_2\>\xi'(1)\, .
\end{align}
Clearly we have, for all $\bx\in\S^{d-1}$, $\blambda_1,\blambda_2\in\S^{n-1}$,
\begin{align}
	\E\big\{G(\bx,\blambda_1)G(\bx,\blambda_2)\big\}=
	\E\big\{J(\bx,\blambda_1)J(\bx,\blambda_2)\big\}\, ,
\end{align}
while, for arbitrary $\bx_1,\bx_2\in\S^{d-1}$, 
$\blambda_1,\blambda_2\in\S^{n-1}$, letting $q:=\<\bx_1,\bx_2\>$,
$t:=\<\blambda_1,\blambda_2\>$, $q,t\in [-1,1]$,
\begin{align*}
	\E\big\{J(\bx_1,\blambda_1)J(\bx_2,\blambda_2)\big\}-
	\E\big\{G(\bx_1,\blambda_1)G(\bx_2,\blambda_2)\big\}
	&=(\xi(1)-\xi(q))t+\xi'(1)(q-1)\\
	&\le \xi'(1)(1-q)t+\xi'(1)(q-1)\\
	&=-\xi'(1)(1-t)(1-q)\le 0\, .
\end{align*}
Hence, by Gordon's comparison inequality \cite{gordon1985some},
\begin{align}
	J_n:=\min_{\bx\in\S^{n-1}}\max_{\blambda\in\S^{d-1}}
	J(\bx,\blambda)\preceq \min_{\bx\in\S^{n-1}}\max_{\blambda\in\S^{d-1}}
	G(\bx,\blambda)\, ,
\end{align}
where $\preceq$ denotes stochastic domination between random variables.
On the other hand, it is immediate to compute
\begin{align}
	\frac{1}{\sqrt{n}}J_n&=\sqrt{\xi(1)}\frac{\|\bg\|}{\sqrt{n}}-
	\sqrt{\xi'(1)}\frac{\|\bh\|}{\sqrt{n}}\nonumber\\
&=\sqrt{\xi(1)}- \sqrt{\frac{\xi'(1)}{\alpha}} +o_P(1)\, .
\label{eq:Jn-Bound}
\end{align}

Let 
\begin{align}
\frac{1}{\sqrt{n}}\min_{\bx\in \S^{d-1}} \big\|\bF(\bx)\big\|&=
\frac{1}{\sqrt{n}}\min_{\bx\in \S^{d-1}} \max_{\blambda\in \S^{n-1}}
\<\blambda,\bF(\bx)\>\\
&=\min_{\bx\in\S^{n-1}}\max_{\blambda\in\S^{d-1}}
	G(\bx,\blambda)+o_P(1)\\
&\succeq J_n +o_P(1)\, .
\end{align}
Hence using Eq.~\eqref{eq:Jn-Bound}, we conclude that,
for any $\delta>0$,
\begin{align}
\lim_{n,d\to\infty}\P\Big(\frac{1}{\sqrt{n}}\min_{\bx\in \S^{d-1}} \big\|\bF(\bx)\big\|
\ge\sqrt{2u_0^{(2)}(\alpha,\xi)}-\delta\Big)=1\, ,
\end{align}
which yields the desired claim.
%
%**********************************************************************************
%
\subsubsection{An upper bound for the case $\xi(t) = \xi_0+t^p$}
\label{app:UB}

We next obtain a bound for pure models $\xi(t) = \xi_0+t^p$. 
For $E\geq2\sqrt{(p-1)/p}$, define 
\begin{align}\label{eq:Theta_p_def}
\Theta_p(E)&:=\frac{1}{2}\log(p-1)-\frac{p-2}{p-1}\frac{E^2}{4}\nonumber\\
&\quad-\sqrt{\frac{p}{p-1}}\frac{E}{4}\sqrt{\frac{p}{p-1}E^2-4}\nonumber\\
&\quad+\log\Big(
\sqrt{\frac{p}{p-1}\frac{E^2}{4}-1}+\sqrt{\frac{p}{p-1}}\frac{E}{2}
\Big)\,.
\end{align}
The meaning of this function is established in \cite{crisanti1995TAP}
(at a heuristic level) and in \cite{AB,auffinger2013random,subag2017complexity,SubagZeitouniConcentration} (rigorously),
and it is worth reminding it here. 
Consider the process $F_{>0,1}(\,\cdot\,)$ (the first coordinate of
$\bF_{>0}(\bx)=\bF(\bx)-\bF(\bzero)$), which is a Gaussian process with covariance 
$\E[F_{>0,1}(\bx^1)F_{>0,1}(\bx^2)]=\<\bx^1,\bx^2\>^p$. Then, the number of local maxima
$\bx$ of this process with $F_{>0,1}(\bx) \approx E\sqrt{d}$  concentrates around 
$\exp(d \Theta_p(E))$. In particular, the function $E\mapsto \Theta_p(E)$
is monotone decreasing for $E>2\sqrt{(p-1)/p}$ and vanishes at $E =\GS(\xi(t)=t^p)$.
With an abuse of notation, we will write $\GS(p)=\GS(\xi(t)=t^p)$. 

For $c\geq0$, define
\begin{align*}
	\varphi_1(c,p) &= \sup_{t,s\geq0}\Big\{ c(\GS(p)+t+s) -\frac12\frac{s^2}{\xi(0)} + \Theta_p(\GS(p)+t)  \Big\},\\
	\varphi_2(c,\alpha) &=  \sup_{0<t\leq 1}\Big\{
	-c\sqrt{\alpha\xi(1)}t-\alpha\frac{t^2-1}{2}+\alpha\log t
	\Big\}\\
	&=-c\sqrt{\alpha\xi(1)}t_*-\alpha\frac{t_*^2-1}{2}+\alpha\log t_*,
\end{align*}
where $t_*=\frac12(\sqrt{c^2\xi(1)/\alpha+4}-c\sqrt{\xi(1)}/\sqrt{\alpha})$.
Note that $\varphi_2(c,\alpha)$ is decreasing in $\alpha>0$.
\begin{proposition}\label{thm:UB_Thm}
	Define 
	\begin{align*}
		\alpha_{\sub}^{(3)}(\xi_0,p):= \inf\Big\{\alpha\ge 0: \;\; \inf_{c>0} \big[\varphi_1(c,p)+\varphi_2(c,\alpha)-\frac12c^2(1+\xi_0) \big]<0 \Big\}\, .
	\end{align*}
	For the pure model $\xi(q) = \xi_0+q^p$ with $p\geq2$, 
	if $\alpha> \alpha_{\sub}^{(2)}(\xi_0,p)$, then there exists $\eps>0$ such that
		\begin{align*}
	 \lim_{n,d\to\infty}\prob\big(\Sol_{n,d}(\eps)= \emptyset \big)
		=1\, .
	\end{align*}
\end{proposition}

Note that
	\begin{equation}
		\label{eq:minmax}
\Sol_{n,d}(\eps)\neq \emptyset \iff \min_{\bx\in\S^{d-1}}\max_{\bv\in\S^{n-1}}\<\bF(\bx), \bv\>\le \sqrt{n\xi(1)\eps}\, .
\end{equation}
We define the two Gaussian processes on $\S^{d-1}\times \S^{n-1}$,
	\[
	f_1(\bx,\bv):= \<\bF(\bx), \bv\>+\sqrt{\xi(1)}\,z,\quad f_2(\bx,\bv):= F_1(\bx)+\sqrt{\xi(1)}\,\<\bv, \bg\>\, ,
	\]
where $\bg\sim\normal(0,\id_n)$ and $z\sim\normal(0,1)$, independent of each other and $\bF(\bx)$. Clearly,
	\begin{align*}
	\E f_1(\bx^1,\bv^1)f_1(\bx^2,\bv^2) &= \xi(\<\bx^1,\bx^2\>)\,\<\bv^1,\bv^2\>+\xi(1)\, ,\\
	\E f_2(\bx^1,\bv^1)f_2(\bx^2,\bv^2) &= \xi(\<\bx^1,\bx^2\>)+\xi(1)\<\bv^1,\bv^2\>\, .
	\end{align*}

For $\bx^1=\bx^2$ the two covariance functions coincide and for general $\bx^1,\bx^2\in \S^{d-1}$,
\[
	\E f_1(\bx^1,\bv^1)f_1(\bx^2,\bv^2)\geq \E f_2(\bx^1,\bv^1)f_2(\bx^2,\bv^2)\,.
\]
Hence, by Gordon's Gaussian comparison inequality \cite{gordon1985some}, we have
\begin{align*}
\min_{\bx\in\S^{d-1}}\max_{\bv\in\S^{n-1}}f_1(\bx^2,\bv^2) \succeq
\min_{\bx\in\S^{d-1}}\max_{\bv\in\S^{n-1}}f_2(\bx^2,\bv^2) \, ,
\end{align*}
where $\succeq$ denotes stochastic domination.
In particular, for any $c>0$,
\begin{align*}
e^{\frac12 c^2 d \xi(1)}\cdot\E \exp\Big\{ -c\sqrt{d} \min_{\bx\in\S^{d-1}}\max_{\bv\in\S^{n-1}}
\<\bF(\bx), \bv\> \Big\} &=\E \exp\Big\{ -c\sqrt{d} \min_{\bx\in\S^{d-1}}\max_{\bv\in\S^{n-1}}f_1(\bx,\bv) \Big\} \\
&\leq \E \exp\Big\{ -c \sqrt{d}\Big(\min_{\bx\in\S^{d-1}} F_1(\bx)+\sqrt{\xi(1)}\|\bg\|\Big) \Big\}\,.
\end{align*}

Using Eq.~\eqref{eq:minmax} and  Markov's inequality, 
\begin{equation}\label{eq:solUB}
	\begin{aligned}
	\prob \Big( \Sol_{n,d}(\eps) \neq \emptyset \Big) & \leq e^{c\sqrt{dn\xi(1)\eps}}\E \exp\Big\{ -c\sqrt{d} \min_{\bx\in\S^{d-1}}\max_{\bv\in\S^{n-1}}\<\bF(\bx), \bv\> \Big\}\\
	&\leq e^{d\eps'-\frac12 c^2d \xi(1)}\, \E e^{-c\sqrt{d\xi(1)}\|\bg\|}\, \E \exp\Big\{ c\sqrt{d} \max_{\bx\in\S^{d-1}} F_1(\bx)\Big\}
	\,,
	\end{aligned}
\end{equation}
where in the second line we defined $\eps':=c\sqrt{\alpha\xi(1)\eps}$.

By Cramer's theorem, for $t\in(0,1]$,
\[
\prob(\|\bg\|<t\sqrt n)=\prob(\|\bg\|^2<t^2 n) = \exp\Big(-n\frac{t^2-1}{2}+n\log t+o(n)\Big)
\]
and therefore
\begin{align}\label{eq:bdphi2}
\lim_{n,d\to\infty}\frac{1}{d}\log\E e^{-c\sqrt{d\xi(1)}\|\bg\|} &= \varphi_2(c,\alpha)\,.
\end{align}

Let $F_1^p(\bx)$ be the random process as defined in \eqref{eq:Fcov} with $\xi(t)=t^p$ and let $z\sim \normal(0,1)$ be 
independent of $F_1^p(\bx)$. Then, $F_1(\bx)\ed F_1^p(\bx)+\xi_0z$ as processes.
By Markov's inequality and \cite[Theorem 2.8]{auffinger2013random}, for $t\geq0$,
\[
\prob\Big(
\max_{\bx\in\S^{d-1}} F^p_1(\bx)\geq \sqrt{d}(E+t)
\Big) \leq e^{d\Theta_p(E+t)+o(d)}\,.
\]
For $s\geq0$, of course, $\prob(\xi_0z\geq \sqrt{d}t)\leq e^{-\frac12\frac{dt^2}{\xi_0^2}+o(d)}$. Combining these facts, one easily sees that
\begin{align}\label{eq:bdphi1}
	\lim_{n,d\to\infty}\frac{1}{d}\log\E \exp\Big\{ c\sqrt{d} \max_{\bx\in\S^{d-1}} F_1(\bx)\Big\} &= \varphi_1(c,p)\,.
\end{align}
The lemma follows from \eqref{eq:solUB}, \eqref{eq:bdphi2} and \eqref{eq:bdphi1}.

%
%**********************************************************************************
%
\subsection{Large-$p$ asymptotics of the satisfiability threshold: 
Proof of Eq.~\eqref{eq:LB-asymp}}
\label{app:LB-asymp}

\subsubsection{Asymptotics of the lower bound $\alpha_{\slb}(\xi_{p,\gamma_0})$}

For $\xi(t)=\xi_0+t^p$, the function $\Psi(r)=\Psi(r;\alpha,\xi_0,p)$ of Eq.~\eqref{eq:PsiDef} reads
\begin{align}
\Psi(r;\alpha,\xi_0,p):= \frac{1}{2}\log(1-r^2) -\frac{\alpha}{2}\log
\left(1-r^{2p}\right)+\alpha \xi_0\frac{r^p}{1+r^p}\, ,
\end{align}
and since $\Psi''(0;\alpha,\xi_0,p)=-1$ for all $\alpha$, we have
\begin{align}
\alpha_{\slb}(\xi_0,p):= \inf\Big\{\alpha\ge 0: \;\; \sup_{r\in [0,1]} \Psi(r;\alpha,\xi_0,p)>0 
 \Big\}\, .
\end{align}

To get an upper bound on $\alpha_{\slb}(\xi_0,p)$, we choose $r=1-M/p$ for some constant $M>0$.
Recalling that $\xi_0=\gamma_0\log p$ for some constant $\gamma_0>1$, it is easy to compute that, as $p\to\infty$,
\begin{align}
\Psi\Big(r= 1-\frac{M}{p}\Big) &= -\frac{1}{2}\log p+\alpha\xi_0 \frac{e^{-M}}{1+e^{-M}}
+O(1)\\
& = \Big(-\frac{1}{2}+\frac{\alpha\gamma_0}{1+e^{M}}\Big)\log p +O(1)\, .
\end{align}
Therefore, for all $p$ large enough
\begin{align}
\alpha_{\slb}(\xi_0=\gamma_0\log p,p)\le \frac{e^{M}+1}{2\gamma_0}\, .
\end{align}
Since $M> 0$ is arbitrary, we get $\limsup_{p\to\infty}\alpha_{\slb}(\xi_0=\gamma_0\log p,p)\le 1/\gamma_0$.

Next we prove the lower bound. For that purpose, we will fix $\gamma_0>1$, assume that
$\xi_0=\gamma_0\log p$, 
$\alpha\le (1-\eps)\gamma_0^{-1}$ and prove that $\sup_{r\in [0,1]} \Psi(r;\alpha,\xi_0,p)\le 0$
for all $p$ large enough.
Note that, under these definitions,  $\Psi(r;\alpha,\xi_0,p)\le \Psi_*(r)$, where
\begin{align}
\Psi_*(r) := \frac{1}{2}\log(1-r^2) -\frac{1}{2}\log
\left(1-r^{2p}\right)+(1-\eps)\frac{r^p}{1+r^p}\,\log p\, .
\end{align}
We split the maximization over $r$ in two regions (depending on a constant $\Delta$ to be fixed below):
\begin{enumerate}
\item $0\le r\le  1- p^{-1+\Delta}$.
 For an absolute constant $C$, we have

\begin{align}
\Psi_*(r)& \le -\frac{1}{2}\, r^2 + C(\log p) r^p\\
&\le  \Big(-\frac{1}{2}+ C'(\log p)e^{-p^{\Delta}}\Big)r^2\, .
\end{align}
\item $1- p^{-1+\Delta}< r <1$. Setting
$r = 1-x/p$, $x\in (0,p^{\Delta})$,  we have
\begin{align*}
\Psi_*(r) &\le \frac{1}{2}\log\Big(\frac{2x}{p}\Big)-\frac{1}{2}
\log(1-e^{-2x})
+\frac{1-\eps}{1+e^{x}} \log p\\
&\le -\frac{\eps}{2}\log p+\frac{1}{2}\log (2x)-\frac{1}{2}
\log(1-e^{-2x})\\
&\le -\frac{\eps}{2}\log p+\frac{1}{2}\log (2p^{\Delta})+1\\
&\le \left(\Delta-\frac{\eps}{2}\right)\log p+2\, .
\end{align*}
\end{enumerate}
Therefore the claim follows by choosing $\Delta=\eps/4$.

\subsubsection{Asymptotics of the upper bound $\alpha_{\sub}(\xi_{p,\gamma_0})$}

Recalling that $\xi_{p,0}(t)=t^p$, we obtain
\begin{align}
	\alpha_{\sub}(\xi_{p,\gamma_0}) \le 
	\alpha_{\sub}^{(1)}(\xi_{p,\gamma_0}) =
	\frac{1}{\gamma_0}\cdot \frac{\GS(\xi_{p,0})^2}{\log p}\, .
\end{align}
Hence the desired claim follows by proving 
$\GS(\xi_{p,0}) \le (1+o_p(1))\sqrt{\log p}$
(indeed, the matching lower bound on $\alpha_{\sub}$ follows from the lower
bound on $\alpha_{\slb}(\xi_{p,\gamma_0})$ in the previous subsection).

To prove the claim about $\GS(\xi_{p,0})$,
recall that  \cite{AB,auffinger2013random}  prove that 
\begin{align}
\GS(\xi_{p,0}) \le\sup\Big\{E>2\sqrt{\frac{p-1}{p}}\, : 
\;\; \Theta_p(E)>0\Big\}\, ,
\end{align}
where $\Theta_p$ is defined in \eqref{eq:Theta_p_def}.
Since $\Theta_p$ is monotone decreasing for $E>2\sqrt{(p-1)/p}$, it suffices to show that
$\Theta_p(E)<0$ for $E=(1+\eps)\sqrt{\log p}$ and all $p$ large enough, for any fixed
$\eps>0$.
Set $L:=\log p$ and $E:=\sqrt{uL}$ with $u>1$ fixed.
Writing $\rho:=\sqrt{p/(p-1)}$, a straightforward expansion gives, as $p\to\infty$,
\begin{align}
\frac{1}{2}\log(p-1) &= \frac{L}{2}+O(p^{-1})\, ,\\
-\frac{p-2}{p-1}\frac{E^2}{4} &= -\frac{uL}{4}+O(p^{-1}L)\, ,\\
-\rho\frac{E}{4}\sqrt{\rho^2E^2-4}
&= -\frac{uL}{4}+\frac12+O(L^{-1})\, ,\\
\log\Big(
\sqrt{\frac{p}{p-1}\frac{E^2}{4}-1}+\sqrt{\frac{p}{p-1}}\frac{E}{2}
\Big)
&= \frac12\log(uL)+O(L^{-1/2})\, .
\end{align}
Combining these identities,
\begin{align}
\Theta_p(\sqrt{uL}) = \frac{1-u}{2}\,L+\frac12+\frac12\log L+O(1)\, .
\end{align}
In particular, for any fixed $u>1$ we have $\Theta_p(\sqrt{uL})<0$ for all $p$ large enough,
whence $\GS(\xi_{p,0})\le \sqrt{uL}=\sqrt{u\log p}$. Since $u>1$ is arbitrary, 
$\GS(\xi_{p,0})\le (1+o_p(1))\sqrt{\log p}$ as claimed.
%
%***********************************************
%
\section{Gradient descent}
\label{sec:Algo_GD}

In this Appendix we study the classical projected gradient descent (GD) algorithm on 
the sphere  $\S^{d-1}$. Our analysis uses techniques recently developed 
to analyze overparametrized neural networks in the neural tangent regime.
A few pointers to this literature include
\cite{du2018gradient,chizat2019lazy,oymak2020towards,arora2019fine,allen2019learning,bartlett_montanari_rakhlin_2021}.
We deliberately follow this type of analysis as it provides a clear example of the gap between what is achievable
by such techniques and sharper characterizations developed in the
next sections.

In Section \ref{sec:Insensitivity} we 
prove that, essentially under the same conditions for
neural tangent  theory to hold, gradient descent is insensitive to 
perturbations.

We we run projected gradient descent
 with respect to the energy function $H(\bx) = \|\bF(\bx)\|_2^2/2$.
At iteration $k$
we take a step along the direction $-\Proj_{\Ts,\bx^k}\nabla H(\bx^k)$, with stepsize $\eta$,
and then project back on the sphere $\S^{d-1}$.
The algorithm,  is defined by the pseudocode of Algorithm \ref{algo:GD}.

This algorithm can be thought of as a discretization of the gradient flow 
dynamics
\begin{align}
\dot{\bx}(t) =-\Proj_{\Ts,\bx(t)}\nabla H(\bx(t))\, ,\label{eq:GFlow_0}
\end{align}
where $\Proj_{\Ts,\bx}$ is the projection onto the tangent space to
the sphere of radius $\|\bx\|_2$ at $\bx$, namely\footnote{In Eq.~\eqref{eq:GFlow_0} we have $\|\bx\|_2=1$ but
 we define $\Proj_{\Ts,\bx}$
more generally for future reference. By convention, we set $\bv/\|\bv\|_2=\bzero$ if $\bv=\bzero$.}
\begin{align}
\Proj_{\Ts,\bx}:= \id_{d}-\frac{\bx\bx^{\sT}}{\|\bx\|^2_2}\, .
\end{align}
Appendix \ref{app:GradientDescent-GF} proves results about gradient flow as well.

\vspace{0.3cm}

\begin{algorithm}[H]
\SetAlgoLined
\KwData{Couplings $\{\bG^{(k)}\}_{k\ge 0}$, stepsize $\eta$, 
number of iterations $K$}
\KwResult{approximate optimizer $\bx\in \S^{d-1}$}
Initialize $\bx^0\sim\Unif(\S^{d-1})$\;
\For{$k\in\{0,\dots,K-1\}$}{
$\bz^{k+1} =\bx^{k}-\eta\Proj_{\Ts,\bx^k}\nabla H(\bx^k)$\;
$\bx^{k+1} = \bz^{k+1}/\|\bz^{k+1}\|_2$\;
}
\KwRet $\bx^K$
\caption{Projected Gradient Descent}\label{algo:GD}
\end{algorithm}

\vspace{0.3cm}

Below, for $\bx\in \S^{d-1}$, $\bU_{\bx}$ is an orthonormal basis for $\Ts_{\bx}$ the orthogonal space to $\bx$.
Further, for $\bx_1,\bx_2\in \reals^d$, define 
$\bU_{\bx_1,\bx_2}:= \bR_{\bx_1,\bx_2}\bU_{\bx_1}$, where 
$\bR_{\bx_1,\bx_2}$ is the rotation that keeps unchanged the space orthogonal to $\bx_1$, $\bx_2$,
and maps $\bx_1$ to $\bx_2$. Finally, 
let $\bD\bF(\bx)\in \reals^{n\times d}$ be the Jacobian of $\bF$ at $\bx$.

Next, we state a general lemma which provides deterministic
conditions for the convergence of GD to a global optimum.
\begin{lemma}\label{lemma:GD}
Consider the Gradient Descent Algorithm \ref{algo:GD}, and fix $\eps_0\in (0,1/2)$. 
Let $\lambda_0:=\sigma_{\min}(\bD \bF(\bx^0)|_{\Ts,\bx^0})$, and $J_n,L_n,M_n$ be given by 
\begin{align}
J_n & := \sup_{\bx_1\neq \bx_2 \in\Ball^{d}(1+\eps_0)}
\frac{\|\bD \bF(\bx_1)-\bD \bF(\bx_2)\|_{\op}}{\|\bx_1-\bx_2\|_2} ,\\
L_n& :=\sup_{\bx_1\neq\bx_2\in\Omega}
\frac{\|\bD\bF(\bx_1)\bU_{x_1}-\bD\bF(\bx_2)\bU_{\bx_1,\bx_2}\|_{\op}}{\|\bx_1-\bx_2\|_2}\, ,\\
M_n &:= \sup_{\bx \in\Ball^{d}(1+\eps_0)}\| \bD \bF(\bx)\|_{\op}\, ,
\end{align}
and assume they satisfy
\begin{align}
\label{eq:ConditionGD}
L_n\|\bF(\bx^0)\|_2< \frac{\lambda_0^2}{16}\, ,
\end{align}
Further assume the step size $\eta$ to be such that
\begin{align}
\eta\le \min\left\{\frac{\eps_0}{
\max_{\bx\in\S^{d-1}}\|\Proj_{\Ts}\bD\bF(\bx)^{\sT} \bF(\bx)\big\|_2};\;
\frac{1}{M_n^2};\; \frac{1}{10\sqrt{n}(M_n+J_n)\max_{\bx\in\S^{d-1}}\| \bF(\bx)\|_2}\right\}\, .\label{eq:ConditionEta}
\end{align}
Then, for all $k\ge 0$ 
\begin{align}
\|\bF(\bx^k)\|_2^2 \le \|\bF(\bx^0)\|_2^2\, e^{-\lambda_0^2\eta k/8}\, .\label{eq:F-Decrease-GD}
\end{align}
\end{lemma}
The proof of this statement is presented in Section
\ref{app:GradientDescent-GD}.

By evaluating the conditions in this statement, we obtain the following.
\begin{theorem}\label{propo:Gradient}
Consider the Gradient Descent Algorithm \ref{algo:GD},
with  $\bF$ the Gaussian process defined in Section \ref{sec:Introduction},
and initialization $\bx^0$ independent of $\bF$. 
Assume $n,d\to\infty$ with $n/d\to\alpha\in [0,1)$. 
Define, for $c_0$ a sufficiently small absolute constant and
\begin{align}
\ualpha_{\sGD}(\xi):= \frac{c_0 \xi'(1)^2}{\xi''(1)\xi(1) \big(\log(\xi'''(1)/\xi''(1))\vee 1\big)} \,. 
\label{eq:AlphaGD}
\end{align} 
 If $\alpha<\ualpha_{\sGD}(\xi)$, and $\eta< 1/(C_1 d)$ with $C_1$ a suitable constant depending on
 $\xi$,
then the following happens with high probability. For all $k\ge 1$,
\begin{align}
\|\bF(\bx^k)\|_2^2& \le 2 n\xi(1)\, \exp\Big(-\frac{\xi'(1)\eta}{16}\big(\sqrt{d}-\sqrt{n}\big)^{2} \cdot  k\Big)\, .
\label{eq:GDRate}
\end{align}
\end{theorem}
Again, we refer to Section \ref{app:GradientDescent-GD} for the proof.

Considering our running example $\xi(t) = \gamma_0\log p+t^p$, cf. Eq.~\eqref{eq:LB-asymp}
and Appendix \ref{app:LB-asymp},
$\ualpha_{\sGD}(\xi)$ is equivalent, up to constants, to
\begin{align}
\ualpha_{\sGD}(\gamma_0,p):= \frac{c_1 }{\gamma_0 (\log p)^2 }\,. 
\end{align} 
It is instructive to compare this result with the lower bound on the threshold for existence of solutions,
see Eq.~\eqref{eq:LB-asymp}. Roughly speaking, for any $\delta>0$ and all $p$ large enough, 
if $1/\log p \ll \alpha\xi_0\le 1-\delta$ then solutions exist with high probability, 
but the gradient descent approach developed here does not guarantee that we can find them. 
We will see that this gap shrinks using the methods
in next sections. 
\subsection{Proof of Theorem \ref{propo:Gradient}: Gradient flow}
\label{app:GradientDescent-GF}

Since calculations are more transparent in the case of gradient flow, we will first 
treat this case, and then outline the modifications that arise for discrete time. 
Since these arguments are standard in the machine learning literature
\cite{du2018gradient,chizat2019lazy,oymak2020towards,arora2019fine,
allen2019learning,bartlett_montanari_rakhlin_2021}
we present them succinctly.

For gradient flow the time $t\in\reals_{\ge 0}$ is continuous and the state is updated according 
to
\begin{align}
\dot{\bx}(t) =-\Proj_{\Ts,\bx(t)}\nabla H(\bx(t))\, .\label{eq:GFlow}
\end{align}
(Recall the definition of cost function $H(\bx)=\|\bF(\bx)\|^2_2/2$.)
\begin{lemma}\label{lemma:GF}
Let $\lambda_0:=\sigma_{\min}(\bD \bF(\bx(0))|_{\Ts,\bx(0)})$
and $L_n := \Lip_{\perp}(\bD \bF;\S^{d-1})$, with the Lipschitz constant 
of Definition \ref{def:Lip}.
If 
\begin{align}
L_n\|\bF(\bx(0))\|_2< \frac{\lambda_0^2}{4}\, ,
\label{eq:LnBound}
\end{align}
then for all $t\ge 0$,
\begin{align}
\|\bF(\bx(t))\|_2^2 \le \|\bF(\bx(0))\|_2^2\, e^{-\lambda_0^2 t/4}\, .\label{eq:F-Decrease}
\end{align}
\end{lemma}
\begin{proof}
We define $\bK(\bx) := \bD\bF(\bx) \Proj_{\Ts,\bx} \bD\bF(\bx)^{\sT}$.
Then the gradient flow equation implies
\begin{align}
\frac{\de \phantom{t}}{\de t}\bF(\bx(t)) =-\bK(\bx(t))\cdot \bF(\bx(t))\, .
\end{align}
Let 
\begin{align}\label{eq:BadDef}
B:= \Big\{\bx\in \S^{d-1}:\; \sigma_{\min}(\bD \bF(\bx)|_{\Ts,\bx})\le \frac{\lambda_0}{2}\Big\}\, .
\end{align}
Obviously $\bx(0)\in B^c$. Further, letting $d_{\S^{d-1}}(\bx_1,\bx_2):=\arccos(\<\bx_1,\bx_2\>)$
denote the geodesic distance on the unit sphere, we have
\begin{align}
d_{\S^{d-1}}(\bx(0),B):= \inf\big\{\, d_{\S^{d-1}}(\bx(0),\bx):\; \bx\in B\, \big\} \ge \frac{\lambda_0}{2L_n}\, .
\label{eq:DistanceBad}
\end{align}
Define $t_*:= \inf\{t:\; \bx(t)\in B \}$. For all $t\le t_*$, we have
\begin{align}
\frac{\de \phantom{t}}{\de t}\|\bF(\bx(t))\|_2^2 \le -\frac{\lambda_0^2}{4}\big\|\bF(\bx(t))\big\|_2^2\, ,
\label{eq:KeyDecrease}
\end{align}
which implies Eq.~\eqref{eq:F-Decrease} for all $t\le t_*$

We next prove that $t_*=\infty$. Indeed, note that for $t\le t_*$,
\begin{align}
\frac{\de \phantom{t}}{\de t}\|\bF(\bx(t))\|_2&= -\frac{1}{\|\bF(\bx(t))\|_2}
\big\|\Proj_{\Ts,\bx(t)} \bD\bF(\bx(t))^{\sT}\bF(\bx(t))\big\|_2^2\\
&\le - \sigma_{\min} (\bD\bF(\bx(t))\Proj_{\Ts,\bx(t)}) \big\|\Proj_{\Ts,\bx(t)} \bD\bF(\bx(t))^{\sT}\bF(\bx(t))\big\|_2\\
& \le -\frac{\lambda_0}{2} \big\|\Proj_{\Ts,\bx(t)} \bD\bF(\bx(t))^{\sT}\bF(\bx(t))\big\|_2\, .
\end{align}
Further, denoting by $\bu(t)\in \Ts_{\bx(t)}$ the unit vector that is tangent to the geodesic between
$\bx(0)$ and $\bx(t)$ at $\bx(t)$, we have
\begin{align}
\frac{\de \phantom{t}}{\de t} d_{\S^{d-1}}(\bx(t),\bx(0))& = -
\<\Proj_{\Ts,\bx(t)} \bD\bF(\bx(t))^{\sT}\bF(\bx(t)),\bu(t)\>\, .
\end{align}
Therefore,  
\begin{align}
\frac{\de \phantom{t}}{\de t}\Big\{ d_{\S^{d-1}}(\bx(t),\bx(0)) + \frac{2}{\lambda_0}\|\bF(\bx(t))\|_2\Big\}
\le 0\, .
\end{align}
whence, for all $t\le t_*$, we have $d_{\S^{d-1}}(\bx(t),\bx(0)) \le 
2\|\bF(\bx(0))\|_2 /\lambda_0$. Recalling Eq.~\eqref{eq:LnBound}, we get $t_*=\infty$ 
as claimed.
\end{proof}

We next prove a version of Theorem \ref{propo:Gradient} for gradient flow.
\begin{theorem} \label{thm:Gradient-App} 
Consider gradient flow, as defined in Eq.~\eqref{eq:GFlow}, 
with respect to the energy function $H(\bx) = \|\bF(\bx)\|_2^2/2$
with  $\bF$ the Gaussian process defined in Section \ref{sec:Introduction},
and initialization $\bx(0)$ independent of $\bF$. 
Assume $n,d\to\infty$ with $n/d\to\alpha\in [0,1)$. 
Define, for $c_0$ a sufficiently small absolute constant,  
\begin{align}
\ualpha_{\sGF}(\xi):= \frac{c_0 \xi'(1)^2}{\xi''(1)\xi(1) \big(\log(\xi^{(3)}(1)/\xi''(1)) \vee 1\big)} \,. 
\end{align}
 If $\alpha<\ualpha_{\sGF}(\xi)$,
then the following happens with high probability. For all $t\ge 0$
\begin{align}
\|\bF(\bx(t))\|_2^2& \le 2 n\xi(1)\, \exp\Big(-\frac{3\xi'(1)}{16}\big(\sqrt{d}-\sqrt{n}\big)^{2} \cdot t\Big)\, .
\end{align}
\end{theorem}
\begin{proof}
This is an immediate application of Lemma \ref{lemma:GF}, whereby
we note that 
\begin{align}
\|\bF(\bx_0)\|^2_2 & = n\xi(1) +o_P(n)\, ,\label{eq:EstGF1}\\
\sigma_{\min}(\bD\bF(\bx(0))|_{\Ts,\bx(0)}) &= \sqrt{\xi'(1)}\big(\sqrt{d}-\sqrt{n}\big) +o_P(\sqrt{n})\, ,
\label{eq:EstGF2}\\
\Lip_{\perp}(\bD\bF;\S^{d-1})&  \le C\sqrt{d\, \xi^{''}(1)\log_+\frac{\xi^{(3)}(1)}{\xi''(1)} }\, .
\label{eq:EstGF3}
\end{align}
where the last inequality holds with high probability for a universal constant $C$.
The first estimate is by the law of large numbers, the second 
by Lemma \ref{lemma:HessianDistr}, which implies $\bD\bF(\bx)\bU_{\bx} = \sqrt{\xi'(q)}\, \bZ$ with 
$\bZ\sim\GOE(n,d-1)$ and $q=\|\bx\|_2^2$ and the Bai-Yin law, and for the last one, we refer to Lemma \ref{lemma:BoundLip}.
\end{proof}

\subsection{Proof of Theorem \ref{propo:Gradient}: Gradient descent}
\label{app:GradientDescent-GD}

The  analogue of Lemma \ref{lemma:GF} for gradient descent is stated in the main text as 
Lemma \ref{lemma:GD}.
\begin{proof}[Proof of Lemma \ref{lemma:GD}]
The key step is to prove that an inequality analogous to the upper bound \eqref{eq:KeyDecrease}
holds for all $k\le  k_{\star}:=\min\{k:\, \bx^k\in B\}$ with $B$ defined as per
Eq.~\eqref{eq:BadDef}, which we copy here for the reader's convenience
\begin{align}
B:= \Big\{\bx\in \S^{d-1}:\; \sigma_{\min}(\bD \bF(\bx)|_{\Ts,\bx})\le \frac{\lambda_0}{2}\Big\}\, .
\end{align}

To this end first note that
\begin{align}
\big\|\bz^{k+1}-\bx^k\big\|= \eta\big\|\Proj_{\Ts}\bD\bF(\bx^k)^{\sT} \bF(\bx^k)\big\|_2\le \eps_0\, .
\end{align}
(In what follows we omit the reference to the point on the sphere from the projector 
$\Proj_{\Ts}$.)

Further there exist $\bxi_{(i)}^{k}\in [\bx^k,\bz^{k+1}]$ (here
$[\bu,\bv]$ denotes the segment between $\bu$ and $\bv\in\reals^d$) such that
\begin{align*}
F_i(\bz^{k+1}) & = F_i(\bx^k)-\eta \nabla F_i(\bx^{k})^{\sT}\Proj_{\Ts}\bD\bF(\bx^k)^{\sT}\bF(\bx^k)
-\eta [\nabla F_i(\bxi_{(i)}^{k})-\nabla F_i(\bx^{k})]^{\sT}\Proj_{\Ts} \bD\bF(\bx^k)^{\sT}\bF(\bx^k)\\
&=: F_i(\bx^k)-\eta \nabla F_i(\bx^{k})^{\sT}\Proj_{\Ts} \bD\bF(\bx^k)^{\sT}\bF(\bx^k) +
\overline{\Delta^k_{i}}\, .
\end{align*}
We have the estimate
\begin{align}
|\overline{\Delta_{i}^k}|&\le \eta J_n\cdot \|\bz^{k+1}-\bx^k\|_2 \cdot
\|\Proj_{\Ts}\bD\bF(\bx^k)^{\Ts}\bF(\bx^k)\|_2\\
&\le J_n\eta^2 \|\Proj_{\Ts}\bD\bF(\bx^k)^{\Ts}\bF(\bx^k)\|_2^2\, .
\end{align}
Further by Pythagoras' theorem $\|\bz^{k+1}\|_2^2 = 1 +\|\bz^{k+1}-\bx^{k}\|^2$, whence,
for some $\bzeta_{(i)}^{k}\in [\bz^{k+1},\bx^{k+1}]$,
\begin{align*}
\big|F_i(\bx^{k+1})-F_i(\bz^{k+1})\big| &= 
\big|\<\nabla F_i(\bzeta_{(i)}^{k}),\bx^{k+1}\>\big| \cdot \big|\|\bz^{k+1}\|_2-1\big|\\
&\le \sup_{\bx\in\Ball^{d}(1+\eps_0)}\| \nabla F_i(\bx)\|_2
\cdot \|\bz^{k+1}-\bx^{k}\|_2^2\\
& \le  M_n \eta^2\big\|\Proj_{\Ts}\bD\bF(\bx^k)^{\sT} \bF(\bx^k)\big\|^2_2\, .
\end{align*}
We therefore obtain that
\begin{align}
\bF(\bx^{k+1})&= \bF(\bx^k)-\eta \bD \bF(\bx^k)\Proj_{\Ts} \bD\bF(\bx^k)^{\sT}\bF(\bx^k) 
+\bDelta^k\, ,\label{eq:CrucialGD1}\\
\|\bDelta^k\|_2 & 
\le \sqrt{n}  (J_n+M_n) \eta^2\big\|\Proj_{\Ts}\bD\bF(\bx^k)^{\sT} \bF(\bx^k)\big\|^2_2
\nonumber\\
& \le \frac{\eta}{10\max_{\bx\in\S^{d-1}}\|\bF(\bx)\|_2}\big\|\Proj_{\Ts}\bD\bF(\bx^k)^{\sT} \bF(\bx^k)\big\|^2_2\, ,\label{eq:GD_Delta}
\end{align}
where the last inequality follows from Eq.~\eqref{eq:ConditionEta}.
Further
\begin{align}
\big\|\bF(\bx^k)-\eta \bD \bF(\bx^{k})\Proj_{\Ts} \bD\bF(\bx^k)^{\sT}\bF(\bx^k) \big\|_2
&\le \|\bF(\bx^k)\|_2+\eta  M_n \|\Proj_{\Ts} \bD\bF(\bx^k)^{\sT}\bF(\bx^k) \big\|_2\nonumber\\
& \le (1+\eta M_n^2)\|\bF(\bx^k)\|_2\nonumber\\
& \le 2\|\bF(\bx^k)\|_2\, .\label{eq:CrucialGD2}
\end{align}
Again, in the last step, we used  condition \eqref{eq:ConditionEta}.
Also, note that by Eq.~\eqref{eq:GD_Delta} and condition \eqref{eq:ConditionEta},
we have
\begin{align}
\|\bDelta^k\|_2 & \le  \frac{\eta}{8\max_{\bx\in\S^{d-1}}\|\bF(\bx)\|_2} M_n^2 \|\bF(\bx^k)\big\|^2_2
\le \frac{1}{10}\|\bF(\bx^k)\big\|_2\, .\label{eq:CrucialGD3}
\end{align}

Using Eqs.~\eqref{eq:CrucialGD1} and \eqref{eq:CrucialGD2}, we get
\begin{align*}
\|\bF(\bx^{k+1})\|_2^2 \le & 
\big\|\bF(\bx^k)-\eta \bD \bF(\bx^{k})\Proj_{\Ts} \bD\bF(\bx^k)^{\sT}\bF(\bx^k) \big\|^2_2+
4\|\bDelta^k\|_2\|\bF(\bx^k)\|_2
+\|\bDelta^k\|_2^2\, .
\end{align*}
Using Eqs.~\eqref{eq:CrucialGD1}  to~\eqref{eq:CrucialGD3}, we get
\begin{align}
\|\bF(\bx^{k+1})\|_2^2 &\le 
\big\|\bF(\bx^k)-\eta \bD \bF(\bx^{k})\Proj_{\Ts} \bD\bF(\bx^k)^{\sT}\bF(\bx^k) \big\|^2_2+
5\|\bDelta^k\|_2\|\bF(\bx^k)\|_2\nonumber\\
&\le \big\|\bF(\bx^k)-\eta \bD \bF(\bx^{k})\Proj_{\Ts} \bD\bF(\bx^k)^{\sT}\bF(\bx^k) \big\|^2_2+
\frac{\eta}{2}\big\|\Proj_{\Ts}\bD\bF(\bx^k)^{\sT} \bF(\bx^k)\big\|^2_2\nonumber\\
&\le \<\bF(\bx^k),\Big(\id-\frac{3}{2}\eta\bK(\bx^k)+\eta^2\bK(\bx^k)^2\Big)\bF(\bx^k)\>\,.
\label{eq:DiscreteRecF}
\end{align}
Since $\eta\lambda_{\max}(\bK(\bx_k)) \le \eta M_n^2\le 1$ by Eq.~\eqref{eq:ConditionEta},
we conclude that
\begin{align*}
\|\bF(\bx^{k+1})\|_2^2 \le & \Big(1-\frac{1}{2}\eta \lambda_{\min}(\bK(\bx^k))\Big)\|\bF(\bx^{k})\|_2^2\,,
\end{align*}
and therefore, for any $k\le k_{\star}$,
\begin{align*}
\|\bF(\bx^{k})\|_2^2 \le & e^{-\eta\lambda_0^2k/8}\|\bF(\bx^{0})\|_2^2\,.
\end{align*}

We are left with the task of proving that $k_{\star}=\infty$. 
To this end we proceed as in the case of gradient flow. Namely, we note
that, for $k\le k_{\star}$ Eq.~\eqref{eq:DiscreteRecF} implies 
(for $\eta\lambda_{\max}(\bK(\bx_k)) \le 1$)
\begin{align}
\|\bF(\bx^{k+1})\|_2^2 & \le \|\bF(\bx^{k})\|_2^2 - \frac{\eta}{2}\big\|\Proj_{\Ts}
\bD\bF(\bx^k)^{\sT}\bF(\bx^k)\big\|^2_2\\
& \le \|\bF(\bx^{k})\|_2^2 - \frac{\lambda_0\eta}{4}\big\|\Proj_{\Ts}
\bD\bF(\bx^k)^{\sT}\bF(\bx^k)\big\|_2\|\bF(\bx^k)\|_2\, ,
\end{align}
and therefore
\begin{align}
\|\bF(\bx^{k+1})\|_2 & \le  \|\bF(\bx^{k})\|_2- \frac{\lambda_0\eta}{8}\big\|\Proj_{\Ts}
\bD\bF(\bx^k)^{\sT}\bF(\bx^k)\big\|_2\, .
\end{align}
Further 
\begin{align}
d_{\S^{d-1}}(\bx^{k+1},\bx^0)\le d_{\S^{d-1}}(\bx^{k+1},\bx^0) +
\eta\big\|\Proj_{\Ts}\bD\bF(\bx^k)^{\sT} \bF(\bx^k)\big\|_2\,.
\end{align}
Therefore, defining $\Psi(\bx) := d_{\S^{d-1}}(\bx,\bx^0) +(8/\lambda_0)\|\bF(\bx)\|_2$,
we have $\Psi(\bx^{k+1})\le \Psi(\bx^{k})$, whence the proof follows.
\end{proof}

We are now in position to prove Theorem \ref{propo:Gradient}.
\begin{proof}[Proof of Theorem \ref{propo:Gradient}]
The proof consists in using the estimates of Section \ref{sec:Basic} to check the assumptions of 
Lemma \ref{lemma:GD}.

More precisely, the assumption in \eqref{eq:ConditionGD} holds with high probability under 
condition  \eqref{eq:AlphaGD}, using the estimates \eqref{eq:EstGF1} to \eqref{eq:EstGF3}
that we obtained in the case of gradient flow. The convergence rate \eqref{eq:GDRate}
follows from Eqs.~\eqref{eq:F-Decrease-GD} and \eqref{eq:EstGF2}.

Finally, the  assumption in \eqref{eq:ConditionEta} on the stepsize, holds with high probability
when $\eta< 1/(C_1 d)$ with $C_1$ large enough, because the following inequalities hold 
with high probability for $C=C(\xi)$ a sufficiently large constant:
\begin{align}
\max_{\bx\in\S^{d-1}}\| \bF(\bx)\big\|_2& \le C\sqrt{d}\, ,\\
\max_{\bx\in\S^{d-1}}\|\Proj_{\Ts}\bD\bF(\bx)^{\sT} \bF(\bx)\big\|_2& \le C\,d \, ,\\
J_n\vee M_n  &\le C\, \sqrt{d}\, .
\end{align}
These inequalities follow from Propositions \ref{propo:MaxValue} and \ref{propo:OPNorm},
thus completing the proof.
\end{proof}

%
%************************************************
%
\subsection{Sensitivity bounds for gradient flow}
\label{sec:Insensitivity}

In this section we study gradient flow on $\S^{d-1}$
as introduced in Appendix \ref{app:GradientDescent-GF}, see Eq.~\eqref{eq:GFlow}
(Riemannian gradient flow on $\S^{d-1}$ with respect
to the cost function $H(\bx)=\|\bF(\bx)\|^2_2/2$.)
We study perturbations of the initialization, and
prove that gradient flow is Lipschitz continuous in such perturbations
as long as energy converges to $0$ at a rate that is integrable (in particular,
faster than $t^{-1}$.)

It is convenient
to denote by $\bx(t;\bz)$ the solution of Eq.~\eqref{eq:GFlow} with initialization
$\bz\in\S^{d-1}$:
\begin{align}
\dot{\bx}(t;\bz) =-\Proj_{\Ts,\bx(t;\bz)}\nabla H(\bx(t;\bz))\, ,\;\;\;\; 
\bx(0;\bz) =\bz\, .
\end{align}
The following lemma establishes a sensitivity bound under conditions on the rate of convergence.
\begin{lemma}\label{lemma:GF-stability-A}
For $\bz_0\in\S^{d-1}$, let $L_n := \Lip_{\perp}(\bD \bF;\S^{d-1})$, with the Lipschitz constant 
of Definition \ref{def:Lip}.
Assume $r_0, K_n$ be such that 
\begin{align}
& \frac{\|\bF(\bz_0)\|_2}
{10\|\bD \bF(\bz_0)|_{\Ts,\bz_0}\|_{\op}} \ge r_0 \, ,\\
&\sup_{\bx\in \S^{d-1}}\big\|\bD^2 \bF(\bx)|_{\Ts\otimes \Ts}\big\|_{\op}
\le K_n\, .
\label{eq:LnBound-2}
\end{align}
Further assume that, for all $\bz\in \S^{d-1}\cap\Ball(\bz_0, r_0)$
\begin{align}
\int_0^{\infty} \|\bF(\bx(t;\bz))\|\, \de t\le T_0 \|\bF(\bz)\|\, . \label{eq:Integrability}
\end{align}
Then, for all $\bz_1,\bz_2\in\S^{d-1}\cap\Ball(\bz_0, r_0)$, we have
\begin{align}
\sup_{t\ge 0}\|\bx(t;\bz_1)-\bx(t;\bz_2)\|_2^2 \le 2 B_n
\, \|\bz_1-\bz_2\|_2\, ,
\;\;\;\;\;\;\;
 B_n :=\exp\Big(2 T_0 K_n\|\bF(\bz_0)\|_2\Big)\, 
.\label{eq:X-Stability-1}
\end{align}
\end{lemma}
\begin{proof}
Define $\eps \mapsto \bz(\eps)$ be the unit-speed geodesic connecting $\bz_1$ and $\bz_2$. 
Namely, $\bz(\eps) = \bz_1\cos \eps+\bv\sin \eps$, 
where $\bv\in \Ts_{\bz_1}$, $\|\bv\|=1$is defined by
 $\bv = \proj_{\Ts,\bz_1}\bz_2/\|\proj_{\Ts,\bz_1}\bz_2\|_2$.
We then have 
\begin{align}
\frac{\de\phantom{\eps}}{\de\eps}\bx(t,\bz(\eps)) =
 \bM(t;\bz(\eps)) \frac{\de\phantom{\eps}}{\de\eps} \bz(\eps) \, ,
\end{align}
where  for $\bz \in \S^{d-1}$, $\bM(t;\bz)$ has row space in $\Ts_{\bz}$ and column space in
 $\Ts_{\bx(t;\bz)}$.
A standard calculation yields (for $\bx(t)=\bx(t;\bz)$):
\begin{align}
\frac{\de\phantom{t}}{\de t}\bM(t;\bz) =-\Proj_{\Ts,\bx(t)}\nabla^2
 H(\bx(t))\Proj_{\Ts,\bx(t)} \bM(t;\bz) \, ,\;\;\;\; 
\bM(0;\bz) =\Proj_{\Ts,\bz}\, .
\end{align}
Therefore
\begin{align}
\|\bM(t;\bz) \|_{\op}\le \exp\left\{\int_0^t \lambda_{\max} \big(-\Proj_{\Ts,\bx(s)}\nabla^2
 H(\bx(s))\Proj_{\Ts,\bx(s)}\big) \,\de s\right\}\,.\label{eq:DiffInit}
\end{align}
Recalling the form of the Hessian 
\begin{align}
\nabla^2
 H(\bx) = \bD\bF(\bx)^\sT\bD\bF(\bx) + \sum_{i=1}^nF_i(\bx) 
 \nabla^2F_i(\bx)\, ,
\end{align}
 we get
\begin{align*}
 \lambda_{\max} \big(-\Proj_{\Ts,\bx}\nabla^2
 H(\bx)\Proj_{\Ts,\bx}\big)& \le 
 \left\|\sum_{i=1}^n F_i(\bx)\nabla^2 F_i(\bx)\right\|_{\op}\\
 &\le  K_n \big\|\bF(\bx)\big\|_2\, .
\end{align*}
Using  the integrability assumption  \eqref{eq:Integrability}, we get
\begin{align*}
\int_0^t \lambda_{\max} \big(-\Proj_{\Ts,\bx(s)}\nabla^2
 H(\bx(s))\Proj_{\Ts,\bx(s)}\big) \,\de s&\le K_n
 \int_{0}^{\infty}\big\|\bF(\bx(s))\big\|_2\de s\\
 &\le K_n T _0\|\bF(\bz)\|_2\, .
\end{align*}
Since  $r_0\le \|\bF(\bz_0)\|_2/10\|\bD \bF(\bz_0)|_{\Ts,\bz_0}\|_{\op}$, we have
 $\|\bF(\bz)\|_2\le 2 \|\bF(\bz_0)\|_2$ for all $\bz\in\S^{d-1}\cap\Ball^d(\bz_0,r_0)$.
Together with Eq.~\eqref{eq:DiffInit} this yields
\begin{align}
\|\bM(t;\bz) \|_{\op} \le \exp\Big( 2 T_0 K_n\|\bF(\bz_0)\|_2\Big) =B_n\, .
\end{align}
Therefore,  letting $\eps_* = \arccos\<\bz_1,\bz_2\>\le 2\|\bz_1-\bz_2\|_2$
\begin{align}
\big\|\bx(t,\bz_2)- \bx(t,\bz_1)\big\|_2&=
\left\|\int_{0}^{\eps_*} \bM(t;\bz(\eps))\dot{\bz}(\eps)\, \de\eps\right\|_2\\
&\le \int_{0}^{\eps_*} \|\bM(t;\bz(\eps))\|_{\op}\, \de\eps\\
&\le 2B_n\|\bz_1-\bz_2\|_2\, .
\, ,
\end{align}
This proves the claim.
\end{proof}

We next specialize the above result to the NT regime. 
\begin{lemma}\label{lemma:GF-stability}
Under the general setting of Lemma 
\ref{lemma:GF-stability-A}  let $\lambda_0 := \sigma_{\min}(\bD \bF(\bz_0)|_{\Ts,\bz_0})$
and
\begin{align}
r_0 := \min \Big( \frac{\lambda_0}{10L_n};\; \frac{\|\bF(\bz_0)\|_2}
{10\|\bD \bF(\bz_0)|_{\Ts,\bz_0}\|_{\op}} \Big)\ .
\end{align}
If
\begin{align}
&L_n\|\bF(\bz_0)\|_2< \frac{\lambda_0^2}{8}\, ,
\end{align}
then, for all $\bz_1,\bz_2\in\S^{d-1}\cap\Ball(\bz_0, r_0)$, we have
\begin{align}
\sup_{t\ge 0}\|\bx(t;\bz_1)-\bx(t;\bz_2)\|_2^2 \le 2 B_n
\, \|\bz_1-\bz_2\|_2\, ,
\;\;\;\;\;\;\;
 B_n :=\exp\Big(\frac{16}{\lambda_0^2}K_n\|\bF(\bz_0)\|_2\Big)\, 
.\label{eq:X-Stability}
\end{align}
\end{lemma}
\begin{proof}
Note that, for  any $\bz\in\S^{d-1}\cap\Ball(\bz_0, r_0)$, the conditions of 
Lemma \ref{lemma:GF} hold.
Therefore
\begin{align}
\int_{0}^{\infty}\big\|\bF(\bx(s;\bz))\big\|_2\de s
 &\le
 \int_{0}^{\infty}\|\bF(\bz)\|_2\, e^{-\lambda_0^2 s/8} \de s\\
 &\le 
  \frac{8}{\lambda_0^2} \|\bF(\bz)\|_2\, .
\end{align}
This shows that condition \eqref{eq:Integrability} in Lemma \ref{lemma:GF-stability-A}
holds. The claim follows by applying that lemma.
\end{proof}

\begin{proposition}
Consider gradient flow with respect to $H(\bx) = \|\bF(\bx)\|_2^2/2$
with  $\bF$ the Gaussian process defined in Section \ref{sec:Introduction}.
Assume $n,d\to\infty$ with $n/d\to\alpha\in [0,1)$ and 
let $\bz_0\in \S^{d-1}$ be an initialization independent of $\bF$.
Define, for $c_0,c_0'$  sufficiently small absolute constantx,  
\begin{align}
\ualpha_{\sGF}(\xi) & := \frac{c_0 \xi'(1)^2}{\xi''(1)\xi(1) \big(\log(\xi^{(3)}(1)/\xi''(1)) \vee 1\big)} \,,\\
\rr_0(\xi,\alpha) & := c_1\sqrt{\frac{\alpha \xi(1)}{ \xi'(1)}}
\end{align}
 Then there exists an absolute constant $C$ such that, if $\alpha<\ualpha_{\sGF}(\xi)$,
then the following happens with high probability. 
For for all $\bz_1,\bz_2\in\S^{d-1}\cap\Ball(\bz_0, \rr_0)$,
\begin{align}
&\sup_{t\ge 0}\|\bx(t;\bz_1)-\bx(t;\bz_2)\|_2 \le \overline{B}_n \|\bz_1-\bz_2\|_2\, ,\\
 & \overline{B}_n  := 2\exp\Big(C\sqrt{\frac{\alpha\xi^{(2)}(1)\xi(1)}{\xi'(1)^2}
 \log_+ \frac{\xi^{(3)}(1)}{\xi^{(2)}(1)}}\Big)\, .
\end{align}
\end{proposition}
\begin{proof}
Recall the following estimates that we already obtained in the proof
Theorem \ref{thm:Gradient-App}:
\begin{align}
\|\bF(\bz_0)\|^2_2 & = n\xi(1) +o_P(n)\, ,\label{eq:EstGF1-Bis}\\
\sigma_{\min}(\bD\bF(\bz_0)|_{\Ts,\bx(0)}) &= \sqrt{\xi'(1)}\big(\sqrt{d}-\sqrt{n}\big) +o_P(\sqrt{n})\, ,
\label{eq:EstGF2-Bis}\\
\Lip_{\perp}(\bD\bF;\S^{d-1})&  \le C\sqrt{d\, \xi^{''}(1)\log_+\frac{\xi^{(3)}(1)}{\xi''(1)} }\, .
\label{eq:EstGF3-Bis}
\end{align}
Further using Proposition  \ref{propo:OPNorm} (cf. Eq.~\eqref{eq:DDFnorm1})
and Gaussian concentration, we obtain that, with high probability
\begin{align}
\sup_{\bx\in \S^{d-1}}\max_{i\le n} \big\|\Proj_{\Ts,\bx}\nabla^2 F_i(\bx)\Proj_{\Ts,\bx}\big\|_{\op}
\le K_n:=C\sqrt{d\xi''(1) \log_+ \frac{\xi^{(3)}(1)}{\xi''(1)}}\, .
\end{align}
Finally, using the the fact that $\bD\bF(\bz_0)|_{\Ts,\bz_0} = \sqrt{\xi'(1)} \bZ\proj_{\Ts,\bz}$
for $\bZ\sim\GOE(n,d)$, we have, with high probability 
\begin{align}
\big\| \bD\bF(\bz_0)|_{\Ts,\bz_0} \big\|_{\op} \le 2\sqrt{\xi'(1)}\big(\sqrt{n}+\sqrt{d}\big)\, .
\end{align}

The claim then follows by applying Lemma \ref{lemma:GF-stability}, whereby condition 
$L_n\|\bF(\bz_0)\|_2< \lambda_0^2/8$ is verified for $\alpha\le \ualpha_{\sGF}(\xi) $
as in the proof of Theorem \ref{thm:Gradient-App}.
Further $B_n\le \overline{B}_n$ follows by substituting the above estimates in
the expression for $B_n$ given in Eq.~\eqref{eq:X-Stability}.
(Note that, without loss of generality, we can assume $\ualpha_{\sGF}(\xi)<1/2$.)

Finally, again applying the estimates above we conclude that, with high probability,
for $\alpha\le \ualpha_{\sGF}(\xi)$
\begin{align}
\frac{\lambda_0}{L_n} & \ge c_1\sqrt{\frac{\xi'(1)}{\xi''(1)\log_+\xi^{(3)}(1)/\xi''(1)}}\, ,\\
\frac{\|\bF(\bz_0)\|_2}
{\|\bD \bF(\bz_0)|_{\Ts,\bz_0}\|_{\op}}  & \ge c_1\sqrt{\frac{\alpha \xi(1)}{ \xi'(1)}}\, .
\end{align}
Since the second lower bound is smaller for $\alpha\le \ualpha_{\sGF}(\xi)$, 
we obtain $\rr_0(\xi,\alpha) \le r_0$ with high probability, thus completing the proof.
\end{proof}
\section{Proof sketch for Theorem \ref{thm:Optimality}}
\label{sec:optimalitypfsketch}

As mentioned after the statement of Theorem \ref{thm:Optimality}, the fact that $u_{\salg,\slb}(\alpha,\xi)=u_{\salg}(\alpha,\xi)$ for $u_{\salg,\slb}(\alpha,\xi)$ of \eqref{eq:ALG-LB}
is proved in \cite{hardness_in_preparation}. Here we sketch how \eqref{eq:optimality} is proved with $u_{\salg,\slb}(\alpha,\xi)$ in the lower bound.

The proof follows the technique developed in \cite{huang2022tight,huang2023hardnessmulti} for optimization of spin glasses.
For any $\rho\in[0,1]$ and $\bF,\tilde\bF\in  \cuF(n,d)$ whose coefficients are jointly 
Gaussian, we say that $\bF$ and $\tilde\bF$ are $\rho$-correlated if
 $\E \{\bF(\bx)\tilde\bF(\by)^{\sT} \} = \rho \E \{\bF(\bx)\bF(\by)^{\sT}\}= \rho \E \{\tilde\bF(\bx)\tilde\bF(\by)^{\sT} \}$.
An algorithm $\bx^{\salg}: \cuF(n,d) \to \S^{d-1}$, $\bF\mapsto \bx^{\salg}(\bF)$ is said to be $(\lambda,\nu)$ overlap concentrated if for any $\rho\in[0,1]$ and $\rho$-correlated $\bF$ and $\tilde\bF$,
\[
\P\Big(
\big|\big\< \bx^{\salg}(\bF),\bx^{\salg}(\tilde\bF)\big\>- \E \big\< \bx^{\salg}(\bF),\bx^{\salg}(\tilde\bF)\big\> \big| \ge \lambda 
\Big) \le \nu\,.
\]

Suppose that $\bx^{\salg}$ is an algorithm which is Lipschitz with constant $M$  for some $M>0$.
From Gaussian concentration, one can show that $\bx^{\salg}$ is $(\lambda,\nu)$ overlap concentrated for arbitrary $\lambda>0$ and $\nu=2e^{-\frac{d\lambda^2}{8M}}$. 
This is the only role of the Lipschitz assumption in the proof, and, in fact, the theorem also holds if we instead assume overlap concentration directly.

For any $\rho\in[0,1]$ and $\rho$-correlated $(\bF,\tilde\bF)$, define
\[
\chi_d(\rho):=\E\big\<\bx^{\salg}(\bF),\bx^{\salg}(\tilde\bF)\big\>\,.
\] 
We will assume that $\chi_d$ is not constant in $\rho$, which is the case only if $\bx^{\salg}$ a.s. outputs some deterministic point (since then $\chi_{d}(0)=\chi_{d}(1)=1$). It is also not difficult to check that $\chi_d(0)\ge 0$ by considering the mean inner-products $\E\<\bx^{\salg}(\bF^{(i)}),\bx^{\salg}(\bF^{(j)})\>$ for a large collection of i.i.d. $\bF^{(i)}$, which are pairwise $0$-correlated.
 
In the sequel we shall assume that  $\chi_d(\rho)\to\chi(\rho)$ pointwise for some $\chi:[0,1]\to[0,1]$. Indeed, we may always move to a subsequence such that the convergence holds, and proving Theorem \ref{thm:Optimality} for such sub-sequences implies it in general (formally, in the rest of the proof we need to consider a sub-sequence on which there is convergence, but this will be implicit from now on).
$\chi$ may have a discontinuity at $\rho=1$ and $\chi(0)$ may be strictly positive. Those cases need to be dealt with separately, and to simplify our discussion, we henceforth assume that neither occurs. 

Using Gaussianity, it can be shown that $\chi_d(t)=\sum_{i=0}^{\infty}\gamma_{d,i}t^i$ for some $\gamma_{d,i}\ge 0$. Under our assumptions, $\chi(t)=\sum_{i=0}^{\infty}\gamma_{i}t^i$ for $\gamma_i=\lim_{d\to\infty}\gamma_{d,i}$ and it is therefore a continuous, convex and strictly increasing on $[0,1]$ with $\chi(0)=0$ and $\chi(1)=1$. Hence,  the function 
$
p(t):=\chi^{-1}(t)
$
is well-defined on $[0,1]$ and belongs to $\cuC^1$. As we explain below, with this variance profile 
$H(\bx^{\salg})/n \ge u_{\slb}(1;p)-\delta$ with high 
probability, for any fixed $\delta>0$, where $u_{\slb}(1;p)$
 is defined 
in Remark~\ref{rmk:LBeqUB}.

\subsection{Hierarchically correlated systems}
Given some integer numbers $\kappa,m\ge1$, let $\TT=\TT(\kappa,m)$
be the $\kappa$-regular tree
of depth $m$ with root $\varnothing$, depth $a$ vertex set $V_{a}=[\kappa]^{a}$
and set of leaves $\LL=V_{m}$. For any two vertices $u,v\in\TT$
let $|u|$ denote the depth of $u$, $\varnothing\leftrightarrow u$ the path
from the root to $u$ and $u\wedge v$ the least common ancestor.
We write $\langle u,v\rangle:=|u\wedge v|/m$ and $u\leq v$ if
$v$ is a descendant of $u$ or $u=v$. 

For each edge $e=(u,v)$ with $|u|=|v|-1=a$, let $\bF^{e}$
be an i.i.d. copy of $\bF$ and by an abuse of notation write $p(e)=p((a+1)/m)-p(a/m)$, where $p(t)=\chi^{-1}(t)$ is the variance profile as above.
For any vertex $u\in\mathbb{T}$, define 
\begin{equation}\label{eq:bFu}
	\bF^{(u)}:=\sum_{e\in\varnothing\leftrightarrow u}\sqrt{p(e)}\bF^{e}
\end{equation}
so that
\begin{equation}
	\E F_{i}^{(u)}(\bx)F_{j}^{(v)}(\by)=\delta_{ij}p(\langle u,v\rangle)\xi(\langle\bx,\by\rangle).\label{eq:correlation-1}
\end{equation}

Note that, for large $d$, 
\begin{equation*}
\E\big\langle\bx^{\salg}(\bF^{(u)}),\bx^{\salg}(\bF^{(v)})\big\rangle=\chi_{d}\circ\chi^{-1}(\langle u,v\rangle)\approx \langle u,v\rangle.
\end{equation*}
By the aforementioned overlap concentration of Lipschitz algorithms and a union bound, for fixed $\kappa$ and $m$, we have that, for large $d$, 
\begin{equation}\label{eq:alg_hierarchy}
	\P\Big(
	\sup_{u,v\in\LL}\big|\big\langle\bx^{\salg}(\bF^{(u)}),\bx^{\salg}(\bF^{(v)})\big\rangle - \langle u,v\rangle \big|\le \eps
	\Big)\ge 1-2\kappa^m e^{-\frac{d\eps^2}{32M}}\,.
\end{equation}

To any leaf of the tree $u\in\LL$, we associate the point $\bx_{u}:=\bx^{\salg}(\bF^{(u)})\in\S^{d-1}$. 
To other vertices $u\in\TT\setminus\LL$ we associate the normalized averages
\begin{equation}
	\bx_{u}:=\sqrt{|u|/m}\cdot\frac{\hat\bx_u}{\|\hat\bx_u\|},\quad\hat\bx_u:=\kappa^{-(m-|u|)}\sum_{v\in\LL(u)}\bx_{v},\label{eq:xu}
\end{equation}
where $\LL(u):=\{v\in\LL:u\leq v\}$ is the set of leaves
in the subtree of $u$. By definition, $\|\bx_u\|^2=|u|/m$. Combining \eqref{eq:alg_hierarchy} with some algebra one can show that, for any $\eps$ there is a constant $c=c(\eps)>0$ such that for any fixed sufficiently large $\kappa$ and any $m$, for any large $d$, 
\begin{equation}\label{eq:xu_hierarchy}
	\P\Big(
	\sup_{u,v\in\TT}\big|\big\langle\bx_u,\bx_v\rangle - \langle u,v\rangle \big|\le \eps
	\Big)\ge 1-2\kappa^m e^{-dc}\,.
\end{equation}

\subsection{Reduction to one-level averages}

Throughout this section, we adopt the notation 
$\ALG=u_{\salg,\slb}(\xi,\alpha):=\inf_{p\in\cuC^1}  u_{\slb}(1;p)$
for the predicted algorithmic theshold. 
Let $\delta>0$ be some number such that $\ALG-\delta>0$.
Suppose that the variance profile $p(t):=\chi^{-1}(t)$ 
is as defined above and let $\varphi:[0,1]\to\R_+$ be some differentiable function such that $\varphi(1)=\ALG-\delta$, $\varphi(0)=\dot{u}_{\slb}(0;p)-\delta'$ and $\dot{\varphi}(t)\le \dot{u}_{\slb}(t;p)-\delta'$, for some $\delta'>0$ and any $t\in[0,1]$.

In the previous section, we constructed the random tree $(\bx_u)_{u\in\TT}$ from the outputs of the correlated systems $\bF^{(u)}$. On this tree, we consider the corresponding energies  $H^{(u)}(\bx_u):=\frac12\|\bF^{(u)}(\bx_u)\|^2$. From Jensen's inequality (see \cite[Proposition 3.6(a)]{huang2023hardnessmulti}) one can show that
\begin{align}
	\P\big(
	\forall u\in \LL:\, H^{(u)}(\bx_u)/n\le \ALG-\delta
	\big)
	&\ge
	\P\big(
	H(\bx^{\salg}(\bF))/n\le \ALG-\delta
	\big)^{\kappa^m}\,.\label{eq:ALGdeltabd}
\end{align}

Let $\bar \calE$ be the event that there exists $u\in\TT\setminus\LL$ such that 
\begin{equation}\label{eq:xuavg}
\frac{1}{n}H^{(u)}(\bx_u) \ge \varphi(|u|/m)\quad\mbox{and}\quad\frac{1}{n\kappa}\sum_{v:\,u\le v,\,|v|=|u|+1} H^{(v)}(\bx_v)\le \varphi((|u|+1)/m)
\end{equation}
and let $\calE_{\varnothing}$ be the event that $H^{(\varnothing)}(\bx_{\varnothing})/n=H^{(\varnothing)}(\bzero)/n\ge \varphi(0)$. 
We claim that the left-hand side of \eqref{eq:ALGdeltabd} is bounded from above by 
$\P(\bar\calE\cup \calE_{\varnothing}^c)$.
Indeed, if $\bar\calE^c\cap \calE_{\varnothing}$ occurs, then  starting at the root $\varnothing$ and iteratively choosing the child with maximal energy $H^{(v)}(\bx_v)$ constructs a path on which  $H^{(u)}(\bx_u)/n\ge \varphi(|u|/m)$. In particular, since $\varphi(1)= \ALG-\delta$, the event in \eqref{eq:ALGdeltabd} does not occur.   

Of course, $\P(\calE_{\varnothing}^c)\le ce^{-cd}$ for some $c=c(\xi,\alpha)$. Hence, denoting by $\calE_{\eps}$ then event from \eqref{eq:xu_hierarchy}, to prove the required bound on the probabilities in \eqref{eq:ALGdeltabd}, it will suffice to show that
\[
\P(\bar\calE\cap  \calE_{\eps})\le e^{-Cd}/C\,.
\]

Define $\bar T(\by,u)=\bar T(\by,u,\kappa,m,\eps)$ and $\bar \Delta(\by,u)=\bar \Delta(\by,u,\kappa,\eps)$ by
\begin{equation}
	\begin{aligned}
		\bar T(\by,u):= & \Big\{(\by_v)_{u\le v, |v|=|u|+1}:\,\forall v\neq v',\\ &\big|\langle\by_v,\by\rangle-|u|/m\big|\vee\big|\langle\by_v,\by_{v'}\rangle-|u|/m\big|\le \eps,\,\langle\by_v,\by_v\rangle=(|u|+1)/m\Big\}\,,\\
		\bar \Delta(\by,u):= & \frac{1}{\kappa}\inf_{ \bar T(\by,u)}\sum_{v:\,u\le v,\,|v|=|u|+1}H^{(v)}(\by_v) \,.
	\end{aligned}
	\label{eq:barT(x)}
\end{equation}
On $\calE_{\eps}$, for any $u\in\TT\setminus\LL$ the $\kappa$ children of $\bx_u$ belong to $\bar T(\bx_u,u)$ and thus
\[
\frac1n\bar\Delta(\bx_u,u) \le \frac{1}{n\kappa}\sum_{v:\,u\le v,\,|v|=|u|+1} H^{(v)}(\bx_v)\,.
\]
Hence, $\P(\bar\calE\cap  \calE_{\eps})\le \P(\calE\cap  \calE_{\eps}) \le \P(\calE)$ where we define $\calE$ as the event that 
there exists $u\in\TT\setminus\LL$ and $\bx$ with $\|\bx\|^2=|u|/m$ such that 
\begin{equation}%\label{eq:xuavg}
	\frac{1}{n}H^{(u)}(\bx) \ge \varphi(|u|/m)\quad\mbox{and}\quad\frac1n\bar\Delta(\bx,u)\le \varphi((|u|+1)/m)\,.
\end{equation}

\subsection{The role of many approximately orthogonal replicas}

Similarly to \eqref{eq:barT(x)}, define 
$T(\bx)=T(\bx,\kappa,\eps,q,q')$ and $\Delta(\bx)=\Delta(\bx,\kappa,\eps,q,q',\rho,\rho')$  by 
\begin{equation}
	\begin{aligned}T(\bx):= & \Big\{(\bx^{1},\ldots,\bx^{\kappa})\in(\Ball^d(1))^{\kappa}:\,\forall i<j\leq\kappa,\\
		& \quad\big|\langle\bx^{i},\bx\rangle-q\big|\vee\big|\langle\bx^{i},\bx^{j}\rangle-q\big|\le \eps,\,\langle\bx^{i},\bx^{i}\rangle=q'\Big\}\,,\\
		\Delta(\bx):=&\frac{1}{\kappa}\inf_{T(\bx)}\sum_{i=1}^{\kappa} H^{(i)}(\bx^{i})\,,
	\end{aligned}
	\label{eq:T(x)}
\end{equation}
where, by an abuse of notation, $H^{(i)}(\bx):=\frac12\|\bF^{(i)}(\bx)\|^2$ and $\bF^{(i)}$ are correlated systems such that 
\begin{equation}
	\E F_{\ell}^{(i)}(\bx)F_{\ell'}^{(j)}(\by)=\delta_{\ell\ell'} \big[ \rho + \delta_{ij}(1-\delta_{i0})(\rho'-\rho)\big]
	\xi(\langle\bx,\by\rangle)\,.\label{eq:correlation-2}
\end{equation}

Note that for any $u$ such that $|u|=i$,  setting $q=i/m$, $q'=(i+1)/m$, $\rho=p(q)$, $\rho'=p(q')$ we have that  $(H^{(u)}(\bx),\bar \Delta(\bx,u))$ and $(H^{(0)}(\bx),\Delta(\bx))$ have the same law, as processes on $\sqrt{q}\S^{d-1}=\{\bx:\,\|\bx\|^2=q\}$. 
Hence, since $|\TT|=\kappa^m$ does not depend on $d,n$, by a union bound, to prove that $\P(\calE)\le e^{-Cd}/C$ and complete the proof of Theorem \ref{thm:Optimality} it is sufficient to show that for any $0\le i\le m-1$,
\begin{equation}\label{eq:Deltabd}
	\P\bigg(\exists \bx\in \sqrt{q}\S^{d-1}:\ \frac{1}{n}H^{(0)}(\bx) \ge \varphi(q)\quad\mbox{and}\quad\frac1n\Delta(\bx)\le \varphi(q')
	\bigg)\le e^{-Cd}/C\,,
\end{equation}
where we may need to adjust the constant $C=C(\xi,\alpha,p,\eps,\kappa,m,q,q')>0$.

Note that for small $\eps$ and $\bx$ such that $\|\bx\|^2=q$, if $(\bx^i)\in  T(\bx)$ then $\bx^i-\bx$ are roughly orthogonal to each other. The first crucial role of this orthogonality is that for any $t,c>0$,
\begin{align}\label{eq:unifconc}
	\P\bigg(\sup_{\bx:\,\|\bx\|^{2}=q}\Big|\Delta(\bx)-\E\big[\Delta(\bx)\,|\,H^{(0)}(\bx)\big]\Big|>nt
	\bigg)\le e^{-cd}\,,
\end{align}
for sufficiently large $d$ and $\kappa$ and small $\eps$.
The main step in proving this bound simply uses Gaussian concentration. A similar idea was first used
in \cite{subag2018landscape,chen2018generalized,chen2020generalized2} 
in the study of the Thouless-Anderson-Palmer approach for mixed $p$-spin models and later utilized in \cite{subag2021following}
in the context of optimization and \cite{huang2022tight,huang2023hardnessmulti}
in the context of algorithmic hardness for spin glasses, which
is closest to our setting.

By the uniform concentration of \eqref{eq:unifconc}, to prove \eqref{eq:Deltabd} it is sufficient to show that, for some small $t>0$,
\[
\frac1n H^{(0)}(\bx)\ge \varphi(q)\implies \frac1n \E\big[\Delta(\bx)\,|\,H^{(0)}(\bx)\big]\ge \varphi(q')+t\,.
\]
If $m$ is sufficiently large (so that $q'-q=1/m$ is small), one can show that the conditional expectation above is increasing with the value of $H^{(0)}(\bx)$. Namely, the above implication follows if we can show that
\begin{equation}
	\frac1n H^{(0)}(\bx)= \varphi(q)\implies \frac1n \E\big[\Delta(\bx)\,|\,H^{(0)}(\bx)\big]\ge \varphi(q')+t\,.
	\label{eq:conditionalDelta}
\end{equation}

To analyze $\E\big[\Delta(\bx)\,|\,H^{(0)}(\bx)\big]$ it is useful to write the processes $\bF^{(i)}$ from  \eqref{eq:correlation-2} explicitly as 
\[
\bF^{(0)} =\sqrt{\rho} \tilde \bF^{(0)} \mbox{ and }\, \bF^{(i)}= \sqrt{\rho} \tilde \bF^{(0)}+\sqrt{\rho'-\rho} \tilde \bF^{(i)}\,,
\]
for $i\ge1$ for i.i.d. copies $\tilde \bF^{(i)}$ of $\bF$. Using this representation, $\Delta(\bx)$ can be expressed as a sum of several terms. The second crucial role of using many approximately orthogonal directions $\bx^i-\bx$ is that one of those terms scales like $1/\sqrt{\kappa}$, and therefore can be made as small as we wish provided that the number of copies $\kappa$ is large enough. Other error terms, which basically correspond to degree $3$ and above terms in a Taylor expansion of the processes $\tilde \bF^{(i)}$ around $\bx$, vanish as $m\to\infty$. For large $\kappa,m$, we have $\rho'-\rho\approx \dot p(q)/m$ while $q'-q=1/m$, and 
$\frac1n \E\big[\Delta(\bx)\,|\,H^{(0)}(\bx)\big]-\frac1n H^{(0)}(\bx)$ can be shown to be approximately equal to $\dot u_{\slb}(q)/m$. The fact that $\dot{\varphi}(q)\le \dot{u}_{\slb}(q;p)-\delta'$ for $\delta'$ independent of $\kappa,m$ allows us to show \eqref{eq:conditionalDelta}.

%
%***********************************************
%
\subsection{Monotonicity of the algorithmic threshold energy}
\label{sec:Monotonicity}

Let $\cL(s;p,\dot{p},u,t,\alpha)= \cL(s;p,\dot{p},u,t)$ be the 
drift of Eq.~\eqref{eq:DriftLowerBound}, whereby we emphasize 
the dependence on $\alpha$. 
For any $\alpha_1\le\alpha_2$, and any $s,p,\dot{p},u,t$, we have
$\cL(s;p,\dot{p},u,t,\alpha_1)\le \cL(s;p,\dot{p},u,t,\alpha_2)$.
Therefore, it follows from Eq.~\eqref{eq:ulb} that 
$u_{\slb}(t;p,\alpha_1)\le _{\slb}(t;p,\alpha_2)$,
whence $u_{\salg,\slb}(\alpha_1,\xi)\le u_{\salg,\slb}(\alpha_2,\xi)$.

%
%***********************************************
%
\section{Interpolating with neural networks: An empirical comparison}
\label{sec:Empirical}

Throughout the paper, we studied the problem of solving the system of nonlinear
equations $\bF(\bW)=\bzero$ with respect to unknowns $\bW\in\S^{d-1}$,  when $\bF$ is a Gaussian 
process. (Throughout this section, the unknowns will be denoted by $\bW$ to match the 
applied literature.)

It is natural to wonder whether the theory developed in this setting can provide
any guidance towards understanding more complex cases in which the 
functions $F_1,\dots, F_n$ are random but non-Gaussian.

In this section we consider the interpolation problem introduced in 
Section \ref{sec:Introduction} and describe a simulation study comparing the Gaussian theory
of the previous section to empirical results. More precisely,
we consider the problem of interpolating random data $(y_i,\bz_i)\in\{+1,-1\}\times\reals^D$
using a two-layer neural network with weights $\bW$.

The setup for these simulations is defined in Section \ref{sec:Setup}.
Section \ref{sec:Sensitivity} investigates the robustness of our results with respect to
various choices in the simulations. Finally, in Section \ref{sec:Comparison}
we compare empirical results in the interpolation problem with the Gaussian theory of the previous
sections. Namely, we compare empirical results with theoretical ones within a 
Gaussian model that approximately matches the covariance structure of the 
interpolation model.

In general, we observe a gap between the Gaussian theory and the neural network model. 
However, the two appear to be in rough qualitative agreement and, in certain cases,
surprisingly close to each other. 

\subsection{Setup and definitions}
\label{sec:Setup}

We are interested in the ability of large neural networks
to interpolate completely unstructured (pure noise) data, a phenomenon that
has attracted considerable attention over the last few years 
\cite{zhang2021understanding,belkin2019reconciling,bartlett_montanari_rakhlin_2021}. 
In order to capture the essence of this problem in a simple setting, 
we assume to be given i.i.d. data $\{(y_i,\bz_i)\}_{i\le n}$
with $y_i\sim \Unif(\{+1,-1\})$ independent of $\bz_i\sim\normal(\bzero,\id_D)$.
We consider a two-layer neural network with $D$ inputs and $m$ hidden units:
\begin{align}
f(\bz;\bW)= \frac{a}{\sqrt{m}}\sum_{j=1}^ms_i\, \sigma(\<\bw_j,\bz\>)\, \, .\label{eq:TwoLayer}
\end{align}
Here $a\in\reals_{>0}$ is a scale parameter that will be fixed independently
of the data. The signs $s_1,\dots,s_{m}\in\{+1,-1\}$ are also fixed
(independent of the data) and uniformly random subject to 
$\#\{i\in [m]:\; s_i=+1\} = \#\{i\in [m]:\; s_i=-1\}=m/2$ (for simplicity we assume $m$ even).
The weights $\bW=(\bw_1,\dots,\bw_m)^{\sT}\in\reals^{m\times D}$,
are instead fit to the data as to 
minimize the empirical risk
\begin{align}
\hR_n(\bW) &:=\frac{1}{2n}\sum_{i=1}^n\big(y_i-f(\bz_i;\bW)\big)^2\, ,
\end{align}
subject to the norm constraints
\begin{align}
\|\bW\|_F^2 = \sum_{i=1}^m\|\bw_i\|_2^2 &\le m\, .\label{eq:WeigthConstraint}
\end{align}
Note that $n\cdot \hR_n(\bW)$ is nothing but the Hamiltonian $H(\bx)= \|\bF(\bx)\|_2^2/2$
of the previous pages, with the replacement $\bx = \bW/\sqrt{m}$ and $F_i(\bx) = y_i-f(\bz_i;\bW)$.
We thus identify the dimension of the optimization problem as $d=mD$.

We attempt to minimize the cost $\hR_n(\bW)$ using stochastic gradient descent 
(SGD). Below are some specifics of our experiments.

\paragraph{Activation function.} We use the standard ELU activation: 
\begin{align}
\sigma(x) = \begin{cases}
x & \mbox{ if $x\ge 0$,}\\
e^{x}-1& \mbox{ if $x< 0$.}\\
\end{cases}
\end{align}

\paragraph{Optimization algorithm.} We use SGD with batch size $|B(k)|=4$:
\begin{align}
\tilde{\bW}^{k+1} &= \bW^k-\frac{\eta_k}{2}\sum_{i\in B(k)} \nabla_{\bw}\big(y_i-f(\bz_i;\bW^k)\big)^2\, ,\\
\bW^{k+1}& = \Big(1\wedge \frac{\sqrt{m}}{\|\tilde{\bW}^{k+1}\|_F}\Big)\tilde{\bW}^{k+1} \, ,\label{eq:Proj}
\end{align}
and stepsize $\eta_k ={\rm lr}/(1+E(k))^{1/2}$, where $E(k)$ is the epoch index. We
will vary the learning rate $\lr$ and number of epochs.

In order to satisfy the $\ell_2$ constraint \eqref{eq:WeigthConstraint},
we project back onto this set at each iteration, as per Eq.~\eqref{eq:Proj}.

\paragraph{Initialization.} We initialize the weights to 
$\bW\sim\normal(0,\eps^2\id_{mD}/D)$, and use  $\eps=0.03$ in simulations.
This corresponds to initializing close to the center of the ball \eqref{eq:WeigthConstraint},
since $\|\bW\|_F/\sqrt{m} = \eps +o_P(1)$.  

\subsection{Dependence on the simulations parameters}
\label{sec:Sensitivity}

\begin{figure}
\includegraphics[width=0.5\linewidth]{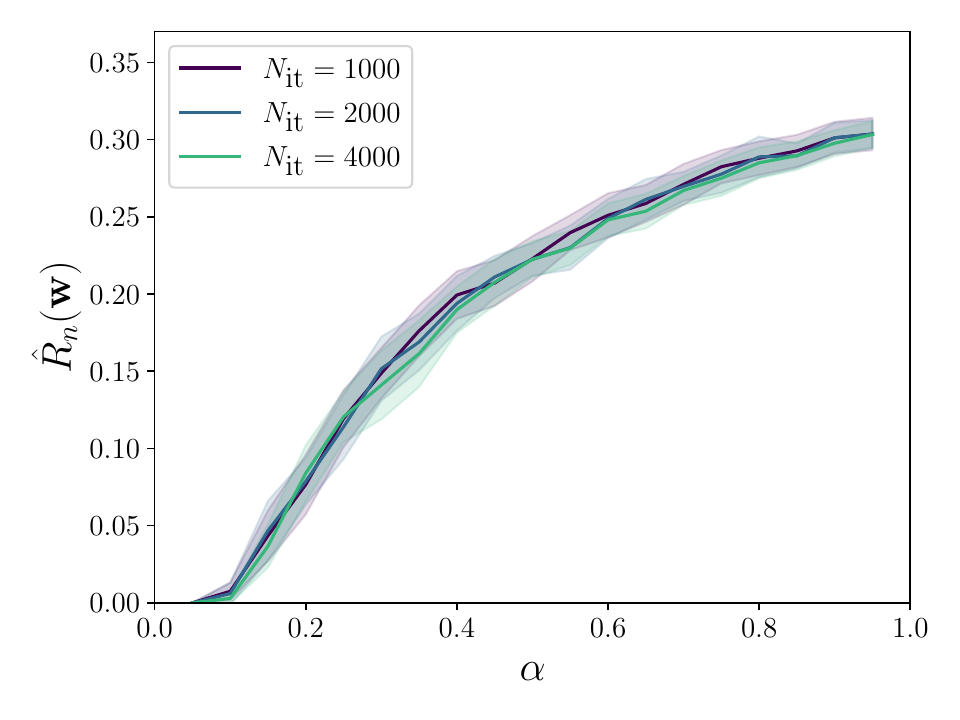}
\includegraphics[width=0.5\linewidth]{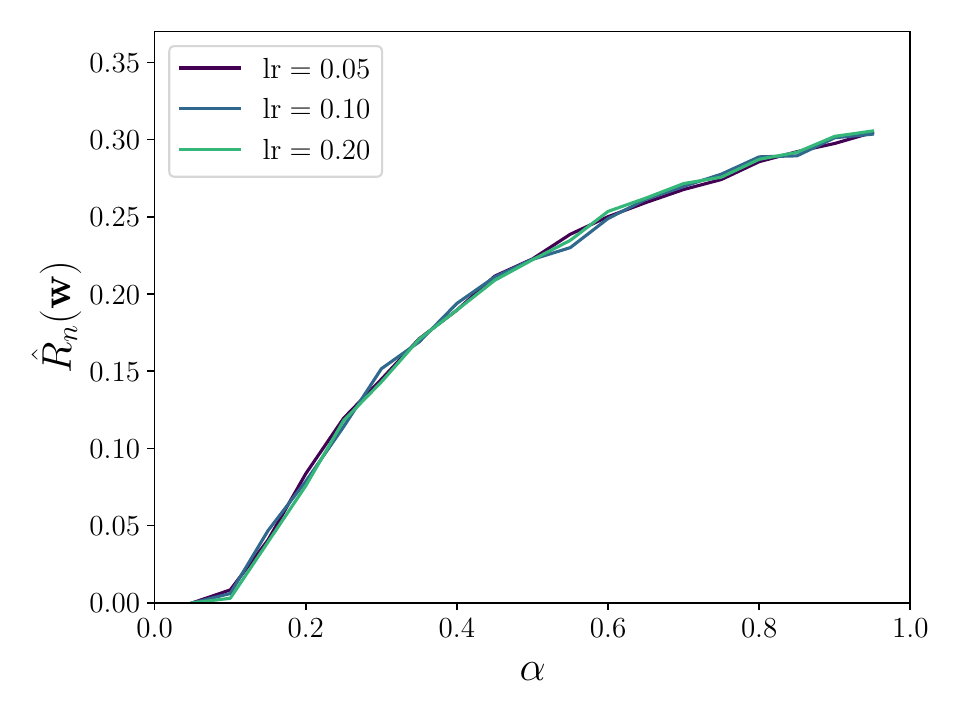}
\caption{Fitting random binary labels using a $2$-layer ELU network \eqref{eq:TwoLayer}, 
using SGD and weights
bounded as per Eq.~\eqref{eq:WeigthConstraint}, as a function of the number
of samples per parameter $n/(mD)$. Left: Dependence of training error on the number of 
epochs $N_{\text{it}}$. Right: Dependence on the learning rate $\lr$.}
\label{fig:Energy-ELU-nit-lr}
\end{figure}

\begin{figure}
\begin{center}
\includegraphics[width=0.5\linewidth]{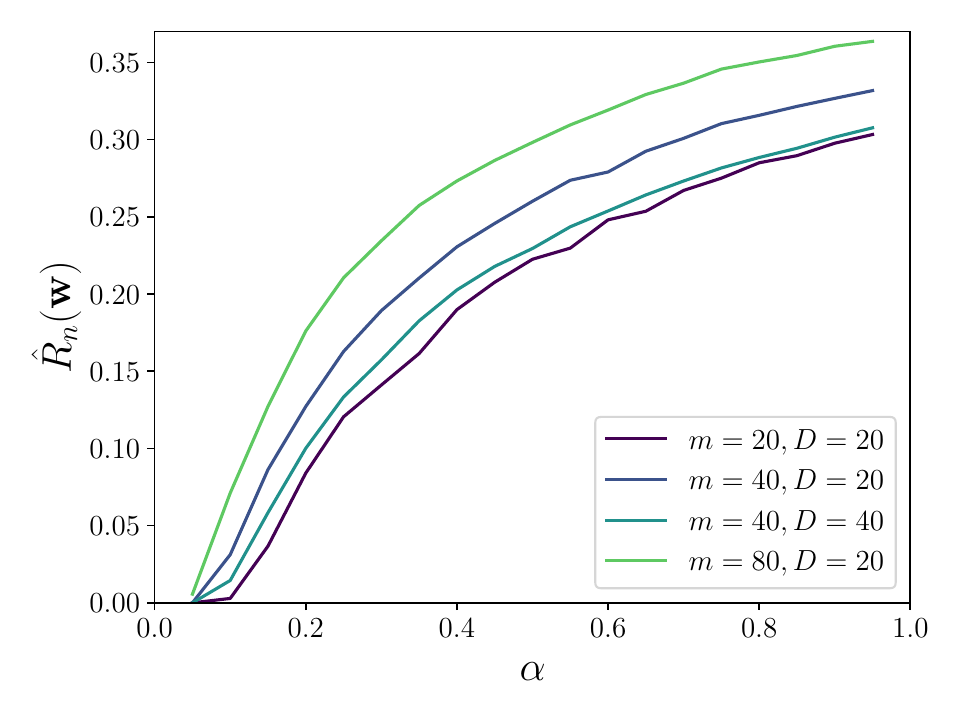}
\end{center}
\caption{Same as in Fig.~\ref{fig:Energy-ELU-nit-lr}. Training error as a function of
 the number of samples per parameter $n/(mD)$. Here we vary the input dimension $D$ 
 and number of neurons $m$.}
\label{fig:Energy-ELU-md}
\end{figure}

In Figures \ref{fig:Energy-ELU-nit-lr}, \ref{fig:Energy-ELU-md} we investigated 
the dependence of training error (energy) achieved  over various parameters of the model and training algorithm.
All of these figures are obtained by setting $a=1$, 
and report the average training error over $20$ realizations of the data,
and independent runs of the algorithm. When present, bands represent an interval of 
plus/minus one standard deviation over these $20$ realizations.

In more detail, we consider the following settings:
\begin{itemize}
\item Figure \ref{fig:Energy-ELU-nit-lr}, left frame: we use $m=D=20$, $\lr =0.1$, and increase 
the number of epochs $N_{\text{it}}$ from $1000$ to $4000$.
\item Figure \ref{fig:Energy-ELU-nit-lr}, right frame:  we use $m=D=20$, $N_{\text{it}}=2000$,
and vary the learning rate $\lr$ between $0.05$ and $0.20$.
\end{itemize}
Our theory does not capture dependence on learning rate or number of iterations. In this respect,
the empirical observation that this dependence is mild in the current setting is reassuring.

Next, the Gaussian theory only depends on the ratio $\alpha=n/d$,
where $d$ is the number of decision variables. In the current setting $d=mD$,
and is therefore interesting to investigate the dependence on $m,D$ while keeping
$\alpha = n/(mD)$ fixed:
\begin{itemize}
\item In Figure \ref{fig:Energy-ELU-md} we use   $\lr =0.1$,  $N_{\text{it}}=4000$, 
and vary the number of neurons $m$ and input dimension $D$.
\end{itemize}
We observe that the training error achieved \emph{does} depend on the architecture,
and most noticeably on the `aspect ratio' $m/D$. On the other hand, this 
dependence is not as dramatic as one might expect. A fourfold increase in aspect 
ratio only changes the training error by $20\%$.

\subsection{Comparison with Gaussian theory}
\label{sec:Comparison}

\begin{figure}
\begin{center}
\includegraphics[width=0.58\linewidth]{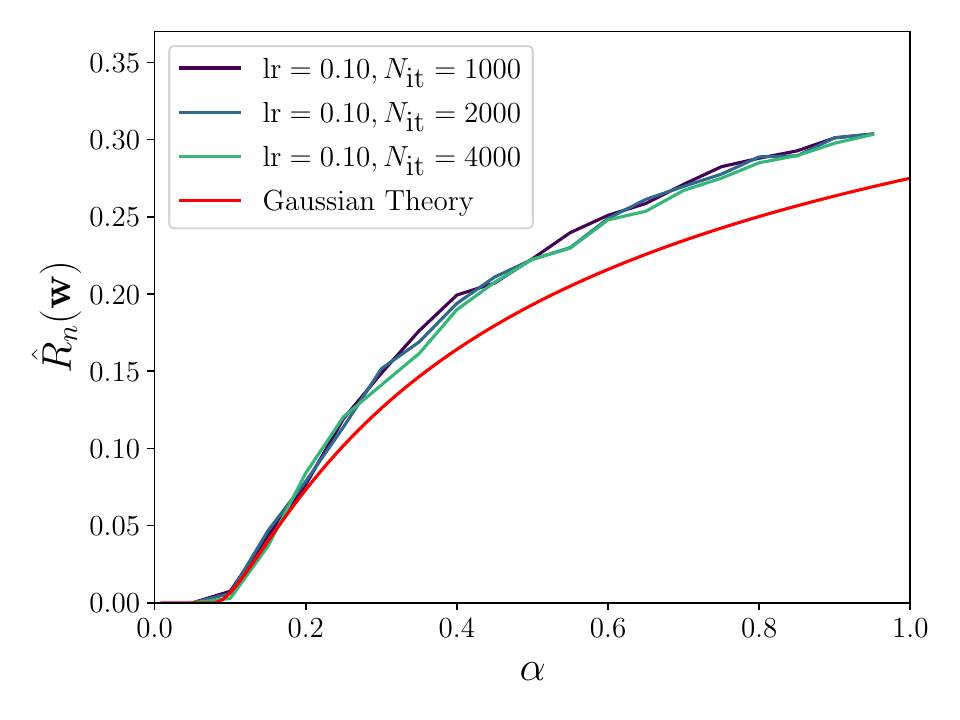}\vspace{-0.05cm}\\
\includegraphics[width=0.58\linewidth]{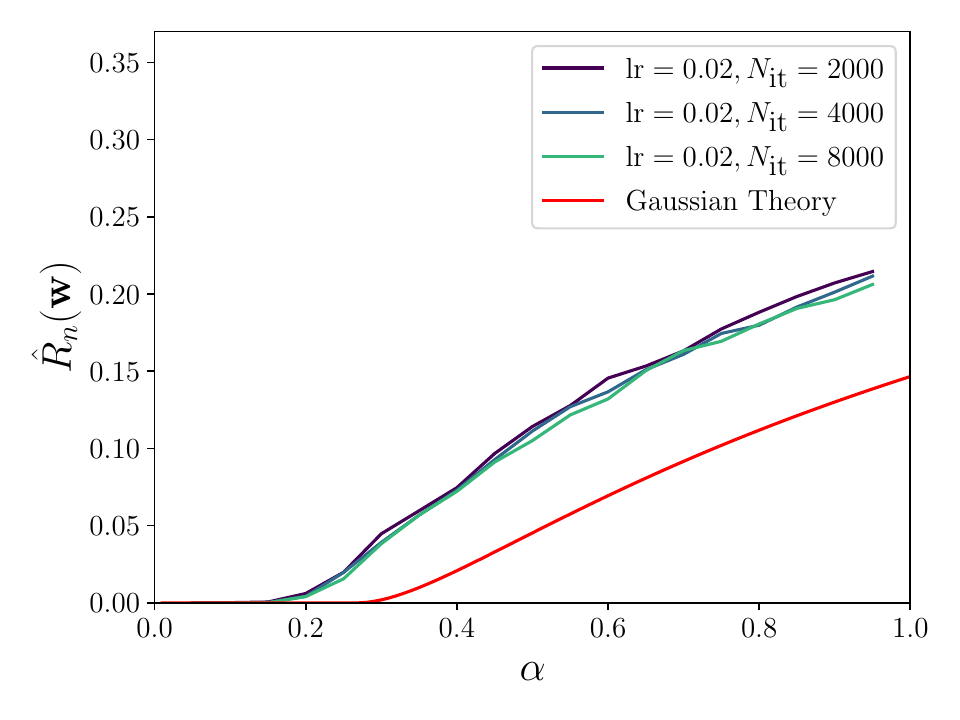}\vspace{-0.05cm}\\
\includegraphics[width=0.58\linewidth]{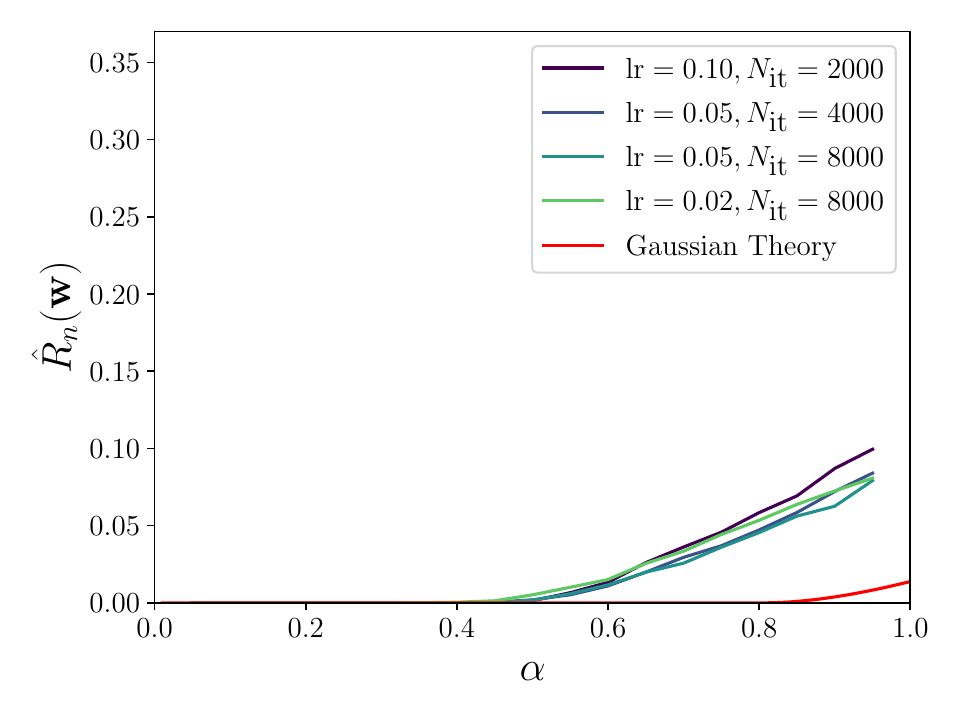}\vspace{-0.05cm}
\end{center}
\caption{Training error od SGD as a function of
 the number of samples per parameter $\alpha=n/(mD)$. From top to bottom: 
 $a\in \{1,2,5\}$. Red lines report the theoretical prediction of Theorem \ref{thm:MainAchievability}
 for the optimal `two-phase' algorithm.
 }
\label{fig:Energy-ELU-Theory}
\end{figure}

We next compare simulation results with the analytical prediction for a Gaussian model 
with matching covariance (see Section \ref{sec:Covariance} for a description of
this prediction). More precisely we compute the prediction
for the two-phase algorithms of 
Section \ref{sec:AlgoDescr}. We expect this to correspond 
to the optimal energy achieved (asymptotically as $n,d\to\infty$) by a broad class of 
efficient algorithms. Hence should provide an upper bound on the energy achieved by SGD
(and an upper bound on the threshold.)

 In Figure \ref{fig:Energy-ELU-Theory} we consider number of neurons $m=20$, and input dimension
$D=20$, and vary $N_{\text{it}}$, $\lr$ as indicated in legends.
We change $a\in\{1,2,5\}$, which of course impacts the Gaussian prediction.
The agreement is surprisingly good for $a=1$, and worsens for $a\in \{2, 5\}$. 
Despite the worse agreement for $a\gtrsim 2$, the Gaussian theory still has the correct 
qualitative behavior, and appears to be a good starting point to analyze the more complex model
\eqref{eq:TwoLayer}.

\subsection{Covariance matching}
\label{sec:Covariance}
We conclude this section by discussing the covariance matching procedure used  
to compute the theoretical predictions in Fig.~\ref{fig:Energy-ELU-Theory}. 
We begin by observing that, \emph{in general}, we cannot expect the Gaussian model to capture the
behavior of the system of nonlinear equations $y_i-f(\bz_i;\bW)=0$, 
with $f(\,\cdot\,;\bW)$ given by Eq.~\eqref{eq:TwoLayer}. 

To see this, consider the
case of a linear activation $\sigma(x) = x$. Then the system of equations reads 
\begin{align}
\Big(\frac{a}{\sqrt{m}}\sum_{j=1}^ms_j\,\bw_j\Big)^{\sT}\bz_i =y_i\, \;\;\;\;\;\forall i\le n\,.
\end{align}
In other words, the model depends on the weight vectors only through the 
$D$-dimensional projection $\bu := am^{-1/2} \sum_{j=1}^ms_j\,\bw_j$. In particular, a
solution exists if and only if $D\ge n$  (for generic data $(\bz_i)_{i\le n}$),
and the total number of parameters is irrelevant. 

In order to take into account the fact that the linear component of $\sigma$
has a much smaller number of degrees of freedom than in the corresponding Gaussian model,
we project out the linear component of activations before matching covariances.
Namely, we define 
$\sigma_{\#}(x) := \sigma(x) -\sigma_1\cdot x$, where $\sigma_1:= \E_{G\sim\normal(0,1)}[G\sigma(G)]$
is the linear coefficient in the Hermite expansion of $\sigma$, and
\begin{align}
f_{\#}(\bz;\bW)
&= \frac{a}{\sqrt{m}}\sum_{j=1}^ms_j\, \sigma_{\#}(\<\bw_j,\bz\>)\, ,\\
F_i(\bW) &:= y_i- f_{\#}(\bz;\bW)\, .
\end{align}
We now compute the covariance of the process $F_i(\bw)$ with respect to the random data $\bz$.
Letting $C_F(\bW;\bW'):= \E_{\bz_i,y_i}[F_i(\bW)F_i(\bW')]$, we have
\begin{align}
C_F(\bW;\bW')& = 1+\E_{\bz}\big\{f_{\#}(\bz;\bW)f_{\#}(\bz;\bW')\big\}\\
&= 1+\frac{a^2}{m}\sum_{j,l=1}^ms_js_l\, 
\E_{\bz}\sigma_{\#}(\<\bw_j,\bz\>)\sigma_{\#}(\<\bw'_l,\bz\>)\big\}\\
& = 1+\frac{a^2}{m}\sum_{j,k=1}^ms_js_l\, \widehat{K}\big(\<\bw_j,\bw'_l\>,\|\bw_j\|^2_2,\|\bw'_l\|^2_2\big)\, ,
\label{eq:RealCovariance}
\end{align}
where we introduced the kernel  $\widehat{K}$
defined by
\begin{align}
\widehat{K}(r_{12},r_{11},r_{22})&:=\E\big[\sigma_{\#}(G_1)\sigma_{\#}(G_2)\big]\, , \\
(G_1,G_2)&\sim\normal(\bzero,\bR) \, \;\; 
\bR:= \left(\begin{matrix}
r_{11} & r_{12}\\
r_{12} & r_{22}
\end{matrix}\right)\, .
\end{align}
We observe that the covariance 
\eqref{eq:RealCovariance} is not invariant under the orthogonal group $\cO(mD)$
but only under ${\mathfrak S}_{m/2}\otimes {\mathfrak S}_{m/2}\otimes \cO(D)$ 
(with ${\mathfrak S}_{m/2}$ denoting the group of permutations in the space of neurons, 
for neurons with same $s_j$, and  $\cO(D)$ denoting 
rotations in the space of weights).
In order to replace it by an orthogonally invariant covariance (depending
only on $\<\bW,\bW'\>$), we make the following approximations:
\begin{enumerate}
\item Since the signs $(s_j)_{ j\le m}$ in Eq.~\eqref{eq:RealCovariance}
are random, we keep only diagonal terms in the sum, and thus replace $C_F(\bW;\bW')$
by $C^{(1)}_F(\bW;\bW')$ whereby
\begin{align*}
C^{(1)}_F(\bW;\bW')&= 
1+\frac{a^2}{m}\sum_{j=1}^m \widehat{K}\big(\<\bw_j,\bw'_j\>,\|\bw_j\|^2_2,\|\bw_j\|^2_2\big)\, .
%\label{eq:RealCovariance}
\end{align*}
Notice that this is accurate at a fixed $\bW$, $\bW'$ but not uniformly over $\bW$, $\bW'$.
\item We approximate the value of $C^{(1)}_F(\bW;\bW')$ by the one taken by $C^{(1)}_F(\bW;\bW')$
for a `typical' pair $\bW$, $\bW'$ with a given inner product $\<\bW,\bW'\>$.
More formally define the set
\begin{align}
\cW_{m,n}(q):=\Big\{\bW,\bW'\in\reals^{m\times D}:\;\; \|\bW\|_F^2=m, \|\bW'\|_F^2=m,
\<\bW,\bW'\> = mq\Big\}\, .
\end{align}
If we draw $(\bW,\bW')\sim \Unif(\cW_{m,n}(q))$, then for a fixed $j\in[m]$,
$\|\bw_j\|_2=1+o_P(1)$, $\|\bw'_j\|_2=1+o_P(1)$  and $\<\bw_j,\bw_j'\>=q+o_P(1)$. As
a consequence, we will have
\begin{align}
C^{(1)}_F(\bW;\bW')&= 
1+a^2\, \widehat{K}\big(q,1,1\big)+o_P(1)\, .
\end{align}
\end{enumerate}

Summarizing, and simplifying notations, the above approximation suggests to use 
a  Gaussian process $\bF^G$ in $d=mD$ dimensions with  
\begin{align}
&\E[F^G_{i}(\bW)F^G_{j}(\bW')] = \xi\big(\<\bW,\bW'\>/m\big)\, ,\\
&\xi(q)  := 1+ a^2\, K_0(q) := 1+a^2[K(q)-\sigma_1^2 q]\, ,\\
&K(q)  :=\E\big[\sigma(G_1)\sigma(G_2)\big]\, , \;\;\;\;\;
(G_1,G_2)\sim\normal\left(\bzero, \left(\begin{matrix}
1 & q\\
q & 1
\end{matrix}\right)\right)\, .
\end{align}
This is exactly the function $\xi(\,\cdot\,)$ used to compare simulations and theory
in the previous sections.

We conclude this section with a warning.  \emph{We do not expect
the theoretical predictions based on the covariance matching here to be 
asymptotically exact as $m,D, n\to\infty$.} In particular, as emphasized
above, the actual covariance structure is not invariant under $\cO(mD)$. 

Nevertheless, the rough agreement between theory and empirical results 
suggests that the Gaussian model (possibly with a more complex covariance structure) might be a 
useful starting point.

\end{document}